\documentclass[12pt]{amsart}

\usepackage{fullpage,url,amssymb,enumerate,colonequals}
\usepackage{mathrsfs}
\usepackage[section]{placeins}
\usepackage{MnSymbol}
\usepackage{extarrows}
\usepackage{lscape}
\usepackage[all,cmtip]{xy}

\usepackage[OT2,T1]{fontenc}
\usepackage{color}

\usepackage[
        colorlinks, citecolor=darkgreen,
        backref,
        pdfauthor={Nuno Freitas}, 
]{hyperref}

\usepackage{comment}
\usepackage{multirow}

\newcommand{\C}{\mathbb{C}}
\newcommand{\F}{\mathbb{F}}
\newcommand{\Fbar}{{\overline{\F}}}

\newcommand{\PP}{\mathbb{P}}
\newcommand{\Q}{\mathbb{Q}}

\newcommand{\Z}{\mathbb{Z}}
\newcommand{\Qbar}{{\overline{\Q}}}

\newcommand{\Ebar}{{\overline{E}}}

\newcommand{\rhobar}{{\overline{\rho}}}

\newcommand{\ff}{\mathfrak{f}}

\newcommand{\calL}{\mathcal{L}}

\newcommand{\calO}{\mathcal{O}}

\newcommand{\vv}{\upsilon}

\newcommand{\fp}{\mathfrak{p}}

\DeclareMathOperator{\Aut}{Aut}

\DeclareMathOperator{\Dic}{Dic}

\DeclareMathOperator{\End}{End}

\DeclareMathOperator{\Frob}{Frob}
\DeclareMathOperator{\Gal}{Gal}

\DeclareMathOperator{\Norm}{Norm}

\DeclareMathOperator{\tr}{tr}

\newcommand{\unr}{{\operatorname{unr}}}

\newcommand{\GL}{\operatorname{GL}}

\newcommand{\SL}{\operatorname{SL}}

\numberwithin{equation}{section}

\newtheorem{theorem}{Theorem}
\newtheorem*{theorem*}{Theorem}
\newtheorem{lemma}{Lemma}
\newtheorem{corollary}{Corollary}
\newtheorem{proposition}{Proposition}

\theoremstyle{definition}
\newtheorem{definition}[equation]{Definition}

\newtheorem{conjecture}[equation]{Conjecture}
\newtheorem{example}[equation]{Example}

\theoremstyle{remark}
\newtheorem{remark}[equation]{Remark}

\newenvironment{psmallmatrix}
  {\left(\begin{smallmatrix}}
  {\end{smallmatrix}\right)}

\definecolor{darkgreen}{rgb}{0,0.5,0}

\DeclareRobustCommand{\SkipTocEntry}[5]{}

\begin{document}

\title{On the symplectic type of isomorphisms of the $p$-torsion of elliptic curves}

\author{Nuno Freitas}
\address{
Instituto de Ciencias Matem\'aticas, CSIC, 
Calle Nicol\'as Cabrera
13--15, 28049 Madrid, Spain}
\email{nuno.freitas@icmat.es}

\author{Alain Kraus}
\address{Sorbonne Universit\'e,
Institut de Math\'ematiques de Jussieu - Paris Rive Gauche,
UMR 7586 CNRS - Paris Diderot,
4 Place Jussieu, 75005 Paris, 
France}
\email{alain.kraus@imj-prg.fr}

\thanks{Freitas is
supported by the
European Union's Horizon 2020 research and innovation programme under the Marie Sk\l{l}odowska-Curie grant 
agreement No.\ 747808 and the grant {\it Proyecto RSME-FBBVA $2015$ Jos\'e Luis Rubio de Francia}.}

\keywords{Elliptic curves, torsion points, Weil pairing, symplectic isomorphism, local fields}

\subjclass[2010]{Primary 11G05, 11G07, Secondary 11D41}

\date{\today}

\begin{abstract} Let $p \geq 3$ be a prime. Let $E/\Q$ and $E'/\Q$ be elliptic curves with isomorphic $p$-torsion modules $E[p]$ and $E'[p]$. Assume further that either 
(i) every $G_\Q$-modules isomorphism $\phi : E[p] \to E'[p]$  
admits a multiple $\lambda \cdot \phi$ with $\lambda \in \F_p^\times$
preserving the Weil pairing; or 
(ii) no $G_\Q$-isomorphism $\phi : E[p] \to E'[p]$ preserves the Weil pairing.
This paper considers the problem 
of deciding if we are in case (i) or (ii).

Our approach is to consider the problem locally at a prime $\ell \neq p$.
Firstly, we determine the
primes $\ell$ for which the local curves $E/\Q_\ell$ and $E'/\Q_\ell$ contain enough information 
to decide between (i) or (ii). Secondly, we establish a collection of criteria, in terms of the standard invariants associated to minimal Weierstrass models of $E/\Q_\ell$ and $E'/\Q_\ell$, to decide between (i) and (ii). We show that our results give a complete solution to the problem by local methods away from~$p$.

We apply our methods to show the non-existence of rational points on certain hyperelliptic 
curves of the form $y^2 = x^p - \ell$ and $y^2 = x^p - 2\ell$ where $\ell$ is a prime; 
we also give incremental results on the Fermat equation $x^2 + y^3 = z^p$.
As a different application, we discuss variants of a question raised by Mazur 
concerning the existence of symplectic isomorphisms 
between the $p$-torsion of two non-isogenous elliptic curves 
defined over $\Q$.  

\end{abstract}

\maketitle

{
 \setlength{\parskip}{0.5ex}
  \hypersetup{linkcolor=black}
  \tableofcontents
}

{\large \part{Motivation and Results}}

\section{Introduction}

Let $p$ be an odd prime. Let $E$ and $E'$ be elliptic curves over $\Q$ and write 
$E[p]$ and $E'[p]$ for their $p$-torsion modules. Write $G_\Q = \Gal(\Qbar/\Q)$ for 
the absolute Galois group of $\Q$.

Let $\phi : E[p] \to E'[p]$ be a $G_\Q$-modules isomorphism;
hence there is an element $d(\phi) \in \F_p^\times$ such that,
for all $P, Q \in E[p]$, the Weil pairings satisfy
$e_{E',p}(\phi(P), \phi(Q)) = e_{E,p}(P, Q)^{d(\phi)}$.
We say that $\phi$ is a {\it symplectic isomorphism} or an
{\it anti-symplectic isomorphism}
if $d(\phi)$ is a square or a non-square modulo~$p$, 
respectively.

We note that, for isomorphic $E[p]$ and~$E'[p]$,
it is possible that isomorphisms with both symplectic types exist;
this occurs if and only if $E[p]$ admits an anti-symplectic automorphism. It is easy to find such examples for $p \leq 5$ (see Example~\ref{Ex:mod5} below), but in fact no such example exists
over~$\Q$ when $p \geq 7$  
(cf. Proposition~\ref{P:typeTriple}) or $p \geq 3$
and the mod~$p$ representation 
$\rhobar_{E,p} : G_\Q \to \GL_2(\F_p)$ attached to~$E$ 
is irreducible (cf. Corollary~\ref{C:problemA}).

Let us briefly mention our main sources of Galois isomorphisms $\phi : E[p] \to E'[p]$. 
Any $\Q$-isogeny $h \; \colon \; E \to E'$ of degree~$n$ coprime to~$p$ restricts
to an isomorphism of $G_\Q$-modules $\phi \; \colon \; E[p] \to E'[p]$. 
In the case of $E$ and $E'$ being non-isogenous curves it is an important conjecture of 
Frey and Mazur that there exists an absolute constant~$C$ such that if 
a $G_\Q$-isomorphism $\phi \; \colon \; E[p] \to E'[p]$ exists then $p < C$. 
Using Cremona's database \cite{lmfdb} together with \cite[Proposition~4]{KO} 
we can find many pairs of non-isogenous curves with isomorphic $p$-torsion 
for $p \leq 17$. In spite of the Frey-Mazur conjecture, we are often interested in considering 
$\phi \; \colon \; E[p] \to E'[p]$ where $E'$ is a fixed elliptic curve, but the existence of the curve~$E$ is 
hypothetical and the prime~$p$ is allowed to grow. This latter setting is of great interest for us 
as it arises in the study of Diophantine equations, as is explained 
in Section~\ref{S:modularmethod}.

We consider triples $(E,E',p)$ where $E/\Q$ and $E'/\Q$ are elliptic curves with 
isomorphic $p$-torsion and such 
that the $G_\Q$-modules 
isomorphisms $\phi :  E[p] \rightarrow E'[p]$ are either all symplectic or all anti-symplectic. 
In this case, we will say that the {\it symplectic type} of~$(E,E',p)$
is respectively symplectic or anti-symplectic. From the discussion above, it follows that, most of the time, when an isomorphism $\phi : E[p] \simeq E'[p]$ exists there is only one possible symplectic type for any such~$\phi$ (for the same $p$), that is
the symplectic type of~$(E,E',p)$ is well defined.

In this work we consider the following problem.

\noindent {\bf Problem A:} 
Given a triple $(E,E',p)$ as above, how do we determine its symplectic type?

If a triple $(E,E',p)$ arises from an isogeny $h \; \colon \; E \to E'$ of degree~$n$ 
not divisible by~$p$, then $d(h|_{E[p]}) = n$ and the symplectic type of $(E,E',p)$ is 
symplectic if $n$ is a square mod~$p$ and anti-symplectic otherwise (Corollary~\ref{C:isogeny}). 
Thus Problem A has an easy solution in this case. 

For a general triple $(E, E', p)$ as in Problem A, in principle,  one could compute
the $p$-torsion fields of $E$ and $E'$, consider the Galois action on 
$E[p]$ and $E'[p]$ and check if they are symplectically or
anti-symplectically isomorphic. However, the degree of the $p$-torsion fields grows very fast
with~$p$ making this method not practical already for $p=7$. 
One could work instead with the $p$-torsion field locally at some prime $\ell$ (assuming we 
proved this will give the correct symplectic type) but this also becomes impractical 
very fast. Another issue with this approach is that it cannot be applied when one of 
the curves is not concrete or the value of $p$ is varying.
Thus, we are interested in having practical methods to determine the 
symplectic type of~($E$, $E'$, $p$) using as little information 
independent of~$p$ as possible. 
We will develop a variety of such methods by working locally. 

Let $\ell \neq p$ be a prime. 
Let $E/\Q_\ell$ and $E'/\Q_\ell$ 
be elliptic curves with isomorphic $p$-torsion and
such that the $\Gal(\Qbar_\ell/ \Q_\ell)$-isomorphisms $\phi : (E/\Q_\ell)[p] \to (E'/\Q_\ell)[p]$ are either all symplectic or all anti-symplectic.\footnote{The definition of a symplectic or anti-symplectic $\Gal(\Qbar_\ell/ \Q_\ell)$-modules isomorphisms is analogous to the one over~$\Q$ given above; see section~\ref{S:notation} for details.}
In this situation, we say that 
a {\bf (local) symplectic criterion exists}
and we call {\bf (local) symplectic criterion} 
to a method/theorem allowing 
to determine the symplectic 
type of one (hence all) such isomorphism~$\phi$.

There are already symplectic criteria available in the literature.
Indeed, the first criterion was established in 1992 by 
the second author with Oesterl\'e \cite{KO} 
and is applicable 
when $E$ and~$E'$
have a common prime~$\ell$ of multiplicative reduction (see Theorem~\ref{T:potMult}). 
Recently, the first author in \cite{F33p} and 
jointly with Naskr{\k e}cki and Stoll in \cite{FNS23n} 
proved other symplectic criteria in some cases of potentially good reduction.

The main goal of this paper is to optimize some of the existing criteria and establish new 
ones in all the remaining cases in which a symplectic criterion can possibly exist. These criteria can then be applied to Problem A as follows. 

Let $(E,E',p)$ be as in Problem~A. We have 
$(E/\Q_\ell)[p] \simeq (E'/\Q_\ell)[p]$ as $\Gal(\Qbar_\ell/ \Q_\ell)$-modules for all primes~$\ell$.
We let
$\calL_{(E,E',p)}$ be the set of primes $\ell \ne p$ for which a 
symplectic criterion exists. 

As we shall explain below,
the symplectic type of~$(E,E',p)$ 
is encoded in the information of
the local curves $E/\Q_\ell$ and
$E'/\Q_\ell$ for any
$\ell \in \calL_{(E,E',p)}$.
More precisely,
the methods in this paper will allow us to do the following. 
\begin{enumerate}
 \item Given a triple $(E,E',p)$ as in Problem~A we completely
 determine the set $\calL_{(E,E',p)}$.
 \item For each $\ell \in \calL_{(E,E',p)}$ we have a symplectic criterion
 to determine the symplectic type of $(E,E',p)$ from $p$ and the Weierstrass models
 of $E/\Q_\ell$ and $E'/\Q_\ell$. Often we only need the standard invariants $(c_4, c_6, \Delta_m)$ attached to minimal models of these local curves and congruence conditions on~$p$.
\end{enumerate}
Note that, although to solve Problem A for a particular $(E,E',p)$ it is enough to find one prime~$\ell$ in part (1), 
we obtain the complete list $\calL_{(E,E',p)}$. In particular, 
when $\calL_{(E,E',p)}$ is empty there is no prime $\ell \ne p$
such that local information at $\ell$ allows us to determine the 
symplectic type of $(E,E',p)$ 
(see Proposition~\ref{P:emptyL} for such an example). Therefore, we obtain an answer to Problem A which
is optimal using only standard local information away from~$p$.

We compile in Tables~\ref{Table:CriteriaList} and~~\ref{Table:TwistingLemmas}
the complete list of symplectic criteria 
in terms of the reduction types of the curves; see Section~\ref{S:results} for the full notation and statements.
\begin{small}
 \begin{table}[htb]
\renewcommand{\arraystretch}{1.25}
\begin{tabular}{|c|c|c|c|c|} \hline
{\sc reduction type of} $E$, $E'$ & {\sc prime} $\ell \neq p$& {\sc Extra conditions} & {\sc criteria}
& {\sc Proof in}
\text{\large\strut}\\\hline
good ($e=1$)  & $(\ell/p) = 1$       &  $p \mid \Delta_\ell, \; p \nmid \beta_\ell, \; \Ebar = \Ebar'$ & Theorem~\ref{T:simpleAbelian} & \S \ref{S:thm12} \\
good ($e=1$)  & $(\ell/p) = 1$       &  $p \mid \Delta_\ell, \; p \nmid \beta_\ell$ & Theorem~\ref{T:mainAbelian} & \S \ref{S:genGood} \\
pot. good $e = 3$ &$\ell \equiv 2 \pmod{3}$ &      & Theorems~\ref{T:mainTame3},~\ref{T:main3torsion} & 
\S \ref{S:thm1}, \S \ref{S:thm2} \\
pot. good $e = 3$ &$\ell \equiv 1 \pmod{3}$ &   $p=3$   & Theorem~\ref{T:e=p=3} & \S \ref{S:thm3}\\
pot. good $e = 3$& $\ell = 3$               & $\tilde{\Delta} \equiv 2 \pmod{3}$ & Theorem~\ref{T:mainWild3} & \S \ref{S:thm4} \\
pot. good $e=4$         &$\ell \equiv 3 \pmod{4}$ &                                     & Theorem~\ref{T:mainTame4} & \S \ref{S:thm5} \\
pot. good $e=4$         &$\ell = 2$               &  $\tilde{c}_4 \equiv 5\tilde{\Delta} \pmod{8}$& Theorem~\ref{T:mainWild4} & \S \ref{S:thm6}\\
pot. good $e=8$         &$\ell = 2$               &    & Theorems~\ref{T:mainWilde8},~\ref{T:mainWilde8II} & \cite{FNS23n}, \S \ref{S:pfThm9} \\
pot. good $e=12$        &$\ell = 3$               &     & Theorem~\ref{T:mainWilde12},~\ref{T:mainWilde12II} & \cite{FNS23n}, \S \ref{S:pfThm9} \\
pot. good $e=24$        &$\ell = 2$               &    & Theorem~\ref{T:Wilde24} & \cite{F33p} \\ \hline 
multiplicative &   any             &  $p \nmid \vv_\ell(\Delta_m)$    & Theorem~\ref{T:potMult} & \cite{KO} \\
\hline 
$E$ multiplicative, $e'=3$ & $\ell \neq 3$               &  $p=3$ & Theorem~\ref{T:mixedReduction} & \S\ref{S:mixedReduction} \\
 \hline 
\end{tabular}
\caption{Summary of symplectic criteria. 
We assume $p \geq 3$ and $(E/\Q_\ell)[p] \simeq (E'/\Q_\ell)[p]$ as $\Gal(\Qbar_\ell/ \Q_\ell)$-modules; 
$e$ and $e'$ are the semistability defects of $E$ and $E'$;
the quantities
$\Delta_\ell$, $\beta_\ell$, 
$\tilde{c}_4$, $\tilde{\Delta}$
and $\Delta_m$ are attached to~$E$.}
\label{Table:CriteriaList}
\end{table}
\end{small}
\begin{small}
 \begin{table}[htb]
 \renewcommand{\arraystretch}{1.25}
\begin{tabular}{|c|c|c|c|c|} \hline
\multicolumn{3}{|c|}{{\sc Twisted reduction types for $E$, $E'$}} & {\sc Twisting Lemmas} & {\sc Proof In} \\ \hline
\multicolumn{3}{|c|}{$e=e'=2$ (twist of good)}                         & Lemmas~\ref{L:lemma3},~\ref{L:lemma4} & \S \ref{S:4lemmas} \\
\multicolumn{3}{|c|}{$e=e'=6$ (twist of $e=e'=3$)}                        
& Lemmas~\ref{L:lemma1},~\ref{L:lemma2} & \S \ref{S:4lemmas}\\
\multicolumn{3}{|c|}{$E$, $E'$ pot. multiplicative}                        
  & Lemma~\ref{L:multTwist} & \cite[p.~444]{SilvermanII} \\
\multicolumn{3}{|c|}{$E$ pot. multiplicative, $e'=6$}  & Lemmas~~\ref{L:lemma1},~\ref{L:lemma2},~\ref{L:multTwist} & \\
\hline
\end{tabular}
\caption{Reduction types that reduce to those in Table~\ref{Table:CriteriaList} after a quadratic twist.}
\label{Table:TwistingLemmas}
\end{table}
\end{small}

We make the following remarks about the tables:
\begin{itemize}
 \item Table~\ref{Table:CriteriaList} contains the full list of conditions under which our criteria are applicable. 
 \item The cases in Table~\ref{Table:TwistingLemmas} sometimes can be reduced to those in Table~\ref{Table:CriteriaList} by twisting.
 More precisely, suppose that $(E/\Q_\ell)[p] \simeq (E'/\Q_\ell)[p]$ and a case 
 in Table~\ref{Table:TwistingLemmas} holds. 
 Apply the relevant twisting lemma to get the value~$d$ to twist both curves. If the twisted curves $dE$ and~$dE'$ satisfy any of the conditions in Table~\ref{Table:CriteriaList}, then applying the  criterion in the same line to $dE$ and~$dE'$ gives the desired conclusion (due to Lemma~\ref{L:twistIso}).
 \item When applying the twisting lemmas in the mixed reduction case (last line of Table~\ref{Table:TwistingLemmas}), we can choose to either apply Lemma~\ref{L:multTwist} to~$E$ or (depending on~$\ell$) Lemmas~\ref{L:lemma1},~\ref{L:lemma2} to~$E'$.
\item Up to line 10 of Table~\ref{Table:CriteriaList}, both curves are assumed to have potentially good reduction; on line 11 both have multiplicative reduction; the last line 
contains the only criterion in the case of mixed reduction types.
 \item To keep Table~\ref{Table:CriteriaList} easier to read, conditions are stated only for~$E$ whenever possible; this is allowed by the assumption $(E/\Q_\ell)[p] \simeq (E'/\Q_\ell)[p]$ which implies the same conditions hold for~$E'$. For example, $e=e'$ up to line 10 (see Proposition~\ref{P:TwistsameType}).
\item The hypothesis $(E/\Q_\ell)[p] \simeq (E'/\Q_\ell)[p]$ is superfluous in the 1st line, as it follows from the stronger extra condition $\Ebar = \Ebar'$ on the residual curves.
\item The cases $e=8,24$ occur only for $\ell=2$ and $e=12$ occurs only for $\ell = 3$ 
 (see \cite{Kraus1990}).
\end{itemize}

We summarize the above discussion into one theorem. 
\begin{theorem*}
Let $\ell$ and $p \geq 3$ be 
different primes.
Let $(E,E',p)$ be as in Problem~ A. 

Then~$\ell \in \calL_{(E,E',p)}$ if and only if, after replacing $E$ and $E'$ by twists according to~Table~\ref{Table:TwistingLemmas} if necessary, 
one of the lines of Table~\ref{Table:CriteriaList} is satisfied. Moreover, the symplectic type of $(E,E',p)$ is determined by the
theorems listed under {\sc Criteria} in that same line.
\end{theorem*}
\begin{proof}
This is a consequence of Theorem~\ref{T:tableEquivalence}, the results stated in Section~\ref{S:results} and Theorem~\ref{T:mainAbelian}.
\end{proof}

\subsection*{Motivation} 
Our interest in
Problem A arises from two 
sources; namely, variants of a question of Mazur
and Diophantine equations. 
We will discuss both in detail in Section~\ref{S:motivation} and 
in Part~\ref{Part:applications} we use 
our main results to obtain applications to these topics.

\subsection*{Possible generalizations} 
It is natural to wonder about
the possibility of extending the present work 
to number fields since Problem A clearly makes sense in that setting.
Following the approach in this paper, that would lead us to consider 
symplectic criteria for elliptic curves defined over finite extensions $F/\Q_{\ell}$. 
To our knowledge, the available results in this direction are 
for the case of curves $E/F$ with (potentially)  multiplicative reduction 
(as the proof over~$\Q_\ell$ applies) or in the special case 
of potentially good reduction covered by \cite[Proposition A.2]{HK2002}. 
Nevertheless, according to the results in Section~\ref{S:existence}, one 
knows that criteria will exist in particular when $\rhobar_{E,p}$ has non-abelian 
image. This will be the case, for instance when $E/F$ gets good reduction over extensions with 
non-abelian inertia subgroups; in the case of cyclic inertia, the existence of criteria will be 
less common as it will depend on which roots of unit belong to~$F$. Furthermore, to 
make the criteria explicit would require consider many cases depending on the ramification of 
the extension $F/\Q_\ell$ and its residue field.

\addtocontents{toc}{\SkipTocEntry}
\subsection*{Acknowledgments}
We thank Benjamin Matschke for many helpful discussions. 
We also thank Michael Bennett, Imin Chen, Lassina Demb\'el\'e and Michael Stoll for 
their comments. Finally, we thank the anonymous referees for carefully reading 
the paper and providing many suggestions and comments which significantly improved the exposition.

All the calculations required for this
work were done using {\tt Magma}~\cite{MAGMA}.

\section{A double motivation}
\label{S:motivation}

\subsection{A question of Barry Mazur} 
\label{S:FreyMazur}

In Mazur's remarkable 1978 paper on isogenies, he asked (see \cite[p. 133]{Mazur}) the following question
\begin{quote}
{\em Are there non-isogenous elliptic curves over $\Q$ with symplectically isomorphic $p$-torsion modules for some $p \geq 7$?}
\end{quote}
An affirmative answer was obtained by the second author and Oesterl\'e  in \cite{KO}, 
using the symplectic criterion for the case of multiplicative 
reduction (see Theorem~\ref{T:potMult}).
They showed that the elliptic curves with Cremona labels 7448e1 and 152a1 
have symplectically isomorphic $7$-torsion modules;
with the same criterion they also showed that the $7$-torsion modules of the 
curves 26a1 and 182c1 are anti-symplectically isomorphic. 

Related to the question above, Mazur also posed the problem of finding 
all elliptic curves~$E'/\Q$ having $p$-torsion module symplectically isomorphic 
to that of a fixed elliptic curve $E/\Q$. Similarly, one can formulate the analogous 
problem where we ask for anti-symplectic isomorphisms instead. In both cases, the 
problem can be rephrased as finding rational points on certain twists 
of the modular curve $X(p)$. These twists are denoted $X_E(p)$ and $X^-_E(p)$
in the symplectic and anti-symplectic case, respectively.
The results developed in this work provide, in particular, quick local tests for the possibility 
of $X_E(p)$ and $X^-_E(p)$ having rational points (see 
Section~\ref{S:appMazur} for more details and examples).

In view of the previous discussion, Kani and Schanz \cite{KaniSchanz} studied the geometry of the surfaces that parametrize pairs of elliptic curves with isomorphic $n$-torsion, and conjectured that for $n \leq 12$ there are infinitely many pairs of such curves with different $j$-invariants (see Conjecture~5 and Question~6 in {\it loc. cit.}). 
This conjecture was firstly established for~$n=7$ by the second author~\cite{Kraus1996}; later, jointly with Halberstadt, they described 
infinitely many 6-tuples of non-isogenous curves with symplectically isomorphic $7$-torsion (see~\cite[Proposition~6.3]{HK2003}).
For $n=11$ the conjecture was proved by Kani and Rizzo~\cite{KaniRizzo} but their proof 
does not provide examples; Fisher~\cite{Fisher} gave alternative proofs for
both cases, including a description of infinitely many pairs of non-isogenous elliptic curves 
for both symplectic types of $7$-torsion isomorphisms and infinitely many 
such pairs with symplectically isomorphic $11$-torsion. 
An analogous infinite family with anti-symplectically isomorphic $11$-torsion
is unknown; note that, in this case,
the corresponding parameterizing surface is of general type (see \cite[Theorem~4]{KaniSchanz}). Recently, Fisher~\cite{Fisher13cong} 
showed that there are infinitely many non-trivial pairs of $13$-congruent
elliptic curves, both in the symplectic and anti-symplectic cases.

The question and the discussion above
are closely related to the powerful Frey-Mazur Conjecture
(see also \cite[Section~4.1]{darmon-SerreConjs}).

\begin{conjecture}[Frey--Mazur] There is a constant $C$ such that the following holds. 
If $E$ and $E'$ are elliptic curves over $\Q$ such that $E[p]$ and $E'[p]$ are isomorphic 
as $G_\Q$-modules for some prime $p > C$, then $E$ and $E'$ are $\Q$-isogenous.
\label{C:FM}
\end{conjecture}

A strong form of this conjecture states that $C=19$.
In the recent work of the first author with Cremona~\cite{CremonaF} 
the authors provide global symplectic criteria that apply in cases 
where local methods of this paper may fail. Furthermore, by applying all the methods from {\it loc. cit} together with the local methods in this work, the authors determined the symplectic type of all mod~$p \geq 7$ congruences between elliptic curves present in the LMFDB database, which currently includes all elliptic curves of conductor less than 500000. They have found that while such congruences exist for each $p \leq 17$, there are none for $p \geq 19$ in the database, in line with the strong form of the Frey-Mazur conjecture.

In Section~\ref{S:appMazur} we 
give a few selected examples found in 
the database that illustrate 
the applicability of our local criteria. 

\subsection{Diophantine equations}
\label{S:modularmethod}
Wiles' proof \cite{Wiles} of Fermat's Last Theorem
pioneered a new strategy to attack Diophantine equations, 
now known as {\it the modular method}. 
Nowadays, when applying it to a Diophantine equation 
the most challenging part is often in the very end where we need to obtain a contradiction. 
Indeed, after applying modularity and
level lowering results, one gets an isomorphism
\begin{equation}
  \rhobar_{E,p} \sim \rhobar_{f,\fp}
\label{E:iso}  
\end{equation}
between the (irreducible) mod~$p$ representation of the Frey curve $E$ and
the mod~$\fp$ representation of a newform $f$ with weight 2 and `small' level $N$, 
for some prime $\fp \mid p$ in $\Qbar$. 
In the proof of FLT we have $N=2$ and there are no candidate newforms $f$, giving a contradiction. 
In essentially every other application of the modular method there are possible forms~$f$,
hence we needed to derive a contradiction to \eqref{E:iso} for all such $f$. 
In \cite{HK2002} the second author and Halberstadt realized
that using symplectic criteria may allow to distinguish
the two Galois representations in \eqref{E:iso}, obtaining the 
desired contradiction for certain~$f$. 

We now explain their method, which is known 
as {\it the symplectic argument}. 
Suppose that~$\eqref{E:iso}$ holds with $f$ a newform
corresponding to the isogeny class of an elliptic curve $W$, hence
there is a $G_\Q$-modules isomorphism $\phi : E[p] \to W[p]$.
Since $\rhobar_{E,p}$ is irreducible,
all such $\phi$ have the same symplectic type,
thus it makes sense to consider $(E,W,p)$.
Now suppose we find two primes $\ell_1, \ell_2 \neq p$ for which
we can determine the symplectic type of $(E,W,p)$ using only local 
information at these primes individually.
It is often the case that the primes $\ell_1$ and $\ell_2$
will give compatible symplectic types only for $p$ satisfying certain congruence conditions, 
therefore we get a contradiction to \eqref{E:iso} for
all the values of $p$ not satisfying those conditions.

At the time of the original application of the symplectic argument, 
the only symplectic criterion available was for the case of multiplicative 
reduction (see Theorem~\ref{T:potMult}). The authors used it to show 
that if $A,B,C$ are odd coprime integers,
then there is a set of exponents $p$ with positive Dirichlet
density such that the equation $Ax^p + By^p + Cz^p = 0$ has no solutions satisfying $xyz \ne 0$. 
The same criterion was recently used by the first author and Siksek \cite{FS3} 
to prove Fermat's Last Theorem over $\Q(\sqrt{17})$ for a set of exponents with density $1/2$. 
Even more recently, the first author in \cite{F33p} and 
jointly with Naskr{\k e}cki and Stoll in \cite{FNS23n} 
developed new symplectic criteria 
(see Theorems~\ref{T:mainWilde12} and~\ref{T:mainWilde8})
and applied them to obtain new results on the Fermat equations 
$x^3 + y^3 = z^p$ and $x^2 + y^3 = z^p$. Further applications of the symplectic 
argument were obtained by the authors of this paper in~\cite{FK1}.

It is therefore promising for applications to Diophantine equations to 
extend the local symplectic criteria available and, with
the results in this paper, we will provide a 
complete toolbox of easy to use local 
symplectic criteria (see Section~\ref{S:results}).
Some of our new results have already seen new Diophantine applications;
indeed, in the joint work of the first author with 
Bennett and Bruni \cite{BenBruFre}, 
Theorem~\ref{T:mainTame4}  
was applied to the equation $x^3 + y^3 = qz^p$ with $q \in \Z$.

Moreover, building on results from \cite{FNS23n},  
we will show in Section~\ref{S:2319} how Theorem~\ref{T:simpleAbelian} sometimes applies 
to reduce the number of curves that we need to find rational points in order to solve the 
very challenging Fermat equation $x^2 + y^3 = z^p$.
Furthermore, in Section~\ref{S:hyperelliptic}, we will use Theorems~\ref{T:Wilde24}~and~\ref{T:mainWilde8II} to improve results from~\cite{IK2006}, about the
non-existence of $\Q$-rational points on hyperelliptic curves of 
the form $y^2=x^p-\ell$ and $y^2=x^p-2 \ell$, where $\ell$ is a prime.

\section{Our approach to the problem of determining the symplectic type} 
\label{S:Overview}

We aim to give an exhaustive solution to the problem of determining the 
symplectic type of triples
$(E,E',p)$ by local methods away from $p$. In this section 
we will explain our approach, summarize what is already in the literature
and what are the novelties of this work. 

\subsection{Notation} \label{S:notation} 
We need notation that will allow us to 
work with fields other than $\Q$.

For a field $K$, we write $\overline{K}$ for an algebraic 
closure and $G_K = \Gal(\overline{K}/K)$ for its absolute Galois group.
For $E$ an elliptic curve defined over~$K$, we write $E[p]$ for its $p$-torsion
$G_K$-module and $\rhobar_{E,p} : G_K \to \Aut(E[p])$ for the corresponding
Galois representation. 

For $K$ a field of characteristic zero or a finite field of characteristic~$\neq p$ 
we fix, for all primes~$p$, a primitive $p$-th root of unity $\zeta_p \in \overline{K}$.
We write $e_{E,p}$ for the Weil pairing on~$E[p]$.

We say that a $\F_p$-basis $(P,Q)$ of $E[p]$ is a \emph{symplectic basis} if
$e_{E,p}(P,Q) = \zeta_p$; we say it is an {\it anti-symplectic basis} if 
$e_{E,p}(P,Q) = \zeta_p^{r}$ with $r$ not a square mod~$p$.

Now let $E /K$ and $E'/K$ be two elliptic curves and
$\phi : E[p] \to E'[p]$ be an isomorphism of $G_K$-modules.
From elementary properties of the Weil pairing it follows
that there is an element $d(\phi) \in \F_p^\times$ such that
\begin{equation}
e_{E',p}(\phi(P), \phi(Q)) = e_{E,p}(P, Q)^{d(\phi)} \quad \text{for all $P, Q \in E[p]$.} 
\label{E:dphi}
\end{equation}
Note that for any $a \in \F_p^\times$ we have $d(a\phi) = a^2 d(\phi)$.
So up to scaling~$\phi$, only the class of~$d(\phi)$ modulo squares matters.
We say that $\phi$ is a \textit{symplectic isomorphism} if
$d(\phi)$ is a square in~$\F_p^\times$, and an \textit{anti-symplectic isomorphism}
if $d(\phi)$ is a non-square. We call this the {\it symplectic type} of $\phi$. 
Note that $\phi$ preserves the Weil pairing precisely when $d(\phi) = 1$.

Finally, we say that $E[p]$ and~$E'[p]$ are
\emph{symplectically} (or \emph{anti-symplectically}) \emph{isomorphic},
if there exists a symplectic (or anti-symplectic) isomorphism of~$G_K$-modules between them.
Note that it is possible that $E[p]$ and~$E'[p]$ are both symplectically
and anti-symplectically isomorphic; this will be the case if and only if
$E[p]$ admits an anti-symplectic automorphism (cf. Example~\ref{Ex:mod5} below).

\subsection{A local approach} 
\hfill
\label{S:localApproach}

For each prime $\ell$ we let $G_{\Q_\ell} \subset G_\Q$ be a decomposition subgroup at~$\ell$.

Let $\ell$ and $p \geq 3$ be different primes.   
Let $E$ and $E'$ be elliptic curves over $\Q$ with isomorphic $p$-torsion.
Let $\phi : E[p] \to E'[p]$ be an
isomorphism of $G_\Q$-modules.
We consider the isomorphism of $G_{\Q_\ell}$-modules 
$\phi_\ell : E[p] \rightarrow E'[p]$ obtained 
from $\phi$ by viewing the curves $E$, $E'$ over~$\Q_\ell$.
The symplectic type of $\phi$ is equal to the symplectic type of $\phi_\ell$. 

Assume that, for some prime $\ell$, the following holds:
\begin{itemize}
 \item[(i)] The $G_{\Q_\ell}$-modules $E[p]$ and $E'[p]$ are 
 not simultaneously symplectically and anti-symplectically isomorphic.
 \item[(ii)] We can decide, from simple local information about $E$ and $E'$
 over $\Q_\ell$, which is the symplectic type in (i).
\end{itemize}
Then, using (ii), we determine the symplectic type of any $G_{\Q_\ell}$-isomorphism
between $E[p]$ and~$E'[p]$ which, in particular, is the symplectic 
type of $\phi_\ell$. Thus $\phi$ also has that symplectic type. 
We conclude that given $(E,E',p)$ as in the introduction we can determine its symplectic type
if we can find a prime $\ell$ for which (i) and (ii) hold. 

In view of (i), we recall and emphasize the following idea that is recurrent 
in this paper. Let $E$ and $E'$ be elliptic curves over $\Q_\ell$ with isomorphic $p$-torsion modules for $p \neq \ell$. We say that a {\bf (local) symplectic criterion exists} if all the $G_{\Q_\ell}$-modules isomorphisms $\phi : E[p] \to E'[p]$ have the same symplectic type.
We can now describe the two main steps of our approach:

\begin{enumerate}
 \item[(I)]
 Let $E/\Q_\ell$ and $E'/\Q_\ell$ be elliptic curves 
 with isomorphic $p$-torsion modules.
 For each possible types of reduction of $E, E'$ we classify, in terms of $\ell$, $p$ and 
 standard information on the curves, when a symplectic criterion exists;
 \item[(II)] For each case of the classification obtained in (I), 
  we give a procedure to determine the symplectic type
  from the value of $p$ and standard information on $E/\Q_\ell$ and $E'/\Q_\ell$.
\end{enumerate}
In other words, in part (I) we classify every possible case in which a 
local symplectic criterion at $\ell \neq p$ exists and in part (II) we provide 
such criteria in every case missing in the  literature.

We note that items (1) and (2) explained in the introduction
will follow from (I) and (II). 

\subsection{What is left to be done?} 
\label{S:left?}
We will now explain which criteria are available in the literature
and which are the novelties of this work. First we need some more notation.

Let $p$ and $\ell$ be primes such that $p \geq 3$ and $\ell \ne p$. 
Write $\Q_\ell^{un}$ for the maximal unramified extension of $\Q_\ell$.
Given an elliptic curve $E/\Q_\ell$ with potentially good reduction 
we write $L=\Q_\ell^{\text{un}}(E[p])$ and $\Phi=\Gal(L/\Q_\ell^{un})$.
We call $L$ the {\it inertial field} of $E$. 
The field $L$ is the minimal extension of $\Q_\ell^{\text{un}}$ where 
$E$ obtains good reduction (see \cite{ST1968}). 
We denote by $e = e(E)$ the order of $\Phi$ which is called the {\it semistability defect of} $E$. 
We say that $E/\Q_\ell$ has {\it tame reduction} if 
$e$ is coprime to $\ell$ and we say it has {\it wild reduction} otherwise. In the presence of a second curve $E'/\Q_\ell$ sometimes we will simply
write $e'$ for its semistability defect $e(E')$.

Let $E/\Q_\ell$ have potentially good reduction but not good reduction, so that $e \ne 1$.
It is well known (see \cite{Kraus1990}) that $\ell$ and $\Phi$ satisfies one of the 
following possibilities:
\begin{itemize}
 \item $\ell \geq 5$ and $\Phi$ is cyclic of order 2,3,4,6;
 \item $\ell = 3$ and $\Phi$ is cyclic of order 2,3,4,6 or isomorphic to $\Dic_{12}$,
 the Dicyclic group of order 12;
 \item $\ell = 2$ and $\Phi$ is cyclic of order 2,3,4,6 or it has order 8 and is isomorphic to the quaternion group $H_8$ or it has order 24 and is isomorphic to $\SL_2(\F_3)$.
\end{itemize}
Below we list the symplectic criteria currently available in the literature. 
Note that all of them 
require $E/\Q_\ell$ and~$E'/\Q_\ell$ to have the same kind of reduction; moreover, in the case of potentially good reduction, the hypothesis $E[p] \simeq E'[p]$ implies the semistability defects satisfy $e=e'$ (see part (A) of Proposition~\ref{P:TwistsameType}), thus no further reference to $E'$ is required.
\begin{itemize}
 \item for any $\ell$ of multiplicative reduction 
 there is \cite[Proposition~2]{KO} (see Theorem~\ref{T:potMult});
 \item for $\ell=2$, $e=24$ there is \cite[Theorem~4]{F33p} (see Theorem~\ref{T:Wilde24}); 
 \item for $\ell=2$, $e=8$  there is \cite[Theorem~4.6]{FNS23n} (see Theorem~\ref{T:mainWilde8});
 \item for $\ell=3$, $e=12$ there is \cite[Theorem~4.7]{FNS23n} (see Theorem~\ref{T:mainWilde12});
 \item for $\ell \equiv 2 \pmod{3}$ and $e=3,6$ 
 there is \cite[Proposition~A.2]{HK2002} but this criterion has fairly restrictive applications due to
its large list of hypothesis. 
\end{itemize}
Suppose that $E/\Q_\ell$ and $E'/\Q_\ell$ have potentially
good reduction and isomorphic $p$-torsion. If their semistability defects satisfy 
$e=e'=6$ or $e=e'=2$,
then there is an element $d \in \Q_\ell$ such that the quadratic twists $dE$ and $dE'$ 
respectively satisfy $e(dE) = e(dE') = 3$ or $e(dE) = e(dE') = 1$ (cf. Proposition~\ref{P:TwistsameType}). 
Since twisting both curves by~$d$ preserves the existence and the symplectic type (Lemma~\ref{L:twistIso})
of a $p$-torsion isomorphism
we conclude that the cases $e=6$ and $e=2$ reduce respectively to the cases $e=3$ and $e=1$ (good reduction).

We conclude that
there are no symplectic criteria available in the literature when
\begin{itemize}
 \item the curves have mixed reduction types, more precisely, 
 $E/\Q_\ell$ has potentially good reduction while $E'/\Q_\ell$ has potentially multiplicative reduction;
\end{itemize}
and when both curves have potentially good reduction 
in the following cases:
\begin{itemize}
 \item the tame case $\ell \neq 3$, $e=3$ when \cite[Proposition~A.2]{HK2002} does not apply;
 \item the wild case $\ell=e=3$;
 \item the tame case $\ell \neq 2$, $e=4$;
 \item the wild case $\ell = 2$ and $e=4$;
 \item the case  of good reduction: any $\ell$ and $e=1$;
 \item also, Theorem~\ref{T:Wilde24} part~(2) contains a procedure
 to detect the symplectic type when $(2/p) = -1$ via the
 `simple formula' $\upsilon_2(\Delta_m(E)) \equiv \upsilon_2(\Delta_m(E')) \pmod{3}$ 
 whilst Theorems~\ref{T:mainWilde8}~and~\ref{T:mainWilde12} do not contain an analogous formula.
\end{itemize}
Thus, to fully complete item (II) in Section~\ref{S:localApproach} 
we have to prove a symplectic criterion or a `simple formula' 
in each of the missing cases described above.
To complete (I) we need to understand, for every type of reduction of $E$, when 
the symplectic type of all $G_{\Q_\ell}$-module isomorphisms $\phi : E[p] \to E'[p]$ 
is the same. This may, for example, result in restrictions 
on the prime $\ell$ or on the invariants $c_4$, $c_6$ and $\Delta_m$ attached to 
a minimal model of $E/\Q_\ell$. A summary of these restrictions together with
a reference for the statement to the corresponding criterion can be 
found in Table~\ref{Table:CriteriaList}.

\subsection{Strategy of proof of a symplectic criterion} 
Let $E$, $E'$ be elliptic curves over~$\Q_\ell$ with potentially good reduction 
and isomorphic $p$-torsion modules. Assume that
all $G_{\Q_\ell}$-modules isomorphisms $\phi : E[p] \to E'[p]$ have the same symplectic type, i.e.
a symplectic criterion exists.
For example,  this happens when $\rhobar_{E,p}(G_{\Q_\ell})$ is non-abelian 
~(see Theorem~\ref{T:conditionRho}).

Fix symplectic bases of $E[p]$ and $E'[p]$.
Let $\phi : E[p] \to E'[p]$ be a $G_{\Q_\ell}$-modules isomorphism 
and write $M_\phi$ for the matrix 
representing it.
We shall show (see Lemma~\ref{L:determinant}) that $\phi$ is symplectic 
if and only if the determinant of $M_\phi$ is a square mod~$p$. Moreover,
we have $\rhobar_{E,p}(g) = M_\phi \rhobar_{E',p}(g) M_\phi^{-1}$ for all $g \in G_{\Q_\ell}$, 
but this equality is not enough to read whether $\det M_\phi$ is a quadratic residue, unless we 
determine the action on the $p$-torsion explicitly. 

We aim to decide the symplectic type of $\phi$ 
from easy information about $E$ and~$E'$. 
The main challenge we face is that we need a method that works for 
general $E$ and $E'$ but keeps track of the very specific information 
required to determine the symplectic type. 

We start by noting (Corollary~\ref{C:notBoth}) that if, for some subgroup $H \subset G_{\Q_\ell}$, we have  $\rhobar_{E,p}(H)$  non-abelian 
then it is enough to find a matrix $M \in \GL_2(\F_p)$ 
satisfying $\rhobar_{E,p}(g) = M \rhobar_{E',p}(g) M^{-1}$ for all $g \in H$,
for which we can determine the square class of 
$\det M$ mod~$p$.
Then we show there are symplectic bases for $E[p]$ and $E'[p]$ such that the images of 
$\rhobar_{E,p}(H)$ and $\rhobar_{E',p}(H)$ land in the same subgroup of $\GL_2(\F_p)$; this implies
that $M$ belongs to the normalizer $N_H$ of $\rhobar_{E,p}(H)$ in $\GL_2(\F_p)$.
Depending on the group structure of $\rhobar_{E,p}(H)$ we can pin down more concretely the 
shape of the matrices in $N_H$. For example, when $\ell = 2$, $H = I_2$, $\rhobar_{E,p}(I_2) \simeq H_8$
and $(2/p) = 1$ then all the matrices in $N_H$ have square determinant mod~$p$; 
this argument plays an essential part in the proof of Theorem~\ref{T:mainWilde8} 
(given in \cite[Theorem~4.6]{FNS23n}).
 
However, for most of the criteria we want to prove, the situation has two extra obstacles;
indeed, (i) there are matrices in $N_H$ with both types of determinant and (ii) the smallest subgroup 
$H$ with non-abelian image is the whole $G_{\Q_\ell}$.
We are therefore obliged to describe more precisely matrices $\rhobar_{E,p}(\sigma)$,
$\rhobar_{E,p}(\tau)$ and $\rhobar_{E',p}(\sigma)$, $\rhobar_{E',p}(\tau)$, generating 
the non-abelian image; the natural choice is to take $\sigma$ a generator of inertia (which is 
often cyclic) and $\tau$ a Frobenius element. 

Assume $E$ obtains good reduction over a field $L$. Then inertia acts via the homomorphism~$\gamma_E : I_\ell \to \Aut(\Ebar) \rightarrow \Ebar[p]$, where $\Ebar$ 
is the reduction of a model of $E/L$ with good reduction (see Lemma~\ref{L:gammaE}).
By showing the existence of Weierstrass models 
with certain properties (e.g. Lemmas~\ref{L:3torsionmodel}~and~\ref{L:gammaEe3}), we are able to describe 
the image of $\gamma_E$ in $\Aut(\Ebar)$ independently of any choice of basis (and
similarly for $E'$ and $\gamma_{E'}$). In particular, this allows to choose a basis where we can control 
the images of $\tau$ but also of $\rhobar_{E,p}(\sigma)$ and $\rhobar_{E',p}(\sigma)$
(e.g. Lemma~\ref{L:goodbasis}). 
This imposes more constraints on the matrices $M \in N_H$ which are allowed. 
To finish we show these are enough restrictions to decide if
$\det M$ is a square mod~$p$ (see, for example, case (1) in the proof of Theorem~\ref{T:mainTame3}). 

\section{A complete list of local symplectic criteria at $\ell \neq p$}
\label{S:results}

Here we state all the symplectic criteria established in this paper together 
with all other criteria available in the literature so 
that we obtain an easy to use and complete list.
We will divide the list according to the type of 
reduction and inertia sizes as described 
in Section~\ref{S:left?}.

To make this section self-contained we will 
introduce and/or repeat all the relevant notation.

Let $p$ and $\ell$ be primes such that $p \geq 3$ and $\ell \ne p$. 

Write $\F_\ell$ for the finite field with $\ell$ elements.

Let $\Q_\ell$ be the field of $\ell$-adic numbers and $\overline{\Q}_\ell$ an algebraic closure.
Write $\Q_\ell^{un} \subset \overline{\Q}_\ell$ 
for the maximal unramified extension of $\Q_\ell$. 
Write $\vv_\ell$ for the $\ell$-adic valuation in $\Q_\ell$ with $\vv_\ell(\ell) = 1$.

Write $G_{\Q_\ell} = \Gal(\overline{\Q}_\ell / \Q_\ell)$ for the absolute Galois group.
Let $I_\ell \subset G_{\Q_\ell}$ denote the inertia subgroup. Note that $\Q_\ell^{un}$
is the field fixed by $I_\ell$. Write $\Frob_\ell \in G_{\Q_\ell}$ for a Frobenius element.

Let $E/\Q_\ell$ be an elliptic curve, $E[p]$ its 
$p$-torsion module
and $\rhobar_{E,p}$ its mod~$p$ representation.

Let $\Delta_m = \Delta_m(E)$ denote the discriminant of a minimal Weierstrass model of $E$. 
Given a minimal model for $E/\Q_\ell$, whenever $c_4 \neq 0$ or $c_6 \neq 0$, 
we define the quantities $\tilde{c}_4$, $\tilde{c}_6$ 
and $\tilde{\Delta}$ by
\[
 c_4 = \ell^{\vv_\ell(c_4)} \tilde{c}_4, \qquad c_6 = \ell^{\vv_\ell(c_6)} \tilde{c}_6, \qquad \Delta_m = \ell^{\vv_\ell(\Delta_m)}\tilde{\Delta}.
\]
When we simultaneously use two elliptic curves $E$ and $E'$ we adapt the 
notation accordingly writing, in particular, $E'[p]$, $e'$, $c_4'$, $\tilde{c}_4'$, $c_6'$, $\tilde{c}_6'$ 
and $\Delta_m'$, $\tilde{\Delta}'$.

Let $E/\Q_\ell$ and $E'/\Q_\ell$ be elliptic curves with isomorphic $p$-torsion.
We say that $E[p]$ and~$E'[p]$ are {\it symplectically isomorphic} 
if there exists a $G_{\Q_\ell}$-isomorphism $\phi \; : \; E[p] \rightarrow E'[p]$
such that the quantity $d(\phi)$ in \eqref{E:dphi} is a square modulo~$p$; we say
$E[p]$ and~$E'[p]$ are {\it anti-symplectically isomorphic} if
$d(\phi)$ is a non-square modulo~$p$.

Let $a \in \Z$. We recall also the definition of the Legendre symbol 
\[
\left( \frac{a}{p} \right) = 
  \begin{cases}
  1 \text{ if } a  \text{ is a square modulo } p, \\
  -1 \text{ if } a  \text{ is not a square modulo } p, \\
  0 \text{ if a is divisible by } p. \
  \end{cases}
\]

\subsection{The case of additive potentially good reduction} \hfill

Given an elliptic curve $E/\Q_\ell$ with potentially good reduction 
we write $L=\Q_\ell^{\text{un}}(E[p])$ and $\Phi=\Gal(L/\Q_\ell^{un})$.
We call $L$ the {\it inertial field} of $E$. 
The field $L$ is independent of $p$ and is the minimal extension of $\Q_\ell^{\text{un}}$ where 
$E$ obtains good reduction (see \cite{ST1968}). 
We denote by $e = e(E)$ the order of $\Phi$. 
The determination of $e$
in terms of the triple of invariants 
$(c_4, c_6, \Delta_m)$ attached to $E$
is given in \cite{Kraus1990};
in particular, we can have $e=1,2,3,4,6,8,12$ or $24$. 
In the presence of a second curve $E'/\Q_\ell$ with potentially good reduction we write~$e'$ 
for~$e(E')$.

Note that when both curves have potentially good reduction, the hypothesis $E[p] \simeq E'[p]$ implies the semistability defects satisfy $e=e'$ (cf. Proposition~\ref{P:TwistsameType}), thus 
in statements of this section
we only need to inculde~$e$.

{\bf (A) Cyclic inertia of order $e=3$ and $e=6$.} 

We start with the list of symplectic criteria for the case $e=3$.

\begin{theorem} Let $\ell \equiv 2 \pmod{3}$ be a prime. 
Let $E$ and $E'$ be elliptic curves over $\Q_\ell$ with potentially good reduction
and $e=3$. Let $p \neq \ell$ be an odd prime.

Suppose that $E[p]$ and $E'[p]$ are isomorphic $G_{\Q_{\ell}}$-modules.

Set $t=1$ if exactly one of $E$, $E'$ has a $3$-torsion point defined over $\Q_\ell$ and $t=0$ otherwise.

Set $r=0$ if $\upsilon_\ell(\Delta_m) \equiv \upsilon_\ell(\Delta_m') \pmod{3}$ and $r=1$ otherwise. 

(i) If $p \geq 5$, then
\[
 E[p] \text{ and } E'[p] \quad \text{are symplectically isomorphic} \quad \Leftrightarrow \quad \left(\frac{\ell}{p}\right)^r \left (\frac{3}{p} \right)^t = 1.
\]
(ii) If $p = 3$, then $E[3]$ and $E'[3]$ are
symplectically isomorphic if and only if~$r=0$.

Moreover, $E[p]$ and $E'[p]$ are not both symplectically and anti-symplectically isomorphic.
\label{T:mainTame3}
\end{theorem}
To apply Theorem~\ref{T:mainTame3} we need to decide when an elliptic curve over
$\Q_\ell$ ($\ell \ne 3$) satisfying $e=3$ has a 3-torsion point over $\Q_\ell$.
The next theorem provides such classification.

\begin{theorem} Let $\ell \ne 3$ and $E/\Q_\ell$ be an elliptic curve
with tame additive reduction with $e=3$. Then the curve $E$ has a $3$-torsion point
defined over $\Q_\ell$ if and only if

(A) $\ell \geq 5$  and $-6\tilde{c}_6$ is a square in $\Q_\ell$ or,

(B) $\ell = 2$ and we are in one of the following cases

\begin{itemize}
 \item[(1)] $(\vv_\ell(c_4),\vv_\ell(c_6),\vv_\ell(\Delta_m)) = (4,5,4)$ and
  \[
   (\tilde{c}_4,\tilde{c}_6) \equiv (7,1) \pmod{8} \quad \text{ or } \quad (\tilde{c}_4,\tilde{c}_6) \equiv (3,5) \pmod{8};
  \]
  
 \item[(2)] $(\vv_\ell(c_4),\vv_\ell(c_6),\vv_\ell(\Delta_m)) \in \{(\geq 6,5,4),(\geq 7,7,8)\}$ 
 and $\tilde{c}_6 \equiv 5 \pmod{8}$;
 
 \item[(3)] $(\vv_\ell(c_4),\vv_\ell(c_6),\vv_\ell(\Delta_m)) = (4,6,8)$ and $(\tilde{c}_4,\tilde{c}_6)$ satisfies one of the following conditions
   $$\begin{array}{llll}
   -   & \tilde{c}_4 \equiv 29 \pmod{32} & \text { and }  & \tilde{c}_6 \equiv 15 \pmod{16} \\
   -   & \tilde{c}_4 \equiv 5 \pmod{32}  & \text { and }  & \tilde{c}_6 \equiv  3 \pmod{16} \\
   -   & \tilde{c}_4 \equiv 13 \pmod{32} & \text { and }  & \tilde{c}_6 \equiv 7  \pmod{16} \\
   -   & \tilde{c}_4 \equiv 21 \pmod{32} & \text { and }  & \tilde{c}_6 \equiv 11 \pmod{16}. \\
   \end{array}$$
\end{itemize}
 \label{T:main3torsion}
\end{theorem}

An analogous results to Theorem~\ref{T:mainTame3}, in the case
$\ell \equiv 1 \pmod{3}$, exists only for $p=3$.

\begin{theorem} Let $\ell \equiv 1 \pmod{3}$ be a prime. 
Let $E$ and $E'$ be elliptic curves over $\Q_\ell$ with potentially good reduction
and $e=3$. 

Suppose that $E[3]$ and $E'[3]$ are isomorphic $G_{\Q_{\ell}}$-modules.
Then
\[
 E[3] \text{ and } E'[3] \quad \text{are symplectically isomorphic} \quad \Leftrightarrow \quad 
 \upsilon_\ell(\Delta_m) \equiv \upsilon_\ell(\Delta_m') \pmod{3}.
\]
Moreover, $E[3]$ and $E'[3]$ are not both symplectically and anti-symplectically isomorphic.
\label{T:e=p=3}
\end{theorem}

We have the following theorem in the wild case $\ell=e=3$.

\begin{theorem} Let $E/\Q_3$ and $E'/\Q_3$ be elliptic curves
with potentially good reduction and~$e=3$, that is
$(\vv_3(c_4), \vv_3(c_6), \vv_3(\Delta_m))$ and $(\vv_3(c_4'), \vv_3(c_6'), \vv_3(\Delta_m'))$
belong to $\{ (2,3,4),(5,8,12)\}$.

Let $p \geq 5$ be a prime and suppose that $E[p]$ and $E'[p]$ are isomorphic $G_{\Q_3}$-modules. 

Let $r=0$ if $\tilde{c}_6 \equiv \tilde{c}_6' \pmod{3}$ and $r=1$ otherwise.

Suppose that $\tilde{\Delta} \equiv 2 \pmod{3}$. Then $\tilde{\Delta}^\prime \equiv 2 \pmod{3}$.
Furthermore,
\[
 E[p] \text{ and } E'[p] \quad \text{are symplectically isomorphic} \quad \Leftrightarrow \quad \left (\frac{3}{p} \right)^r = 1
\]
Moreover, $E[p]$ and $E'[p]$ are not both symplectically and anti-symplectically isomorphic.
\label{T:mainWild3}
\end{theorem}

The case $e=6$ can be reduced to 
the case $e=3$ by 
taking an adequate 
simultaneous 
quadratic twist of~$E$ and~$E'$
by some $d \in \Q_\ell$. The following two lemmas describe explicitly the value 
of~$d$ we need to take. 
We remark that, although in Lemma~\ref{L:lemma1} the twist is independent of the curve, this is not the case in Lemma~\ref{L:lemma2}. However, it follows from Proposition~\ref{P:TwistsameType} that, under the hypothesis $E[p] \simeq E'[p]$, the same twist works for both curves.

\begin{lemma}
\label{L:lemma1} Let $\ell \geq 3$ and $E/\Q_\ell$ be an elliptic curve 
with potentially good reduction with~$e=6$. 

Then the quadratic twist $E'/\Q_\ell$ of $E$ by $\sqrt{\ell}$ has potentially good reduction with $e'=3$.
\end{lemma}

\begin{lemma} \label{L:lemma2} 
Let $E/\Q_2$ be an elliptic curve with potentially good reduction with $e=6$.

Write $t=(c_4(E), c_6(E), \Delta_m(E))$. Then~$t$ satisfies one of the following cases and we
let $u\in \lbrace -2,-1,2\rbrace$ be defined as follows:
 \[
u = 
  \begin{cases}
   -1 \ \text{ if } \ v(\Delta_m) \in \lbrace 4,8\rbrace, \\
  2 \ \text{ if } \ v(\Delta_m)=10 \quad \text{and} \quad \tilde{c}_6\equiv 1 \pmod 4,\\
  -2 \ \text{ if } \ v(\Delta_m)=10  \quad \text{and} \quad \tilde{c}_6\equiv -1 \pmod 4,\\
   2 \ \text{ if } \ t=(6,9,14) \quad \text{and} \quad \tilde{c}_6\equiv -1 \pmod 4,\\
  -2 \ \text{ if } \ t=(6,9,14) \quad \text{and} \quad \tilde{c}_6\equiv 1 \pmod 4,\\
 2 \ \text{ if } \ t=(\geq 9,10,14) \quad \text{and} \quad \tilde{c}_6\equiv 1 \pmod 4, \\
  -2 \ \text{ if } \ t=(\geq 9,10,14) \quad \text{and} \quad \tilde{c}_6\equiv -1 \pmod 4. \
  \end{cases}
\]
Then the quadratic twist $E'/\Q_2$ of $E$ by $\sqrt{u}$ has potentially good reduction with $e'=3$. 
\end{lemma}

\

{\bf (B) Cyclic inertia of order $e=4$.}

\begin{theorem} Let $\ell \equiv 3 \pmod{4}$ be a prime. 
Let $E$ and $E'$ be elliptic curves over $\Q_\ell$ with potentially good reduction
and $e=4$. Let $p \geq 5$ be a prime. 

Set $r=0$ if $\upsilon_\ell(\Delta_m) \equiv \upsilon_\ell(\Delta_m') \pmod{4}$ and $r=1$ otherwise. 

Set $t=1$ if exactly one of $\tilde{\Delta}$, $\tilde{\Delta}'$ is a square mod~$\ell$ and $t=0$ otherwise.

Suppose that $E[p]$ and $E'[p]$ are isomorphic $G_{\Q_{\ell}}$-modules. Then
\[
 E[p] \text{ and } E'[p] \quad \text{are symplectically isomorphic} \quad \Leftrightarrow \quad \left(\frac{\ell}{p}\right)^r \left (\frac{2}{p} \right)^t = 1
\]
Moreover, $E[p]$ and $E'[p]$ are not both symplectically and anti-symplectically isomorphic.
\label{T:mainTame4}
\end{theorem}

We have the following theorem in the wild case $\ell=2$, $e=4$.

\begin{theorem} Let $E/\Q_2$ and $E'/\Q_2$ be elliptic curves with potentially good reduction
and $e=4$, that is $(\vv_2(c_4), \vv_2(c_6), \vv_2(\Delta_m))$ and $(\vv_2(c_4'), \vv_2(c_6'), \vv_2(\Delta_m'))$ belong to $\{ (5,8,9),(7,11,15)\}$. 

Let $p \geq 3$ be a prime and suppose that $E[p]$ and $E'[p]$ are isomorphic $G_{\Q_2}$-modules. 

Let $r=0$ if $\tilde{c}_6 \equiv \tilde{c}_6' \pmod{4}$ and $r=1$ otherwise.

Suppose that $\tilde{c}_4 \equiv 5\tilde{\Delta} \pmod{8}$. Then $\tilde{c}_4' \equiv 5\tilde{\Delta}' \pmod{8}$. Moreover,
\[
 E[p] \text{ and } E'[p] \quad \text{are symplectically isomorphic} \quad \Leftrightarrow \quad \left (\frac{2}{p} \right)^r = 1
\]
Moreover, $E[p]$ and $E'[p]$ are not both symplectically and anti-symplectically isomorphic.
\label{T:mainWild4}
\end{theorem}

{\bf (C) Non-abelian inertia.}

Let $H_8$ denote the quaternion group, $\Dic_{12}$ the 
dicyclic group with 12 elements and $\SL_2(\F_3)$ the special linear group
of degree 2 and coefficients in $\F_3$.

We remark that the next theorems do not require the assumption 
of $E[p]$ and $E'[p]$ being isomorphic $G_{\Q_\ell}$-modules, because the weaker 
assumption that the curves have the same inertial field 
already guarantees, for all $p$, an isomorphism at the level of inertia. However, in practice 
we often have
$E[p] \simeq E'[p]$ which implies the equality of inertial fields.

The following is \cite[Theorem~4]{F33p}.  

\begin{theorem}
  Let $E$ and~$E'$ be elliptic curves over~$\Q_2$
  with potentially good reduction. Assume they have the same inertial field
  $L$ and $e=24$, so that $\Gal(L/\Q_2^{un}) \simeq \SL_2(\F_3)$.
  
  Let $p \geq 3$ be a prime. Then $E[p]$ and $E'[p]$ are isomorphic as $I_2$-modules. Furthermore,
\begin{enumerate}[\upshape(1)]
 \item if $(2/p)=1$ then $E[p]$ and $E'[p]$ are symplectically isomorphic $I_2$-modules. 
 \item if $(2/p)=-1$ then $E[p]$ and $E'[p]$ are symplectically isomorphic $I_2$-modules
 if and only if $\upsilon_2(\Delta_m(E)) \equiv \upsilon_2(\Delta_m(E')) \pmod{3}$. 
\end{enumerate}
Moreover, $E[p]$ and $E'[p]$ are not both symplectically and anti-symplectically isomorphic.
\label{T:Wilde24}
\end{theorem}

The following is part (1) of \cite[Theorem~4.6]{FNS23n}. 

\begin{theorem} \label{T:mainWilde8}
  Let $E$ and~$E'$ be elliptic curves over~$\Q_2$
  with potentially good reduction. 
  Assume they have the same inertial field $L$ 
  and $e=8$, so that $\Gal(L/\Q_2^{un}) \simeq H_8$.
  
  Let $p \geq 3$ be a prime. Then $E[p]$ and $E'[p]$ are isomorphic as $I_2$-modules
  and not simultaneously symplectically and anti-symplectically isomorphic.
  
  Moreover, if $(2/p) = 1$ then $E[p]$ and $E'[p]$ are symplectically isomorphic $I_2$-modules.
\end{theorem}

Note that \cite[Theorem~4.6]{FNS23n} contains a part~(2) stating that 
if $(2/p) = -1$, then $E[p]$ and $E'[p]$ are symplectically isomorphic $I_2$-modules
if and only if $E[3]$ and~$E'[3]$ are symplectically isomorphic $I_2$-modules.
When compared to part~(2) of Theorem~\ref{T:Wilde24} this description is not satisfactory, 
because it does not provide an `easy formula' to determine the 
symplectic type when $(2/p) = -1$. 
For that purpose we will prove the following.
\begin{theorem} \label{T:mainWilde8II}
  Let $E$, $E'$ and $p$ be as in Theorem~\ref{T:mainWilde8}, so that $\Gal(L/\Q_2^{un}) \simeq H_8$ and also $E[p]$ and $E'[p]$ are isomorphic as $I_2$-modules
  and not simultaneously symplectically and anti-symplectically isomorphic.
  Assume further that $(2/p) = -1$. 
  
  Then, after twisting both curves by $2$ if necessary, 
  we can assume that $E$ and $E'$ have either 
  both conductor $2^5$ or both conductor $2^8$. 
  Moreover, $E[p]$ and $E'[p]$ are symplectically isomorphic if and only if 
  one of the following holds:
  \begin{enumerate}[(A)]
   \item Both $E$ and $E'$ have conductor $2^5$ and are in the same case of Table~\ref{Table:thm8};
   \item Both curves have conductor $2^8$ and $\tilde{c}_4 \equiv \tilde{c}_4' \pmod{4}$.
  \end{enumerate}
\end{theorem}
\begin{small}
\begin{table}[htb]
$$
\begin{array}{|c|c|c|} \hline
 \text{{\sc Case}} & \text{\sc Conductor} &(\vv_2(c_4), \vv_2(c_6), \vv_2(\Delta_m)), \; \; \tilde{c}_4 \pmod{4}  \\ \hline \hline
 \multirow{2}{*}{\text{(a)}}  & \multirow{2}{*}{$2^5$} & (4, n \geq 7,  6), \; \; \tilde{c}_4 \equiv -1 \\
                              &                        &  (7,9,12)           \\ \hline \hline
 \multirow{2}{*}{\text{(b)}}  & \multirow{2}{*}{$2^5$} & (6, n \geq 10, 12), \; \; \tilde{c}_4  \equiv 1 \\
                              &                        & (4,6,9)            \\ \hline 
\end{array}
$$
\caption{Conditions for Theorem~\ref{T:mainWilde8II};
when $\tilde{c}_4$ does not ocurr in a row, 
there is no restriction on it.}
\label{Table:thm8}
\end{table}
\end{small}
Since twisting both curves by the same element preserves the existence and symplectic type of an $I_2$-modules isomorphism
$\phi : E[p] \to E'[p]$ the previous theorem also covers the case of conductor $2^6$ and $e=8$.

The following theorem is part (1) of \cite[Theorem~4.7]{FNS23n}. 

\begin{theorem} \label{T:mainWilde12}
  Let $E$ and~$E'$ be elliptic curves over~$\Q_3$ with potentially good reduction. 
  Assume they have the same inertial field $L$ and $e=12$, so that 
  $\Gal(L/\Q_3^{un}) \simeq \Dic_{12}$.
  
  Let $p \geq 5$ be a prime. Then $E[p]$ and $E'[p]$ are isomorphic as $I_3$-modules
  and not simultaneously symplectically and anti-symplectically isomorphic.
   
  Moreover, if $(3/p) = 1$ then $E[p]$ and $E'[p]$ are symplectically isomorphic $I_3$-modules;
\end{theorem}

Analogously to Theorem~\ref{T:mainWilde8II}, 
we will complement the previous theorem with 
an easy procedure to determine the symplectic type when $(3/p) = -1$.

\begin{theorem} \label{T:mainWilde12II}
  Let $E$, $E'$ and $p$ be as in Theorem~\ref{T:mainWilde12}, so that $\Gal(L/\Q_3^{un}) \simeq \Dic_{12}$ and also $E[p]$ and $E'[p]$ are isomorphic as $I_3$-modules
  and not simultaneously symplectically and anti-symplectically isomorphic.
  Assume further that $(3/p) = -1$.

  Then $E[p]$ and $E'[p]$ are symplectically isomorphic if and only if 
  both $E$ and $E'$ satisfy the same case of Table~\ref{Table:thm10}.
  
\end{theorem}

\begin{small}
\begin{table}[htb]
$$
\begin{array}{|c|c|c|} \hline
 \text{{\sc Case}} & \text{\sc Conductor} &(\vv_3(c_4), \vv_3(c_6), \vv_3(\Delta_m)), \; \; \tilde{\Delta} \pmod{9}  \\ \hline \hline
 \multirow{3}{*}{\text{(a)}} & \multirow{3}{*}{$3^3$}  & (n \geq 2, 3,3), \; \; \tilde{\Delta} \not\equiv 2,4 \\
                             &  & (2, 4, 3)        \\
                             &  & (4,6,11)          \\ \hline \hline
 \multirow{3}{*}{\text{(b)}} &  \multirow{3}{*}{$3^3$} & (n \geq 4, 6,9), \; \; \tilde{\Delta} \not\equiv 2,4 \\
			     &  & (2,3,5)             \\
                             &  & (4,7,9)             \\ \hline \hline
 \multirow{2}{*}{\text{(c)}} &  \multirow{2}{*}{$3^5$} & (n \geq 3, 4, 5) \\
                             &   & (n \geq 6, 8, 13) \\ \hline \hline
 \multirow{2}{*}{\text{(d)}} &  \multirow{2}{*}{$3^5$} & (n \geq 4, 5, 7)    \\
                             &   & (n \geq 5, 7, 11)    \\ \hline
\end{array}
$$
\caption{Conditions for Theorem~\ref{T:mainWilde12II}; 
when $\tilde{\Delta}$ does not ocurr in a row, 
there is no restriction on it.}
\label{Table:thm10}
\end{table}
\end{small}
We remark that Theorems~\ref{T:mainWilde8II}~and~\ref{T:mainWilde12II} are optimal in the
sense that it is possible to find elliptic curves $E$, $E'$
satisfying the hypothesis of each of the listed cases.
 
\subsection{The cases $e=2$ and of good reduction} \hfill
\label{S:resultsGood}

We first note that the case $e=2$ can be reduced to 
the case of good reduction ($e=1$)
by 
taking an adequate 
simultaneous 
quadratic twist of~$E$ and~$E'$
by some $d \in \Q_\ell$. The following two lemmas describe explicitly the value 
of~$d$.
We remark that, although in Lemma~\ref{L:lemma3} the twist is independent of the curve, this is not the case in Lemma~\ref{L:lemma4}. It follows from Proposition~\ref{P:TwistsameType} that, under the hypothesis $E[p] \simeq E'[p]$, the same twist works for both curves.

\begin{lemma}
\label{L:lemma3} Let $\ell \geq 3$ and $E/\Q_\ell$ be an elliptic curve 
with potentially good reduction with~$e=2$. 

Then the quadratic twist of $E$ by $\sqrt{\ell}$ has good reduction.
\end{lemma}

\begin{lemma} \label{L:lemma4} 
Let $E/\Q_2$ be an elliptic curve with potentially good reduction with $e=2$.

Write $t=(c_4(E), c_6(E), \Delta_m(E))$.
Then~$t$ satisfies one of the following cases and we
let $u\in \lbrace -2,-1,2\rbrace$ be defined as follows:
 \[
u = 
  \begin{cases}
  2 \ \text{ if } \ t=(\geq 6,6,6) \quad \text{and} \quad \tilde{c}_6\equiv 1 \pmod 4, \\
  -2 \ \text{ if } \ t=(\geq 6,6,6) \quad \text{and} \quad \tilde{c}_6\equiv -1 \pmod 4, \\
  -1 \ \text{ if } \ t=(4,6,12) \quad \text{or} \quad t=(\geq 8,9,12), \\
  2 \ \text{ if } \ t=(6,9,18) \quad \text{and} \quad \tilde{c}_6\equiv -1 \pmod 4, \\
 -2 \ \text{ if } \ t=(6,9,18) \quad \text{and} \quad \tilde{c}_6\equiv 1 \pmod 4. \
  \end{cases}
\]
Then the quadratic twist of $E/\Q_{2}$ by $\sqrt{u}$ has good reduction.
\end{lemma}

To state the symplectic criterion in the case of good reduciton 
we need some more notation.

Let $D$ be a negative discriminant, i.e. $D < 0$ satisfies $D \equiv 0,1 \pmod{4}$.
Denote by $\calO_D$ the imaginary quadratic order of discriminant $D$, 
viewed inside the field of complex numbers $\C$. 
Consider the following polynomial
\[
  \mathcal{P}_D = \prod_{\calO_D \subset \End(\C/\mathfrak{a})} (x - j_{\C/\mathfrak{a}}),
\]
where the product runs over the
isomorphism classes of the elliptic curves $\C/\mathfrak{a}$ whose endomorphisms ring contains
the order $\calO_D$. We have $\mathcal{P}_D \in \Z[x]$. 
Extend the definition of $\mathcal{P}_D$
to all $D \leq 0$ by setting $\mathcal{P}_0 = 0$ and
$\mathcal{P}_D = 1$ for $D \equiv 2,3 \pmod{4}$. 

Let $E/\Q_\ell$ be an elliptic curve with good reduction and $j$-invariant $j_E$. 
Write $\Ebar /\F_\ell$ for the elliptic curve obtained by reducing mod~$\ell$
a minimal model of $E$. Define the integer quantities
\[
 a_\ell = (\ell + 1) - \#\Ebar(\F_\ell) \qquad \text{ and } \qquad 
 \Delta_\ell = a_\ell^2 - 4\ell.
\]
Note that $\Delta_\ell \neq 0$ and, from the Hasse-Weil bound, 
we have $\lvert a_\ell \rvert \leq 2 \sqrt{\ell}$, 
hence $\Delta_\ell < 0$. The value $a_\ell$ is also called
the trace of Frobenius.
Define also the quantity $\beta_\ell$ by the formula
\[
 \beta_\ell = \sup_{h > 0} \{ \; h : h^2 \mid \Delta_\ell \; \text{ and } \; 
 \mathcal{P}_{\Delta_\ell/h^2}(j_E) \equiv 0 \pmod{\ell}  \}
\]
which is an integer because $\Delta_\ell \neq 0$. We can finally state the criterion.

\begin{theorem} Let $\ell \neq p$ be primes with $p \geq 3$. 
Let $E$ and $E'$ be elliptic curves over $\Q_\ell$ with good reduction
and $\F_\ell$-isomorphic residual elliptic curves.
Then

1) $E[p]$ and $E'[p]$ are isomorphic $G_{\Q_{\ell}}$-modules.

2) Assume further that $p \mid \Delta_\ell$ and $p \nmid \beta_\ell$, or equivalently, that 
$\rhobar_{E,p}(\Frob_\ell)$ has order a multiple of~$p$.
Then $E[p]$ and $E'[p]$ are symplectically isomorphic 
and not anti-symplectically isomorphic. 

\label{T:simpleAbelian}
\end{theorem}

In~\cite{Ccode} there is a {\tt Magma} code available 
to compute all the quantities relevant for Theorem~\ref{T:simpleAbelian}.

\begin{remark} Note that
Theorem~\ref{T:simpleAbelian} differs from the other criteria as it does not start from the assumption $E[p]\simeq E’[p]$; instead it first obtains this conclusion from a stronger condition on the residual curves. We give in 
Theorem~\ref{T:mainAbelian} a generalization of part 2) of Theorem~\ref{T:simpleAbelian} starting from the natural condition $E[p]\simeq E’[p]$
but, unfortunately, that theorem contains other hypothesis which are hard to verify in practice. 
We decided to not include Theorem~\ref{T:mainAbelian} in this list to keep in this section only criteria which are easy to use.
\end{remark}

\subsection{The case of potentially multiplicative reduction} \hfill

The following criterion is \cite[Proposition~2]{KO}. 

\begin{theorem}
  Let $\ell \neq p$ be primes with $p \geq 3$. 
  Let $E$ and $E'$ be elliptic curves over $\Q_\ell$ with multiplicative reduction.
  Suppose that $E[p]$ and $E'[p]$ are isomorphic $G_{\Q_{\ell}}$-modules. 
  
  Assume further that $p \nmid v_{\ell}(\Delta_m)$. Then $p \nmid v_{\ell}(\Delta_m')$. 
  Furthermore, 
  \[
 E[p] \text{ and } E'[p] \quad \text{are symplectically isomorphic} \quad  
 \Leftrightarrow \quad \left (\frac{v_\ell(\Delta_m)/v_\ell(\Delta_m')}{p} \right) = 1.
 \]
 Moreover, $E[p]$ and $E'[p]$ are not both symplectically and anti-symplectically isomorphic.
  \label{T:potMult}
\end{theorem}

The previous theorem requires both curves to have multiplicative reduction. In case the
curves have additive potentially multiplicative reduction we first need to twist them by an elment~$d \in \Q_\ell$ to make the reduction multiplicative.
More precisely, we
apply the following well known fact (see \cite[p.~442--444]{SilvermanII} and \cite[Lemme~1]{CaliKraus}).
\begin{lemma}
\label{L:multTwist} 
Let $E/\Q_\ell$ be an elliptic curve 
with potentially multiplicative reduction. 

Then  the quadratic twist of $E$ by~$\sqrt{-c_6(E)}$ has split multiplicative reduction.
\end{lemma}

Note that, although the twist depends on~$E$, 
it follows from Proposition~\ref{P:TwistsameType} that, under the hypothesis $E[p] \simeq E'[p]$, the same twist works also for~$E'$.
\subsection{The case of mixed reduction types} \hfill

All the symplectic criteria above require $E/\Q_\ell$ and $E'/\Q_\ell$ to simultaneously have potentially multiplicative reduction or potentially good reduction. The following is the unique symplectic criteria that exists in the mixed situation. It is very restrictive as it only applies for $p=3$.

\begin{theorem} Let $\ell \neq 3$ be a prime.
Let $E/\Q_\ell$ be an elliptic curve with multiplicative reduction and
$E'/\Q_\ell$ an elliptic curve with potentially good reduction and $e'=3$.

Suppose that $E[3]$ and $E'[3]$ are isomorphic $G_{\Q_{\ell}}$-modules.
Then $3 \nmid \vv_\ell(\Delta_m)$. Moreover,
\[
 E[3] \text{ and } E'[3] \quad \text{are symplectically isomorphic} \quad \Leftrightarrow \quad 
 \upsilon_\ell(\Delta_m) \equiv \upsilon_\ell(\Delta_m') \pmod{3}.
\]
Moreover, $E[3]$ and $E'[3]$ are not both symplectically and anti-symplectically isomorphic.
\label{T:mixedReduction}
\end{theorem}
We note that the case 
where $p=3$, $E/\Q_\ell$ has additive potentially multiplicative reduction and
$E'/\Q_\ell$ potentially good reduction with $e'=6$ reduces to Theorem~\ref{T:mixedReduction} by twsting. Indeed, Lemma~\ref{L:multTwist} applied to~$E$ gives the value
of~$d$ we need to twist by.
Although~$d$ depends on~$E$, 
it follows from Lemma~\ref{L:twistIso} that, under the hypothesis $E[3] \simeq E'[3]$, the same twist will reduce 
the semistability defect $e'=6$ to $e(dE')=3$.

{\large \part{The existence of local symplectic criteria} \label{S:conditionOnRho}}

\section{Existence of symplectic criteria in terms of the image of $\rhobar_{E,p}$}
\label{S:existence}

Recall from Section~\ref{S:localApproach} that, for elliptic curves $E/\Q_\ell$ and $E'/\Q_\ell$ 
with isomorphic $p$-torsion modules, we say that a
local symplectic criterion exists if the 
$G_{\Q_\ell}$-module isomorphisms
between $E[p]$ and $E'[p]$ are either all symplectic or all anti-symplectic.
As explained in the same section
we want to classify when a symplectic criterion exists for
all the possible types of reduction of $E$ and~$E'$, as 
summarized in Table~\ref{Table:CriteriaList}.
Here we will give an intermediate step towards this classification
by describing when a symplectic criterion exists in terms of conditions
on the image $\rhobar_{E,p}(G_{\Q_\ell})$ as a subgroup of~$\GL_2(\F_p)$. 

We shall prove the more general Theorem~\ref{T:conditionRho} below.  
We start by extending the idea above for any field $F \subset \Qbar_\ell$. Given elliptic curves $E/F$ and $E'/F$ 
with isomorphic $p$-torsion $G_F$-modules, we say that a
{\bf (local) symplectic criterion exists} if the 
$G_F$-module isomorphisms
between $E[p]$ and $E'[p]$ are either all symplectic or all anti-symplectic.

\begin{theorem} \label{T:conditionRho}
Let $\ell$ and $p$ be different primes with $p \geq 3$.
Let $F \subset \Qbar_\ell$ be a field. Let
$E/F$ and $E'/F$ be elliptic curves with isomorphic $p$-torsion modules.
Then a symplectic criterion exists if and only if one of the following conditions holds:
\begin{itemize}
 \item[(A)] $\rhobar_{E,p}(G_F)$ is non-abelian;
 \item[(B)] $\rhobar_{E,p}(G_F)$ is generated, up to conjugation,  
by $\begin{pmatrix} a & 1 \\ 0 & a \end{pmatrix}$ where $a \in \F_p^*$.
\end{itemize}
\end{theorem}
We will derive this theorem from considerations on elliptic curves over more general fields. 

Let $K$ denote a field of characteristic zero or a finite field with characteristic $\neq p$.
For an elliptic curve $E/K$ 
a $\F_p$-basis $(P,Q)$ of $E[p]$ is called a symplectic basis if $e_{E,p}(P,Q) = \zeta_p$, 
where $\zeta_p \in \overline{K}$ is a fixed $p$-th root of unity.

The following lemma generalizes \cite[Lemma~1]{F33p} 
and plays a crucial r\^ole in our arguments.

\begin{lemma} \label{L:sympcriteria}
  Let $E$ and $E'$ be two elliptic curves defined over~$K$
  with isomorphic $p$-torsion.
  Fix symplectic bases for $E[p]$ and $E'[p]$. Let $\phi \; \colon \; E[p] \to E'[p]$ be
  an isomorphism of $G_K$-modules and write $M_\phi$ for the matrix representing~$\phi$ with
  respect to the fixed bases.

  Then $d(\phi)= \det M_\phi$ and $\phi$ is a symplectic isomorphism if and only if $\det(M_\phi)$ is a square mod~$p$; otherwise $\phi$ is anti-symplectic. Here, $d(\phi)$ is the quantity defined by equation~\eqref{E:dphi}. 

  Moreover, $E[p]$ and~$E'[p]$ are not simultaneously symplectically and anti-symplectically isomorphic if and only if the centralizer of $\rhobar_{E,p}(G_K)$ in $\GL_2(\F_p)$ contains only matrices  with square determinant.
  \label{L:determinant}
\end{lemma}

\begin{proof}
  Let $P,Q \in E[p]$ and $P',Q' \in E'[p]$ be symplectic bases. We have that
  \[ e_{E',p}(\phi(P),\phi(Q)) = e_{E',p}(P',Q')^{\det(M_\phi)}
                               = {\zeta_p}^{\det(M_\phi)}
                               = e_{E,p}(P,Q)^{\det(M_\phi)},
  \]
  so $d(\phi) = \det(M_\phi)$. This implies the first assertion.

  Let $\phi, \beta : E[p]\to E'[p]$ be two isomorphisms of $G_K$-modules.
  Then $\lambda = \phi^{-1} \circ \beta$ is an automorphism of the
  $G_K$-module $E[p]$. Furthermore, for any choice of bases
  for $E[p]$ and $E'[p]$, the matrix $M_\lambda$ representing $\lambda$ lies
  in the centralizer $C$ of $\rhobar_{E,p}(G_K)$ in $\GL_2(\F_p)$. 
  Conversely, if we fix a basis for $E[p]$ and take any $M \in C$
  there is a $G_K$-automorphism $\lambda$ of $E[p]$ such that $M_\lambda = M$ 
  and we can define the $G_K$-modules isomorphism 
  $\beta= \phi \circ \lambda : E[p]\to E'[p]$. 
  
  We have the relation of determinants  $\det(M_\beta) = \det(M_\lambda) \det(M_\phi)$.
  Now, choosing symplectic bases for $E[p]$ and $E'[p]$, the first part of the lemma 
  shows that $\phi$, $\beta$ have the same symplectic type if and only if 
  $\det(M_\lambda)$ is a square mod~$p$. The second statement now follows. 
\end{proof}

The following corollary determines the symplectic type of a special type of 
$G_K$-isomorphisms.

\begin{corollary} \label{C:isogeny}
Let $E/K$ and $E'/K$ be elliptic curves. 
Let $\phi : E \to E'$ be an isogeny of degree~$n$ not divisible by~$p$, so that it
restricts to an isomorphism $\phi|_{E[p]} : E[p] \to E'[p]$. 

Then $\phi|_{E[p]}$ is symplectic if $(n/p) = 1$
and anti-symplectic if $(n/p) = -1$.
\end{corollary}
\begin{proof}
Let $\hat{\phi}$ be the dual isogeny of~$\phi$ (see~\cite[III.6]{SilvermanI} for definition and properties of~$\hat{\phi}$).
The isogenies $\phi$ and~$\hat{\phi}$ 
are adjoint with respect to the Weil pairing. Therefore,
\[ e_{E',p}(\phi(P), \phi(Q)) = e_{E,p}(P, \hat{\phi}\phi(Q))
                              = e_{E,p}(P, nQ) = e_{E,p}(P,Q)^n
\]
and it follows from \eqref{E:dphi} that $d(\phi) \equiv n \pmod{p}$. 
The result now follows from 
Lemma~\ref{L:sympcriteria}.
\end{proof}

For a field $k$ let $M_2(k)$ denote the ring of $2$ by $2$ matrices 
with coefficients in $k$ and denote by $I_2$ the identify matrix in $M_2(k)$.
The following lemma shows that the condition on the centralizer in Lemma~\ref{L:determinant}
is always satisfied for non-abelian subgroups of $\GL_2(k)$.

\begin{lemma} Let $k$ be a field and and $G$ a subgroup of $\GL_2(k)$. If $G$ is non-abelian
then the centralizer $C_{\GL_2(k)}(G)$ are the scalar matrices. In particular, every matrix
in $C_{\GL_2(k)}(G)$ has square determinant mod~$p$.
\label{L:centralizerNonAbelian}
\end{lemma}
\begin{proof} Clearly, the scalar matrices belong to $C_{\GL_2(k)}(G)$.
Suppose now the lemma is false. There is a matrix 
$g \in C_{\GL_2(k)}(G)$ with minimal polynomial of degree 2. 

We claim that the centralizer of $g$ in $M_2(k)$ is $k[g]$. 
Then $C_{\GL_2(k)}(g)$ is $k[g] \cap \GL_2(k)$ and, 
in particular, $C_{\GL_2(k)}(g)$ is abelian. This is a contradiction,
since the non-abelian group $G$ is contained in $C_{\GL_2(k)}(g)$. 
The last statement follows from $\det (\lambda \cdot I_2) = \lambda^2$.

We now prove the claim. 
Let $V$ be a 2-dimensional $k$-vector space 
and identify $\GL_2(k)$ with $\GL(V)$.
Write $A$ for the centralizer of $g$ in $M_2(k) \simeq \End(V)$.
Since $g$ is not scalar, there exists a line which is not stable by~$g$. 
So there exists $v \in V$ such that $(v,g(v))$ is a basis of $V$. The map  $A \to V$ defined by
$T\mapsto T(v)$ is linear and injective (if $T(v)=0$, since $gT = Tg$ we also have $T(g(v))=0$, so $T=0$). 
We deduce that the dimension of $A$ is at most $2$. 
Moreover, $k[g] \subset A$, and its dimension is $2$ because the minimal polynomial
of $g$ is of degree 2. Hence $A=k[g]$ as claimed.
\end{proof}

Next we will classify the abelian subgroups of $\GL_2(\F_p)$ whose
centralizer contains only matrices with square determinant mod~$p$.
Recall from \cite[\S2.1]{Serre72} that a subgroup $C \subset \GL_2(\F_p)$ 
is a {\it split Cartan subgroup} if it is conjugate to 
$\left(\begin{smallmatrix}
  * & 0 \\
  0 & *
 \end{smallmatrix}\right)
$
and a {\it non-split Cartan subgroup} if $C \simeq \F_{p^2}^*$.

\begin{proposition} Let $V$ be a $\F_p$-vector space of dimension $2$ 
and $g \in \GL(V)$ a non-scalar element of order coprime to $p$. 
Then the centralizer $C(g)$ of $g$ in $\GL(V)$ is a Cartan subgroup.
\label{P:centralizer}
\end{proposition}
\begin{proof}
Let $g$ be as in the statement and write $H=C(g) \subset \GL(V)$ for 
its centralizer.
Let $f$ be the minimal polynomial of $g$ which is of degree 2 because
$g$ is not scalar; also $f$ is separable because the order of $g$ is 
not divisible by $p$. We now divide into cases:

(1) Suppose $f = (X-a)(X-b)$ with $a,b \in \F_p$ distinct roots.
Then $g$ stabilizes two distinct lines which are also stable under $H$. 
So $H$ is a split Cartan subgroup.

(2) Suppose $f$ is irreducible over $\F_p$.
Let $\F \subset \End(V)$ be the set of elements $T$ such that $gT=Tg$.
From the last part of the proof of Lemma~\ref{L:centralizerNonAbelian}, 
we have that $\F=\F_p[g]$. We have that  $\F_p[X]/(f)$ is isomorphic to $\F$ via the morphism $P\mapsto P(g)$, thus $\F$ is a field. Finally, $\F^*=H$ and we conclude $H$ is a non-split Cartan subgroup, as desired.

\end{proof}

Before proving Theorem~\ref{T:conditionRho} we need one more lemma.

\begin{lemma} Let $p \geq 3$ be a prime and $H$ an abelian subgroup of $G=\GL_2(\F_p)$.
Then the following statements are equivalent:
\begin{itemize}
 \item[(a)] all matrices in the centralizer $C_G(H)$ have determinant which is a square mod $p$;
 \item[(b)] $H$ is cyclic of order divisible by $p$;
 \item[(c)] $H$ is generated (up to conjugation) by a matrix of the 
form $\left( \begin{smallmatrix} a & 1 \\ 0 & a \end{smallmatrix} \right)$ where $a \in \F_p^*$.
\end{itemize}
\label{L:CentralizerAbelian}
\end{lemma}
\begin{proof} Suppose all matrices in the centralizer $C_G(H)$ have determinant 
which is a square mod~$p$. Since $H$ is abelian we have
\[ 
H \subset C_G(H) \subset C_G(g) \quad \text{ for all }  g \in H. 
\]
If $H$ is made up of scalar matrices then $C_G(H) = G$ contains matrices
with non-square determinant (since $p \neq 2$). Thus we can assume there is at least one
generator of $H$ which is not a scalar matrix. Write $g_1, \ldots, g_k$ for the 
non-scalar generators of $H$. 

Suppose $H$ has order not divisible by $p$. 
The centralizers $C_G(g_i)$ are Cartan subgroups of $G$ for all $i$
by Proposition~\ref{P:centralizer}.
We now divide into two cases:

(i) Suppose there exists a generator $g = g_j$ whose centralizer is 
maximal among the centralizers of the $g_i$, that is $C_G(g_i) \subset C_G(g)$ 
for all $i$. Let $h \in H$ be given by 
$h = g_1^{d_1} \cdot \ldots \cdot g_k^{d_k}$; now any $s \in  C_G(g)$ commutes with
all $g_i$ so it commutes with $h$ and we conclude that $s  \in  C_G(H)$. Thus 
$C_G(H) = C_G(g)$ is a Cartan subgroup of $G$; in particular it contains matrices
with non-square determinant (since $p \ne 2$).

(ii) Suppose there is no generator $g$ as in (i). 
Then there are generators $g_i \ne g_j$ such that $C_G(g_i) \not\subset C_G(g_j)$
and $C_G(g_j) \not\subset C_G(g_i)$.
Then $H \subset C_G(g_i) \cap C_G(g_j)$ where the intersection is 
the subgroup of invertible scalar matrices. Thus $C_G(H) = G$ and again it contains
matrices with non-square determinant. 

We conclude that under our hypothesis on $C_G(H)$ both cases (i) and (ii) are impossible;
thus $H$ has order divisible by $p$. 
Since $H$ is abelian then, up to 
conjugation, it is contained 
in the Borel subgroup by \cite[Proposition~15]{Serre72}.
Also, all elements of $H$ have one eigenvalue with multiplicity two, 
otherwise $H$ is not abelian.
Thus $H$ is 
isomorphic to $D \times U$ where $U$ is unipotent of order $p$
and $D$ is a subgroup of scalar matrices with order coprime to $p$. 
In particular, $D$ and $U$ are cyclic with coprime order, hence $H$ is cyclic
and generated by a matrix $g = \begin{psmallmatrix} a & 1 \\ 0 & a \end{psmallmatrix}$ where $a \in \F_p^*$.
We conclude that (a) implies (b) and (b) is equivalent to (c).

To finish we note that the centralizer $C_G(g)$ is the group of matrices of the form
$\begin{psmallmatrix} \lambda & b \\ 0 & \lambda \end{psmallmatrix}$ with $\lambda \ne 0$; thus
all of them have square determinant, showing that (c) implies (a).
\end{proof}

\begin{remark} For $p=2$, it is still true that, in the previous theorem, 
(b) is equivalent to (c), 
since all the elements of order~$2$ in $\GL_2(\F_2)$ are conjugated 
to~$\begin{psmallmatrix} 1 & 1 \\ 0 & 1 \end{psmallmatrix}$.
\end{remark}

We now prove the main theorem of this section.

\begin{proof}[Proof of Theorem~\ref{T:conditionRho}]
From Lemma~\ref{L:determinant} with $K = F$ we see that
a symplectic criterion exists precisely 
when the centralizer of $\rhobar_{E,p}(G_F)$ 
in $\GL_2(\F_p)$ contains only matrices with square determinant mod~$p$.
From Lemma~\ref{L:centralizerNonAbelian} we know 
this is always the case when $\rhobar_{E,p}(G_F)$ 
is non-abelian, hence part (A) follows. Part (B) follows directly from 
condition (c) in Lemma~\ref{L:CentralizerAbelian}.
\end{proof}

We extract an important  consequence of the previous discussion. 

\begin{corollary} Let $E$ and $E'$ be elliptic curves over~$K$ with isomorphic
$p$-torsion. Let $H \subset G_{K}$ be a subgroup such that $\rhobar_{E,p}(H)$ 
is a non-abelian subgroup of $\GL_2(\F_p)$. 

Then $E[p]$ and $E'[p]$ cannot be simultaneously symplectically 
and anti-symplectically isomorphic $G_K$-modules.

Moreover, the symplectic type of any $G_K$-isomorphism $\phi : E[p] \simeq E'[p]$ is determined by the 
$H$-module structures of $E[p]$ and $E'[p]$.
\label{C:notBoth}
\end{corollary}
\begin{proof} 
Let $\phi : E[p] \simeq E'[p]$ be a $G_K$-isomorphism and 
write $\phi|_H : E[p] \simeq E'[p]$ for the isomorphism of $H$-modules obtained by restricting the action. Fix symplectic basis for $E[p]$ and $E'[p]$ and let 
$M$ and $M_H$ be matrices representing $\phi$ and $\phi_H$ in these basis, respectively. We have that
\[
 \rhobar_{E',p}(h) = M\rhobar_{E,p}(h)M^{-1}, 
 = M_H\rhobar_{E,p}(h)M_H^{-1}, 
\]
for all $h \in H$, hence $M^{-1}M_H$ centralizes $\rhobar_{E,p}(H)$. Since this group is non-abelian by assumption, Lemma~\ref{L:centralizerNonAbelian} implies 
$\det(M)=\det(M_H)\cdot m^2$. The lemma now follows from Lemma~\ref{L:determinant}.
\end{proof}

\subsection{Revisiting Problem~A}
We will now discuss how the ideas above relate to the setup of Problem A. 
Recall from the introduction that, given elliptic curves 
$E/\Q$ and $E'/\Q$ with isomorphic $p$-torsion, we consider the triple~$(E,E',p)$ in Problem A only if it has a well defined symplectic type.
That is, when all the $G_\Q$-isomorphisms between $E[p]$
and $E'[p]$ have the same symplectic type. 
With this in mind, we have the following consequence of the above discussion.

\begin{corollary} \label{C:problemA}
Let $E/\Q$ and $E'/\Q$ be elliptic curves with isomorphic $p$-torsion. Suppose that 
$\rhobar_{E,p}$ is irreducible. Then all the $G_\Q$-isomorphisms between $E[p]$
and $E'[p]$ have the same symplectic type, that is, the symplectic type of 
$(E,E',p)$ is well defined.
\end{corollary}
\begin{proof}
The mod~$p$ representation $\rhobar_{E,p} : G_\Q \to \GL_2(\F_p)$ is odd and
irreducible, hence it is absolutely irreducible (see \cite[Lemma~24]{DahmenPhD}).
This forces its image to be non-abelian (see \cite[VII.47, Proposition~19]{BourbAlgII})
and the conclusion follows from 
Corollary~\ref{C:notBoth}.
\end{proof}

There are however elliptic curves $E$ and $E'$ over~$\Q$ satisfying $E[p] \simeq E'[p]$ for which isomorphisms with 
both symplectic types exist. This occurs if and only if $E[p]$ admits an anti-symplectic automorphism; we illustrate
this with the following example.

\begin{example} \label{Ex:mod5}
Take $p=5$ and $E$, $E'$ to be the curves 
with Cremona labels $11a1$ and $1342c2$, respectively; 
both have $5$-torsion module isomorphic to $\mu_5 \times \Z/5\Z$.
Now let $P,Q \in E[5]$ and $P',Q' \in E'[5]$ be basis such that $P$, $P'$ are defined over~$\Q$.
The map defined by $P \mapsto P'$ and $Q \mapsto n\cdot Q'$ (with $5 \nmid n$)
is a symplectic $G_\Q$-isomorphism if and only if $n$ is a square mod~$5$.  
Moreover, the automorphism $\alpha$ of $E[5]$ given by $\alpha(P) = P$ and $\alpha(Q) = 2Q$ 
is anti-symplectic because $2$ is not a square modulo~$5$. 
Thus, the symplectic 
type of $(11a1,1342c2,5)$ is not well defined.
\end{example}
It is a direct consequence of the Corollary~\ref{C:problemA} 
and Mazur's~\cite{Mazur} 
work on isogenies that for $p > 163$ there is only one symplectic type of 
isomorphisms $E[p] \simeq E'[p]$, provided that one exists. The next proposition 
shows that this is already the case for $p \geq 7$. Therefore, 
the problematic cases as in the previous example can only occur for $p \leq 5$.
\begin{proposition} \label{P:typeTriple}
Let $p \geq 7$ be a prime.
If $E/\Q$ is an elliptic curve, then $\rhobar_{E,p}(G_\Q)$ is non-abelian. Thus, if $E'/\Q$ is another elliptic curve with $E[p] \simeq E'[p]$ as $G_\Q$-modules, then the symplectic type of $(E,E',p)$ is well-defined.
\end{proposition}
\begin{proof} Note that the case
when~$\rhobar_{E,p}$ is irreducible follows from Corollary~\ref{C:problemA} and its proof, so we can suppose  that $\rhobar_{E,p}$ is reducible.
The reducibility of $\rhobar_{E,p}$ means that
\[
 \rhobar_{E,p} \simeq \begin{pmatrix} \epsilon_1 & * \\ 0 & \epsilon_2 \end{pmatrix} \quad \text{ with } \quad
 \epsilon_1 \epsilon_2 = \chi_p, 
\] 
where $\epsilon_i : G_\Q \to \F_p^*$ and $\chi_p$ is the mod~$p$ cyclotomic character. 
Since $\chi_p$ is surjective onto $\F_p^*$ it cannot be a square of a character, hence 
$\epsilon_1 \neq \epsilon_2$. 

From Mazur's~\cite{Mazur} results on isogenies 
we know that $p \in \{7, 11, 17, 19, 37, 43, 67, 163\}$.
Moreover, from \cite[Remark 2.1]{GRSS} it follows that $* \neq 0$. 
Since the characters on the diagonal are different the
image of $\rhobar_{E,p}$ is non-abelian, proving the first statement. 

The second statement follows from
Corollary~\ref{C:notBoth}.
\end{proof}

\begin{remark}
 The first statement in Proposition~\ref{P:typeTriple} also follows from \cite[Theorem~1.1]{GL}.
\end{remark}

Let us finish this section by clarifying part of our terminology.
Note that the name `symplectic criterion' is reserved for the setting of curves over local fields (cf. discussion before Theorem~\ref{T:conditionRho}). 
In view of Lemma~\ref{L:determinant}, which gives a condition for the existence of a well defined symplectic type only in terms of 
the image of $\rhobar_{E,p}$ (and not depending on the field where $E$ is defined), this might seem confusing. This terminology is already well established in the literature and the 
reason is that we are trying to solve a global problem (Problem A) by a local method (a symplectic criterion).
Therefore, the statement of Theorem~\ref{T:conditionRho}
needs to be restricted to local fields, as it is about
symplectic criteria.

In particular, when a triple 
$(E,E',p)$ has a well defined symplectic type, this does not mean there exists a symplectic criterion associated with it. 
Indeed, there can be a {\bf global} symplectic type for 
$(E,E',p)$ and nevertheless, 
for all primes~$\ell \neq p$,
the {\bf local} $G_{\Q_\ell}$-isomorphic modules
$(E/\Q_\ell)[p]$ 
and $(E/\Q_\ell)[p]$ being symplectically and anti-symplectically isomorphic.
In fact, the (global) symplectic type of 
$(E,E',p)$ is determined by a (local) symplectic criterion if and only if the set $\calL_{(E,E',p)}$ (defined in the introduction) is non-emppty; we refer to Proposition~\ref{P:emptyL}
for an example where
$\calL_{(E,E',p)}$ is empty. 

\section{Symplectic criteria with $\rhobar_{E,p}(G_{\Q_\ell})$ abelian}
\label{S:abelian}

The proof of the symplectic criterion for the case when $E$ and $E'$ have (potentially) multiplicative 
reduction (see \cite[Proposition~2]{KO})
exploits the fact that, when $E$ is a 
Tate curve such that $p \nmid \vv_\ell(\Delta_m(E))$,   
the image of inertia $\rhobar_{E,p}(I_\ell)$ satisfies
condition~(B) of Theorem~\ref{T:conditionRho}. 
This section is the first step in understanding, 
which criteria may exist under the same condition~(B) when 
at least one of the curves 
has potentially 
good reduction.
\begin{proposition} Let $\ell$ and $p \geq 3$ be different primes. 
Let $E/\Q_\ell$ be an elliptic curve with potentially good reduction
and semistability defect $e$. Suppose that $\rhobar_{E,p}(G_{\Q_\ell})$ is abelian and its centralizer 
in $\GL_2(\F_p)$ contains only matrices with square determinant. Then
\begin{enumerate}
 \item If $p \geq 5$ we have $e=1$ or $e=2$;  
 \item If $p = 3$ we have $\ell \equiv 1 \pmod{3}$ and $e=1,2,3$ or $6$.
\end{enumerate}
\label{P:abeliane}
\end{proposition}

\begin{proof} From Lemma~\ref{L:CentralizerAbelian} it follows that, up to conjugation, $\rhobar_{E,p}(G_{\Q_\ell})$ is generated by a matrix of the form 
$\begin{psmallmatrix} a & 1 \\ 0 & a \end{psmallmatrix}$ where $a \in \F_p^*$.
The determinant of $\rhobar_{E,p}$ is the mod~$p$ cyclotomic character (\cite[Ch. I, \S I-3]{SerreBook})  
which is unramified at all $\ell \neq p$. Then, for an element of inertia 
$g \in I_\ell \subset G_{\Q_\ell}$, we have 
\[
\rhobar_{E,p}(g) = \begin{pmatrix} b & c \\ 0 & b \end{pmatrix} \quad \text{ with } \quad
\det \rhobar_{E,p}(g) = b^2 = 1.
\]
Moreover, we know from Section~\ref{S:left?}
that the image of inertia is cyclic of order $e \in \{1,2,3,4,6\}$, hence
we have~$c=0$ if $p \geq 5$; this proves (1).
If $p=3$, then we can have $c\neq 0$ and we see that $e=3,6$ is also possible.
Further, the condition $\det \rhobar_{E,3} = \chi_3 = 1$ implies $\ell \equiv 1 \pmod{3}$,
proving~(2).
\end{proof}

Let $E/\Q_\ell$ be an elliptic curve with potentially good reduction. If $E$ 
obtains good reduction at most after a quadratic twist, i.e. 
with semistability defect $e=1,2$, then $\rhobar_{E,p}(G_{\Q_\ell})$ 
is abelian (see \cite[Lemma~9]{FK2}). 
Conversely, we conclude from Proposition~\ref{P:abeliane} 
that, for $p\geq 5$, symplectic criteria under condition~(B) of Theorem~\ref{T:conditionRho}
may exist only in the case of good reduction (up to twist).
For $p=3$, the cases $e=1,2$ are possible but
$e=3,6$ also needs to be considered (and are covered by Theorems~\ref{T:e=p=3}~and~\ref{T:mixedReduction}).

The case where both $E$ and $E'$ have good reduction is summarized in the first row of Table~\ref{Table:CriteriaList};
the extra conditions on that row are established by Propositions~\ref{P:tableConditionI} and~\ref{P:AbelianCentralizer}.

\begin{proposition} Let $\ell$ and $p \geq 3$ be primes satisfying $(\ell/p) = -1$. 
Let $E/\Q_\ell$ be an elliptic curve with potentially good reduction with $e \in \{1,2 \}$.

Then the centralizer of $\rhobar_{E,p}(G_{\Q_\ell})$ in $\GL_2(\F_p)$ 
contains matrices with non-square determinant.
\label{P:tableConditionI}
\end{proposition}

\begin{proof} 
Write $G = \GL_2(\F_p)$. Let $E' = dE$ be a quadratic twist with good reduction.
We have $\rhobar_{E',p} \sim \rhobar_{E,p} \otimes \chi_d$
and $C_G(\rhobar_{E,p}(G_{\Q_\ell})) = C_G(\rhobar_{E',p}(G_{\Q_\ell}))$.

Let $\Frob_\ell$ be a Frobenius element in $G_{\Q_\ell}$.
The group $\rhobar_{E',p}(G_{\Q_\ell})$ is cyclic and generated 
by $\rhobar_{E',p}(\Frob_\ell)$. In particular, it is contained in its own 
centralizer and since $\det \rhobar_{E',p}(\Frob_\ell) = \ell$ is not a square
the result follows.
\end{proof}

{\large \part{The criterion in the case of good reduction}}
\label{P:good}

The objective of this part is to 
establish symplectic criteria in the case where both $E$ and~$E'$ have good reduction,   
that is Theorems~\ref{T:simpleAbelian}~and~\ref{T:mainAbelian}.

\section{The action of Frobenius} \label{S:good}

For a curve $E/\Q_\ell$ with good 
reduction, the image of 
$\rhobar_{E,p}$ is cyclic and 
generated by the image of Frobenius. 
To study the centralizer of $\rhobar_{E,p}(G_{\Q_\ell})$ in this case, we will use the explicit matrix describing the Frobenius action given by
Centeleghe~\cite{Centeleghe}.

We first need to recall some notation and a result 
from~{\it loc. cit.}.

Let $D$ be a negative discriminant, i.e. $D < 0$ satisfies $D \equiv 0,1 \pmod{4}$.
Denote by $\calO_D$ the imaginary quadratic order of discriminant $D$, 
viewed inside the field of complex numbers $\C$. 
Consider the following polynomial
\[
  \mathcal{P}_D = \prod_{\calO_D \subset \End(\C/\mathfrak{a})} (x - j_{\C/\mathfrak{a}}),
\]
where the product runs over 
the isomorphism classes of 
the elliptic curves $\C/\mathfrak{a}$ whose endomorphisms ring contains
the order $\calO_D$. From \cite[p. 144]{LangEF} we known that $\mathcal{P}_D \in \Z[x]$ and that the $j$-invariants $j_{\C/\mathfrak{a}}$ and $j_{\C/\mathfrak{a}'}$ lie in the same $G_\Q$-orbit 
if and only if $\End_\C(\C/\mathfrak{a}) = \End_\C(\C/\mathfrak{a}')$. 
The irreducible factors of $\mathcal{P}_D$ are the Hilbert class polynomials 
of the corresponding imaginary orders.
We extend the definition of the polynomial to all $D \leq 0$ by setting $\mathcal{P}_0 = 0$ and
$\mathcal{P}_D = 1$ for $D \equiv 2,3 \pmod{4}$.

Let $E/\Q_\ell$ be an elliptic curve with good reduction and $j$-invariant $j_E$. 
Write $\Ebar /\F_\ell$ for the elliptic curve obtained by reducing mod~$\ell$
a minimal model of $E$. Define the integer quantities
\[
 a_\ell = (\ell + 1) - \#\Ebar(\F_\ell) \qquad \text{ and } \qquad 
 \Delta_\ell = a_\ell^2 - 4\ell.
\]
Note that $\Delta_\ell \neq 0$ and, from the Hasse-Weil bound, 
we have $\lvert a_\ell \rvert \leq 2 \sqrt{\ell}$, 
hence $\Delta_\ell < 0$. The value $a_\ell$ is also called
the trace of Frobenius.
Define also the quantity $\beta_\ell$ by the formula
\[
 \beta_\ell = \sup_{h > 0} \{ \; h : h^2 \mid \Delta_\ell \; \text{ and } \; 
 \mathcal{P}_{\Delta_\ell/h^2}(j_E) \equiv 0 \pmod{\ell}  \}
\]
which is an integer because $\Delta_\ell \neq 0$.

Since $p \neq 2$
it follows from \cite[Theorem~2]{Centeleghe}
that there is a $\Z_p$-basis of the $p$-adic Tate 
module $T_p(\Ebar)$ such that the action of 
the Frobenius element $\Frob_\ell$ 
on it is given by the $\GL_2(\Z_p)$ 
matrix
\begin{equation}
  F_\ell = \begin{pmatrix} \frac{a_\ell \beta_\ell - \Delta_\ell}{2\beta_\ell} 
 & \frac{\Delta_\ell (\beta_\ell^2 - \Delta_\ell)}{4\beta_\ell^3} \\ 
 \beta_\ell & \frac{a_\ell \beta_\ell + \Delta_\ell}{2\beta_\ell} \end{pmatrix}.
 \label{E:FrobChar0}
\end{equation}
The reduction map induces an identification of $E[p](\Qbar_\ell)$ with $\Ebar[p](\Fbar_\ell)$ which is Galois equivariant. Therefore, we can reduce $F_\ell$ to obtain the matrix 
giving the action of the Frobenius of $\Gal(\Fbar_\ell/\F_\ell)$
on $\Ebar[p](\Fbar_\ell)$ which is equal to
the matrix $\rhobar_{E,p}(\Frob_\ell)$,  giving the action of $\Frob_\ell$ on $E[p]$. 
Suppose further that $p \mid \Delta_\ell$ and $p \nmid \beta_\ell$. 
Then there exists a basis of $E[p]$ where the action of $\Frob_\ell$ is via the matrix
\begin{equation}
  \rhobar_{E,p}(\Frob_\ell) = 
 \begin{pmatrix} \frac{a_\ell}{2}  & 0 \\ 
 \beta_\ell &  \frac{a_\ell}{2} \end{pmatrix} \pmod{p} \quad 
 \text{ with} \quad \beta_\ell \not\equiv 0 \pmod{p}.
 \label{E:FrobCharp}
\end{equation}

\begin{proposition} Let $\ell \neq p$ be primes with $p \geq 3$.
Let $E/\Q_\ell$ be an elliptic curve with good reduction.
Then $\rhobar_{E,p}(G_{\Q_\ell})$ has a centralizer in $\GL_2(\F_p)$ containing 
only matrices with square determinant if and only if
$p \mid \Delta_\ell$ and $p \nmid \beta_\ell$.
\label{P:AbelianCentralizer}
\end{proposition}
\begin{proof} 
Since $E/\Q_\ell$ has good reduction then $H = \rhobar_{E,p}(G_{\Q_\ell})$
is generated by~$\rhobar_{E,p}(\Frob_\ell)$,  hence it is abelian (cyclic). 
Suppose $H$ has a centralizer in $\GL_2(\F_p)$ containing 
only matrices with square determinant.

Moreover, from Lemma~\ref{L:CentralizerAbelian} it follows that (up to conjugation) 
$H$ is generated by a matrix of the
form $\begin{psmallmatrix} a & 1 \\ 0 & a \end{psmallmatrix}$ where $a \in \F_p^*$.
In particular, (i) $\rhobar_{E,p}(\Frob_\ell)$ has two equal eigenvalues and (ii) 
$\rhobar_{E,p}(\Frob_\ell)$ has order divisible
by $p$. Since the characteristic polynomial of $F_\ell$ (given in \eqref{E:FrobChar0})
is $x^2 - a_\ell x + \ell$ 
it follows that (i) occurs exactly when $p \mid \Delta_\ell$; assuming 
$p \mid \Delta_\ell$ we see from \eqref{E:FrobCharp} that if (ii) holds 
then $p \nmid \beta_\ell$.

For the other direction, suppose that $p \mid \Delta_\ell$ and $p \nmid \beta_\ell$. Thus
$\rhobar(\Frob_\ell)$ is given by \eqref{E:FrobCharp} which generates a cyclic group
whose centralizer are the matrices of the form
$\begin{psmallmatrix} \lambda & 0 \\ b & \lambda \end{psmallmatrix}$ with $\lambda \ne 0$. In particular, they all have square determinant, as desired.
\end{proof}

\begin{corollary} Let $\ell \neq p$ be primes with $p \geq 3$.
Let $E/\Q_\ell$ be an elliptic curve with good reduction.
Then $\rhobar_{E,p}(\Frob_\ell)$ has order divisible by $p$
if and only if 
$p \mid \Delta_\ell$ and $p \nmid \beta_\ell$.
\label{C:FrobOrderp}
\end{corollary}
\begin{proof} Since $E$ has good reduction we have that $H = \rhobar_{E,p}(G_{\Q_\ell})$ is 
cyclic and generated by $\Frob_\ell$. 
Suppose that $\rhobar_{E,p}(\Frob_\ell)$ has order divisible by $p$. Thus by Lemma~\ref{L:CentralizerAbelian}
we conclude that the centralizer of $H$ in $\GL_2(\F_p)$ has only matrices with square determinant, hence 
by Proposition~\ref{P:AbelianCentralizer} we conclude that 
$p \mid \Delta_\ell$ and $p \nmid \beta_\ell$. The other direction follows by observing that all 
the steps in the previous argument are equivalences.
\end{proof}

\section{Proof of Theorem~\ref{T:simpleAbelian}}
\label{S:thm12}

By hypothesis, we can choose minimal models with good reduction for $E$ and $E'$ such that the residual curves satisfy~$\Ebar = \Ebar'$. 
Let $\varphi : E[p] \rightarrow \overline{E}[p]$ 
and $\varphi' : E'[p] \rightarrow \overline{E}[p]$ be the reduction morphisms 
which are Galois-equivariant.
Write $\rhobar$ for the representation giving the action of $\Gal(\Fbar_\ell / \F_\ell)$ on~$\Ebar[p]$. The element $\Frob_\ell \in G_{\Q_\ell}$ 
is a lift of the Frobenius 
automorphism $\bar{\tau} \in \Gal(\Fbar_\ell / \F_\ell)$ and we have
\[ \varphi \circ \rhobar_{E,p}(\Frob_\ell) = \rhobar(\bar{\tau}) \circ \varphi 
 = \varphi' \circ \rhobar_{E',p}(\Frob_\ell). 
\]
Fix a symplectic basis for $\Ebar[p]$ and let $\bar{N}$ be the matrix representing $\rhobar(\bar{\tau})$ in that basis.
Lift the fixed basis to symplectic bases of $E[p]$ and $E'[p]$ using $\varphi$, $\varphi'$, respectively. The lifted basis are symplectic
and in these basis the matrices representing $\varphi$ and $\varphi'$ in $\GL_2(\F_p)$ are the identity. Thus, viewed as elements of $\GL_2(\F_p)$, we have 
$\rhobar_{E,p}(\Frob_\ell) = \rhobar_{E',p}(\Frob_\ell) = \bar{N}$.
This implies that $E[p]$ and $E'[p]$ are isomorphic $G_{\Q_\ell}$-modules, 
because in the case of good reduction the action of $\Frob_\ell$ 
suffices to determine them. This proves part 1).

Moreover, in the same symplectic bases, we have
\[
 \rhobar_{E,p}(\Frob_\ell) = M \rhobar_{E',p}(\Frob_\ell) M^{-1} \quad  
 \text{ for all }  \quad M \in C_{\GL_2(\F_p)}(\rhobar_{E,p}(G_{\Q_\ell})).
\]
Suppose $p \mid \Delta$ and $p \nmid \beta_\ell$; then by
Proposition~\ref{P:AbelianCentralizer} it follows that $C_{\GL_2(\F_p)}(\rhobar_{E,p}(G_{\Q_\ell}))$ 
contains only matrices with square determinant and 
part 2) follows 
from Lemma~\ref{L:sympcriteria}. 

Finally, we note that the equivalent hypothesis 
on the order of $\rhobar_{E,p}(\Frob_\ell)$ is Corollary~\ref{C:FrobOrderp}.

\section{A more general theorem}
\label{S:genGood}

For an elliptic curve $E/\Q_\ell$, two points $P,Q \in E[p]$ form 
an anti-symplectic basis if their Weil-pairing satisfies $e_{E,p}(P,Q) = \zeta_p^{r}$ 
with $r$ not a square mod~$p$. For elliptic curves $E/\Q_\ell$ and $E'/\Q_\ell$ with good reduction
we let $\Delta_\ell \neq 0$, $\beta_\ell$ and $\Delta_\ell' \neq 0$, $\beta'_\ell$ be the quantities
defined in Section~\ref{S:good}, respectively.
We will prove the following generalization of Theorem~\ref{T:simpleAbelian} part 2).

\begin{theorem} Let $E$ and $E'$ be elliptic curves over $\Q_\ell$ with good reduction.
Let $p \geq 3$ be a prime and suppose that $E[p]$ and $E'[p]$ 
are isomorphic $G_{\Q_{\ell}}$-modules.

Assume that $p \mid \Delta_\ell$ and $p \nmid \beta_\ell$. 
Then $p \mid \Delta_\ell'$ and $p \nmid \beta_\ell'$. 

Write $s=0$ if there are simultaneously symplectic basis or anti-symplectic basis of $E[p]$ and $E'[p]$ 
such that the action of $\Frob_\ell$ on $E[p]$ and $E'[p]$ is given by \eqref{E:FrobCharp}; write
$s=1$ otherwise. Then
\[
 E[p] \text{ and } E'[p] \quad \text{are symplectically isomorphic} \quad 
 \Leftrightarrow \quad (-1)^s \left (\frac{\beta_\ell' / \beta_\ell}{p} \right) = 1.
\]
Moreover,
$E[p]$ and $E'[p]$ are not both symplectically and 
anti-symplectically isomorphic. 
\label{T:mainAbelian}
\end{theorem}
\begin{proof}
Write $G = \GL_2(\F_p)$ and~$C = C_G(\rhobar_{E,p}(G_{\Q_\ell}))$,
~$C' = C_G(\rhobar_{E',p}(G_{\Q_\ell}))$ for the centralizers.

Suppose $p \mid \Delta_\ell$ and $p \nmid \beta_\ell$. 
From Proposition~\ref{P:AbelianCentralizer} it follows that $C$ 
 contains only matrices with square determinant. 
 Since $\rhobar_{E,p} \sim \rhobar_{E',p}$ we have that $C'$
 has the same property, hence $p \mid \Delta_\ell'$ and $p \nmid \beta_\ell'$ by Proposition~\ref{P:AbelianCentralizer}.
 Now from the last statement of Lemma~\ref{L:sympcriteria} it follows that 
 $E[p]$ and $E'[p]$ cannot be both symplectically and anti-symplectically isomorphic. 
 This proves the first and last claim of the theorem.
 
Choose basis of $E[p]$ and $E'[p]$ such that the action of $\Frob_\ell$ is given by 
\eqref{E:FrobCharp}, that is, 
 \begin{equation}
 \rhobar_{E,p}(\Frob_\ell) = 
 \begin{pmatrix} \frac{a_\ell}{2}  & 0 \\ 
 \beta_\ell &  \frac{a_\ell}{2} \end{pmatrix} 
  \quad \text{ and } \quad 
 \rhobar_{E',p}(\Frob_\ell) = 
 \begin{pmatrix} \frac{a_\ell}{2}  & 0 \\ 
 \beta_\ell' &  \frac{a_\ell}{2} \end{pmatrix}, 
 \label{E:MatrixReps}
\end{equation}
where the matrices are in $\GL_2(\F_p)$, $\beta_\ell, \beta_\ell' \ne 0$
and we used $a_\ell \equiv a_\ell' \pmod{p}$ since $\rhobar_{E,p} \sim \rhobar_{E',p}$.

Since the matrices above generate the same subgroup
of $\GL_2(\F_p)$ (even when $\beta_\ell \ne \beta_\ell'$)
it follows that
$\rhobar_{E,p}(\Frob_\ell) = M \rhobar_{E',p}(\Frob_\ell) M^{-1}$ for some $M$ 
in the normalizer $N_G(\rhobar_{E,p}(\Frob_\ell))$.
Write $M = \begin{psmallmatrix} a  & 0 \\ c &  d \end{psmallmatrix}$. 
From the identity
\[
\begin{pmatrix} \frac{a_\ell}{2}  & 0 \\ 
 \beta_\ell &  \frac{a_\ell}{2} \end{pmatrix} = 
 \begin{pmatrix} a  & 0 \\ c &  d \end{pmatrix} \begin{pmatrix} \frac{a_\ell}{2}  & 0 \\ 
 \beta_\ell' &  \frac{a_\ell}{2} \end{pmatrix}
 \begin{pmatrix} a  & 0 \\ c &  d \end{pmatrix}^{-1}  
\]
we see that $a/d = \beta_\ell' / \beta_\ell$, therefore 
$\det M$ and $\beta_\ell' / \beta_\ell$ are simultaneously
a square or a non-square mod~$p$.
 
Suppose $s=0$ and that the matrix representations in~\eqref{E:MatrixReps} 
were obtained using symplectic bases
for both $E[p]$ and $E'[p]$. From Lemma~\ref{L:sympcriteria} it follows that 
$E[p]$ and $E'[p]$ are symplectically isomorphic if and only if $\beta_\ell' / \beta_\ell$ is 
a square mod~$p$.

Suppose $s=0$ and that the matrix representations in~\eqref{E:MatrixReps} 
were obtained using anti-symplectic bases for both $E[p]$ and $E'[p]$.
Then applying a change of basis with non-square
determinant to $E[p]$ and $E'[p]$ we obtain matrix representations with respect to a symplectic
basis to which we can again apply Lemma~\ref{L:sympcriteria}. 
Note that the determinant of the matrix representing the isomorphism in this new choice of bases 
is obtained by multiplying $\det M$ by the product of 
two non-squares which is a square. Thus again by Lemma~\ref{L:sympcriteria} we have 
$E[p]$ and $E'[p]$ symplectically isomorphic if and only if $\beta_\ell' / \beta_\ell$ is 
a square mod~$p$.

Suppose $s=1$. Then the matrix representations in~\eqref{E:MatrixReps} 
were obtained using one symplectic basis and one anti-symplectic basis.
Now, after applying one change of basis with non-square determinant 
we can assume both basis are symplectic. This
changes $\det M$ by a non-square. From Lemma~\ref{L:sympcriteria} we have 
$E[p]$ and $E'[p]$ symplectically isomorphic if and only if $\beta_\ell' / \beta_\ell$ 
is not square mod~$p$.

This exhausted all the cases for the value of
$(-1)^s \left (\frac{\beta_\ell' / \beta_\ell}{p} \right)$
and completes the proof.
\end{proof}

We did not included Theorem~\ref{T:mainAbelian} in Section~\ref{S:results} because it is not as simple to use as the other criteria. This is due to the fact that determining the value of~$s$ in its statement requires plenty of computations; see Example~\ref{Ex:mod7} for an application of Theorem~\ref{T:mainAbelian}.

{\large \part{Elliptic curves with potentially good reduction}}

We will now study various phenomena of elliptic curves with potentially good reduction
which are essential to the main proofs in later sections. The results in this part
are also of independent interest.

\section{An useful Weierstrass model}

Let $E/\Q_\ell$ be an elliptic curve with potentially good reduction. 
We write $K=\Q_\ell(E[p])$.
For all $m \in \Z_{\geq 3}$ coprime to $\ell$ the extension $L = \Q_\ell^{un}(E[m]) = \Q_\ell^{un} K$ is the minimal extension of~$\Q_\ell^{un}$ where $E$ achieves good reduction. 
We write $e = [L : \Q_\ell^{un}]$ for the semistability defect of~$E$. 
For a minimal model for $E/\Q_\ell$ we recall the 
quantities $\tilde{c}_4$, $\tilde{c}_6$, $\tilde{\Delta}$ 
defined in Section~\ref{S:results} 
for $c_4 \neq 0$ or $c_6 \neq 0$~by
\[
 c_4 = \ell^{\vv_\ell(c_4)} \tilde{c}_4, \qquad c_6 = \ell^{\vv_\ell(c_6)} \tilde{c}_6, \qquad \Delta_m = \ell^{\vv_\ell(\Delta_m)}\tilde{\Delta}.
\]
\begin{proposition} \label{P:usefulmodel}   
Let $E/\Q_\ell$ be an elliptic curve
having additive potentially good reduction, semistability defect $e$ and conductor $N_E$.
Suppose further $e \neq 2,6,24$ and $N_E \neq 2^6$ if $e=8$.

Then $E$ has a minimal model whose invariants 
$c_4$, $c_6$ and $\Delta$ satisfy, in relation to $e$, 
the conditions in Table~\ref{Table:model}. Moreover, 
we can assume the minimal model to be of the form
\begin{equation}
 y^2 = x^3 + a x + b,\quad a = -\frac{c_4}{48}, \quad b = -\frac{c_6}{864}.
 \label{E:usefulmodel}
\end{equation}

\begin{footnotesize}
\begin{table}[htb]
$$
\begin{array}{|c|c|c|c|c|} \hline
\text{\sc Case}   & \text{\sc Prime } \ell & e &  (\vv(c_4), \vv(c_6), \vv(\Delta_m)) & \text{\sc Extra Conditions} \\ \hline
 A_3       &  \ell \geq 5   & 3 & (\geq 2, 2, 4) \; \text{ or } \; (\geq 3, 4, 8) & \text{none} \\
 A_4       &  \ell \geq 5   & 4 & (1, \geq 2, 3) \; \text{ or } \; (3, \geq 5, 9) & \text{none} \\ \hline
 B_3       &  \ell = 3      & 3 & (2, 3,4) \; \text{ or } \; (5, 8, 12) & \text{none} \\
 B_{4,i}   &  \ell = 3      & 4 & ( 2, \geq 5, 3) \; \text{ or } \; (4, \geq 8, 9); & \text{none} \\
 B_{4,ii}  &  \ell = 3      & 4 & (\geq 2, 3, 3) & \tilde{\Delta} \equiv 2,4 \pmod{9} \\
 B_{4,iii} &  \ell = 3      & 4 & (\geq 4, 6, 9) & \tilde{\Delta} \equiv 2,4 \pmod{9} \\ \hline
 
 C_4       &  \ell = 2      & 4 & (5, 8, 9) \; \text{ or } \; ( 7, 11, 15) & N_E = 2^8 \\
 C_{3,i}   &  \ell = 2      & 3 & (4,5,4) & \tilde{c}_4 \equiv -1 \pmod{4}, \; \tilde{c}_6 \equiv 1 \pmod{4} \\
 C_{3,ii}   &  \ell = 2      & 3 & (\geq 6,5,4) & \tilde{c}_6 \equiv 1 \pmod{4} \\
 C_{3,iii}   &  \ell = 2      & 3 & (4,6,8) & \tilde{c}_6 \equiv -1 \pmod{4}, \;
  \tilde{\Delta} \equiv -1 \pmod{4} \\
 C_{3,iv}   &  \ell = 2      & 3 & (\geq 7,7,8) & \tilde{c}_6 \equiv 1 \pmod{4} \\ \hline
 D_{a}    &  \ell = 2      & 8 & (4, n \geq 7,  6) & \tilde{c}_4 \equiv -1 \pmod{4}, \; N_E = 2^5 \\
 D_{b}    &  \ell = 2      & 8 & (6, n \geq 10, 12) & \tilde{c}_4 \equiv 1 \pmod{4}, \; N_E = 2^5 \\ 
 D_{c}    &  \ell = 2      & 8 & (7, 9, 12) & N_E = 2^5 \\ 
 D_{d}    &  \ell = 2      & 8 & (4,6,9)    & N_E = 2^5 \\
 D_{e}    &  \ell = 2      & 8 & (5, n \geq 9,  9)  & N_E = 2^8 \\ 
 D_{f}    &  \ell = 2      & 8 & (7, n \geq 12, 15) & N_E = 2^8 \\ \hline
 
 G_{a}   & \ell = 3 & 12 & (n \geq 2, 3,3) & \tilde{\Delta} \not\equiv 2,4 \pmod{9}, \; N_E = 3^3 \\ 
 G_{b}   & \ell = 3 & 12 & (n \geq 4, 6,9) & \tilde{\Delta} \not\equiv 2,4 \pmod{9}, \; N_E = 3^3 \\ 
 G_{c}   & \ell = 3 & 12 & (2, 4, 3) & N_E = 3^3 \\ 
 G_{d}   & \ell = 3 & 12 & (2,3,5) & N_E = 3^3 \\ 
 G_{e}   & \ell = 3 & 12 & (4,7,9) & N_E = 3^3 \\
 G_{f}   & \ell = 3 & 12 & (4,6,11) & N_E = 3^3 \\
 G_{g}   & \ell = 3 & 12 & (n \geq 3, 4, 5) & N_E = 3^5 \\ 
 G_{h}   & \ell = 3 & 12 & (n \geq 4, 5, 7) & N_E = 3^5 \\ 
 G_{i}   & \ell = 3 & 12 & (n \geq 5, 7, 11) & N_E = 3^5 \\ 
 G_{j}   & \ell = 3 & 12 & (n \geq 6, 8, 13) & N_E = 3^5 \\ \hline
\end{array}
$$
\caption{Relation between the semistability defect $e$ of 
an elliptic curve $E/\Q_\ell$ and the invariants $c_4(E)$, $c_6(E)$ and $\Delta_m(E)$ of 
a minimal model for $E$.}
\label{Table:model}
\end{table}
\end{footnotesize}
\end{proposition}
\begin{proof} From Section~\ref{S:left?} we know that $e \in \{ 2,3,4,6,8,12,24\}$
for an elliptic curve with additive potentially good reduction;
from the restrictions on $e$ in the statement 
it follows that the cases we have to consider are $e=3,4,8,12$.

For $\ell \geq 5$ the value of $e$ is given by (see \cite[p.355]{Kraus1990}) 
\begin{equation}
\label{E:denominator}
e=\text{denominator of}\  \frac{\vv_\ell(\Delta)}{12}.
\end{equation}
Now, from the assumption $e=3,4$ and 
\cite[Table~I]{pap} it follows that $(\vv(c_4), \vv(c_6), \vv(\Delta_m))$
is given as in the first two lines of the table.

For the pairs $(\ell,e)$ in the table with
$\ell = 2,3$, we see from the 
tables in \cite[pp. 356--359]{Kraus1990} that a minimal model satisfies the 
conditions on $(\vv(c_4), \vv(c_6), \vv(\Delta_m))$
and also the extra conditions on $\tilde{c}_4$, $\tilde{c}_6$ 
and $\tilde{\Delta}$. Further, the claims on $N_E$ follow 
from a case checking in \cite[tables~II and~IV]{pap} and 
the tables of \cite[pp. 39--40]{CaliThesis}.

We conclude from the above that
for each choice of $\ell$ and~$e$ in Table~\ref{Table:model} there exist a minimal model of $E/\Q_\ell$ satisfying the claimed conditions on $(\vv(c_4), \vv(c_6), \vv(\Delta_m))$ and on $\tilde{c}_4$, $\tilde{c}_6$, $\tilde{\Delta}$. Denote such model by $W'$ and 
write $c_4$, $c_6$, $\Delta_m$ for the standard invariants associated to~$W'$. 
From \cite[p. 43]{SilvermanI} we see that $W'$ can be transformed 
into a model of the form 
\[ W''/\Q_\ell \; : \; y^2 = x^3 - 27c_4 x - 54c_6. \]
Now the transformation $x = 6^2 x'$, $y = 6^3 y'$ transforms $W''$ into a model 
as in the statement
\[
W / \Q_\ell \; \; : \; \; y^2 = x^3 + a x + b,\quad a = -\frac{c_4}{48}, \quad b = -\frac{c_6}{864}.
\]
We observe that
\[
  \vv_2(48) = 4, \quad \vv_2(864) = 5, \quad \vv_3(48) = 1, \quad \vv_3(864) = 3
\]  
and  $\vv_\ell(48)=\vv_\ell(864) = 0$ if $\ell \geq 5$. Looking to the
valuations of $c_4$ and $c_6$ given in Table~\ref{Table:model} 
we conclude that $W$ satisfies $a, b \in \Z_\ell$ in all cases. 
Moreover, $c_4(W) = c_4$, $c_6(W) = c_6$, and the relation
\[
 2^6 3^3 \Delta(W) = c_4(W)^3 - c_6(W)^2 = c_4^3 - c_6^2 = 2^6 3^3 \Delta_m 
\]
gives $\Delta(W) = \Delta_m$. Thus $W$ is a minimal model satisfying all the conditions.
\end{proof}

\begin{remark} The cases $e=2,6,24$ are not covered by the previous proposition because the information we need when $e=2,6$ is given by Lemmas~\ref{L:lemma1}--\ref{L:lemma4} and the
the case $e=24$ does not have to be considered in this work. The latter case is already completely solved 
by \cite[Theorem~4]{F33p}, which is restated above as Theorem~\ref{T:Wilde24}. 
\end{remark}

\section{The field of good reduction}
\label{S:goodF}

Since we are dealing with curves $E/\Q_\ell$ with potentially good reduction it is natural to seek 
explicit descriptions of an extension $F/\Q_\ell$ such that $E/F$ has good reduction. 
For the purpose of proving symplectic criteria, we can further assume that 
the Galois group of the extension $K/\Q_{\ell}$ defined by the $p$-torsion of field 
of $E$ is non-abelian. In Theorem~\ref{T:goodOverF} below we describe the field $F$ 
in the cases of our interest; before we proceed to its proof we need to introduce 
some facts about certain Galois representations attached to $E$.

Let $I_\ell \subset W_\ell \subset G_{\Q_\ell} = \Gal(\Qbar_\ell / \Q_\ell)$ 
denote respectively the inertia and Weil subgroups of $G_{\Q_\ell}$.

Denote by $\omega$ the unramified quasi-character 
giving the action of $W_\ell$ on the roots of unity.

A $2$-dimensional complex Weil-Deligne representation of $\Q_\ell$ 
consist of a pair $(\sigma,N)$ such that $\sigma$ is a representation $\sigma : W_\ell \to \GL_2(\C)$ 
and $N$ is a nilpotent endomorphism of $\C^2$ satisfying
$\sigma(g)N\sigma(g)^{-1} = \omega(g)N$ for all $g \in W_\ell$;
see \cite[Section~3]{roh94} for further details. 

For an elliptic curve $E/\Q_\ell$ of conductor $N_E$ 
there is a Weil-Deligne representation $(\sigma_{E},N)$
attached to~$E$; 
when $E$ has potentially good reduction the 
nilpotent operator $N = 0$ because inertia has finite image
(see \cite[\S 13 and \S 14]{roh94}). The conductor of this representation coincides with~$N_E$ which can be computed using Tate's algorithm
(see~\cite[\S 9, \S 10, \S 11]{SilvermanII} and references there for a detailed discussion).

We call the restriction $\tau := \sigma_E|_{I_\ell}$ {\it the inertial type of $E$}.
The field extension of $L/\Q_\ell^{un}$ fixed by~$\tau$ is the minimal extension
of $\Q_\ell^{un}$ where $E$ obtains good reduction; this is the same extension 
we called the {\it inertial field of $E$} in Section~\ref{S:left?}. 

In \cite[Main Theorem]{DFV} a classification of all possible inertial types arising from elliptic curves
over~$\Q_\ell$ is given; in the case of potentially good reduction we can 
have either {\it principal series} or {\it supercuspidal} types. The former correspond to 
the case where $E/\Q_\ell$ obtains good reduction over some abelian extension 
while the latter corresponds to the case of a non-abelian extension. 
In our setting, we have that $E/K$ has good reduction and $K/\Q_\ell$ is non-abelian so that all
the curves in Theorem~\ref{T:goodOverF} below have supercuspidal inertial types; we refer 
to~\cite[Section~2]{DFV} for a summary of definitions and properties of inertial types.

\begin{theorem} Let $\ell \ne p$ be primes with $p \geq 3$.
Let $E/\Q_\ell$ be an elliptic curve with potentially good reduction, 
semistability defect $e$ and conductor $N_E$. Assume further that 
its $p$-torsion field extension $K = \Q_\ell(E[p])/\Q_{\ell}$ is non-abelian. 

Then there is a non-Galois totally ramified extension $F/\Q_\ell$ 
of degree $e$ such that $E/F$ has good reduction. 
More precisely, in each of the following cases, 
$E$ obtains good reduction over exactly one of the listed fields.
\begin{enumerate}
 \item if $e=3,4$ and $(\ell,e) = 1$ then $\ell \equiv -1 \pmod{e}$ and 
 we can take $F = \Q_\ell(\ell^{1/e})$;
 \item if $\ell = 2$ and $e=4$ then $F$ is defined by
 \[
  f_1 = x^4+12x^2+6 \quad \text{or} \quad f_2 = x^4+4x^2+6;
 \]
 \item if $\ell=2$, $e=8$ and $N_E = 2^5$ then $F$ is defined by
 \[
 g_1 = x^8 + 8x^4 + 336 \quad \text{ or } \quad g_2 = x^8 +4x^6 + 28x^4 + 20;
\]
 \item if $\ell=2$, $e=8$ and $N_E = 2^8$ then $F$ is defined by
 \[
 g_3 = x^8+20x^4+52 \quad \text{or} \quad g_4 = x^8+4x^4+84;
 \]
 \item if $\ell = e = 3$ then $F$ is defined by the polynomial $x^3+3x^2+3$;
 \item if $\ell = 3$, $e=12$ and $N_E = 3^3$ then $F$ is defined by
 \[
 h_1 = x^{12}-3x^{11}-3x^{10}+3x^9+3x^5-3x^4+3x^3+3 \quad \text{or} \quad h_2 = x^{12} + 3x^4 + 3;  
\]
\item if $\ell = 3$, $e=12$ and $N_E = 3^5$ then $F$ is defined by one of
 \begin{eqnarray*}
 h_3 &  = & x^{12} + 9x^{10} + 9x^9 - 9x^8 + 6x^6 + 9x^5 - 9x^4 - 3x^3 + 9x^2 - 9x-12, \\
 h_4 & = & x^{12} + 9x^{11} + 9x^{10} + 9x^9 + 9x^8 - 9x^7 - 12x^6 - 9x^2 - 3, \\
 h_5 & = & x^{12} - 9x^{11} + 9x^9 - 9x^8 + 9x^7 - 12x^6 + 3x^3 + 9x^2 + 9x - 12. 
\end{eqnarray*}
\end{enumerate}
Moreover, in case (2) either $E$ or its quadratic twist by $-1$ has good reduction 
over the field defined by $f_1$; in case (4) 
either $E$ or its quadratic twist by $2$ has good reduction over the field defined by $g_3$.
\label{T:goodOverF}
\end{theorem}
\begin{proof}
We first prove part (1). From \cite{Kraus1990} it follows that the Kodaira type of $E$ is IV or IV*
when $e=3$ and III or III* when $e=4$. Since $F=\Q_\ell(\ell^{1/e})/\Q_\ell$ is a tame extension 
a direct application of part (3) of \cite[Theorem~3]{DDKod} implies that $E/F$ has good reduction.

Suppose $F/\Q_\ell$ is Galois, hence cyclic. It follows that
that $L' = \Q_\ell^{un} F$ is abelian over $\Q_\ell$
and $E/L'$ has good reduction. 
By minimality we have $L = \Q_\ell^{un} K \subset L'$, therefore $L/\Q_{\ell}$ is 
abelian; since the Galois group of $K/\Q_{\ell}$ is a quotient of that of $L/\Q_{\ell}$ it follows that $K/\Q_{\ell}$ is also abelian, a contradiction.
We conclude that $F/\Q_\ell$ is not Galois.

Finally, we observe that $F$ is Galois if and only if $\ell \equiv 1 \pmod{e}$ and 
since $(\ell,e) = 1$ it follows that $\ell \equiv -1 \pmod{e}$, concluding the proof 
of (1).

We will now prove (2). 

Recall from the discussion preceding the theorem that our hypothesis imply that $E$ has a supercuspidal inertial type.
From \cite[Table~1]{DFV} we see that, 
for $\ell= 2$ and $e=4$, there are only two such types denoted $\tau_{sc,2}(5,4,4)$ and $\tau_{sc,2}(5,4,4) \otimes \varepsilon_{-4}$
in {\it loc. cit.}. 
In particular, these two types 
differ by the character corresponding to the quadratic twist by $-1$.

Write $F_i$ for the field defined by $f_i$. 
We check that the curve with Cremona label $256a1$ has good reduction over $F_1$ and
the curve $256d1$ has good reduction over $F_2$, hence their inertial fields 
are $L_1 = \Q_2^{un} F_1$ and $L_2 = \Q_2^{un} F_2$, respectively. Moreover 
$256a1$ has bad reduction over~$F_2$ and since $L_i / F_i$ is 
unramified it follows that $L_1 \ne L_2$; 
this shows that the fields $L_1$, $L_2$ are the inertial fields of the two possible types. 
Thus any elliptic curve $E$ satisfying the hypothesis must obtain good reduction
over exactly one of $L_1$ or $L_2$. Again, since $L_i / F_i$ is 
unramified we conclude that $E$ has good reduction
over exactly one of $F_1$ or $F_2$, proving (2).  

Part (3) follows similarly to (2). 
Indeed, for $\ell =2$, $e=8$ and $N_E = 2^5$ 
from~\cite[Table~1]{DFV} we see there are 
again exactly two inertial supercuspidal types, 
denoted $\tau_{sc,2}(-4,3,4)$ and $\tau_{sc,2}(-20,3,4)$ in {\it loc cit.}.
The conclusion now follows as above, where we use the curves $96a1$, $288a1$ 
instead of $256a1$ and $256d1$.

Part (4) follows similarly using 
the curves $256b1$, $256c1$. We note that, in this case, the two inertial types are related by a quadratic twist by~$2$. 

Part (5) follows similarly using 
the curve $162d1$.

Part (6) follows similarly using 
the curves $27a1$, $54a1$.

Part (7) follows similarly using 
the curves $243a1$, $243b1$, $972a1$. 

Finally, the last two statements follow because in cases (2) and (4) the two possible inertial types are related by quadratic twist by -1 and 2, respectively.
\end{proof}

We extract the following useful consequence from the 
proof of part (1) of Theorem~\ref{T:goodOverF}.

\begin{corollary} Let $\ell \ne p$ be primes with $p \geq 3$. 
Let $E/\Q_\ell$ be an elliptic curve with potentially good reduction with $e=3,4$. 
Write $K = \Q_\ell(E[p])$. Suppose there is a degree $e$ cyclic extension~$F/\Q_\ell$ such 
that $E/F$ has good reduction. Then $K/\Q_\ell$ is abelian.

Furthermore, if $(\ell,e) = 1$ then $K/\Q_\ell$ 
is non-abelian if and only if $\ell \equiv -1 \pmod{e}$.
\label{C:nonabelianTame}
\end{corollary}
\begin{proof} 
We note that the first paragraph in 
the previous proof holds for $F$;
thus, to prove the corollary,
we are left to show that $K/\Q_\ell$ abelian 
implies $\ell \equiv 1 \pmod{e}$. Indeed,
suppose $K/\Q_\ell$ is abelian. Then $L = \Q_\ell^{un} F  = \Q_\ell^{un} K$ is abelian 
over~$\Q_\ell$ and 
the Galois closure $\overline{F}/\Q_\ell$, which is a subfield of $L$, is also abelian. 
But when $\ell \equiv -1 \pmod{e}$ then $\overline{F}/\Q_\ell$ is a $D_e$-extension, hence non-abelian,
giving a contradiction. We conclude $\ell \equiv 1 \pmod{e}$, as desired.
\end{proof}

To apply Theorem~\ref{T:goodOverF} we need to know when a given curve has non-abelian 
$p$-torsion field extension. The previous corollary answers this 
in the case of tame potentially good reduction with $e=3,4$; 
the next propositions provide the answer in the case of wild reduction.

\begin{proposition} Let $p \geq 5$ be a prime. 
Let $E/\Q_3$ be an elliptic curve with potentially good reduction with $e=3$. 
Then $K/\Q_3$ is non-abelian if and only if $\tilde{\Delta} \equiv 2 \pmod{3}$.
\label{P:nonabelianWild3}
\end{proposition}
\begin{proof} First we note that 
condition 2 of \cite[Proposition~2.1]{Diamond} 
does not hold when $K/\Q_3$ is non-abelian and $p \nmid e$.

From Proposition~\ref{P:usefulmodel} we know that
$(\vv(c_4), \vv(c_6), \vv(\Delta_m)) = (2,3,4)$ or $(5,8,12)$.

Suppose $(\vv(c_4), \vv(c_6), \vv(\Delta_m)) = (2, 3,4)$. We are in the case $\vv_3(j)=2$ 
of Table~2 in \cite[Appendix]{Diamond}; since $p \nmid e=3$,  we conclude that
$\rhobar_{E,p}$ is of type {\bf V}, which holds
if and only if $j^*$ is not a square in $\Q_3$. 
Finally, $j^* = j_E - 1728 = c_6^2 / \Delta_m$ is a square if and only if $\Delta_m$ is a square; 
since $\Delta_m = 2^{\vv(\Delta_m)} \tilde{\Delta}$ and $\vv(\Delta_m)$ is even the statement follows.

Suppose $(\vv(c_4), \vv(c_6), \vv(\Delta_m)) = (5, 8,12)$. 
We are in the case $\vv_3(j)=3$ and $\vv_3(j^*)=4$ 
of Table~2 in \cite[Appendix]{Diamond}; note that the valuation 
of the discriminant in that line is achieved by a quadratic twist
which does not change the type of the representation. The result 
now follows as in the first case.
\end{proof}

\begin{proposition} Let $p \geq 3$ be a prime.
Let $E/\Q_2$ be an elliptic curve with potentially good reduction with $e=4$. 
Then $K/\Q_2$ is non-abelian if and only if $\tilde{c}_4 \equiv 5\tilde{\Delta} \pmod{8}$.

Furthermore, $\tilde{c}_4 \equiv 1 \pmod{8}$ or $\tilde{c}_4 \equiv 5 \pmod{8}$. 
\label{P:nonabelianWild4}
\end{proposition}
\begin{proof} First we note that 
condition 2 of \cite[Proposition~2.1]{Diamond} 
does not hold when $K/\Q_2$ is non-abelian and $p \nmid e$.
Since $e=4$ we see from Table~3 in \cite[Appendix]{Diamond}
that we are in the case $\vv_2(j) = 6$ and $\vv_2(j^*) = 7$; moreover, 
the field extension $K/\Q_2$ 
is non-abelian if and only if $\rhobar_{E,p}$ is of type~{\bf V}, 
that is, if and only if $j_E / 64 \equiv 5 \pmod{8}$.

From Proposition~\ref{P:usefulmodel} we have
$(\vv(c_4), \vv(c_6), \vv(\Delta_m)) = (5,8,9)$ or $(7, 11, 15)$. 

Suppose $(\vv(c_4), \vv(c_6), \vv(\Delta_m)) = (5,8,9)$.
We now compute
\[
 \frac{j_E}{2^6} = \frac{1}{2^6}\frac{2^{3\vv(c_4)}\tilde{c}_4^3}{2^{\vv(\Delta_m)}\tilde{\Delta}} = 
 \frac{\tilde{c}_4^3}{\tilde{\Delta}} \equiv 5 \pmod{8}
\]
and since $\tilde{c}_4^3 \equiv \tilde{c}_4 \pmod{8}$ the result follows. 

In the case $(\vv(c_4), \vv(c_6), \vv(\Delta_m)) = (7, 11, 15)$ we first take a quadratic twist
to reduce to the previous case. Since this does not change $j_E/2^6$ the result follows in the same way;
this concludes the proof of the first statement.

We now prove the second statement. We have
that $\tilde{c}_4 \equiv 5\tilde{\Delta} \pmod{8}$. Therefore, 
from Proposition~\ref{P:usefulmodel} and  
the equality $c_4^3 - c_6^2 = 1728 \Delta_m$ we obtain
\[
 \tilde{c}_4^3 - 2 \tilde{c}_6^2 \equiv 7 \tilde{c}_4 \pmod{8}.
\]
Note that $\tilde{c}_4$, $\tilde{c}_6 \equiv 1,3,5,7 \pmod{8}$.
Replacing all these possibilities in the congruence above we see 
that the congruence is only satisfied when
$\tilde{c}_4 \equiv 1,5 \pmod{8}$, as desired.
\end{proof}

We are now able to decide from standard information on $E/\Q_\ell$ if the hypothesis of 
Theorem~\ref{T:goodOverF} hold. This only determines the field $F$ of good reduction in cases (1) 
and~(5).
In the remaining cases, tables~\ref{Table:FoverQ2}, \ref{Table:FoverQ2II}, \ref{Table:FoverQ2III}, 
\ref{Table:FoverQ3I},~\ref{Table:FoverQ3II} 
describe~$F$ in terms of 
minimal $c_4$, $c_6$, $\Delta_m$; note the lines are indexed by the cases in Table~\ref{Table:model}.
The content of these tables follows as a consequence of 
Lemma~\ref{L:rescurvewilde2}, Theorems~\ref{T:coordchanges8}, 
Theorem~\ref{T:coordchanges12} and their proofs. 
Indeed, the changes of coordinates in Lemma~\ref{L:rescurvewilde2} and
Section~\ref{S:coordchanges}
lead to models with good reduction over the fields appearing in these tables; 
moreover, the $\ell$-adic precision used to define those coordinate changes is
the same or greater than that used in the tables 
in this section. For example, Table~\ref{Table:(e)} focuses on case $D_e$ of Table~\ref{Table:model} and the changes of coordinates there
lead to models of~$E$ 
with good reduction over the field defined by the polynomial 
$g_3$ together with a concrete residual elliptic curve; for this we 
work mod~$16$, more precisely, the 
values $\tilde{c}_4 \equiv 1, 3, 9, 11 \pmod{16}$. 
Since these values satisfy $\tilde{c}_4 \equiv 1, 3 \pmod{8}$ this proves 
the first line of Table~\ref{Table:FoverQ2III}.
\begin{footnotesize}
\begin{table}[htb]
$$
\begin{array}{|c|c|c|} \hline
\text{\sc case}   & \tilde{c}_4 \pmod{8}, \; \tilde{c}_6 \pmod{4}  & \mbox{Field} \\ \hline
      C_4 \; : \;     (5,8,9) &   1,\; 1 \text{ or } 5,\; 3  & f_2 \\
      C_4 \; : \;     (5,8,9) &   1,\; 3 \text{ or } 5,\; 1  & f_1 \\
\hline          
     C_4 \; : \;    (7,11,15) &   1,\; 1 \text{ or } 5,\; 3 & f_1 \\
     C_4 \; : \;    (7,11,15) &   1,\; 3 \text{ or } 5,\; 1 & f_2 \\
\hline
\end{array}
$$
\caption{Field of good reduction in case (2) of Theorem~\ref{T:goodOverF}.}
\label{Table:FoverQ2}
\end{table}
\end{footnotesize}

\begin{footnotesize}
\begin{table}[htb]
$$
\begin{array}{|c|c|c|c|} \hline
\text{\sc case}   &  \tilde{\Delta} \pmod{4}  &  \tilde{c}_4 \pmod{4} & \mbox{Field} \\ \hline
      D_a, \; n \geq 8 \; \text{ or } D_b, \; n \geq 11 & & & g_1 \\
      D_a, \; n = 7 \text{ or } D_b, \; n = 10 & & & g_2 \\ 
      D_c  &  & 1, \; 3 & g_1, \; g_2, \; \text{ respectively} \\
      D_d  &  1, \; 3 & & g_1, \; g_2, \; \text{ respectively} \\ \hline
\end{array}
$$
\caption{Field of good reduction in case (3) of Theorem~\ref{T:goodOverF}.}
\label{Table:FoverQ2II}
\end{table}
\end{footnotesize}

\begin{footnotesize}
\begin{table}[htb]
$$
\begin{array}{|c|c|c|c|} \hline
\text{\sc case}   &  \tilde{c}_4 \pmod{8} & \mbox{Field} \\ \hline
      D_e         &  1 \; \text{ or } \; 3              & g_3  \\
      D_e         &  5 \; \text{ or } \; 7              & g_4  \\
      D_f         &  1 \; \text{ or } \; 3              & g_4  \\
      D_f         &  5 \; \text{ or } \; 7              & g_3  \\ \hline
\end{array}
$$
\caption{Field of good reduction in case (4) of Theorem~\ref{T:goodOverF}.}
\label{Table:FoverQ2III}
\end{table}
\end{footnotesize}

\begin{footnotesize}
\begin{table}[htb]
$$
\begin{array}{|c|c|c|c|} \hline
\text{\sc case}   &  \tilde{\Delta} \pmod{3}  & \mbox{Field} \\ \hline
      G_a, \; n=2 \; \text{ or } \;  G_b, \; n=4 & & h_1 \\ 
      G_a, \; n \geq 3 \; \text{ or } \;  G_b, \; n \geq 5 & & h_2 \\
      G_c, \; G_d, \; G_e, \; G_f &  1     &           h_1  \\
      G_c, \; G_d, \; G_e, \; G_f &  2     &           h_2  \\ \hline
\end{array}
$$
\caption{Field of good reduction in case (6) of Theorem~\ref{T:goodOverF}.}
\label{Table:FoverQ3I}
\end{table}
\end{footnotesize}

\begin{footnotesize}
\begin{table}[htb]
$$
\begin{array}{|c|c|c|c|} \hline
\text{\sc case}   &  \tilde{\Delta} \pmod{9}  & \mbox{Field} \\ \hline
      G_h, \; G_j  &  8, \; 5, \; 2     &           h_4, \; h_3, \; h_5, \; \text{ respectively } \\
      G_g, \; n \geq 4 \; \text{ or } \; G_i, \; n \geq 6  &  8, \; 5, \; 2 &   h_3, \; h_4, \; h_5, 
      \; \text{ respectively } \\ \hline 
      \text{\sc case}   &  \tilde{\Delta} \pmod{9}, \; \tilde{c}_4 \pmod{3}  & \mbox{Field} \\ \hline
      G_g, \; n = 3 \; \text{ or } \; G_i, \; n = 5  &  2, 2 \text{ or } 5, 1      & h_3  \\
      G_g, \; n = 3 \; \text{ or } \; G_i, \; n = 5  &  2, 1 \text{ or } 8, 2             & h_4  \\
      G_g, \; n = 3 \; \text{ or } \; G_i, \; n = 5  &  5, 2 \text{ or } 8, 1           & h_5  \\ \hline
\end{array}
$$
\caption{Field of good reduction in case (7) of Theorem~\ref{T:goodOverF}.}
\label{Table:FoverQ3II}
\end{table}
\end{footnotesize}

Although the content of 
tables~\ref{Table:FoverQ2}, \ref{Table:FoverQ2II}, \ref{Table:FoverQ2III}, 
\ref{Table:FoverQ3I},~\ref{Table:FoverQ3II} is contained in the tables 
presented later in the paper we decided to include here more friendly tables 
devoted only to the field of good reduction as we think this is of wider interest.

We remark that proofs for the tables in this section which do not make use of lenghty calculations with changes of coordinates exist. However, 
since the explicit changes of coordinates 
in Theorems~\ref{T:coordchanges8} and~\ref{T:coordchanges12}
are fundamental for the proofs of some of our main results we do not include the
simpler proofs here; we simply illustrate them by proving the following theorem which implies the content of Table~\ref{Table:FoverQ2II}.

For $i=1,2$, we let
$F_i$ be the field defined 
by the polynomial $g_i$ in
case~(3) of Theorem~\ref{T:goodOverF}.

\begin{theorem} 
\label{T:good8I}
Let $E/\Q_2$ be an elliptic curve with potentially good reduction, 
semistability defect $e=8$ and conductor $N_E = 2^5$. 

Then $E$ has good reduction over $F_1$ or $F_2$ according to the cases:
\begin{itemize}
 \item[(a)] if $(c_4,c_6,\Delta_m)$ satisfy case $D_a$ in Table~\ref{Table:model} then
 $E$ has good reduction over $F_1$ if and only if $n \geq 8$ and over $F_2$ if and only if $n=7$;
 \item[(b)] if $(c_4,c_6,\Delta_m)$ satisfy case $D_b$ in Table~\ref{Table:model} then
 $E$ has good reduction over $F_1$ if and only if $n \geq 11$ and over $F_2$ if and only if $n=10$;
 \item[(c)] if $(c_4,c_6,\Delta_m)$ satisfy case $D_c$ in Table~\ref{Table:model} then
 $E$ has good reduction over $F_1$ if and only if $\tilde{c}_4 \equiv 1 \pmod{4}$ and over $F_2$
 if $\tilde{c}_4 \equiv 3 \pmod{4}$; 
 \item[(d)] if $(c_4,c_6,\Delta_m)$ satisfy case $D_d$ in Table~\ref{Table:model} then
 $E$ has good reduction over $F_1$ if and only if $\tilde{\Delta} \equiv 1 \pmod{4}$ 
 and over $F_2$ if and only if $\tilde{\Delta} \equiv 3 \pmod{4}$.
\end{itemize} 
\end{theorem}
\begin{proof} 
Write $K = \Q_2(E[3])$ for the $3$-torsion field of $E$. From \cite[Lemma~6]{DDRoot}
there is a field $F \subset K$ totally ramified of degree 8 whose splitting
field is $K$ and $\Gal(K/\Q_2)$ is the order 16 group denoted $H_{16}$ in {\it loc. cit.}.
Since $K/F$ is unramified it follows that $E/F$ has 
good reduction and $L = \Q_2^{un} F = \Q_2^{un} F_i$ is the minimal extension of $\Q_2^{un}$ where $E$ gets good reduction, 
where $i=1,2$ by Theorem~\ref{T:goodOverF} part~(3).
Since the discriminant valuation of the fields $F_1$ and $F_2$ 
is 16 we conclude this is also true for $F$.
Consulting the LMFDB database \cite{lmfdb} for local fields 
we see that, for discriminant valuation~16, the only fields with 
the required inertia and Galois group structure are (up to $\Q_2$-isomorphism) the fields 
$F_1$, $F_2$, so $F$ is one of them.

We now need to decide which of the two fields $F$ is.
Note that $F_1$ and $F_2$ have exactly 
one quadratic subfield, more precisely,
\[
 M = \Q_2(\sqrt{-5}) = \Q_2(\sqrt{3}) \subset F_1 \quad  \text{ and } \quad M = \Q_2(\sqrt{-1}) = \Q_2(\sqrt{7}) \subset F_2,
\]
so to decide which of the possibilites is $F$ 
it is enough to determine which of the two possible $M$ lies in $K$. 
Let $\Delta^{1/3}$ be the unique cubic root of $\Delta_m$ in $\Q_2$. We consider the quantities
\[
 A = c_4 - 12 \Delta^{1/3} \qquad B = c_4^2 + 12 c_4 \Delta^{1/3} + (12 \Delta^{1/3})^2 
\]
and, from \cite[Lemma~6]{DD2015}, we see that 
\begin{equation} \label{E:M}
  M = \Q_2(\sqrt{A}) = \Q_2(\sqrt{B}).
\end{equation}
We also have the relation
\[
 c_4^3 - c_6^2 = 12^3 \Delta \quad \Leftrightarrow \quad c_4^3  - (12 \Delta^{1/3})^3 = c_6^2 
 \quad \Leftrightarrow \quad AB = c_6^2.
\]
We will now consider the four cases $(a)$, $(b)$, $(c)$ and $(d)$, separately.

(a) Suppose $(c_4(E),c_6(E),\Delta_m(E))$ satisfy case $D_a$ in Table~\ref{Table:model}.
From $c_4^3 - c_6^2 = 12^3 \Delta_m$ we obtain
\[
 2^{12} \tilde{c}_4^3 - 2^{2n}\tilde{c}_6^2 = 2^{12}3^3 \tilde{\Delta} 
 \quad \Leftrightarrow \quad
 \tilde{c}_4^3 -  (3\tilde{\Delta}^{1/3})^3 = (2^{n-6}\tilde{c}_6)^2 
  \]
which gives 
 \[
 (\tilde{c}_4  - 3 \tilde{\Delta}^{1/3})(\tilde{c}_4^2 + 3\tilde{c}_4 \tilde{\Delta}^{1/3} + (3\tilde{\Delta}^{1/3})^2)  
 = (2^{n-6}\tilde{c}_6)^2
\]
and we write $\tilde{A}$ and $\tilde{B}$ respectively for the two factors on the left hand side. 
Note that $B = 2^8 \tilde{B}$ and $\tilde{B} \equiv 1 \pmod{2}$ hence $\tilde{A} \equiv 0 \pmod{2^{2n-12}}$. 
From \eqref{E:M} we see that $M$ is determined by $\sqrt{B} = 2^4 \sqrt{\tilde{B}}$ and so it is determined by
$\tilde{B} \pmod{8}$. 

Note that $\tilde{\Delta} \equiv \tilde{\Delta}^{1/3} \pmod{8}$ and since 
$n \geq 7$ we have that $\tilde{A} \equiv \tilde{c}_4  - 3 \tilde{\Delta} \equiv 0 \pmod{4}$.

Running through all the possible values of $\tilde{c}_4$ and $\tilde{\Delta}$ modulo $8$ such that
\[
\tilde{\Delta} \equiv 1 \pmod{2}, \quad \tilde{c}_4 \equiv -1 \pmod{4}, \quad  \tilde{c}_4  - 3 \tilde{\Delta} \equiv 0 \pmod{4} 
\]
we find that $\tilde{B} \equiv 3,7 \pmod{8}$ 
(this was expected from the possibilities for $M$). Furthermore, one also checks that we have
$\tilde{B} \equiv 7 \pmod{8}$ if and only if $\tilde{c}_4  - 3 \tilde{\Delta} \equiv 4 \pmod{8}$,
proving that $E$ has good reduction over $F_1$ 
if and only if $n \geq 8$ and over $F_2$ if and only if $n=7$, as desired.

(b) Suppose $(c_4(E),c_6(E),\Delta_m(E))$ satisfy case $D_b$ in Table~\ref{Table:model}.
Arguing analogously to the previous case, we conclude that
\[
 (\tilde{c}_4  - 3 \tilde{\Delta}^{1/3})(\tilde{c}_4^2 + 3\tilde{c}_4 \tilde{\Delta}^{1/3} + (3\tilde{\Delta}^{1/3})^2)  
 = (2^{n-9}\tilde{c}_6)^2
\]
and (again) $M$ is determined by $\sqrt{B} = 2^6 \sqrt{\tilde{B}}$ and so it is determined by
$\tilde{B} \pmod{8}$; we also have $\tilde{A} \equiv \tilde{c}_4  - 3 \tilde{\Delta} \equiv 0 \pmod{4}$
because $n \geq 10$. As above, running through all the possible values 
of $\tilde{c}_4$ and $\tilde{\Delta}$ modulo $8$ such that
we find that $\tilde{B} \equiv 3,7 \pmod{8}$ and, furthermore,
$\tilde{B} \equiv 7 \pmod{8}$ if and only if $\tilde{c}_4  - 3 \tilde{\Delta} \equiv 4 \pmod{8}$.
This proves that $E$ has good reduction over $F_1$ 
if and only if $n \geq 11$ and over $F_2$ if and only if $n=10$, as desired.

(c) Suppose $(c_4(E),c_6(E),\Delta_m(E))$ satisfy case $D_c$ in Table~\ref{Table:model}.
Write the quantities $A$ and $B$ as
\[
 A = 2^6 (2\tilde{c}_4 - 3 \tilde{\Delta}) = 2^6 \tilde{A} 
 \qquad  \text{and} \qquad B = 2^{12} (2^2 \tilde{c}_4^2 + 6 \tilde{c}_4 \tilde{\Delta}^{1/3} 
 + (3 \tilde{\Delta}^{1/3})^2)  = 2^{12} \tilde{B}
\]
and note it follows from \eqref{E:M} that $M = \Q_2(\sqrt{\tilde{A}}) = \Q_2(\sqrt{\tilde{B}})$. 
Running through all the possible values of $\tilde{c}_4$ and $\tilde{\Delta}$ modulo $8$
we check that $\tilde{A} \equiv \tilde{B} \pmod{8}$ implies
\[
 \tilde{A} \equiv \tilde{B} \equiv 3 \pmod{8} \qquad \text{ or } \qquad \tilde{A} \equiv \tilde{B} \equiv 7 \pmod{8}
\]
as expected. Moreover, in the former case we have 
$\tilde{c}_4 \equiv 1,5 \pmod{8}$
and in the latter we have $\tilde{c}_4 \equiv 3,7 \pmod{8}$, as desired.

(d) Suppose $(c_4(E),c_6(E),\Delta_m(E))$ satisfy case $D_d$ in Table~\ref{Table:model}.
Write the quantities $A$ and~$B$ as
\[
 A = 2^4 (\tilde{c}_4 - 6 \tilde{\Delta}) = 2^4 \tilde{A} 
 \qquad  \text{and} \qquad B = 2^8 (\tilde{c}_4^2 + 6 \tilde{c}_4 \tilde{\Delta}^{1/3} 
 + (6 \tilde{\Delta}^{1/3})^2)  = 2^{8} \tilde{B}
\]
and again we need to have $M = \Q_2(\sqrt{\tilde{A}}) = \Q_2(\sqrt{\tilde{B}})$. 
Running through all the possible values of $\tilde{c}_4$ and $\tilde{\Delta}$ modulo $8$
we check that $\tilde{A} \equiv \tilde{B} \pmod{8}$ implies $\tilde{c}_4 \equiv 1 \pmod{8}$
and the necessary conditions
$\tilde{A} \equiv \tilde{B} \equiv 3 \pmod{8}$ or $\tilde{A} \equiv \tilde{B} \equiv 7 \pmod{8}$
are satisfied respectively when $\tilde{\Delta} \equiv 1,5 \pmod{8}$ or $\tilde{\Delta} \equiv 3,7 \pmod{8}$, 
as desired.
\end{proof}

\section{The Galois group of the $p$-torsion field in the cases $e=3,4$.}

Let $E/\Q_\ell$ be an elliptic curve with potentially good reduction with $e=3,4$ and 
$p\geq 3$.

Write $\rhobar_{E,p}$ for the mod~$p$ Galois representation attached to $E$ and $\PP(\rhobar_{E,p})$ for its projectivization. 
Let $K = \Q_\ell(E[p])$ denote the $p$-torsion field of $E$, that is the field fixed by the kernel of $\rhobar_{E,p}$.
Let also $K_{proj} \subset K $ be the field fixed by the kernel of $\PP(\rhobar_{E,p})$. Write $G = \Gal(K/\Q_\ell)$.

Let $f$ be the residual degree of $K$, that is $f = [K_{un} : \Q_\ell]$ where $K_{un}$ is the maximal unramified subfield of $K$. The inertia subgroup of $G$ is $\Gal(K/K_{un})$ and it is cyclic of order $e$.
We have $[K : \Q_\ell] = [K : K_{un}][K_{un} : \Q_\ell] = ef$.

The objective of this section is to determine the group structure of $G = \Gal(K/\Q_\ell)$ in the non-abelian case, which is given in Proposition~\ref{P:groupStructure}. 
We start by proving the following lemma describing the relation between $K$ and the field $F$
given by Theorem~$\ref{T:goodOverF}$.

\begin{lemma} 
Let $p \neq \ell$ be primes with $p \geq 3$.
Let $E/\Q_\ell$ be an elliptic curve with potentially good reduction 
with semistability defect $e=3,4$. 
Suppose its $p$-torsion field $K$ is a non-abelian
extension of $\Q_{\ell}$.

Let $F$ be the degree $e$ totally ramified field given by Theorem~$\ref{T:goodOverF}$. 
Then 
\begin{itemize}
 \item[(i)]  if $e = 3$  then $F \subset K$;
 \item[(ii)] if  $\ell \geq 3$ and $e = 4$ then  $K \cap F = F$ or $K \cap F = \Q_\ell(\sqrt{\ell})$;
 \item[(iii)] if $\ell = 2$ and $e = 4$ then $K \cap F = F$ or $K \cap F = \Q_2(\sqrt{-2})$.
\end{itemize}
\label{L:FinK}
\end{lemma}
\begin{proof} Part (i) follows from the criterion of N\'eron-Ogg-Shafarevich (see~\cite[VII.7]{SilvermanI}). Indeed,
suppose that $F$ is not contained in $K$; then $K \cap F = \Q_\ell$ 
(because $F$ has prime degree) and $\Gal(FK/F) \simeq \Gal(K/K\cap F) = \Gal(K/\Q_\ell)$.
Since $E/F$ has good reduction it follows that $\Gal(FK/F)$ is cyclic which contradicts
the fact that $\Gal(K/\Q_\ell)$ is non-abelian.

Suppose $e=4$. Let $N = \rhobar_{E,p}(\sigma)$, where $\sigma$ generates the inertia subgroup~$\Gal(K/K_{un}) \subset G$;
thus $N$ has order $e=4$. We shall show that $N$ has projective order 2.
Indeed, write $m(x)$ for the minimal polynomial of $N$ and $I$ for the identity matrix in $\GL_2(\F_p)$. Since $N^4=I$ we have $m(x)$ divides $x^4 - 1 = (x-1)(x+1)(x^2 + 1)$. 
If $m$ is of degree one, since $\det N=1$ (as for $\ell \neq p$
the inertia subgroup $I_\ell \subset G_{\Q_\ell}$ lands in $\SL_2(\F_p)$ because the mod~$p$ cyclotomic character is unramified at $\ell$) 
we obtain $N=\pm I$ hence $N$ has order $\leq 2$, a contradiction. Thus $m=x^2+1$ and we have $N^2=-I$, as desired.

Since $K/\Q_{\ell}$ is non-abelian and $p \nmid 4$ condition 2 in \cite[Proposition~2.1]{Diamond} does not hold.
Furthermore, condition 4 in \cite[Propositions~2.2~and~2.4]{Diamond} do not hold;
we conclude that  the equivalent conditions of \cite[Proposition~2.3]{Diamond} hold
(note that we already knew this in the case $\ell=2$ from 
the proof of Proposition~\ref{P:nonabelianWild4}).
From condition 3 in \cite[Proposition~2.3]{Diamond} it follows that $K_{proj}$ has
Galois group $D_2 \simeq V_4$ with inertia subgroup of order 2. We conclude that
the field $K_{proj} \subset K$ is the compositum of the unramified quadratic extension 
$U/\Q_\ell$ with a ramified quadratic extension of $\Q_\ell$.

Observe that, since $F$ has degree 4, then $K \cap F$ is either $\Q_\ell$, quadratic or $F$. 

Suppose $\ell \geq 3$; then $K_{proj} = U (\sqrt{\ell})$. 
In this case $F = \Q_\ell(\ell^{1/4})$ then
$\Q_\ell(\sqrt{\ell}) \subset K \cap F$.

Suppose $\ell = 2$; thus $F = F_1$ or $F_2$, where $F_i$ is the field defined by 
the polynomial $f_i$ in Theorem~\ref{T:goodOverF} part (2); in both cases $F$ 
has $\Q_2(\sqrt{-2})$ as the unique quadratic subfield. Since $L = \Q_2^{un} F = \Q_2^{un} K$
we conclude that $\Q_2(\sqrt{-2}) \subset K \cap F$, as desired.
\end{proof}

We can now describe the group structure of $G$
in the non-abelian case.

\begin{proposition} 
Let $p \neq \ell$ be primes with $p \geq 3$.
Let $E/\Q_\ell$ be an elliptic curve with potentially good reduction 
with semistability defect $e=3,4$. 
Suppose that $G = \Gal(K/\Q_\ell)$ is non-abelian. 

Let $F$ be given by Theorem~$\ref{T:goodOverF}$ 
and if $e=4$ let also $K_2$ be the quadratic subfield of $F$, so that $K \cap F = F$ or $K_2$ by Lemma~$\ref{L:FinK}$.

Then there are generators $\sigma, \tau \in G$ such that $\sigma \in \Gal(K/K_{un})$ is
an inertia generator and  $\tau \in \Gal(K/K \cap F) \subset G$. 
Moreover, 
\begin{enumerate}
 \item[(i)]  If $K \cap F = F$ then
 \[\tau^f = 1, \quad \sigma^e = 1, \quad \tau \sigma \tau^{-1} = \sigma^{-1}; \]
 \item[(ii)] If $K \cap F = K_2$ then
\[
\tau^{2f} = 1, \qquad  \sigma^4 = 1, \qquad \tau \sigma \tau^{-1} = \sigma^{-1}.
\]
\end{enumerate}
Also, when $p \geq 5$ or $p=3$ and $e=4$, the order of $G$ is not divisible by $p$.

We can further assume that $\tau$ acts on the residue field of $K$ as the Frobenius
automorphism.
\label{P:groupStructure}
\end{proposition}

\begin{proof} 
From Theorem~\ref{T:goodOverF} we know $E/F$ has good reduction. The $p$-torsion field of
$E/F$ is $K F$, hence $\Gal(KF/F) \simeq \Gal(K/K\cap F)$ is cyclic. 
We let $\tau$ be a generator of $\Gal(K/K\cap F)$. 

Since $\Gal(K/K_{un})$ is normal in $G$ 
the canonical map
\begin{equation}
     \frac{\Gal(K/K\cap F)}{\Gal(K/K_{un}) \cap \Gal(K/K\cap F)}
     \longrightarrow \frac{\Gal(K/K_{un}) \Gal(K/K\cap F)}{\Gal(K/K_{un})}
     \label{E:canonical}
\end{equation}
is an isomorphism of groups. Moreover, we also have
\[
 \tau \sigma \tau^{-1} = \sigma^{k}, \qquad k \in \{0,1,\ldots,e-1\}.
\]
We note that $|\Gal(K/\Q_\ell)| = ef$ and $|\Gal(K/K_{un})| = e$ and divide into the cases:

(i) Suppose $K\cap F = F$. Then $\Gal(K/F) \simeq \Gal(K_{un}/\Q_\ell)$, hence $\tau$ has order $f$. 
Therefore,
\[
\tau^f = 1, \qquad  \sigma^e = 1, \qquad \tau \sigma \tau^{-1} = \sigma^k
\]
and 
\[
|\Gal(K/K \cap F)|=f, \quad |\Gal(K/K_{un}) \cap \Gal(K/K\cap F)| = 1. 
\]
Thus isomorphism \eqref{E:canonical} gives $|\Gal(K/K_{un}) \Gal(K/K\cap F)|=ef$
showing that $\tau$ and $\sigma$ generate $G$.

(ii) Suppose $K \cap F = K_2$; then $e=4$ by Lemma~\ref{L:FinK}. 
Thus $4f = [K : \Q_\ell] = [K : K_2][K_2 : \Q_\ell]$ implies 
$[K : K \cap F] = 2f$, that is $\tau$ has order $2f$. Therefore,
\[
\tau^{2f} = 1, \qquad  \sigma^4 = 1, \qquad \tau \sigma \tau^{-1} = \sigma^k
\]
and
\[
|\Gal(K/K \cap F)|=2f, \quad |\Gal(K/K_{un}) \cap \Gal(K/K\cap F)| = 2. 
\]
Thus the isomorphism \eqref{E:canonical} gives $|\Gal(K/K_{un}) \Gal(K/K\cap F)|=4f$
showing again that $\tau$ and $\sigma$ generate $G$.

Note that if $k=0$ or $e=4$ and $k=2$ then $(\tau \sigma \tau^{-1})^2 = 1$ which implies $\tau \sigma^2 = \tau$, a contradiction; since $G$ is non-abelian we also have $k \ne 1$; 
thus $k=2$ if $e=3$ or $k=3$ if $e=4$. We conclude $k \equiv -1$ mod~$e$ in both cases (i) and (ii), as desired.

The claim on the order of $G$ follows 
from \cite[Proposition~2.3]{Diamond} and the assumption 
that $G$ is non-abelian (see also the 3rd paragraph in the proof of Lemma~\ref{L:FinK}).

To finish the proof, we observe that the arguments above work for any choice of 
generator $\tau$ of $\Gal(K/K\cap F)$; since the residue field of $K \cap F$ is
$\F_\ell$ we can also assume that $\tau$ acts on the residue field of $K$ 
as the Frobenius automorphism.
\end{proof}

\subsection{Explicit description of $\tau$ and $\sigma$ in some tame cases}
\label{S:TauGamma}
Let $E/\Q_\ell$ satisfy $e=3,4$ and $\ell \equiv -1 \pmod{e}$ so that $G$ is non-abelian by Corollary~\ref{C:nonabelianTame}; suppose further $F \cap K = F$. 

We have that $F = \Q_\ell(\pi)$, where $\pi^e = \ell$; the Galois closure of $F$ 
is $\Q_\ell(\pi, \mu_e) \subset K$, where $\mu_e$ are the $e$-th roots of unity.
Let $q = \ell^f$ be the cardinality of the residue field $\F_q$ of $K$. Fix $\zeta \in K$ a 
a root of unity of order $q-1$. Since $\mu_e \subset K_{un}$ 
then $q - 1 = er$ for some positive integer $r$.
Set $\zeta_e := \zeta^r$, which is is a root of unity of order $e$ in $K_{un}$. 
Since $F$ is totally ramified of 
degree $e = [K : K_{un}]$ we have $K = K_{un}F = \Q_\ell(\zeta, \pi)$.

We define the elements $\tau, \sigma \in G = \Gal(K/\Q_\ell)$
as follows:
\[
 \tau \in \Gal(K/F), \quad \tau(\zeta) = \zeta^\ell, \quad \tau(\pi) = \pi
\]
and
\[
\sigma \in \Gal(K/K_{un}), \quad \sigma(\zeta) = \zeta, \quad \sigma(\pi) = \zeta_e \pi. 
\]
Clearly, $\tau$, $\sigma$ generate $\Gal(K/\Q_\ell)$ and satisfy $\tau^f = 1$, $\sigma^e = 1$. 
Finally, 
\[ 
\tau \sigma \tau^{-1} (\pi) = \tau(\zeta_e \pi) = \zeta_e^\ell \pi = \sigma^\ell(\pi) 
\]
and  
\[
\tau \sigma \tau^{-1} (\zeta) = \zeta = \sigma^\ell(\zeta) 
\]
showing $\tau \sigma \tau^{-1} = \sigma^\ell$. Now $\ell \equiv -1 \pmod{e}$, 
implies $\tau \sigma \tau^{-1} = \sigma^{-1}$ as in Proposition~\ref{P:groupStructure}.

\section{Proof of Theorem~\ref{T:main3torsion}}
\label{S:thm2}

For an elliptic curve $E/\Q_\ell$ with potentially good reduction
we denote by~$e$ its semistability defect and, when $c_4 \neq 0$ or $c_6 \neq 0$, we write
\[
 c_4 = \ell^{\vv(c_4)} \tilde{c}_4, \quad c_6 = \ell^{\vv(c_6)} \tilde{c}_6, 
 \quad \Delta = \ell^{\vv(\Delta_m)} \tilde{\Delta},
\]
where $c_4$, $c_6$ and $\Delta_m$ are the standard invariants of a minimal model.

We first prove the following proposition.

\begin{proposition} Let $\ell \equiv -1 \pmod{3}$ be a prime. Let $E/\Q_\ell$ be an elliptic curve with
additive potentially good reduction. If $E$ has a $3$-torsion point defined over $\Q_\ell$ then $e = 3$.
Conversely, if $e=3$ then the mod~$3$ representation $\rhobar_{E,3}$ has $S_3$
image and is isomorphic to either
\[
 \begin{pmatrix} 1 & * \\ 0 & \chi_3 \end{pmatrix} \qquad \text{ or } \qquad \begin{pmatrix} \chi_3 & * \\ 0 & 1 \end{pmatrix}.
\] 
\label{P:mod3}
\end{proposition}
\begin{proof} We have $\det \rhobar_{E,3} = \chi_3 \neq 1$. Suppose $E$ has a 3-torsion point. Then
\[
 \rhobar_{E,3} \sim \begin{pmatrix} 1 & * \\ 0 & \chi_3 \end{pmatrix}.
\] 
Since $e \neq 1$ and $\chi_3$ is unramified it follows that $* \ne 0$ and $e=3$. 

Suppose $e=3$. We can assume that, up to conjugation, the image of the inertia 
subgroup $I_\ell \subset G_{\Q_\ell}$ via $\rhobar_{E,3}$ is generated by 
$U = \begin{pmatrix} 1 & 1 \\ 0 & 1 \end{pmatrix}.$
Hence the image of $\rhobar_{E,3}$ must be contained in the normalizer of $\langle U \rangle$ 
which is the Borel subgroup
of $\GL_2(\F_3)$. Therefore we have
\[
 \rhobar_{E,3} \sim \begin{pmatrix} \epsilon_1 & * \\ 0 & \epsilon_2 \end{pmatrix}, \quad \text{ with } 
 \quad \epsilon_1\epsilon_2 = \chi_3.
\]
with $\epsilon_i$ unramified characters of order dividing 2. The character $\chi_3$ is the 
unique non-trivial such character thus $\epsilon_1 = \chi_3$ and $\epsilon_2 = 1$
or the other way around. Clearly the image is $S_3$.
\end{proof}

\subsection{Proof of Theorem~\ref{T:main3torsion} part (A)} 
\label{S:TorsionPartA}
Write $\vv = \vv_\ell$ for the valuation in $\Q_\ell$. Let $E/\Q_\ell$ be given by the minimal model in Proposition~\ref{P:usefulmodel}
\[
 y^2 = x^3 + a x + b,\quad a = -\frac{c_4}{48}, \quad b = -\frac{c_6}{864}.
\]
Its 3-division polynomial is $3f$, where 
\begin{equation}
  f = X^4 + 2aX^2 + 4bX - \frac{a^2}{3}.
  \label{E:3divpol}
\end{equation}
Note that $b \ne 0$. From the Newton polygon for $f$ we see there is a root $x_0$
with integer valuation and the remaining roots have non-integer valuations
(for this we use cases A of Table~\ref{Table:model}).
More precisely, if $a=0$ then $x_0 = 0$ is the unique root in $\Q_\ell$;
if $a \neq 0$ then there is a unique root $x_0 \in \Z_\ell$ whose valuation 
satisfies $\vv(x_0) = 2\vv(a) - 2$ if $\vv(\Delta_m) = 4$ 
or $\vv(x_0) = 2\vv(a) - 4$ if $\vv(\Delta_m) =8$.

Write $b = \ell^{\vv(b)} b'$ and observe that
\[
 \frac{x_0^3 + a x_0 + b}{\ell^{\vv(b)}} = \frac{x_0^3 + a x_0}{\ell^{\vv(b)}} + b'.
\]

We have $\vv(x_0^3 + a x_0) \geq \vv(b) + 1$. Indeed,
if $\vv(\Delta_m) = 4$ we have $\vv(a) \geq 2$ and
\[
 \vv(x_0^3 + ax_0) \geq \text{min}(6\vv(a) - 6,3\vv(a) - 2) = 3\vv(a) - 2 \geq 4 > \vv(b) = 2;
\]
if $\vv(\Delta_m) = 8$ we have $\vv(a) \geq 3$ and
\[
 \vv(x_0^3 + ax_0) \geq \text{min}(6\vv(a) - 12,3\vv(a) - 4) = 3\vv(a) - 4 \geq 5 > \vv(b) = 4.
\]
Thus
\[
 \frac{x_0^3 + a x_0}{\ell^{\vv(b)}} \equiv 0 \pmod{\ell} \quad \text{ and } 
 \quad \vv(x_0^3 + a x_0 + b) = \vv(b).
\]

Write $x_0^3 + a x_0 + b = u \ell^{\vv(b)}$. Since $\vv(b)$ is even we conclude that $x_0^3 + a x_0 + b$ is a square in $\Q_\ell$ if and only if $u \equiv b' \pmod{\ell}$ is a square in $\F_\ell$.
Clearly $E$ has a 3-torsion point over $\Q_\ell$ if and only if 
$x_0^3 + a x_0 + b$ is a square in $\Q_\ell$.
The result now follows because
\[
 b' = -\frac{\tilde{c}_6}{864} = -6 \tilde{c}_6 s^2 \; \text{ for } \; s = \frac{1}{72} \in \Q_\ell.
\]
\subsection{Proof of Theorem~\ref{T:main3torsion} part (B)}
\label{S:TorsionPartB}

We will need the following auxiliary statement.

\begin{proposition} Let $\ell \equiv -1 \pmod{3}$ be a prime. 
Let $E/\Q_\ell$ be an elliptic curve with potentially 
good reduction with $e=3$ given by a minimal model $y^2 = x^3 + ax + b$ 
as in Proposition~\ref{P:usefulmodel}. Then there exist a unique element $x_0 \in \Z_\ell$ which is the $x$-coordinate of a $3$-torsion point $P=(x_0,y_0)$ of $E$. Moreover, $P$ is defined over $\Q_\ell$ 
if and only if $x_0^3 + a x_0 + b$ is a $\Z_\ell$-square.
\label{P:rootf}
\end{proposition}

\begin{proof} The last statement is clear from the equality $y_0^2 = x_0^3 + a x_0 + b$.

Let $f$ be the 3-division polynomial \eqref{E:3divpol} and $G=\Gal(\Q_\ell(E[3])/\Q_\ell)$. 
We have from Proposition~\ref{P:mod3} that $G \simeq S_3$.
Note that if $f$ is irreducible then $4 \mid |G|$ and if $f$ has at least 2 roots 
in $\Q_\ell$ then $3 \nmid |G|$, a contradiction.
\end{proof}

\begin{remark}
Note that in the proof of Section~\ref{S:TorsionPartA} we had proved 
Proposition~\ref{P:rootf} for $\ell \geq 5$ using Newton polygons.
Trying to prove the case $\ell = 2$ with Newton polygons works
in cases $C_{3,i}$, $C_{3,ii}$ and $C_{3,i}$ of Table~\ref{Table:model},
but not in case $C_{3,iv}$. In the latter we can only conclude 
that $f$ has 4 roots with valuation 0. On the other hand, the
previous proof covers $\ell \geq 5$ but does not provide the
information on $\vv_\ell(x_0)$ needed 
in Section~\ref{S:TorsionPartA}.
\end{remark}

We shall now prove part (B) of Theorem~\ref{T:main3torsion}. Write $\vv = \vv_2$
for the valuation in $\Q_2$.
Consider the minimal model for $E/\Q_2$ given by Proposition~\ref{P:usefulmodel}
\[
 y^2 = x^3 + a x + b,\quad a = -\frac{c_4}{48}, \quad b = -\frac{c_6}{864}
\]
whose 3-division polynomial is $3f$, where
\[
 f = X^4 + 2aX^2 + 4bX - \frac{a^2}{3}.
\]
Let $x_0 \in \Z_2$ be the root of $f$ given by Proposition~\ref{P:rootf} and
write
\[
 u = x_0^3 + a x_0 + b.
\]

We now divide the proof into the cases $C_{3,i}$, $C_{3,ii}$, $C_{3,iii}$ and $C_{3,iv}$
of Proposition~\ref{P:usefulmodel}.

(i) Suppose $(\vv(c_4),\vv(c_6),\vv(\Delta_m)) = (4,5,4)$ and $\tilde{c}_4 \equiv -1 \pmod{4}$, $\tilde{c}_6 \equiv 1 \pmod{4}$.

We have $\vv(a) = \vv(b) = 0$, hence $\vv(x_0) = 0$. By looking modulo 2 it follows that $u$  is also a unit. Therefore, $u$ its a square if and only if $u \equiv 1 \pmod{8}$.
The equalities
\[
 a = - \frac{\tilde{c}_4}{3}, \quad \text{ and } \quad b = - \frac{\tilde{c}_6}{27}
\]
imply $a \equiv 3, 7 \pmod{8}$ and $b \equiv 1,5 \pmod{8}$. Direct computations allow us to check that
\[
 (a,b) \equiv (3,1),(7,5) \pmod{8} \Rightarrow x_0 \equiv 7 
 \pmod{8} \text{ and } u \equiv 5 \pmod{8};
\]
\[
 (a,b) \equiv (3,5),(7,1) \pmod{8} \Rightarrow x_0 \equiv 3 \pmod{8} 
 \text{ and } u \equiv 1 \pmod{8}.
\]
We have $\tilde{c}_4 \equiv 5a \pmod{8}$ and $\tilde{c}_6 \equiv 5b \pmod{8}$. Part (1) follows.

(ii) Suppose $(\vv(c_4),\vv(c_6),\vv(\Delta_m)) = (\geq 6,5,4)$ and $\tilde{c}_6 \equiv 1 \pmod{4}$.

We have $\vv(a) \geq 2$ and $\vv(b) = 0$, hence $\vv(x_0) \geq 2$. Thus $u \equiv b \pmod{8}$.
Since we also have $\tilde{c}_6 \equiv 5b \pmod{8}$ this proves (2) in this case.

(iii) Suppose $(\vv(c_4),\vv(c_6),\vv(\Delta_m)) = (\geq 7,7,8)$ and $\tilde{c}_6 \equiv 1 \pmod{4}$.

We have $a \equiv 0 \pmod{8}$ and $\vv(b) = 2$. Hence $\vv(x_0) \geq 2$ and $\vv(u) = 2$.
Therefore we have to compute $u$ modulo 32. From the previous we have that
\[
 a \equiv 0, 8, 16, 24 \pmod{32} \qquad b \equiv 4, 12, 20, 28 \pmod{32}
\]
and for each choice of $a$, $b$ we compute the roots of $f$ mod~$32$ and for each root
we check if $u$ is a square mod 32. We find that $u$ is a square in $\Z_2$ if and only
if $b \equiv 4 \pmod{32}$. Finally,
\[
 b = - \frac{2^7 \tilde{c}_6}{3^3 2^5} = - \frac{4\tilde{c}_6}{27}
\]
implies $b \equiv 4 \pmod{32}$ if and only if $\tilde{c}_6 \equiv 5 \pmod{8}$.
This completes the proof of (2).

(iv) Suppose $(\vv(c_4),\vv(c_6),\vv(\Delta_m)) = (4,6,8)$, 
$\tilde{c}_6 \equiv -1 \pmod{4}$ and $\tilde{\Delta} \equiv -1 \pmod{4}$.

We have $\vv(a) = 0$ and $\vv(b) = 1$. We also have
\[
 b = - \frac{2 \tilde{c}_6}{27} \equiv \frac{2}{27} \equiv 6 \pmod{8}.
\]
Moreover, 
\[
 1728\Delta_m = 2^6 3^3 2^8 \tilde{\Delta} = c_4^3 - c_6^2 = 2^{12}(\tilde{c}_4^3 - \tilde{c}_6^2)  
\]
implies
\[
 \tilde{c}_4^3 - \tilde{c}_6^2 \equiv 2^2 3^3 \tilde{\Delta} \equiv - 2^2 3^3 \equiv 4 \pmod{16}.
\]
It follows $\tilde{c}_4 \equiv 5 \pmod{8}$. Thus $a = - \tilde{c}_4 / 3 \equiv 1 \pmod{8}$. 
We conclude $\vv(x_0) = 0$.
From the previous
\[
 a \equiv 1, 9, 17, 25 \pmod{32} \qquad b \equiv 6, 14, 22, 30 \pmod{32}.
\]
Similar calculation as in (iii) show that $\vv(u) = 2$ and that $u/2^2$ is a square
mod 8 (hence $u$ a square in $\Z_2$) if and only if 
\[
 (a,b) \equiv (1,6), (9,14), (17,22), (25,30) \pmod{32}.
\]
Part (3) now follows from
\[
 \tilde{c}_4 \equiv 29a \pmod{32} \quad \text{ and } \quad b = - \frac{2\tilde{c}_6}{27},
\]
concluding the proof.

\section{Proof of Lemmas~\ref{L:lemma1},~\ref{L:lemma2},~\ref{L:lemma3} and ~\ref{L:lemma4}}
\label{S:4lemmas}

To decide the symplectic type of an isomorphism between 
two elliptic curves over~$\Q_\ell$ with potentially good reduction 
with semistability defect  of order $e=2$ or $e=6$ we first need to 
take a suitable quadratic twist to respectively reduce to the 
cases $e=1$ or $e=3$. The lemmas~\ref{L:lemma1},~\ref{L:lemma2},~\ref{L:lemma3} 
and ~\ref{L:lemma4} describe which twist we have to take in each case; 
the objective of this section is to prove these lemmas.

Write $\vv = \vv_\ell$. Let $E/\Q_{\ell}$ be an elliptic curve with potentially good reduction and semistability defect of order 
$e\in \lbrace 2,6\rbrace$.
We write 
\[ t=\left(v(c_4),v(c_6),v(\Delta_m)\right) \] 
for the standard invariants of 
a minimal Weierstrass model of $E/\Q_{\ell}$. 

\begin{lemma} Let $E$ be as above. Then $E$ admits a minimal model of the form
\begin{equation}
\label{(1.1)} 
y^2=x^3-{c_4\over 48}x-{c_6\over 864}.
\end{equation}

\begin{proof} This follows similarly to the proof of Proposition~\ref{P:usefulmodel}.
\end{proof}
\end{lemma}

\subsection{Proof of Lemma~\ref{L:lemma1} and Lemma~\ref{L:lemma3}} 

Let $E'/\Q_\ell$ be the quadratic twist of $E$ by $\sqrt{\ell}$, 
which is given by the integral model
\[ (W') \ : \ Y^2=X^3-\frac{c_4\ell^2}{48}X-\frac{c_6\ell^3}{864}.\]
Let $c_4(W')$, $c_6(W')$ and $\Delta(W')$ be the standard invariants  associated to  $(W')$. 
We have
\begin{equation}
\label{(1.5)} 
c_4(W')=\ell^2c_4,\quad c_6(W')=\ell^3c_6 \quad \text{and}\quad 
\Delta(W')=\ell^6\Delta_m.
\end{equation}
Assume first $\ell\geq 5$. 
It follows from \eqref{E:denominator}, the assumption $e=2,6$ 
and \cite[Table~I]{pap} that $t = (\vv(c_4), \vv(c_6), \vv(\Delta_m))$
satisfies
\begin{equation*}
t\in  
  \begin{cases}
\lbrace (2,\geq 3,6), (\geq 3,3,6)\rbrace \quad \text{ if } \ e=2, \\
\lbrace (\geq 1,1,2), (\geq 4,5,10)\rbrace \quad \text{if} \ e=6.\
  \end{cases}
\end{equation*}
If $e=2$, then $v(\Delta_m)=6$ and~\eqref{(1.5)} leads to  $\vv(\Delta(W'))=12$. This implies that $(W')$ is not a minimal model for $E'/\Q_{\ell}$, so necessarily
$E'/\Q_{\ell}$ has good reduction.

If $e=6$ then $v(\Delta_m)\in \lbrace 2,10\rbrace$. 

In case $v(\Delta_m)=2$, one has  $v(\Delta(W'))=8$. Then the formula 
 \eqref{E:denominator} applied to $E'/\Q_{\ell}$ then implies that the semistability  defect of $E'/\Q_{\ell}$ is $e' = 3$.
 
In case $v(\Delta_m)=10$, one has  $v(\Delta(W'))=16$. We deduce that the valuation of the  minimal discriminant of  $E'/\Q_{\ell}$ is equal to $4$ and again by \eqref{E:denominator} 
it follows that $e' = 3$.

This proves Lemmas~\ref{L:lemma1}~and~\ref{L:lemma3} for $\ell\geq 5$.
 
Suppose now $\ell=3$. From \cite[Corollary, p. 355]{Kraus1990} it follows that
\begin{equation*}
t\in  
  \begin{cases}
\lbrace (2,3,6),  (3,\geq 6,6)\rbrace \quad \text{ if } \ e=2, \\
\lbrace (3,5,6), (4,6,10)\rbrace \quad \text{if} \ e=6.\
  \end{cases}
\end{equation*}
If $e=2$ then $v(\Delta)=6$ and the formulas~\eqref{(1.5)} give
 $$\left(v(c_4(W')),v(c_6(W')),v(\Delta(W'))\right)\in \lbrace  (4,6,12), (5,\geq 9,12)   \rbrace.$$
 We deduce from \cite[p. 126, Table II]{pap}  that the  model $(W')$ is not minimal, 
thus $E'/\Q_3$ has good reduction over $\Q_3$, as desired.

If $e=6$ then from the possibilites for $t$ and \eqref{(1.5)} we see that
 $$(v(c_4(W')),v(c_6(W')),v(\Delta(W'))\in \lbrace  (5,8,12), (6,9,16) \rbrace.$$

In case $(v(c_4(W')),v(c_6(W')),v(\Delta(W'))=(5,8,12)$, 
then \cite[Corollary, p. 355]{Kraus1990} implies that $e'=3$, as desired.

If $(v(c_4(W')),v(c_6(W')),v(\Delta(W'))=(6,9,16) $ then $(W')$ is not minimal; 
further, a minimal model for $E'$ will satisfy $\vv(\Delta_m') = 4$ and again
the result follows from \cite[Corollary, p. 355]{Kraus1990}.

\subsection{Proof of Lemma~\ref{L:lemma4}} 
Let $E/\Q_2$ be an elliptic curve with potentially good reduction and 
semistability defect $e=2$ so that, by \cite[p. 357, Corollary]{Kraus1990}, we have
\begin{equation}
t\in \lbrace (\geq 6,6,6), (4,6,12), (\geq 8,9,12), (6,9,18)\rbrace.
\end{equation}
In particular, $c_6 \neq 0$ and so we can consider the quantity
$$\tilde{c}_6=\frac{c_6}{2^{v(c_6)}},$$
showing that the definition of $u$ in statement of Lemma~\ref{L:lemma4}
always makes sense.

For the rest of this proof we will adopt the notations used in \cite{pap}.

We divide the proof into cases according to $t$.

1) Suppose $t=(\geq 6,6,6)$ and $\tilde{c}_6\equiv 1 \pmod 4$. 

An integral  model  of $E'/\Q_2$, the quadratic twist of $E/\Q_2$ by $\sqrt{2}$, is
$$(W')Ê\ : \ Y^2=X^3-\frac{c_4}{12}X-\frac{c_6}{108}.$$
It satisfies
$$\left(v(c_4(W'),v(c_6(W'),v(\Delta(W')\right)=(\geq 8,9,12).$$
We shall prove that $(W')$ is not minimal, which implies that $E'/\Q_2$ has good reduction. For this, we use the Table~IV and Proposition~6 of \cite{pap}.
We have
$$b_8(W')=-a_4(W')^2,$$
hence the congruence
$$b_8(W')\equiv 0 \pmod {2^8}.$$
So we can choose $r=0$ in Proposition~6 of {\it{loc. cit.}}. Moreover, one has 
$$b_6(W')=4a_6(W')=-\frac{c_6}{27}.$$
Since $\tilde{c}_6\equiv 1 \pmod 4$, one has $-\frac{\tilde{c}_6}{27}\equiv 1 \pmod 4$. The equality $v(c_6)=6$ then implies 
$$b_6(W') \equiv 2^6 \pmod {2^8},$$
and we obtain  our assertion with $x=8$.

2) Suppose $t=(\geq 6,6,6)$ and $\tilde{c}_6\equiv -1 \pmod 4$. 

We proceed as above. An equation of $E'/\Q_2$, the quadratic twist of $E/\Q_2$ by $\sqrt{-2}$,  is 
$$(W') \ : \ Y^2=X^3-\frac{c_4}{12}X+\frac{c_6}{108}.$$
One has again
$b_6(W')\equiv 2^6 \pmod {2^8},$
hence   the result. 

\medskip
For the next two cases below, we will denote by  $b_2$, $b_4$, $b_6$ and $b_8$ the standard invariants 
associated to  the  equation \eqref{(1.1)} of $E/\Q_2$.
\medskip

3) Suppose $t=(4,6,12)$.

An equation of $E'/\Q_2$, the quadratic twist of $E/\Q_2$ by $\sqrt{-1}$,  is 
$$(W') \ : \ Y^2=X^3-\frac{c_4}{48}X+\frac{c_6}{864}.$$
We  will use Table IV and Proposition 4 of \cite{pap} to prove that
$(W')$ is not minimal, establishing that $E'/\Q_2$ has good reduction.

From Table IV in {\it loc. cit.}  and the assumption made on ~$t$, the elliptic curve $E/\Q_2$ corresponds either to the case 7 of Tate (with Neron type 
$C_{5,\nu}$ and
$\nu = 4$) or a non-minimal model.

One has $b_2=0$, $v(b_4)=1$, $v(b_6)=3$ and $v(b_8)=0$.
So there exists $r\in \Z_2$, with $v(r)=0$,  such that (conditions (a) and (b) of Proposition~4)
$$b_8+3rb_6+3r^2b_4+3r^4\equiv 0 \pmod {32}\quad \text{and}\quad r\equiv 1 \pmod 4.$$
Furthermore,
$$b_2(W')=0,\quad b_4(W')=b_4,\quad b_6(W')=-b_6,\quad b_8(W')=b_8.$$
We conclude that the integer $-r$ satisfies the condition (a) of 
the same Proposition 4 for the equation $(W')$.
One has $-3r\equiv 1 \pmod 4$, so condition (b) of this proposition with $s=1$  implies the assertion.

4) Suppose $t=(\geq 8,9,12)$.

Again an equation of $E'/\Q_2$  is
$$(W') \ : \ Y^2=X^3-\frac{c_4}{48}X+\frac{c_6}{864}.$$
We use Proposition~6 of \cite{pap}. The elliptic curve $E/\Q_2$ corresponds to the  case 10 of Tate.  One has
$b_8\equiv 0 \pmod {2^8},$
so $r=0$ satisfies the required condition of this proposition for the equation \eqref{(1.1)}. Since equation \eqref{(1.1)} is  minimal, we deduce (by \cite[Prop~6]{pap}) that  $b_6$ is not a square modulo~$2^8$, so we have
$$b_6=-\frac{c_6}{216}=-\frac{2^6 \tilde{c}_6}{27}\equiv 2^6\tilde{c}_6\equiv -2^6 \pmod {2^8},$$
where the last congruence follows due to $\tilde{c}_6\equiv -1 \pmod 4$. 
From the equality $b_6(W')=-b_6$ it follows that $b_6(W')$ is a square modulo~$2^8$, hence the $(W')$ is not minimal, as desired.

5) Suppose $t=(6,9,18)$. 

Let $\varepsilon=\pm 1$, so that $\tilde{c}_6\equiv \varepsilon \pmod 4$. 
An integral equation of  the quadratic twist of $E/\Q_2$ 
by~$\sqrt{-2\varepsilon}$ is 
$$(W') \ : \ Y^2=X^3-\frac{c_4}{2^6\cdot 3}X+\varepsilon\frac{c_6}{2^8\cdot 27}.$$
It satisfies
$$\left(v(c_4(W')),v(c_6(W')),v(\Delta(W'))\right)=(4,6,12).$$
We will apply Proposition 4 of \cite{pap}.
Recall that $\tilde{c}_4 = c_4/2^{v(c_4)}$.
From the assumption on~$t$, we have
$\tilde{c}_4^3\equiv \tilde{c}_6^2\pmod {32}$,
which implies
$$\tilde{c}_4\equiv 1,9,17,25 \pmod {32}.$$
Moreover, one has $\varepsilon \tilde{c}_6\equiv 1 \pmod 4$, so from the condition (a) for $(W')$, there exists $r\in \Z_2$ such that 
$$-\tilde{c}_4^2+8r -18\tilde{c}_4r^2+27r^4\equiv 0 \pmod {32}.$$
For all the values of $\tilde{c}_4$ modulo $32$, we then verify that this congruence
is satisfied with $r=-1$. 
The condition (b) then implies that $(W')$ is not minimal, so  $E'/\Q_2$ has good reduction, 
completing the proof.

\subsection{Proof of Lemma~\ref{L:lemma2}}
Let $E/\Q_2$ be an elliptic curve with potentially good reduction and 
semistability defect $e=6$. 
Whenever it makes sense we again consider the quantities
$$\tilde{c}_4=\frac{c_4}{2^{v(c_4)}},\quad \tilde{c}_6=\frac{c_6}{2^{v(c_6)}}\quad  \text{and}\quad \tilde{\Delta}=\frac{\Delta_m}{2^{v(\Delta_m)}}.$$
From \cite[Corollary, p. 357]{Kraus1990}, the triple $t=\left(v(c_4),v(c_6),v(\Delta)\right)$
 is one of the following:
\begin{equation*}
(4,5,4)\ \ \text{with} \  \tilde{c}_4\equiv \tilde{c}_6\equiv -1 \pmod 4\quad \text{or} \quad  (\geq 6,5,4) \ \text{with} \  \tilde{c}_6\equiv -1 \pmod 4,
\end{equation*}
\begin{equation*}
(4,6,8)\ \ \text{with} \  \tilde{c}_6\equiv -\tilde{\Delta}\equiv 1 \pmod 4\quad \text{or} \quad  (\geq 7,7,8) \ \text{with} \  \tilde{c}_6\equiv -1 \pmod 4,
\end{equation*}
\begin{equation*}
(6,8,10) \ \text{with} \  \tilde{c}_4\equiv -1 \pmod 4 \quad \text{or}\quad  (\geq 8,8,10),
\end{equation*}
\begin{equation*}
(6,9,14) \ \text{with} \  \tilde{\Delta}\equiv -1 \pmod 4 \quad \text{or}\quad  (\geq 9,10,14).
\end{equation*}
The possibilities for $t$ show that $c_6 \neq 0$ in all cases but $c_4 = 0$ is a possibility; we remark in advance that, in such cases, the quantity $\tilde{c}_4$ does not take part of the arguments below.

Let $E'/\Q_2$ be
the quadratic twist of $E/\Q_2$ by $\sqrt{u}$. A Weierstrass equation of $E'/\Q_2$ is 
\begin{equation*}
\label{(1.6)} 
(W')  \ : \ Y^2=X^3-\frac{c_4 }{48 u^2}X-\frac{c_6}{864 u^3}.
\end{equation*}

The change of variables
 \[
  \begin{cases}
x=u X \\
y=u\sqrt{u} Y,\
  \end{cases}
\]
is an isomorphism between  $E/\Q_2$ and $E'/\Q_2$. One has 
\begin{equation}
\label{(1.7)} 
c_4=u^2c_4(W'),\quad c_6=u^3c_6(W') \quad \text{and}\quad  \Delta_m=u^6\Delta(W').
\end{equation}
Let $\tilde{c}_4(W')$,  $\tilde{c}_6(W')$ and $\tilde{\Delta}(W')$ be the analogue of $\tilde{c}_4, \tilde{c}_6$ and $\tilde{\Delta}$ for the equation $(W')$.

1) Suppose  $v(\Delta_m) \in \lbrace 4,8\rbrace$, so that $u=-1$. From~\eqref{(1.7)} we have the equalities
$$\Delta_m=\Delta(W'),\quad  \tilde{c}_4(W')=\tilde{c}_4 \quad \text{and} \quad \tilde{c}_6(W')=-\tilde{c}_6$$
and  \cite[Corollary, p. 357]{Kraus1990}  implies that $E'/\Q_2$ satisfies $e'=3$.

2) Suppose $v(\Delta_m)=10$, so that $u=\pm 2$.  We obtain
$$v(\Delta(W'))=4,\quad v(c_4(W))=v(c_4)-2 \quad \text{and}\quad \tilde{c}_4(W')=\tilde{c}_4.$$
In case $\tilde{c}_6\equiv 1 \pmod 4$, one has $u=2$, so $\tilde{c}_6(W') = \tilde{c}_6$.
If  $\tilde{c}_6\equiv -1 \pmod 4$, one has $u=-2$, hence
$\tilde{c}_6(W')=-\tilde{c}_6.$ In both cases, $(W')$ is minimal and $e'=3$ by {\it loc. cit.}.

3) Let us suppose  $t=(6,9,14)$. We have $u=\pm 2$ and 
$$v(\Delta(W'))=8,\quad v(c_4(W))=v(c_4)-2 \quad \text{and}\quad \tilde{c}_4(W')=\tilde{c}_4.$$

If $\tilde{c}_6\equiv -1 \pmod 4$, so that $u=2$, we have $\tilde{c}_6(W')=\tilde{c}_6$. If $\tilde{c}_6\equiv 1 \pmod 4$, so that $u=-2$ then $\tilde{c}_6(W')=-\tilde{c}_6$. The conclusion follows as above.

4) In case $t=(\geq 9,10,14)$ a similar argument applies, concluding the proof.

\section{The completeness of Table~\ref{Table:CriteriaList}}

In Theorem~\ref{T:conditionRho}  we gave a classification of when a symplectic criteria exists in terms of the image of $\rhobar_{E,p}$ as a subgroup of $\GL_2(\F_p)$. Our main objective 
in this section is to show that 
Tables~\ref{Table:CriteriaList}~and~\ref{Table:TwistingLemmas}
capture the same information.
More precisely, we will prove 
the following.

\begin{theorem} 
\label{T:tableEquivalence}
Let $\ell$ and $p \geq 3$ be different primes. 
Let $E/\Q_\ell$ and $E'/\Q_\ell$ be elliptic curves such that $E[p] \simeq E'[p]$ as $G_{\Q_{\ell}}$-modules.
Then,  a symplectic criterion exits if and only if the two following equivalent conditions are satisfied.

(a) The condition (A) or (B) of Theorem~\ref{T:conditionRho} holds.

(b) After replacing $E$ and $E'$ by a simultaneous quadratic twist prescribed by Table~\ref{Table:TwistingLemmas}
if necessary, the conditions of a line in Table~\ref{Table:CriteriaList} are satisfied.
\end{theorem}

Before we prove this theorem we give some auxiliary results. 

For $K$ a field of characteristic zero or a finite field of characteristic~$\neq p$ 
we fix, for all primes~$p$, a primitive $p$-th root of unity $\zeta_p \in \overline{K}$.
Recall that, for an elliptic curve $E/K$,  
we donete by~$e_{E,p}$ the Weil pairing on~$E[p]$. 
We write $G_K = \Gal(\overline{K}/K)$.

The following lemma is a well known fact, but due to its importance everywhere in this work, we include a 
proof for completeness.

\begin{lemma} \label{L:twistIso}
Let $K$ be a field of characteristic 0. Let $p \geq 3$ be a prime. Let $E$ and $E'$ be elliptic curves over 
$K$ such that $E[p]$ and $E'[p]$ are isomorphic 
$G_K$-modules.
Let $d \in K^*$ and write $W=dE$, $W'=dE'$ for the quadratic twists of $E$ and $E'$ by~$d$. 

Then $W[p]$ and $W'[p]$ are isomorphic $G_K$-modules.
Moreover, $W[p]$ and $W'[p]$ are  symplectically (anti-symplectically) isomorphic if and only if $E[p]$ and $E'[p]$ are symplectically (anti-symplectically) isomorphic.
\end{lemma}
\begin{proof} 
Let $d \in K$ be be a non-square,
otherwise there is nothing to prove. 

Write $\epsilon : G_K \to \{ \pm 1 \}$ for the character of $G_K$
corresponding to $K(\sqrt{d})$.

Let $C/K$ be an elliptic curve. We write $dC/K$
for its quadratic twist by~$d$. 
We can choose short Weierstrass models (since char $K = 0$)
\[
 C \; : \; y^2 = x^3+ax+b \quad \text{ and } \quad  dC \; : \; y^2 = x^3 + a d^2 x+d^3b.
\]
The map from $C$ to $dC$ given by
$(x,y) \mapsto (dx,d\sqrt{d}y)$ becomes 
an isomorphism over $K(\sqrt{d})$. 
Let $T_C : C[p](\overline{K}) \to dC[p](\overline{K})$ 
denote the map on $p$-torsion induced 
by it.
With a direct calculation we see
that
\begin{equation} \label{E:eps}
 T_C(\sigma(P)) = \epsilon(\sigma)\cdot\sigma(T_C(P)) \quad \text{ for all } \sigma \in G_K 
 \text{ and all } P \in C[P]. 
 \end{equation}
Also, twisting $dC$ by~$d$ recovers $C$ so we have $T_C^{-1} = T_{dC}$. 

The Weil pairing is invariant 
under $\overline{K}$-isomorphism 
(cf. \cite[Proposition~7.4.1 (e)]{DiamondShurman}) so,
for all $P, Q \in C[p](\overline{K})$, we have 
\begin{equation}\label{E:weil}
 e_{dC,p}(T_C(P),T_C(Q)) = e_{C,p}(P,Q).
\end{equation}
Now let $\phi : E[p] \to E'[p]$ be a $G_K$-isomorphism given by assumption. We consider the map
\[
 \phi_d \; : \; W[p] \to W'[p] \quad \text{ given by } \quad
 \phi_d := T_ {E'} \circ \phi \circ T_W, 
 \]
which is the composition of three $\F_p$-vector space isomorphisms. 

Moreover, for $P \in W[p](\overline{K})$ and $\sigma \in G_K$, we have
\begin{eqnarray*}
 \phi_d(\sigma(P)) & = & T_{E'} \circ \phi(\epsilon(\sigma)\sigma(T_W(P)) = T_{E'}(\sigma(\phi(\epsilon(\sigma)T_W(P)))) \\ 
 &=& \epsilon(\sigma)\sigma(T_{E'}(\phi(\epsilon(\sigma)T_W(P)))) = 
 \epsilon(\sigma)\epsilon(\sigma)\sigma(T_{E'}(\phi(T_W(P)))) \\
 & = &\sigma(\phi_d(P)), \
\end{eqnarray*}
where we have used~\eqref{E:eps} twice and the fact 
that $\phi$ is $G_K$-linear. Thus, $\phi_d$ is  a
$G_K$-isomorphism, 
proving the first statement. 

Let $d(\phi)$ and $d(\phi_d)$ be the quantities attached to $\phi$ and $\phi_d$ as in~\eqref{E:dphi}, respectively.

For all $P, Q \in W[p]$, we have 
$ 
e_{W',p}(\phi_d(P),\phi_d(Q)) =  e_{W,p}(P,Q)^{d(\phi_d)}
$
and
\begin{eqnarray*}
  e_{W',p}(\phi_d(P),\phi_d(Q)) & = & e_{W',p}(T_ {E'}(\phi \circ T_W(P)),T_ {E'} (\phi \circ T_W(Q))) \\ 
   & = & e_{E',p}(\phi \circ T_W(P),\phi \circ T_W(Q)) \\
   & = & e_{E,p}(T_W(P),T_W(Q))^{d(\phi)} \\
   & = & e_{W,p}(P,Q)^{d(\phi)}, \\
\end{eqnarray*}
thus $d(\phi_d) \equiv d(\phi) \pmod{p}$. In particular, $d(\phi_d)$ is a square mod~$p$ if and only if $d(\phi)$ is a square mod~$p$, proving the second statement.
\end{proof}

This lemma gives us a very useful corollary.

\begin{corollary}\label{C:symplectTwist}
 Let $\ell$ and $p$ be different primes with $p \geq 3$. Let $K \subset \Qbar_\ell$ be a field. Let
$E/K$ and $E'/K$ be elliptic curves with isomorphic $p$-torsion modules. Let $d \in K^*$.

Then a symplectic criterion exists for $E$ and $E'$ if and only if it exists for $dE$ and $dE'$.
\end{corollary}
\begin{proof} By definition, a symplectic criterion exists for $E$ and $E'$ if and only if the symplectic types of all the 
$G_K$-isomorphisms $\phi : E[p] \simeq E'[p]$ are the same. It is now direct from Lemma~\ref{L:twistIso} that this condition holds if and only if it holds for $dE[p]$ and $dE'[p]$.
\end{proof}

We will also need well known facts 
from the theory of the Tate curve, which  
we summarize in the following proposition. 
For $p \geq 3$ a prime we fix~$\zeta_p \in \Qbar_\ell$ a $p$-th root of unity.

\begin{proposition} \label{P:theoryTate}
Let $\ell$ and $p \geq 3$ be different primes. 
Let $C/\Q_\ell$ be an elliptic curve with potentially multiplicative reduction, minimal 
discriminant~$\Delta_m(C)$ and $j$-invariant~$j_C$. 
Then there exists a symplectic basis of $C[p]$ in which the $p$-torsion representation is of the form
\begin{equation}
 \label{E:tate0}
 \rhobar_{C,p} \simeq 
 \begin{pmatrix} \chi \chi_p & h \\ 0 & \chi \end{pmatrix},
\end{equation}
where $\chi_p$ is the mod~$p$ cyclotomic 
character and $\chi$ is 
character of order at most 2. Moreover, 
\begin{enumerate}
 \item $C$ has additive reduction if and only if~$\chi$ is ramified.
 \item $C$ has split multiplicative reduction if and only if~$\chi = 1$.
 \item We have $h = 0 \iff \#\rhobar_{C,p}(I_\ell) = 1  \iff p \mid \vv_\ell(\Delta_m(C))$.
 \item For $\sigma \in I_\ell$, we have $\sigma(j_C^{1/p}) = \zeta_p^{h(\sigma)} \cdot j_C^{1/p}$ for any choice of~$j_C^{1/p}$ a $p$-th root of $j_C$.
\end{enumerate}
\end{proposition}
\begin{proof} 
Everything in this proposition follows from the well known theory of the Tate curve.
This is treated in detail in~\cite[Chap. V]{SilvermanII}. In particular, see Proposition~6.1 and Exercise~5.13. 
For the claim that~\eqref{E:tate0} holds in a symplectic basis we refer to~\cite{FukudaWeilPairing}; in particular, see
Section 2.4 and Remark~2.10 there.

We shall only prove (4) which is a small modification of a standard fact. 
It is well known that after at most a quadratic twist $C$ is $\Q_\ell$-isomorphic to a Tate curve~$E_q/\Q_\ell$ with Tate parameter~$q \in \Q_\ell^*$. 
Fix $q^{1/p} \in \Qbar^*_\ell$ a $p$-th root of~$q$.
Moreover, it is also known that, for~$\sigma \in I_\ell$, 
the value of~$h$ in 
in~\eqref{E:tate0} is given by the action 
of~$\sigma$ on~$q^{1/p}$ therefore, we have to show this is the same as the action of~$\sigma$ on~$j_C^{1/p}$.
This follows from the fact that $q \cdot j_C$ is a $p$-th power in~$\Q_\ell$; indeed, we have $\vv_\ell(q) > 0$, therefore
\[
j_C \cdot q\equiv 1+744q +O(q^2) \implies j_C \cdot q \equiv 1\pmod{\ell}
\]
and, since $\ell \neq p$, it follows from Hensel's lemma that $q \cdot j_C = s^p$ with $s \in \Q_\ell$.

Finally, we show the value of $h(\sigma)$ in (4) is independent of the choice of~$j_C^{1/p}$. 
Fix~$j_C^{1/p}$ a $p$-th root of $j_C$ and let $r = \zeta_p^k \cdot j_C^{1/p}$ be another such $p$-th root.
Thus,
\[ \sigma(r) = \sigma (\zeta_p^k \cdot j_C^{1/p}) = \zeta_p^k \cdot  \sigma(j_C^{1/p}) =  \zeta_p^k \cdot \zeta_p^{h(\sigma)} \cdot j_C^{1/p} 
= \zeta_p^{h(\sigma)} \cdot r,
\]
as desired; here the second equality follows because $\sigma \in I_\ell$ and $\zeta_p \in \Q_\ell^{un}$.
\end{proof}

We now recall some notation
from Section~\ref{S:notation}.
Let $\ell$ and $p \geq 3$ be different primes. 
Given an elliptic curve $E/\Q_\ell$ with potentially good reduction 
we write $L=\Q_\ell^{\text{un}}(E[p])$ for the {\it inertial field} of $E$. It is the minimal extension of $\Q_\ell^{\text{un}}$ where 
$E$ obtains good reduction.  
We call $e = e(E) = \#\Gal(L/\Q_\ell^{\text{un}})$ the {\it semistability defect of} $E$. 
Recall that $e \in \{1, 2, 3, 4, 6, 8, 12, 24\}$.

\begin{proposition} 
\label{P:TwistsameType}
Let $\ell$, $p \geq 3$ be different primes. Let $E$ and $E'$ be elliptic curves over $\Q_\ell$ such that
$E[p]$ and $E'[p]$ are isomorphic $G_{\Q_{\ell}}$-modules.

(A) Suppose $E$ and $E'$ have
potentially good reduction with semistability defect $e$ and 
$e'$, respectively. Then $e = e'$.

Assume further $e \in \{2,6\}$. 
Then, for $d \in \Q_\ell^*$ described by Lemma~\ref{L:lemma1},~\ref{L:lemma2},~\ref{L:lemma3}~or~\ref{L:lemma4}, 
the twists $W=dE$,
$W' = dE'$ satisfy $W[p] \simeq W'[p]$ as $G_{\Q_{\ell}}$-modules
and 
\begin{itemize}
 \item[(a)] $e(W) = e(W')=3$ if $e=6$.
 \item[(b)] $e(W) = e(W')=1$ if $e=2$.
\end{itemize}

(B) Suppose $E$ and $E'$ have
potentially multiplicative reduction. Then, for $d \in \Q_\ell$ described by Lemma~\ref{L:multTwist},
the twists $W=dE$, $W' = dE'$ have split multiplicative reduction and $W[p] \simeq W'[p]$.
\end{proposition}
\begin{proof} 
We first prove (A).  
The inertial field of $E$ is also the field fixed
by the restriction of $\rhobar_{E,p}$ to $G_{\Q_\ell^{\unr}}$ and similarly for $E'$.
From $\rhobar_{E,p} \simeq \rhobar_{E',p}$ it follows that 
the inertial fields satify
$
L=\Q_\ell^{\text{un}}(E[p]) = \Q_\ell^{\text{un}}(E'[p]) = L',
$
hence $e = \#\Gal(L/\Q_\ell^{\text{un}}) = \#\Gal(L'/\Q_\ell^{\text{un}}) = e'$, as desired.

Suppose now $e=6$. From Lemma~\ref{L:lemma1}~or~\ref{L:lemma2}
we conclude there is an element $d \in \Q_\ell$ such 
that $W = dE$ has semistability defect $e=3$. From Lemma~\ref{L:twistIso} it follows that 
$W[p] \simeq W'[p]$ and from part (A) applied with $W$ and $W'$
it follows that $e(W) = e(W') = 3$, proving (a). The proof of~(b) is similar, where we replace Lemmas~\ref{L:lemma1}~and~\ref{L:lemma2} 
by Lemmas~\ref{L:lemma3}~and~\ref{L:lemma4}. 

We now prove (B). By Lemma~\ref{L:multTwist},
since~$E$ has potentially multiplicative reduction, 
there exists~$d \in \Q_\ell$ such that $W = dE$ has split multiplicative reduction 

From Lemma~\ref{L:twistIso}
we get $W[p] \simeq W'[p]$.
To complete the proof we will 
show that $W'$ has split multiplicative reduction.
Applying Proposition~\ref{P:theoryTate} 
with $C=W$ gives
\[
\begin{pmatrix} \chi_p & h \\ 0 & 1 \end{pmatrix} \simeq
\rhobar_{W,p} \simeq \rhobar_{W',p} 
\]
and since $W'$ also has potentially multiplicative reduction (as this is the case for $E'$) by applying Proposition~\ref{P:theoryTate} 
now with $C=W'$ we conclude that $W'$ has split multiplicative reduction (from part (2) because $\chi = 1$).
\end{proof}

\begin{proposition} \label{P:multTwist},
Let $\ell$, $p \geq 3$ be different primes. Let $E$ and $E'$ be elliptic curves over $\Q_\ell$ such that
$E[p]$ and $E'[p]$ are isomorphic $G_{\Q_{\ell}}$-modules.

Suppose that $E$ has potentially multiplicative reduction and $E'$ has potentially good
reduction. 

Then, 
after taking a simultaneous 
qudratic twist of $E$ and $E'$ by some $d \in \Q_\ell$ if necessary,
we are in one of the following cases: 
\begin{itemize}
 \item[(i)] $E$ has split multiplicative reduction, 
 $E'$ has good reduction and 
 $p \mid \vv_\ell(\Delta_m(E))$, 
 where $\Delta_m(E)$ is the discriminant of a minimal model for~$E$.
 \item[(ii)] $E$ has split multiplicative reduction, 
 $E'$ has potentially good reduction with semistability defect $e(E')=3$
 and $p=3$. 
\end{itemize}
\end{proposition}
\begin{proof} From Lemma~\ref{L:multTwist} 
there is $d \in \Q_\ell$ such that the $W = dE$ has split multiplicative reduction. 
Applying Proposition~\ref{P:theoryTate}
with $C=W$ gives
\[
\rhobar_{W,p} \simeq 
 \begin{pmatrix} \chi_p & h \\ 0 & 1 \end{pmatrix} \quad \text{and} \quad 
 \#\rhobar_{W,p}(I_\ell) = 1 
 \iff
p \mid \vv_\ell(\Delta_m(W)).
\]
Since $\chi_p$ is unramified at~$\ell$, we have that $\rhobar_{W,p}(I_\ell)$ is of order 1 or~$p$. 

Write $W' = dE'$. This curve
has potentially good reduction. 
From Lemma~\ref{L:twistIso}, we have $W[p] \simeq W'[p]$ as $G_{\Q_{\ell}}$-modules. In particular, 
$\# \rhobar_{W,p}(I_\ell) = \# \rhobar_{W',p}(I_\ell)= e(W')$. We now split into cases according to the two possible orders of $\rhobar_{W,p}(I_\ell)$ above: 

(i) Suppose $\# \rhobar_{W,p}(I_\ell) = 1$.
Then $p \mid \vv_\ell(\Delta_m(W))$ and $e(W') = 1$ (i.e. $W'$ has good reduction) 

(ii) Suppose $\# \rhobar_{W,p}(I_\ell) = p$. 
Since $e(W') \in \{1, 2, 3, 4, 6, 8, 12, 24\}$ we conclude $p= e(W')=~3$.
\end{proof}

\begin{proposition} \label{P:mixedCase}
Let $\ell$, $p \geq 3$ be different primes. Let $E$ and $E'$ be elliptic curves over $\Q_\ell$ such that
$E[p]$ and $E'[p]$ are isomorphic $G_{\Q_{\ell}}$-modules.

Suppose that $E$ has potentially multiplicative reduction and $E'$ has potentially good
reduction. 

If a symplectic criterion exists, 
after taking a simultaneous 
quadratic twist of $E$ and $E'$ by some $d \in \Q_\ell$ if necessary,
then we are in case (ii) of Proposition~\ref{P:multTwist}.
\end{proposition}
\begin{proof}
The existence of a symplectic criterion is invariant by taking a simultaneous quadratic twist by~$d \in \Q_\ell^*$ (cf. Corollary~\ref{C:symplectTwist}). 
From Proposition~\ref{P:multTwist}, 
after twisting both curves by some~$d$ if needed,
we can assume that one of the two cases
in that proposition is satisfied.
In both cases $E$ has split multiplicative reduction and 
Proposition~\ref{P:theoryTate} with $C = E$ gives
\begin{equation} \label{E:tate1}
\rhobar_{E,p} \simeq 
 \begin{pmatrix} \chi_p & h \\ 0 & 1 \end{pmatrix}.
\end{equation}
Assume case (i) of Proposition~\ref{P:multTwist}. Then the image of $\rhobar_{E',p}$ is cyclic (generated by the image of Frobenius) and $\rhobar_{E',p}(I_\ell) = 1$. 
Hence, if a symplectic criterion exits, we are in case (B) of Theorem~\ref{T:conditionRho}.
Therefore, from the isomorphism $E[p] \simeq E'[p]$ we conclude that in~\eqref{E:tate1} we have
$\chi_p = 1$ 
and $h \neq 0$. Since $E$ has multiplicative reduction and $h \neq 0$, it follows from Proposition~\ref{P:theoryTate}
that
$\#\rhobar_{E,p}(I_\ell) \neq 1$,
a contradiction with 
$\# \rhobar_{E',p}(I_\ell) = 1$.

We conclude we are in case (ii)
of Proposition~\ref{P:multTwist}.
\end{proof}

\begin{proof}[Proof of Theorem~\ref{T:tableEquivalence}]
From Theorem~\ref{T:conditionRho} we know that (a) is equivalent to the existence of a symplectic criterion.

Suppose (b). Thus, after a simultaneous quadratic twist by some~$d \in \Q_\ell^*$, the conditions in a line of Table~\ref{Table:CriteriaList}
are satisfied.
Then the theorems in the fourth column of that line can be applied (by our assumption) to $dE$ and $dE'$
from which it follows that a symplectic criterion exists for $dE$ and $dE'$. From Corollary~\ref{C:symplectTwist} we conclude that a criterion exists for $E$ and $E'$.

Now we shall show that if a symplectic criterion exists for $E$ and $E'$ then, after at most a simultaneous twist prescribed by Table~\ref{Table:TwistingLemmas}, the curves
$dE$ and $dE'$ satisfy the conditions in a line of
Table~\ref{Table:CriteriaList}. 

We divide the proof in three cases according to the reduction types of~$E$ and $E'$.

{\sc Case 1:} We consider first the case of mixed reduction types. After interchanging $E$ and $E'$ 
if necessary, we can
suppose that $E$ has potentially
multiplicative reduction
and $E'$ potentially good
reduction. 

Suppose a symplectic criterion exists. From Proposition~\ref{P:mixedCase}, after possibly twisting both curves by an element $d \in \Q_\ell$ (given by Lemma~\ref{L:multTwist} applied to~$E$), we conclude that
$E/\Q_\ell$ has split multiplicative reduction and $p=e(E')=3$. This corresponds to the last line of Table~\ref{Table:CriteriaList}. 

Moreover, in case a ramified twist was taken, we see that before
we had $e(E')=6$ and $E/\Q_\ell$ with additive potentially multiplicative reduction. This corresponds to the reduction types in the last line of Table~\ref{Table:TwistingLemmas}. Furthermore, instead of Lemma~\ref{L:multTwist} applied to~$E$, we could have applied Lemma~\ref{L:lemma1} (if $\ell \geq 3$) or Lemma~\ref{L:lemma2} (if $\ell = 2$) 
to $E'$, as this would also reduce the inertia sizes on both curves from 6 to 3 (due to Lemma~\ref{L:twistIso}). This establishes the last line of Table~\ref{Table:TwistingLemmas}. 

{\sc Case 2:} Suppose that $E$ and $E'$ have potentially multiplicative reduction. 
From Proposition~\ref{P:TwistsameType} part (B),
after replacing the curves by a twist given by Lemma~\ref{L:multTwist} if necessary, we can assume 
both curves have split  multiplicative reduction. 

Suppose that $p \mid \vv_\ell(\Delta_m(E))$. Applying Proposition~\ref{P:theoryTate} with $C=E$, we conclude that
$\rhobar_{E,p} \simeq \chi_p \oplus 1$, thus none of the conditions in Theorem~\ref{T:conditionRho} are satisfied, so a symplectic criterion does not exist.
Suppose a symplectic criterion exists, hence 
$p \nmid \vv_\ell(\Delta_m(E))$.

This corresponds to the 11th line of Table~\ref{Table:CriteriaList} and 3rd line of Table~\ref{Table:TwistingLemmas}.

{\sc Case 3:} Suppose that both 
$E/\Q_\ell$ and $E'/\Q_\ell$ have potentially good reduction. Write $e$ and~$e'$ for their semistability defects. 
From Proposition~\ref{P:TwistsameType} and its proof
we see that, up to twist prescribed by lines 1 or 2 of Table~\ref{Table:TwistingLemmas}, we have $e=e'$ and can assume $e \in \{1,3,4,8,12,24\}$.
Recall also that $e=8,24$ only occurs for $\ell = 2$ and $e=12$ for $\ell = 3$.

We subdivide into cases:

(a) Suppose $e=8,12$ or $24$ and that a symplectic criterion exist
(actually this is automatic since inertia is non-abelian), then the conditions in line 8, 9 or 10 of Table~\ref{Table:CriteriaList} are satisfied, respectively.

(b) Suppose $e=3$ or $e=4$ and that a symplectic criterion exists. 

Suppose further that condition~(B) of 
Theorem~\ref{T:conditionRho} holds. Then, 
by Proposition~\ref{P:abeliane},
we have $p=e=3$ and $\ell \equiv 1 \pmod{3}$.
This corresponds to the 4th line of Table~\ref{Table:CriteriaList}.
Thus, we can assume 
condition (A) holds, hence
the $p$-torsion field $K=\Q_\ell(E[p])$ is non-abelian. 
We subdivide further:

(i) wild case 
$\ell = e = 3$. 
Then Proposition~\ref{P:nonabelianWild3}
gives $\tilde{\Delta} \equiv 2 \pmod{3}$.

(ii) wild case 
$\ell = 2$ and $e = 4$. 
Then Proposition~\ref{P:nonabelianWild4}
gives $\tilde{c}_4 \equiv 5 \tilde{\Delta} \pmod{8}$.

(iii) tame reduction. Then Corollary~\ref{C:nonabelianTame}
gives $\ell \equiv -1 \pmod{e}$.

Cases (i) and (ii) correspond, respectively, to lines 5 and 7 of Table~\ref{Table:CriteriaList}; case (iii) corresponds to lines 3 and 6.

(c) Suppose $e=1$, i.e. $E$ has good reduction, and that a symplectic criterion exists. 
From Proposition~\ref{P:AbelianCentralizer} we have
$p \mid \Delta_\ell$ and $p \nmid \beta_\ell$; 
Proposition~\ref{P:tableConditionI} 
implies $(\ell/p) = 1$; this corresponds to the second line; assuming further $\Ebar = \Ebar'$ gives the first line.

This completes the equivalence with (b). 
\end{proof}

{\large \part{The morphism $\gamma_E$}}

\section{Explicit description of $\gamma_E$}
\label{S:mapgammaE}

Let $E/\Q_\ell$ be an elliptic curve with potentially good reduction. 
Let $\pi_L$ be a uniformizer in $L=\Q_\ell^{\text{un}}(E[p])$ and
write $\Ebar$ for the elliptic curve over $\Fbar_\ell$ obtained by 
reduction mod~$(\pi_L)$ of a model of $E / L$ with good reduction
and $\varphi : E[p] \rightarrow \overline{E}[p]$ for the reduction morphism. 
Let $\Aut(\Ebar)$ be the automorphism group of $\Ebar$ over $\Fbar_\ell$. 

In \cite[Section 2]{ST1968} it is proved that $\Phi=\Gal(L/\Q_\ell^{un})$ 
acts on $\Ebar$ by $\Fbar_\ell$-automorphisms.\footnote{The arguments in \cite[Section 2]{ST1968} hold in the more general setting of abelian varieties with potentially good reduction.} 
Indeed, write $\psi : \Aut(\Ebar) \rightarrow \GL(\Ebar[p])$ for the natural
injective morphism. The action of $\Phi$ on $L$ induces an injective homomorphism 
$\gamma_E : \Phi \rightarrow \Aut(\overline{E})$. Furthermore, for $\sigma \in \Phi$ we have
\begin{equation}
\label{E:phi}
 \varphi \circ \rhobar_{E,p}(\sigma) = \psi(\gamma_E(\sigma)) \circ \varphi.
\end{equation}

The proof of our main results relies on explicit computations with the morphism $\gamma_E$.
Since the arguments in \cite[Section 2]{ST1968} are not explicit enough for our purpose,
in this section we will give a direct proof of the existence of $\gamma_E$ and establish
\eqref{E:phi}.

Consider a Weierstrass model over $\Q_\ell$ 
\[
W : y^2+a_1xy+a_3y=x^3+a_2x^2+a_4x+a_6. 
\]
From \cite[Proposition VII.1.3]{SilvermanI} there is a change of coordinates
\begin{equation}
 \label{E:coordchangeprime} 
x=u^2x'+r, \qquad y=u^3y'+u^2sx'+t, \quad u,r,s,t \in \calO_L
\end{equation}
which transforms $W$ into a minimal model with good reduction over $L$ 
\[
W' :  y'^2+a_1'x'y'+a_3'y'=x'^3+a_2'x'^2+a_4'x'+a_6'.
\]
Moreover, from Table~3.1 in \cite{SilvermanI}, we have
\begin{equation}
\label{E:coordchange}
\begin{cases}
ua'_1   & =    a_1+2s  \\
u^2a'_2 & =  a_2-sa_1+3r-s^2 \\
u^3a'_3 & =  a_3+ra_1+2t \\
u^4a'_4 & =  a_4-sa_3+2ra_2-(t+rs)a_1+3r^2-2st \\
u^6a'_6 & =  a_6+ra_4+r^2a_2+r^3-ta_3-t^2-rta_1. 
\end{cases}
\end{equation}
and the standard invariants are related by
\[
 c_4=u^4c'_4,\quad c_6=u^6c'_6,\quad \Delta=u^{12}\Delta'.
\]
Take $\sigma \in \Phi$. We claim that the coordinate change
\begin{equation} \label{E:recall1}
x=\sigma(u)^2 x''+\sigma(r),\quad y=\sigma(u)^3 y''+\sigma(u)^2\sigma(s) x''+\sigma(t) 
\end{equation}
transforms $W$ into the model 
\[
W'' : y''^2+\sigma(a_1')x'' y''+\sigma(a_3')y''=x''^3+\sigma(a_2')x''^2+\sigma(a_4')x''+\sigma(a_6')
\]
which is in fact minimal.
Indeed, write $a_i''$ for the coefficients of $W''$ obtained by the previous coordinate change.
We have
\begin{equation}
 \begin{cases}
\sigma(u)a''_1   & =  a_1+2\sigma(s)  \\
\sigma(u)^2a''_2 & =  a_2-\sigma(s)a_1+3\sigma(r)-\sigma(s)^2  \\
\sigma(u)^3a''_3 & =  a_3+\sigma(r)a_1+2\sigma(t) \\
\sigma(u)^4a''_4 & =  a_4-\sigma(s)a_3+2\sigma(r)a_2-(\sigma(t)+\sigma(r)\sigma(s))a_1+3\sigma(r)^2
-2\sigma(s)\sigma(t) \\
\sigma(u)^6a''_6 & =  a_6+\sigma(r)a_4+\sigma(r)^2a_2+\sigma(r)^3-\sigma(t)a_3
-\sigma(t)^2-\sigma(r)\sigma(t)a_1.
\end{cases}
\end{equation}
Since $a_i$ are fixed by $\sigma$ by applying $\sigma$ to \eqref{E:coordchange} we see that $\sigma(a'_i)=a''_i$. Clearly, $a_i'' \in \calO_L$ and since $\sigma(u)$ has the same valuation
as $u$ the model $W''$ is minimal as desired. Therefore, there exists a coordinate change
\begin{equation}
\label{E:coordchangehat} 
x'=\hat{u}^2x'' + \hat{r}, \quad y'=\hat{u}^3y''+\hat{u}^2 \hat{s}x''+ \hat{t}, \quad u,r,s,t \in \calO_L, \quad \upsilon_L(\hat{u}) = 0
\end{equation}
which transforms $W'$ into $W''$.
Thus we have
\begin{equation}
\label{E:autoeqn}
 \begin{cases}
\hat{u}\sigma(a'_1)   & =  a'_1+2\hat{s}  \\
\hat{u}^2\sigma(a'_2) & =  a'_2-\hat{s}a'_1+3\hat{r}-\hat{s}^2  \\
\hat{u}^3\sigma(a'_3) & =  a'_3+\hat{r}a'_1+2\hat{t} \\
\hat{u}^4\sigma(a'_4) & =  a'_4-\hat{s}a'_3+2\hat{r}a'_2-(\hat{t}+\hat{r}\hat{s})a'_1+3\hat{r}^2
-2\hat{s}\hat{t} \\
\hat{u}^6\sigma(a'_6) & =  a'_6+\hat{r}a'_4+\hat{r}^2a'_2+\hat{r}^3-\hat{t}a'_3
-\hat{t}^2-\hat{r}\hat{t}a'_1.
\end{cases}
\end{equation}

\begin{lemma} The following equalities hold
\[
 \hat{u} = \frac{\sigma(u)}{u}, \quad \hat{r} = \frac{\sigma(r) - r}{u^2}, \quad \hat{s} = \frac{\sigma(s)-s}{u}, \quad
 \hat{t} = \frac{\sigma(t)-t-s(\sigma(r) - r)}{u^3}.
\]
\label{L:hats}
\end{lemma}
\begin{proof} From the equations \eqref{E:coordchangeprime} and \eqref{E:coordchangehat}
it follows that 
\[x=u^2\hat u^2 x''+u^2\hat r+r,\]
\[y=u^3\hat u^3 y''+(u^3\hat u^2 \hat s+u^2\hat u^2 s)x''+u^3\hat t+u^2 s\hat r+ t.\]
Recall from \eqref{E:recall1} that
\[
x=\sigma(u)^2 x''+\sigma(r),\quad y=\sigma(u)^3 y''+\sigma(u)^2\sigma(s) x''+\sigma(t).
\]
The result now follows by equating all the coefficients. 
\end{proof}

Since $\sigma \in \Phi$ the images of $\sigma(a_i')$ and $a_i'$ modulo $(\pi_L)$ 
are the same element in the residual field $\Fbar_\ell$.
Therefore the models $W'$ and $W''$ reduce to the same elliptic curve $\Ebar$ 
over $\Fbar_\ell$. By reducing the equations \eqref{E:autoeqn} we obtain an automorphism of $\Ebar$
which we denote by $\gamma_E(\sigma)$. 

\begin{definition} Let $\sigma \in \Phi$. We define $\gamma_E(\sigma) \in \Aut(\Ebar)$ to 
be the automorphism given by 
\[
 (X,Y) \mapsto (\xi_1^2 X + \xi_2, \xi_1^3 Y + \xi_3\xi_1^2 X + \xi_4)
\]
where
\[
 \xi_1 = \hat{u} + (\pi_L), \quad \xi_2 = \hat{r} + (\pi_L), \quad \xi_3 = \hat{s} + (\pi_L), \quad 
 \xi_4 = \hat{t} + (\pi_L).
\]
\label{D:gammaE}
\end{definition}

\begin{lemma}
 \label{L:gammaE}
 For all $\sigma \in \Phi$ we have
 \[
  \varphi \circ \rhobar_{E,p}(\sigma) = \psi(\gamma_E(\sigma)) \circ \varphi,
 \]
where $\varphi : W'[p] \to \Ebar[p]$ is the reduction morphism.
In particular, $\gamma_E$ is a well defined injective homomorphism.
\end{lemma}
\begin{proof} Let $P=(x,y)$ be a $L$-point in $E[p]$ satisfying the model $W$.
We have $x,y \in \calO_L$ (see \cite[VII.3 Theorem~3.4]{SilvermanI})
and $\sigma(P) = (\sigma(x),\sigma(y))$ also satisfies
the model $W/\Q_\ell$.
Write $(x',y')$ and $(x_1,y_1)$ respectively for the coordinates of $P$ and $\sigma(P)$ in $W'$. 
Thus
\[
 \varphi(P) = (X,Y) = (x' + (\pi_L), y' + (\pi_L)), \quad 
 \varphi(\sigma(P)) = (X_1,Y_1) = (x_1 + (\pi_L), y_1 + (\pi_L)).
\]
We have $x = u^2 x' + r$ hence $\sigma(x) = \sigma(u)^2 \sigma(x') + \sigma(r)$.
Furthermore $\sigma(x) = u^2 x_1 + r$ thus
\[
 u^2 x_1 + r = \sigma(u)^2 \sigma(x') + \sigma(r)
\]
and we conclude
\[
 x_1 = \hat{u}^2 \sigma(x') + \hat{r}.
\]
Working in an analogous way with the $y$-coordinate gives
\[
 y_1 = \hat{u}^3 \sigma(y') + \hat{u}^2 \hat{s} \sigma(x') + \hat{t}.
\]
By reducing mod~$(\pi_L)$ the two previous equations and taking into 
account that $x' \equiv \sigma(x')$ mod~$(\pi_L)$ we obtain equalities
\[
 X_1 = \bar{\hat{u}}^2 X + \bar{\hat{r}}, \qquad Y_1 = \bar{\hat{u}}^3 Y + \bar{\hat{s}}\bar{\hat{u}}^2 X + \bar{\hat{t}}.
\]
By definition of $\gamma_E(\sigma)$ the previous two equalities mean 
\[
 \varphi(\sigma(P)) = \gamma_E(\sigma)(\varphi(P)),
\]
which implies the equality in the statement. 
From it follows that
$\gamma_E$ is a homomorphism. Finally, note that
if $\sigma = 1$, then from Lemma~\ref{L:gammaE} we have 
\[
 \hat{u} = 1 \quad \text{ and } \quad \hat{r} = \hat{s} = \hat{t} = 0,
\]
so by Definition~\ref{D:gammaE} we conclude 
$\gamma_E(\sigma)=1$ 
(since $\xi_1=1$ and $\xi_2=\xi_3=\xi_4=0$).
\end{proof}

\section{The morphism $\gamma_E$ in the tame case $e=3$}
\label{S:gammaEtame3}

Let $E/\Q_\ell$ be an elliptic curve having potentially good
reduction with $e=3$ and $p$-torsion field~$K$.
Let $\ell \equiv 2 \pmod{3}$, so that
$K/\Q_\ell$ is non-abelian by Corollary~\ref{C:nonabelianTame}. 
We have already fixed $\zeta_3 \in K_{un} \subset K$ 
in Section~\ref{S:TauGamma}. For $\ell \equiv 1 \pmod{3}$ we now 
fix a cubic root of unity $\zeta_3 \in \Q_\ell$. 

Write $F = \Q_\ell(\pi)$ where $\pi^3 = \ell$. 
Note that $\pi$ is a uniformizer of $L = \Q_\ell^{un}F$
and let $\omega_3 \in \Fbar_\ell$ be the cubic root of unity 
given by $\omega_3 \equiv \zeta_3 \pmod{\pi}$. 
Let $\sigma$ be the generator 
of $\Phi = \Gal(L/\Q_\ell^{un})$ given by $\sigma(\pi) = \zeta_3 \pi$.
Note that, when $\ell \equiv 2 \pmod{3}$, it coincides with the $\sigma$ 
given by Proposition~\ref{P:groupStructure} and Section~\ref{S:TauGamma}.
The following results compute $\gamma_E(\sigma)$ explicitly in the case
that $E$ has a $3$-torsion point defined over $\Q_\ell$.

\begin{lemma} Let $\ell \neq 3$ be a prime. 
Let $E/\Q_\ell$ be an elliptic curve with potentially good reduction 
and $e=3$. Suppose further that $E$ has a $3$-torsion point in $\Q_\ell$. 
Then $E$ admits a minimal Weierstrass model over $\Q_\ell$ of the form
\[
 y^2 + axy + by = x^3, \qquad \Delta_m  = b^3(a^3 - 27b)
\]
where $b = \ell^\alpha u_0$ with $u_0 \in \Z_\ell^*$, 
$\alpha=1,2$ and $\upsilon_\ell(b) < 3\upsilon_\ell(a)$. 
In particular, $\upsilon_\ell(\Delta_m) = 4\alpha$.
\label{L:3torsionmodel}
\end{lemma}

\begin{proof} Write $\vv = \vv_\ell$.
Since $e=3$ we have $\vv(\Delta_m) \not\equiv 0 \pmod{3}$ and the same is
true for the discriminant $\Delta$ of any other model of $E$.

Since $E$ has a 3-torsion point over $\Q_\ell$ there is a model of $E/\Q_\ell$
of the form
\[
 y^2 + axy + by = x^3, \quad a, b \in \Z_\ell, \qquad \Delta = b^3(a^3 - 27b),
\]
where $(0,0)$ has order 3 (see \cite[pp. 89, Remark~2.2]{Husemoller}). 
Observe that $b \neq 0$ and the discriminant is not necessarily minimal.

We will now show that $\vv(b) < 3\vv(a)$. If $\vv(b) > 3\vv(a)$ then $\vv(\Delta) = 3\vv(b)+3\vv(a)$
is a multiple of 3 which is impossible; thus $\vv(b) \leq 3\vv(a)$. 

Suppose $\vv(b) = 3\vv(a)$. Then $\vv(a^3 - 27b) \geq \vv(b)$ which implies
$\vv(\Delta) = 12\vv(a)$ when the equality holds. Again, since 
$3 \nmid \vv(\Delta)$ we conclude $\vv(a^3 - 27b) > \vv(b)$,
hence $\vv(\Delta) > 4\vv(b)$.

Now, the equality $a^3 - 24b = a^3 - 27b + 3b$
implies $\vv(a^3 - 24b) = \vv(b)$. We also have
\[
 j_E = \frac{a^3(a^3 - 24b)^3}{\Delta} \quad \text{ and } \quad  \vv(j_E) \geq 0,
\]
from which follows $4\vv(b) \geq \vv(\Delta)$, a contradiction. 
Thus $0 \leq \vv(b) < 3\vv(a)$ as claimed and, in particular, $\upsilon(a) > 0$.

The coordinates change $x = \ell^{2\lambda} x'$, $y = \ell^{3\lambda} y'$ replaces $a$ and $b$ by $a\ell^{-\lambda}$ and $b\ell^{-3\lambda}$ respectively, hence by choosing an appropriate $\lambda$ we can assume that $\upsilon(b) = \alpha \in \{0,1,2\}$. 

If $\alpha = 0$ then $\upsilon(\Delta) = 0$ which is impossible because $3 \nmid \upsilon(\Delta)$;
if $\alpha = 1$, the relation $\upsilon_\ell(b) < 3\upsilon_\ell(a)$ implies $\upsilon(\Delta) = 4$
and if $\alpha = 2$ it implies $\upsilon(\Delta) = 8$. In particular, the model is minimal in 
both cases (see \cite[Remark~1.1]{SilvermanI}).
\end{proof}

If, in Lemma~\ref{L:3torsionmodel}, we impose the extra condition $\ell \equiv 2 \pmod{3}$, then more can be said.

\begin{lemma} Let $\ell \equiv 2 \pmod{3}$ be a prime. 
Let $E/\Q_\ell$ be an elliptic curve with potentially good reduction
and $e=3$. 
Suppose that $E$ has a $3$-torsion point in $\Q_\ell$. 

Then there is a model of $E/F$ with good reduction and residual curve $\Ebar : Y^2 + Y = X^3$.

Moreover, $\gamma_E(\sigma) : (X,Y) \mapsto (\omega_3^{2\alpha}X,Y)$ where $\alpha = \vv_\ell(\Delta_m)/4$.
\label{L:gammaEe3}
\end{lemma}

\begin{proof} From Lemma~\ref{L:3torsionmodel} there is a model of $E/\Q_\ell$ of the form
\[
 y^2 + axy + by = x^3, \quad a, b \in \Z_\ell, \quad \Delta_m  = b^3(a^3 - 27b)
\]
where $b = \ell^\alpha u_0$ with $u_0 \in \Z_\ell^*$, 
$\alpha=1,2$ and $\upsilon_\ell(b) < 3\upsilon_\ell(a)$. 

We can assume $b=\ell^\alpha$.
Indeed, since $3 \nmid \ell - 1$ it follows from Hensel's 
lemma that $u_0 = u_1^3$ with $u_1$ in $\Z_\ell^*$. 
Then the change of coordinates $x = u_1^2 x'$, $y = u_1^3 y'$ 
leads to the model over $\Z_\ell$
\[
 W \; : \; y^2 + \frac{a}{u_1} xy + \ell^\alpha y = x^3,
\]
as desired. From this model, the change of variables $x = \pi^{2\alpha} x'$, $y = \pi^{3\alpha} y'$
give rise to a model over $F$ defined by
\[
 W' \; : \; y'^2 + \frac{a}{u_1 \pi^\alpha} x' y' + y' = x'^3 \quad \text{with} 
 \quad \upsilon_F(\Delta(W')) = \upsilon_F(\Delta_m) - 12\alpha = 0.
\]
Since $\upsilon_F(a/u_1 \pi^\alpha) = 3\upsilon(a) - \alpha > 0$ the model $W'$ is minimal with 
good reduction and its residual elliptic curve is $\overline{W'} = \Ebar$.
We have $\sigma(\pi) = \zeta_3 \pi$. In the notation of Lemma~\ref{L:hats},
we have
\[
 \hat{u} = \frac{\sigma(\pi^\alpha)}{\pi^\alpha} = \zeta_3^\alpha, \qquad \hat{r}=\hat{s}=\hat{t}=0
\]
and since $\zeta_3 \equiv \omega_3 \pmod{\pi}$ from the 
the definition of $\gamma_E$ follows directly that $\gamma_E(\sigma) \in \Aut(\Ebar)$ is the
automorphism $(X,Y) \mapsto (\omega_3^{2\alpha}X,\omega_3^{3\alpha}Y) = (\omega_3^{2\alpha}X, Y)$.
\end{proof}

The following is similar to Lemma~\ref{L:gammaEe3}, where the
condition $\ell \equiv 2 \pmod{3}$ is replaced by 
working over $\Q_\ell^{un}$.

\begin{lemma}
Let $\ell \neq 3$ be a prime. 
Let $E/\Q_\ell$ be an elliptic curve with potentially good reduction
and $e=3$. 
Suppose that $E$ has a $3$-torsion point in $\Q_\ell$. 

Then there is a model of $E/L$ with good reduction and residual curve $\Ebar : Y^2 + Y = X^3$.

Moreover, $\gamma_E(\sigma) : (X,Y) \mapsto (\omega_3^{2\alpha}X,Y)$ where $\alpha = \vv_\ell(\Delta_m)/4$.
\label{L:gammaEe3II} 
\end{lemma}
\begin{proof} 
From Lemma~\ref{L:3torsionmodel}, there is a model of $E/\Q_\ell$ of the form
\[
 y^2 + axy + by = x^3, \quad a, b \in \Z_\ell, \quad \Delta_m  = b^3(a^3 - 27b)
\]
where $b = \ell^\alpha u_0$ with $u_0 \in \Z_\ell^*$, 
$\alpha=1,2$ and $\upsilon_\ell(b) < 3\upsilon_\ell(a)$. 
Since $\vv_\ell(u_0)=0$, there exist~$u_1 \in \Q_\ell^{un}$ 
such that $u_1^3 = u_0$ (for Lemma~\ref{L:gammaEe3}
we actually showed that $u_1 \in \Z_\ell^*$ since $\ell \equiv 2 \pmod{3}$).
Now, arguing as in the proof of Lemma~\ref{L:gammaEe3} we also obtain the models
$W$ and $W'$ but defined over $\Q_\ell^{un}$ and $L$, respectively.
Moreover, the exact same calculation shows that $\gamma_E(\sigma)$ is as claimed.
\end{proof}

\section{The morphism $\gamma_E$ in the wild case $e=3$}

Let $E/\Q_3$ be an elliptic curve having potentially good
reduction with $e=3$ and non-abelian $p$-torsion field extension.
From Theorem~\ref{T:goodOverF} part (5), 
$E$ obtains good reduction over the field
\[ F=\Q_3(t) \quad \text{ where } \quad  t^3 + 3t^2 + 3 = 0.\]
Write $\vv = \vv_3$ for the valuation in $\Q_3$ and $\vv_F$
for that in $F$ satisfying $\vv_F(t) = 1$.

Since $F$ has residue field $\F_3$, every element of $\calO_F$ 
admits an unique $t$-adic expansion of the form
$\sum_{k=0}^{\infty} c_k t^k$ with $c_k \in \{ 0,1,2 \}$ and,
in particular,
\[
3 = \mu t^3, \qquad \text{ where } \qquad \mu = 2 + t^2 + t^3 + 2t^4 + st^5, \quad s \in \calO_F. 
\]
From Proposition~\ref{P:usefulmodel} we know that
$E$ admits a minimal model of the form
\begin{equation}
\label{E:modelwild3}
 y^2 = x^3 + a x + b,\quad a = -\frac{c_4}{48}, \quad b = -\frac{c_6}{864}
\end{equation}
whose invariants satisfy $(\vv(c_4), \vv(c_6), \vv(\Delta_m)) = (2, 3,4)$ or $(5, 8, 12)$.

We consider again the quantities $\tilde{c}_4$, $\tilde{c}_6$ and $\tilde{\Delta}$ defined by
\[
 c_4 = 3^{\vv(c_4)} \tilde{c}_4, \qquad c_6 = 3^{\vv(c_6)} \tilde{c}_6, \qquad \Delta_m = 3^{\vv(\Delta)}\tilde{\Delta}
\]
and consider the $3$-adic expansion
\begin{equation} \label{E:c4c6}
  \tilde{c}_6 = a_0 + a_1 3 + a_2 3^2 + O(3^3), \quad 
  \text{ where } \quad a_i \in \{0,1,2\}, \quad a_0 \ne 0.
\end{equation}
To compute $\gamma_E$ we first
need to describe change of coordinates leading 
to a model of $E/F$ with good reduction.
To do this we divide into cases according to $(\vv(c_4), \vv(c_6), \vv(\Delta_m))$.

\noindent {\sc Case I:} Suppose $(\vv(c_4), \vv(c_6), \vv(\Delta_m)) = (5,8,12)$. 

Define the change of coordinates
\begin{equation}
 x=u^2 x' + r, \quad y=u^3 y' \quad \text{ where } \quad u = t^3, \qquad r = (a_0 + y_1 t)t^5
 \label{E:rII}
\end{equation}
and $y_1 = 2, 0, 1$ respectively if $a_1 = 0, 1, 2$.

\noindent {\sc Case II:} Suppose $(\vv(c_4), \vv(c_6), \vv(\Delta_m)) = (2, 3, 4)$. 

We have $c_4^3 - c_6^2 = 12^3 \Delta_m$, hence $\tilde{c}_4^3 - \tilde{c}_6^2 = 2^6 3 \Delta_m$
and Proposition~\ref{P:nonabelianWild3} implies
$\tilde{c}_6 \equiv 2 , 7 \pmod{9}$ and $\tilde{c}_4 \equiv 1 , 4, 7 \pmod{9}$.

Write $\beta_1 = 0, 2, 1$ respectively if $\tilde{c}_4 \equiv 1,4,7 \pmod{9}$.

Now define the change of coordinates
\begin{equation}
 x=u^2 x' + r, \quad y=u^3 y' \quad \text{ where } \quad u = t, \qquad r = y_0 + y_0 t + y_2 t^2,
 \label{E:r}
\end{equation}
where $y_0$ and $y_2$ are given according to the following cases:

\begin{small}
\begin{table}[htb]
$$
\begin{array}{|c|c|} \hline
\tilde{c}_6 \equiv 7 \pmod{9}                  & \tilde{c}_6 \equiv 2 \pmod{9} \\ \hline
y_0 = 2                                        &  y_0 = 1 \\    
y_2 = 2 \text{ if } \beta_1 - a_2 = 0          &  y_2=1 \text{ if }  2\beta_1 - a_2 \in \{4,1,-2 \}  \\
y_2 = 1 \text{ if } \beta_1 - a_2 \in \{-2,1\} &  y_2=2 \text{ if } 2\beta_1 - a_2 \in \{ 3,0 \}  \\
y_2 = 0 \text{ if } \beta_1 - a_2 \in \{-1,2\} &  y_2=0 \text{ if } 2\beta_1 - a_2 \in \{-1,2 \} \\
\hline
\end{array}
$$
\end{table}
\end{small}

\begin{lemma} 
Let $p \geq 5$ be a prime. 
Let $E/\Q_3$ be an elliptic curve with potentially good reduction with $e=3$
and non-abelian $p$-torsion field extension $K/\Q_3$. 

If $(\vv(c_4), \vv(c_6), \vv(\Delta_m)) = (2, 3,4)$ or $(5,8,12)$ then
the change of coordinates \eqref{E:r} or~\eqref{E:rII}, respectively,
transforms the model \eqref{E:modelwild3} into a minimal model of $E/F$
with good reduction and residual elliptic curve 
$\Ebar \; : \; Y^2 = X^3 + X$.
\label{L:rescurvewilde3}
\end{lemma}
\begin{proof} 
Suppose that $(\vv(c_4), \vv(c_6), \vv(\Delta_m)) = (5,8,12)$.

We have $c_4^3 - c_6^2 = 12^3 \Delta_m$, hence
$\tilde{c}_4^3 - 3 \tilde{c}_6^2 = 2^6 \tilde{\Delta}$ which
implies $\tilde{c}_4 \equiv \tilde{\Delta} \equiv 2 \pmod{3}$, where 
the last congruence is due to Proposition~\ref{P:nonabelianWild3}. 
Over $F$ we have
\[
 a = - \frac{\mu^4 t^{12} \tilde{c}_4}{16} \qquad \text{and} \qquad b = - \frac{\mu^5 t^{15} \tilde{c}_6}{32}.
\]
The given change of coordinates leads to a model
\[
 W \; : \; y^2 = x^3 + a_2' x^2 + a_4' x + a_6', \qquad \vv_F(\Delta_W) = e\vv(\Delta_m) - 12\vv_F(t^3) = 0
\]
where 
\[
 t^6 a_2' = 3r, \quad t^{12} a_4' = a + 3r^2, \quad t^{18} a_6' = b + r a + r^3.
\]
Since $\vv_F(r) = 5$ and $\vv_F(a) = 12$ we have $\vv_F(a_2') \geq 2$ and $\vv_F(a_4') = 0$.
Moreover,
\[
 t^{12} a_4' = a + 3r^2 \quad \Leftrightarrow \quad a_4' = -\frac{\mu^4 \tilde{c}_4}{16} + \mu t (a_0 + y_1 t)^2.
\]
Since $\mu \equiv 2 \pmod{t^2}$ we obtain
$a_4' \equiv 2 \tilde{c}_4 \equiv 2 \tilde{\Delta} \equiv 1 \pmod{t}$.

To finish the proof we have to check that $\vv_F(a_6') \geq 1$. We have
\[
 t^{18} a_6' = b + r a + r^3 \Leftrightarrow \quad 
 32 t^3 a_6' = -\mu^5 \tilde{c}_6 - 2 \mu^4 \tilde{c}_4 t^2 (a_0 + y_1 t) + 32(a_0 + y_1 t)^3
\]
and since $b_0 \equiv \tilde{c}_4 \equiv \tilde{\Delta} \equiv 2 \pmod{3}$ we also have
\[
 2 t^3 a_6' \equiv (1+t^2)(a_0 + 2a_1 t^3) + 2t^2(a_0 + y_1t) + 5(a_0 + y_1 t)^3 \pmod{t^4}.
\]
hence
\[
 2 t^3 a_6' \equiv (a_0 + 5 a_0^3) + (a_0 + 2a_0)t^2 + (2b_1  + 2y_1 + 5 y_1^3) t^3 \pmod{t^4}.
\]
Now, in both cases $a_0=1$ or $a_0=2$ we obtain the congruence
\[
 a_6' \equiv 2(1+2a_1 + y_1) \pmod{t}
\]
and the result follows from the definition of $y_1$.

The case $(\vv(c_4), \vv(c_6), \vv(\Delta_m)) = (2, 3,4)$ follows by similar
calculations but with more cases (as expected by the amount of cases required 
to define the values of $y_0$ and $y_2$). 

Alternatively, as the conclusions 
of the lemma will follow from having sufficient but finite $p$-adic precision on the 
coefficients of the Weierstrass models involved, we could use instead a 
computer program to run over all the cases, testing for the desired conclusions.
\end{proof}

We fix $\omega_4 \in \Fbar_3$ to be a 4th root of unity.

The discriminant of the polynomial defining~$F$ is $D = -3^4 \cdot 7$ and $7$ is a square in~$\Q_3$.
Thus, $\Q_3(\sqrt{D}) = \Q_3(\sqrt{-1})$ and 
$\overline{F} = \Q_3(s,t)$ where $s^2 = -1$.
The residue field of $\overline{F}$ is $\F_9$ and so
we can see $\omega_4$ in $\F_9^* \subset \F_9$.
We write also $L =\Q_3^{un}\cdot F$.

\begin{lemma} Let $p \geq 5$ be a prime. 
Let $E/\Q_3$ have potentially good reduction with $e=3$
and non-abelian $p$-torsion field extension $K/\Q_3$, 
so that $E$ has good reduction over~$L$ and by Lemma~$\ref{L:rescurvewilde3}$ there is a good reduction model $E/F$ reducing 
to $\Ebar \; : \; Y^2 = X^3 + X$.

Let $a_0 \equiv \tilde{c}_6(E) \pmod{3}$ be defined by $\eqref{E:c4c6}$.
Then there is a generator $\sigma \in \Gal(L/\Q_3^{un})$
such that $\gamma_E(\sigma) \in \Aut(\Ebar)$ 
is the order $3$ automorphism given by
\[
 (X,Y) \mapsto (X + 2 a_0 \omega_4,Y).
\]
Moreover, $\sigma$ is independent of $E$ (as long as $e=3$ and $K/\Q_3$ is non-abelian).
\label{L:gammEwild3}
\end{lemma}
\begin{proof} Let $\sigma$ be a generator of the order 3 cyclic inertia subgroup of
$\Gal(\overline{F}/\Q_3)$. By explicit computations in the field $\overline{F}$ 
one checks that 
\[
\vv_{\overline{F}}(\sigma(t)-t) = 2, \qquad \vv_{\overline{F}}(\sigma(t)+ t) = 1, \qquad
\vv_{\overline{F}}(\sigma(t)^5-t^5) = 6, \qquad \vv_{\overline{F}}(\sigma(t)^6 - t^6) = 9
\]
and also
\[
\frac{\sigma(t)}{t} \equiv \frac{\sigma(t^3)}{t^3} \equiv 1 \pmod{t}.
\]
Furthermore, by replacing $\sigma$ with $\sigma^2$ if necessary, we can also assume that
\[
\frac{\sigma(t)-t}{t^2} \equiv \omega_4 \pmod{t} \quad \text{ and } \quad
\frac{\sigma(t)^5-t^5}{t^6} \equiv 2\omega_4 \pmod{t}.
\]
Abusing notation we let $\sigma$ be the generator of $\Phi = \Gal(L/\Q_3^{un})$ 
that lifts the previously fixed~$\sigma$. We shall shortly show that $\sigma$
satisfies the desired properties for any $E$ satisfying the hypothesis. 
In particular, it is independent of the elliptic curve $E$ as long as
$e=3$ and $K/\Q_3$ is non-abelian, proving the last statement.

Let $E/\Q_3$ be as in the statement, so that  
$E$ obtains good reduction over $F$ by Theorem~\ref{T:goodOverF}.

Suppose that $(\vv(c_4), \vv(c_6), \vv(\Delta)) = (2,3,4)$.
From Definition~\ref{D:gammaE} and Lemma~\ref{L:hats} we will compute $\gamma_E(\sigma)$
by using the coordinate change and model $E/F$ in Lemma~\ref{L:rescurvewilde3} (which reduces 
to $\Ebar$). Indeed, we have to compute 
\[
 \hat{u} = \frac{\sigma(t)}{t} \pmod{t} \qquad \text{and} \qquad \hat{r} = \frac{\sigma(r) - r}{t^2} \pmod{t}.
\]
Observe that
\[
\hat{r} = \frac{\sigma(r) - r}{t^2} = y_0\frac{(\sigma(t)-t)}{t^2} + y_2\frac{\sigma(t)-t}{t^2}(\sigma(t)+t)
\]
and that $y_0 \equiv 2 a_0 \pmod{3}$. From the valuations and congruences above we obtain
\[
 \hat{u} \equiv 1 \pmod{t} \quad \text{ and } \quad \hat{r} \equiv y_0 \omega_4 \equiv 2a_0 \omega_4 \pmod{t}. 
\]

Suppose now that $(\vv(c_4), \vv(c_6), \vv(\Delta)) = (5,8,12)$.
Similarly, we compute 
\[
 \hat{u} = \frac{\sigma(t^3)}{t^3} \pmod{t} \qquad \text{and} \qquad \hat{r} = \frac{\sigma(r) - r}{t^6} \pmod{t}.
\]
Observe that
\[
\hat{r} = a_0\frac{\sigma(t)^5-t^5}{t^6} + y_1\frac{\sigma(t)^6-t^6}{t^6}
\]
and from the valuations and congruences above we obtain again
\[
 \hat{u} \equiv 1 \pmod{t} \quad \text{ and } \quad \hat{r} \equiv 2a_0 \omega_4 \pmod{t}.
\]

\end{proof}

\section{The morphism $\gamma_E$ in the tame case $e=4$}
\label{S:gammaEtame4}

Let $\ell \geq 3$ be a prime. Fix a primitive $4$-th root 
of unity $\zeta_4 \in \Q_\ell^{un} \subset \Qbar_\ell$
and $\omega_4 \in \Fbar_\ell$ such that $\zeta_4$ 
reduces to $\omega_4$. 
Let $F = \Q_\ell(\pi)$, where 
$\pi \in \Qbar_\ell$ satisfies $\pi^4 = \ell$. 
The field $F$ is totally ramified of degree 4 and 
it is Galois if and only if $\ell \equiv 1 \pmod{4}$. 

Let $E/\Q_\ell$ an elliptic curve with potentially good reduction 
with semistability defect $e=4$ and $p$-torsion field $K$; 
from Theorem~\ref{T:goodOverF} we know that $E/F$ has good reduction.  
Recall from Section~\ref{S:left?} that $L = \Q_\ell^{un} K = \Q_\ell^{un}F$  
is the minimal extension of $\Q_\ell^{un}$ where $E$ gets good reduction; 
it is also the unique degree 4 tame extension of $\Q_\ell^{un}$ 
hence, in particular, it does not depend on $E$. 

Write $\vv = \vv_\ell$. Recall also the quantity $\tilde{\Delta}$ defined by
$\Delta_m = \ell^{\vv(\Delta_m)} \tilde{\Delta}$, where 
$\Delta_m = \Delta_m(E)$ denotes the discriminant of a minimal model of $E$.

\begin{lemma} 
Let $E/\Q_\ell$ be as above.
\begin{itemize}
 \item[(i)] There is a model of $E/\Q_\ell$ of the form
 \[
  E \; : \; y^2 = x^3 + ax^2 + bx, \qquad \Delta = \Delta_m = 2^4 b^2 (a^2 - 4b)
 \]
 and, moreover, we have either
 \[
  \vv(a) \geq 1, \; \vv(b) = 1, \; \vv(a^2 - 4b) = 1, \quad \text{ or } \quad \vv(a) \geq 2, \;
  \vv(b) = 3, \; \vv(a^2 - 4b) = 3.
 \]

 \item[(ii)] There is a model of $E/F$ with good reduction of the form
\[
E' \; : \; y^2 = x^3 + a_2' x^2 + a_4' x, \quad \text{ where } \quad a_2' = \frac{a}{\pi^{2\alpha}}, \quad 
 a_4' = \frac{b}{\pi^{4\alpha}}, \quad \alpha = \frac{\upsilon(\Delta_m)}{3};
\]
moreover, $\upsilon_F(a_2')> 0$, $\upsilon_F(a_4') = 0$ and $E'$ has a residual curve of the form
\[
\Ebar \; : \; Y^2 = X^3 + \bar{u} X, \quad \bar{u} \in \F_\ell^*.
\]
\end{itemize}
\label{L:Ebar4}
\end{lemma}
\begin{proof}
From \cite[p. 355]{Kraus1990} we know that $E$ has Kodaira type III or III* (because $e=4$). 
Furthermore, from \cite[p. 406, Ex. 4.48]{SilvermanII} there is a minimal model of $E/ \Q_\ell$ of the form
\[
 y^2 = x^3 + a_2 x^2 + a_4 x + a_6. 
\]
which for Type III also satisfies
\[
\upsilon(\Delta_m) = 3, \quad \upsilon(a_2) \ge 1, \quad \upsilon(a_4) = 1, \quad \upsilon(a_6) \ge 2,  
\]
and for Type III* satisfies
\[
\upsilon(\Delta_m) = 9, \quad \upsilon(a_2) \ge 2,\quad  \upsilon(a_4) = 3,\quad \upsilon(a_6) \ge 5.
\]
From the Newton polygon for $f=x^3 + a_2 x^2 + a_4 x + a_6$ or $f=x^2 + a_2 x + a_4$ (when $a_6 = 0$) it follows that $f$ has a root $\beta$ such that $\upsilon(\beta) \geq 1$ if $E$ has Type III and 
$\upsilon(\beta) \geq 2$ if $E$ has Type III*. The other roots have non integer 
valuation. In particular, $E$ has exactly one point of $P=(\beta,0)$ of order 2 defined over $\Q_\ell$.
Applying the translation $(x,y) \mapsto (x+\beta,y)$ gives an integer model for $E/\Q_\ell$
of the form
\[
  E \; : \; y^2 = x^3 + ax^2 + bx, \qquad \Delta = \Delta_m = 2^4 b^2 (a^2 - 4b).
\]
Moreover, from the usual formulas for coordinate changes it follows that if $E$ has Type III we have
$\upsilon(a) \geq 1$, $\upsilon(b) = 1$ and $\upsilon(a^2 - 4b) = 1$
and if $E$ has Type III* we have
$\upsilon(a) \geq 2$, $\upsilon(b) = 3$ and $\upsilon(a^2 - 4b) = 3$.
This proves (i).

Set $\alpha = \upsilon(\Delta_m)/3$ and $u = \pi^\alpha$, where $\pi$ is a uniformizer of $F$.
Note that $\upsilon_F(\ell) = \upsilon(\pi^4) = 4$ and $\upsilon_F(u^k) = k \alpha$. 
The change of coordinates $(x,y) \mapsto (u^2 x,u^3 y)$ applied to $E$
leads to the integral model over $F$
\[
 E' \; : \; y^2 = x^3 + a_2' x^2 + a_4' x, \quad \text{ where } \quad a_2' = \frac{a}{u^2}, \quad 
 a_4' = \frac{b}{u^4}, \quad \Delta' = \Delta_m u^{-12}
\]
which satisfies
\[
\upsilon_F(\Delta') = \upsilon_F(\Delta_m) - 12 \upsilon_F(u) = 4 \upsilon(\Delta_m) - 12 \alpha = 0. 
\]
Observe that $\upsilon_F(au^{-2}) \geq 1$ and $\upsilon_F(bu^{-4}) = 0$. 
Thus the residual elliptic curve of $E/F$ has the 
form $\Ebar \; : \; y^2 = x^3 + \overline{u} x$ with $\overline{u} \in \F_\ell^*$, proving (ii).
\end{proof}

\begin{lemma} 
Let $E/\Q_\ell$ be as above. 
Then $E$ has full $2$-torsion over $F$ if and only if $\tilde{\Delta}$
is a square mod~$\ell$.
\label{L:full2torsion}
\end{lemma}
\begin{proof}
From Lemma~\ref{L:Ebar4} we know that $E$ has a model of the form $y^2 = x^3 + ax^2 + bx$ over $\Q_\ell$. By base change we get a model of the same form over $F$.
The 2-torsion points of such model are $(0,0)$ and $(x_0,0)$ where $x_0$ is a root of the polynomial 
$x^2 + ax + b$. The values of $x_0$ are defined in $F$ if and only if $a^2 - 4b$ is a square in $F$.
Since $\Delta_m = 2^4 b^2 (a^2 - 4b)$ it follows that $E$ has full 2-torsion over $F$ if and only if $\Delta_m$ is a square in $F$. 

To finish the proof we shall show that $\Delta_m$ is a square in $F$ if and only if $\tilde{\Delta}$ is a square mod~$\ell$.
Indeed, we have that $\Delta_m = \ell^{\vv(\Delta_m)}\tilde{\Delta}$ with $\tilde{\Delta} \in \Z_\ell^*$, hence over $F$ we have $\Delta_m = \pi^{4\vv(\Delta_m)}\tilde{\Delta}$.
Then $\Delta_m$ is a square in $F$ if and only if $\tilde{\Delta}$ is a square in $F$.
Finally, note that $\tilde{\Delta}$ is a square in $F$ if and only if its reduction
is a square in $\F_\ell$ (by  Hensel's lemma using the fact that
$\ell \neq 2$ and the residue field $\calO_F/(\pi)$ is $\F_\ell$).
\end{proof}

We remark that the previous lemmas do not require $K/\Q_\ell$ to be non-abelian.

Suppose now $\ell \equiv -1 \pmod{4}$, so that $F$ is not Galois and $K/\Q_\ell$ is non-abelian 
by Corollary~\ref{C:nonabelianTame}. 
Let $\sigma, \tau \in \Gal(K/\Q_\ell)$ be given by Proposition~\ref{P:groupStructure}.
Identify $\sigma$ with the corresponding element in $\Gal(L / \Q_\ell^{un})$;
we can further assume that $\sigma$ acts on $L$ 
by $\sigma(\pi) = \zeta_4 \pi$ (note that $\pi$ is not always in $K$). 
We can now determine $\gamma_E(\sigma)$ in the case of our interest.

\begin{lemma} Let $\ell \equiv 3 \pmod{4}$ be a prime. 
Let $E/\Q_\ell$ be an elliptic curve with potentially good reduction
with $e=4$. Write $\Ebar$ for the curve $Y^2 = X^3 - X$. 
\begin{itemize}
\item[(A)] Suppose that $E/F$ has a full $2$-torsion.
Then there is a model of $E/F$ with good reduction reducing to $\Ebar$. 
\item[(B)] Suppose that $E/F$ does not have full $2$-torsion. 
Then $E/\Q_\ell$ admits a $2$-isogeny defined over $\Q_\ell$ to an elliptic curve $W/\Q_\ell$
with the following properties:
\begin{itemize}
 \item $W/\Q_\ell$ has tame additive reduction with $e(W)=4$,
 \item $W/F$ has full $2$-torsion,
 \item $W/F$ admits a model with good reduction whose reduction is $\overline{W} = \Ebar$,
 \item $\upsilon_\ell(\Delta_m(E)) =  \upsilon_\ell(\Delta_m(W))$.
\end{itemize}
\end{itemize}

Moreover, $\alpha = \upsilon(\Delta_m)/3$ is odd. If we 
let $T_\alpha \in \Aut(\Ebar)$ be given by $(X,Y) \mapsto (-X,\omega_4^{3\alpha}Y)$,

then in case (A) we have $\gamma_E(\sigma) = T_\alpha$ and in case (B) we have 
$\gamma_W(\sigma) = T_\alpha$.
\label{L:modelEtame4}
\end{lemma}
\begin{proof} From Lemma~\ref{L:Ebar4} (i) we can write a model over $\Z_\ell$ of the form
\[
  E \; : \; y^2 = x^3 + ax^2 + bx, \qquad \Delta = \Delta_m = 2^4 b^2 (a^2 - 4b).
\]
It is well known (e.g. \cite[Theorem 2.2]{DD2015}) that $E$ admits the $2$-isogenous curve 
\[
  W \; : \; y^2 = x^3 - 2a x^2 + (a^2 - 4b)x, \qquad \Delta_W = 2^8 b (a^2 - 4b)^2.
\] 
and from the valuations in Lemma~\ref{L:Ebar4} (i) we see that $\vv(\Delta_m(E)) = \vv(\Delta_m(W))$, 
proving the last claim of (B).  

From Lemma~\ref{L:Ebar4} (ii) there are models over $F$ with good reduction 
\[
 E' \; : \; y'^2 = x'^3 + \frac{a}{\pi^{2\alpha}} x'^2 + \frac{b}{\pi^{4\alpha}} x', 
 \quad \upsilon_F(\frac{a}{\pi^{2\alpha}}) > 0, \quad \upsilon_F(\frac{b}{\pi^{4\alpha}}) = 0
\]
and
\[
 W' \; : \; y'^2 = x'^3 + \frac{-2a}{\pi^{2\alpha}} x'^2 + \frac{a^2 - 4b}{\pi^{4\alpha}} x', 
 \quad \upsilon_F(\frac{-2a}{\pi^{2\alpha}})> 0, \quad \upsilon_F(\frac{a^2 - 4b}{\pi^{4\alpha}}) = 0.
\]

We shall shortly see that $\calO_F^*/(\calO_F^*)^4 = \{ \pm 1 \}$ from which
we conclude 
\begin{equation}
 a_4'(E')= b/\pi^{4\alpha} = \pm u_0^4, \qquad u_0 \in \calO_F^*.
\label{E:b}
\end{equation}
Therefore, the change of variables
$x' = u_0^{2} x''$, $y' = u_0^{3} y''$ transforms the model $E'/F$ into another model $E''/F$
with $a_4'' = \pm 1$ whose residual curve is $Y^2 = X^3 \pm X$. The analogous 
statement holds for $W$.

We now prove \eqref{E:b}. 
Indeed, let $\mu \in \calO_F^*$. Since $4 \nmid \ell-1$ we have
$\F_\ell^*/(\F_\ell^*)^4 = \{ \pm 1 \}$, hence
$\mu$ or $-\mu$ is a 4th-power in $\calO_F/(\pi) = \F_\ell$. Since
$\ell \ne 2$, from Hensel's lemma, there is $\mu_0 \in F$ such
that $\mu_0^4 = \mu$ or $\mu_0^4 = -\mu$, that is
$\calO_F^*/(\calO_F^*)^4 = \{ \pm 1 \}$ and \eqref{E:b} follows.

Observe that
\[
 \frac{a^2 - 4b}{b} = \frac{a^2}{b} - 4 \equiv -4 \pmod{\pi}.
\]
Since $\ell \equiv -1 \pmod{4}$ implies -1 is not a square mod~$\ell$
and $\calO_F/(\pi) = \F_\ell$ we conclude that exactly one of $b$ or 
$a^2 - 4b$ is a square in $F$.

Suppose now that $E$ has full 2-torsion over $F$, hence its discriminant is a square.
From the formula for $\Delta_m$ it follows that $a^2 - 4b$ is a square in $F$, hence $b$ is not a square in $F$. 
Thus $a_4'(E')= b/\pi^{4\alpha} = - u_0^4$ and the residual curve of $E''$ is $\Ebar : Y^2 = X^3 - X$.
This proves (A).

Suppose next that $E$ does not have full 2-torsion. It follows from the proof of Lemma~\ref{L:full2torsion}
that $a^2 - 4b$ is not a square in $F$, hence $b$ is a square in $F$.
Thus (from the same proof) 
$W$ has full 2-torsion over $F$, $a_4'(W')= - u_0^4$ and the residual curve of 
$W''$ is $\Ebar : Y^2 = X^3 - X$. This proves (B).

We will now prove the last statement. We need to compute $\gamma_E(\sigma) \in \Aut(\Ebar)$.

The arguments proving (A) and (B) 
show that we go from the model for $E/\Q_\ell$ into the
minimal model $E''/F$ reducing to $\Ebar : Y^2 = X^3 - X$ by a transformation
$x = u^2 x''$, $y = u^{3} y''$ where $u = u_0 \pi^\alpha$ where $u_0 \in \calO_F$.
In particular, the same is true over $L = \Q_\ell^{un} K = \Q_\ell^{un} F$.

We claim that $u_0 \in \Z_\ell^* \subset \Q_\ell^{un}$ and compute
\[
 \hat{u} = \frac{\sigma(u)}{u} = \frac{u_0 \zeta_4^\alpha \pi^\alpha}{u_0 \pi^\alpha} = \zeta_4^\alpha 
 \quad \text{ and } \quad \hat{r} = \hat{s} = \hat{t} = 0
\]
which means that $\gamma_E(\sigma)$ is given by
$(X,Y) \mapsto (\omega_4^{2\alpha}X,\omega_4^{3\alpha}Y)$.
Since $\alpha = 1$ when the Kodaira type is III and $\alpha = 3$ when it is III* 
(see begining of the proof of Lemma~\ref{L:Ebar4}),
we see that $\alpha$ is odd and hence
we have 
$\omega_4^{2\alpha} = -1$. 
Thus $\gamma_E(\sigma) = T_\alpha$ in case (A). 
The analogous argument works for $W$ in case (B).

We now prove the claim to complete the proof. Indeed, from \eqref{E:b} we have
$b/\ell^{\alpha} = b/\pi^{4\alpha} = \pm u_0^4$, where $u_0 \in \calO_F^*$.
Since $b/\ell^{\alpha}$ is a unit in $\Z_\ell$ and
any unit of $\Q_\ell^{un}$ is a fourth power in $\Q_\ell^{un}$ 
there exist $u_1 \in \Q_\ell^{un}$ such that $\pm u_0^4 = b/\ell^{\alpha} = u_1^4$. 
Thus $u_1/u_0$ is a root of unity of order prime to $\ell$, so $u_0/u_1 \in \Q_\ell^{un}$  and
we conclude $u_0 \in \Q_\ell^{un} \cap \calO_F^* = \Z_\ell^{*}$, as claimed.
\end{proof}

\section{The morphism $\gamma_E$ in the wild case $e=4$}

Write $\vv = \vv_2$ and 
$\Ebar$ for the elliptic curve over $\F_2$ of equation
\[
 \Ebar \; : \;  Y^2 + Y = X^3. 
\]
Let $E/\Q_2$ be an elliptic curve with potentially good reduction with $e=4$ 
and $p$-torsion field~$K$.
From Proposition~\ref{P:usefulmodel} we know that $E$ admits a minimal model of the form
\begin{equation}
\label{E:modelwild4}
 y^2 = x^3 + a_4 x + a_6,\quad a_4 = -\frac{c_4}{48}, \quad a_6 = -\frac{c_6}{864}
\end{equation}
with invariants satisfying $(\vv(c_4), \vv(c_6), \vv(\Delta_m)) = (5,8,9)$ or $(7, 11, 15)$.
Consider again the quantities $\tilde{c}_4$, $\tilde{c}_6$ and $\tilde{\Delta}$ given by
\[
 c_4 = 2^{\vv(c_4)} \tilde{c}_4, \qquad c_6 = 2^{\vv(c_6)} \tilde{c}_6, \qquad \Delta_m = 2^{\vv(\Delta)}\tilde{\Delta}
\]
and suppose $\tilde{c}_4 \equiv 5 \tilde{\Delta} \pmod{8}$, so that $K/\Q_2$ is non-abelian 
by Proposition~\ref{P:nonabelianWild4}. Then
$E$ has good reduction over $F_1$ or $F_2$, 
the fields respectively defined by the polynomials $f_1$ or $f_2$ in 
part~(2) of Theorem~\ref{T:goodOverF}. Moreover, either $E$ or 
its quadratic twist by $-1$, denoted $E_{-1}$, has good reduction over $F_1$.
Recall from Theorem~\ref{T:goodOverF} that
\[
F_1  =  \Q_2(t) \quad \text{ where } \quad t^4+12t^2+6 = 0
\]
has residue field $\F_2$ and satisfies $\vv_{F_1}(t) = 1$. 
Therefore, every element of $\calO_{F_1}$ admits an unique $t$-adic 
expansion of the form
$\sum_{k=0}^{\infty} a_k t^k$ with $a_k \in \{ 0,1 \}$;
in particular, we have
\[
2 = \mu t^4 \quad \text{ where } \qquad \mu = 1 + t^6 + t^8 + t^{10} + t^{14} + t^{16} + t^{18} + st^{20}, \quad s \in \calO_{F_1}. 
\]
We have $\tilde{c}_4 \equiv 1,5 \pmod{8}$ (see Proposition~\ref{P:nonabelianWild4}), 
hence $\tilde{c}_4$ has a 2-adic expansion of the form
\begin{equation}
 \tilde{c}_4 = 1 + \beta_2 2^2 + \beta_3 2^3 + \beta_4 2^4 + s' 2^5, 
 \quad \text{ with } \quad \beta_i \in \{0,1\},  \quad s' \in \Z_2
\label{E:c42adic}
\end{equation}
and a straightforward calculation using \eqref{E:c42adic} and the 
expansions for $\mu$ shows that
\begin{equation}
  \tilde{c}_4 = 1 + \beta_2 t^8 + \beta_3 t^{12} + \beta_4 t^{16}
  + (\beta_2 + \beta_3) t^{18} + s t^{20}, \qquad s \in \calO_{F_1}.
  \label{E:c4pWild4}
\end{equation}
Similarly, $\tilde{c}_6$ has a $2$-adic expansion of the form
\begin{equation}
 \tilde{c}_6 = 1 + \alpha_1 2 + \alpha_2 2^2 + \alpha_3 2^3 + \alpha_4 2^4 + s' 2^5, 
 \quad \text{ with } \quad \alpha_i \in \{0,1\},  \quad s' \in \Z_2
\label{E:c62adic}
\end{equation}
and we obtain
\begin{equation}
\tilde{c}_6 = 1 + \alpha_1 t^4 + \alpha_2 t^8 + \alpha_1 t^{10} + (\alpha_1 + \alpha_3) t^{12} 
+ s t^{14}, 
\quad s \in \calO_{F_1}.
\label{E:c6pWild4}
\end{equation}

\begin{lemma} Let $E/\Q_2$ be as above, so that $E$ or $E_{-1}$ has good reduction
over $F_1$. 

Define~$\beta = \alpha_1 + \beta_3 \pmod{2}$, where
$\beta_3$ and $\alpha_1$ be defined by \eqref{E:c42adic} and \eqref{E:c62adic}, respectively. 
Furthermore, let $(u,r,s,T)$ be given according to the cases:
\begin{enumerate}
 \item[(i)] if $(\vv(c_4), \vv(c_6), \vv(\Delta_m)) = (5,8,9)$ define
\begin{equation*}
u = t^3, \quad r = t^2 + (1-\beta_3)t^6 + \beta t^7, \quad s = 
t + \beta_3 t^3, \quad T = t^5 + \beta t^7 +  \beta t^8;
\end{equation*}
 \item[(ii)] if $(\vv(c_4), \vv(c_6), \vv(\Delta_m)) = (7,11,15)$ define 
 \begin{equation*}
 u = t^5, \quad r = t^6 + \beta_3 t^{10}, \quad s = t^3 + (1-\beta_3)t^5,\quad T = t^{11} + \beta t^{13}.
\end{equation*}
\end{enumerate}
Then the change of coordinates 
\[
 x=u^2 x' + r, \qquad y=u^3 y' + u^2 s x' + T,
\]
transforms a model of the form \eqref{E:modelwild4} of whichever among $E$ or $E_{-1}$ has good reduction 
over~$F_1$, into a minimal model $W/ F_1$ with good reduction 
and residual curve $\overline{W} = \Ebar$.
\label{L:rescurvewilde2}
\end{lemma}
\begin{proof} Write $F = F_1$. 
Suppose $(\vv(c_4), \vv(c_6), \vv(\Delta_m)) = (7,11,15)$.
Write 
\[
 W \; : \; y^2 + a_1' xy + a_3'y = x^3 + a_2' x^2 + a_4' x + a_6',
\]
for the model of $E/F$ obtained after applying the 
change of coordinates in part (ii) of the statement to 
the model \eqref{E:modelwild4}. We observe that
\[
 \vv_F(\Delta_W) = e\vv(\Delta_m) - 12\vv_F(u) = 4 \cdot 15 - 12\vv_F(t^5) = 
 60 - 60 = 0,
\]
hence $W$ is minimal with good reduction and residual curve $\Ebar$ 
if we show that 
\[
 \vv_F(a_i') \geq 1, \quad \text{ for } i = 1,2,4,6 \quad \text{ and } \quad \vv_F(a_3') = 0.
\]
From \eqref{E:coordchange}, we know that
\[
 ua_1' = 2s, \quad u^2 a_2' = 3r - s^2, \quad u^3 a_3' = 2T, \quad u^4 a_4' = a_4 + 3r^2 - 2sT, 
 \quad u^6 a_6' = a_6 + r a_4 + r^3 - T^2.
\]
We have $\vv_F(2) = 4$ and from the formulas for $u$, $r$, $s$ and $T$ we see that 
\[
 \vv_F(u) = 5, \quad \vv_F(r) = 6, \quad \vv_F(s) = 3, \quad \vv_F(T) = 11. 
\]
Hence $\vv_F(a_1') = 2$ and $\vv_F(a_3') = 0$; moreover,
\begin{eqnarray*}
 3r - s^2 & = & 3(t^6 + \beta_3 t^{10}) - t^6 - 2(1-\beta_3) t^8 - 
 (1-\beta_3) t^{10}  \\
          & = & (\mu - 1)t^{10} - \mu(1-\beta_3)t^{12} + \mu^2 \beta_3 t^{18}
\end{eqnarray*}
and since $\vv_F(\mu-1) = 6$ it follows $\vv_F(3r - s^2) \geq 12$, thus $\vv_F(a_2') \geq 2$.
We note that the previous calculations also hold if we replace $E$ by $E_{-1}$.
We now want to compute 
\[
 a_4 + 3r^2 - 2sT \pmod{t^{21}} \quad \text{ and } \quad a_6 + r a_4 + r^3 - T^2 \pmod{t^{31}}.
\]
Working modulo $t^{21}$, we have
\[
 a_4 = \frac{-2^3}{3} \tilde{c}_4 \equiv \frac{-\mu^3}{3} (t^{12} + \beta_2 t^{20}) 
 \equiv t^{12} + t^{18} + \beta_2 t^{20} \pmod{t^{21}} 
\] 
and also 
\[
 r^2 \equiv t^{12} \pmod{t^{21}}, \qquad 
 2sT \equiv t^{18} + (1 + \beta - \beta_3) t^{20} \pmod{t^{21}},
\]
therefore,
\[
 u^4 a_4' = a_4 + 3r^2 - 2sT \equiv 4t^{12} + (1+ \beta + \beta_2 - \beta_3)t^{20} 
 \equiv (\beta + \beta_2 - \beta_3)t^{20} \pmod{t^{21}}.
\]
With a similar but lengthy calculation, we can show that 
\begin{eqnarray*}
 u^6 a_6' =  a_6 + r a_4 + r^3 - T^2 &\equiv & (\beta_3 + \beta_2 - \beta)t^{26} + (\alpha_1 + \beta_3 - \beta)t^{28} 
 + \beta_2 \beta_3 t^{30} \pmod{t^{31}} \\ 
 & \equiv & (\beta_3 + \beta_2 - \beta)t^{26} + \beta_2 \beta_3 t^{30} \pmod{t^{31}}, \ 
\end{eqnarray*} 
where we used the definition of $\beta$ for the last congruence.

Note that replacing $E$ by $E_{-1}$ does not change the values of $\beta_i$ and replaces 
$\alpha_1$ by $1-\alpha_1$; thus it replaces $\beta$ by $1-\beta$ in the previous congruences
for $u^4a_4'$ and $u^6a_6'$.

Suppose now $\beta = 0$ and $u^4a_4' \equiv 0 \pmod{t^{21}}$. Then $(\beta_2,\beta_3) = (1,1)$ 
or $(0,0)$, which implies $u^6a_6' \equiv 0 \pmod{t^{31}}$; also, if 
$\beta = 1$ and $u^4a_4' \equiv 0 \pmod{t^{21}}$ then $u^6a_6' \equiv 0 \pmod{t^{31}}$. Clearly, a similar conclusion holds for the congruences with $1-\beta$ replacing $\beta$. 
We conclude that the change of coordinates will lead $E$ or $E_{-1}$ into a model with the 
claimed properties if
\[
 (\beta + \beta_2 - \beta_3)t^{20} \equiv 0 \pmod{t^{21}} \quad \text{or} 
 \quad (1-\beta + \beta_2 - \beta_3)t^{20} \equiv 0 \pmod{t^{21}}
\]
respectively. If the first congruence holds, then $E/F_1$ has good reduction and we are done. 
Suppose it fails and $\beta = 0$; then $(\beta_2,\beta_3) = (1,0)$ or $(0,1)$, hence the second 
congruence holds and $E_{-1}/F_1$ has good reduction and we are done. 
The same conclusion holds if  the first congruence fails with $\beta = 1$, completing 
the proof in the case $(\vv(c_4), \vv(c_6), \vv(\Delta_m)) = (7,11,15)$.
The case $(\vv(c_4), \vv(c_6), \vv(\Delta_m)) = (5,8,9)$ follows by a similar argument.
\end{proof}

Let $\omega_3 \in \F_{4} \subset \Fbar_2$ be a fixed cube root of unity.

\begin{lemma} Let $E/\Q_2$ be as above and, if necessary, 
twist it by $-1$, so that $E/F_1$ has good reduction.
From
Lemma~$\ref{L:rescurvewilde2}$ there is a model of $E/F_1$ 
with good reduction and residual curve~$\Ebar$.
Let $L = \Q_2^{un} F_1$ and $\alpha_1$ be defined by $\eqref{E:c62adic}$. 

Then there is a generator $\sigma$ of $\Phi = \Gal(L/\Q_2^{un})$
such that $\gamma_E(\sigma) \in \Aut(\Ebar)$  is the order $4$ 
automorphism given by
\[
 (X,Y) \mapsto (X + 1,Y + X + \alpha_1 + \omega_3). 
\]
Moreover, $\sigma$ is independent of $E$ 
as long as $e=4$, its $p$-torsion field extension is non-abelian and $E/F_1$ has good reduction.
\label{L:gammEwild4}
\end{lemma}

\begin{proof} Write $F = F_1$ and let $\overline{F}$ be its Galois closure. 
Let $\sigma$ be a generator of the order 4 cyclic inertia subgroup of
$\Gal(\overline{F}/\Q_2)$. Computations in the field $\overline{F}$ 
show that
\begin{equation}
\vv_{\overline{F}}(\sigma(t)^3-t^3) = 5, \quad \vv_{\overline{F}}(\sigma(t)^6 - t^6) = 10, \quad
\vv_{\overline{F}}(\sigma(t)^7-t^7) = 9, \quad \vv_{\overline{F}}(\sigma(t)^8-t^8) = 22
\label{E:valuationsF1}
\end{equation}
and
\begin{equation}
\vv_{\overline{F}}(\sigma(t)^2-t^2) = 6, \qquad \frac{\sigma(t)}{t} \equiv \frac{\sigma(t) - t}{t^3} \equiv \frac{\sigma(t)^2 - t^2}{t^6} \equiv \frac{\sigma(t)^7 - t^7}{t^9} \equiv 1 \pmod{t}.
\label{E:congruencesF1}
\end{equation}
Furthermore, by replacing $\sigma$ with $\sigma^3$ if necessary, we can also assume that
\begin{equation}
\frac{\sigma(t)^5 - t^5 - t (\sigma(t)^2 - t^2)}{t^9} \equiv \omega_3 \pmod{t}.
\label{E:congruencesF1II}
\end{equation}
Abusing notation we let $\sigma$ be the generator of $\Phi = \Gal(L/\Q_2^{un})$ 
that lifts the previously fixed~$\sigma$. We shall shortly show that $\sigma$
satisfies the desired properties for any $E$ satisfying the hypothesis. 
In particular, it is independent of the elliptic curve $E$ as long as
$e(E)=4$, $K = \Q_2(E[p])/\Q_2$ is non-abelian and $E/F_1$ has good reduction, proving the 
last statement.

Let $E/\Q_2$ be an elliptic curve with potentially good reduction with $e=4$. 
Let $p \geq 3$ and suppose that its $p$-torsion field extension $K/\Q_2$ is non-abelian 
and that $E$ has good reduction over $F$. We now divide into two cases:

{\sc Case 1:} Suppose that $E$ satisfies $(\vv(c_4), \vv(c_6), \vv(\Delta)) = (5,8,9)$.
We will compute $\gamma_E(\sigma)$ by using the coordinate change 
in Lemma~\ref{L:rescurvewilde2} part (i) (which transforms $E/\Q_2$ into
a minimal model with good reduction for $E/F$ reducing to $\Ebar$). 
Indeed, we have to compute the reduction modulo $t$ of
\[
 \hat{u} = \frac{\sigma(u)}{u}, \quad \hat{r} = \frac{\sigma(r) - r}{u^2},
\quad \hat{s} = \frac{\sigma(s)-s}{u}, \quad \text{ and } \quad 
\hat{T} = \frac{\sigma(T) - T - s(\sigma(r) - r)}{u^3}.
\]
We have that
\begin{eqnarray*}
\hat{u} & = & \frac{\sigma(t)^3}{t^3}, \\
\hat{r} & = & \frac{\sigma(t)^2-t^2}{t^6} + (1-\beta_3)\frac{\sigma(t)^6-t^6}{t^6} + \beta \frac{\sigma(t)^7-t^7}{t^6}, \\
\hat{s} & = & \frac{\sigma(t) - t}{t^3} + \beta_3\frac{\sigma(t)^3 - t^3}{t^3}, \\
\hat{T} & = & \beta \frac{\sigma(t)^7 - t^7}{t^9} - \beta_3 \frac{t^3 (\sigma(t)^2 - t^2)}{t^9} 
	      + \frac{\sigma(t)^5 - t^5 - t (\sigma(t)^2 - t^2)}{t^9}
	      \\
	& + &  \frac{\beta (\sigma(t)^8 - t^8) - s(1-\beta_3)(\sigma(t)^6 - t^6) - s\beta(\sigma(t)^7 - t^7)}{t^9}. \
\end{eqnarray*}

We have $\beta = \alpha_1 + \beta_3 \pmod{2}$ by definition. 
Now, from \eqref{E:valuationsF1}, \eqref{E:congruencesF1} and \eqref{E:congruencesF1II} we obtain
\[
 \hat{u} \equiv \hat{r} \equiv \hat{s} \equiv 1 \pmod{t} \quad \text{ and } 
 \quad \hat{T} \equiv \beta - \beta_3 \equiv \alpha_1 + \omega_3 \pmod{t}. 
\]

{\sc Case 2:} Suppose that $E$ satisfies $(\vv(c_4), \vv(c_6), \vv(\Delta)) = (7,11,15)$.
We will compute $\gamma_E(\sigma)$ by using the coordinate change 
in Lemma~\ref{L:rescurvewilde2} part (ii). Indeed, we have that
\begin{eqnarray*}
\hat{u} & = & \frac{\sigma(t)^5}{t^5}, \\
\hat{r} & = & \frac{\sigma(t)^6-t^6}{t^{10}} + \beta_3\frac{\sigma(t)^{10}-t^{10}}{t^{10}}, \\
\hat{s} & = & \frac{\sigma(t)^3 - t^3}{t^5} + (1-\beta_3)\frac{\sigma(t)^5 - t^5}{t^5}, \\
\hat{T} & = & \beta \frac{(\sigma(t)^{13} - t^{13})}{t^{15}} - (1-\beta_3) \frac{t^5 (\sigma(t)^6 - t^6)}{t^{15}} 
	      + \frac{\sigma(t)^{11} - t^{11} - t^3 (\sigma(t)^6 - t^6)}{t^{15}} \\
	& + &  \beta_3 \frac{-s(\sigma(t)^{10} - t^{10})}{t^{15}}. \
\end{eqnarray*}
With further computations in the field $\overline{F}$ we verify that $\sigma$ also satisfies
\begin{equation*}
\vv_{\overline{F}}(\sigma(t)^5-t^5) = 7, \quad \vv_{\overline{F}}(\sigma(t)^{10} - t^{10}) = 14, \quad
\end{equation*}
and
\begin{equation*}
\frac{\sigma(t)^6 - t^6}{t^{10}} \equiv \frac{\sigma(t)^3 - t^3}{t^5} \equiv 
\frac{\sigma(t)^{13} - t^{13}}{t^{15}} \equiv 1, \quad
\frac{\sigma(t)^{11} - t^{11} - t^3 (\sigma(t)^6 - t^6)}{t^{15}} \equiv \omega_3 + 1 \pmod{t}.
\end{equation*}
From this and \eqref{E:valuationsF1}, \eqref{E:congruencesF1} we obtain
\[
 \hat{u} \equiv \hat{r} \equiv \hat{s} \equiv 1 \pmod{t} \quad \text{ and } 
 \quad \hat{T} \equiv \beta - (1 -\beta_3) + \omega_3 + 1 \equiv \alpha_1 + \omega_3 \pmod{t}. 
\]
In both cases, we conclude that $\gamma_E(\sigma) \in \Aut(\Ebar)$ is given by
\[
 (X,Y) \mapsto (\hat{u}^2 X + \hat{r} \pmod{t}, \; \hat{u}^3 Y + \hat{u}^3 \hat{s} X + \hat{T} \pmod{t})
 = (X + 1,Y + X + \alpha_1 + \omega_3),
\]
concluding the proof.
\end{proof}

\section{The morphism $\gamma_E$ in the wild case $e=8$}
\label{S:coordchanges12}

Let $E/\Q_2$ be an elliptic curve with $e=8$ and conductor $2^5$ or $2^8$.
From parts (3) and (4) of Theorem~\ref{T:goodOverF} we know that $E$ obtains good reduction 
over one of the fields~$F_i$ defined by one of the polynomials $g_i$, for $i=1,2,3,4$. 
For $i=1,2,3$, we fix an uniformizer $\pi$ in $F_i$ as follows.
\begin{small}
\begin{table}[htb]
$$
\begin{array}{|c|c|} \hline
\text{ Field }                  & \pi \text{ is a root of } f(x) \text{ given by }  \\ \hline
g_1 &  x^8 - 4x^7 + 54x^6 - 200x^5 + 680x^4 - 3300x^3 + 5400x^2 - 10500x + 36750 \\    
g_2 &  x^8 - 4x^7 + 154x^6 + 40x^5 + 6680x^4 - 29336x^3 + 61764x^2 - 63604x + 27130 \\
g_3 &  x^8 + 20x^6 + 216x^5 + 3142x^4 + 624x^3 + 8892x^2 - 8424x + 5382 \\
\hline
\end{array}
$$
\caption{Eisenstein polynomials defining uniformizers in $F_i$}
\label{Table:uniformizer}
\end{table}
\end{small}

Due to reasons that will be explained in the proof of Theorem~\ref{T:mainWilde8II},
in this section we do not need to compute explicitly the image of $\gamma_E$. It will 
be enough for us to describe the change of coordinates leading 
the model \eqref{E:usefulmodel} into a minimal model of $E/F_i$ with good reduction.
More precisely, we consider changes of coordinates of the form
\begin{equation}
\label{E:coordinates} 
 x=u^2 x' + r, \qquad y=u^3 y' + u^2 s x' + t \qquad \text{ where } \quad u,r,s,t \in \calO_{F}
\end{equation}
and we will prescribe the value of $(u,r,s,t)$, in terms of the standard invariants of $E$.

Fix the following elliptic curves over $\F_2$:
\[
 \Ebar_1 : y^2 + y = x^3 + x^2 + x, \qquad \Ebar_2 : y^2 + y = x^3 + x^2 + x + 1,
\]
\[
 \Ebar_3 : y^2 + y = x^3 + 1, \qquad \Ebar_4 : y^2 + y = x^3. 
\]

\begin{theorem} Let $E/\Q_2$ be an elliptic curve with potentially 
good reduction with $e=8$ and conductor $2^5$ or $2^8$, so that one of the $D$ cases in 
Table~\ref{Table:model} is satisfied. Moreover, in cases $D_e$ and $D_f$ assume further that $E$
has good reduction over the field defined by $g_3$ 
(which is true up to quadratic twist by $\sqrt{2}$ by Theorem~\ref{T:goodOverF}).

Then the change of coordinates \eqref{E:coordinates} 
transforms the model \eqref{E:usefulmodel} into a minimal model with good reduction over $F$ and
residual curve $\Ebar$, where the relevant information is given by 
Tables~\ref{Table:(a)}, ~\ref{Table:(b)}, ~\ref{Table:(c)}, ~\ref{Table:(d)}, ~\ref{Table:(e)}~and~\ref{Table:(f)} in case $D_a$, $D_b$, $D_c$, $D_d$, $D_e$ and $D_f$, respectively.
\label{T:coordchanges8}
\end{theorem}
The proof of this theorem relies on very lengthy and repetitive calculations. 
To illustrate the arguments involved we will sketch the proof of Lemma~\ref{L:lemma22}.
This lemma proves Theorem~\ref{T:coordchanges8} in the case of $E/\Q_2$ satisfying case $D_a$ of Table~\ref{Table:model}, that is, the content of Table~\ref{Table:(a)} is correct.
Alternatively, since the conclusions we aim for depend only on a finite amount of precision for 
the coefficients of the model \eqref{E:usefulmodel} one can write a {\tt Magma} program that starts 
from the model \eqref{E:usefulmodel} applies the change of coordinates and tests for all the conclusions. 

\begin{lemma} \label{L:lemma22}
Let $E/\Q_2$ satisfy case~$D_a$ in Table~\ref{Table:model}.
\begin{enumerate}
 \item If $n \geq 8$, 
 set $\mu = 1 + \pi^5 + \pi^7$ if $\tilde{c}_4 \equiv 7 \pmod{8}$ or
 $\mu = 1 + \pi^4 + \pi^5 + \pi^7$ if $\tilde{c}_4 \equiv 3 \pmod{8}$.
 \item If $n = 7$, 
  set $\mu = 1 + \pi^5$ if $\tilde{c}_4 \equiv 3 \pmod{8}$ or
  $ \mu = 1 + \pi^4 + \pi^5$ if $\tilde{c}_4 \equiv 7 \pmod{8}$.
\end{enumerate}
Here $\pi$ denotes the uniformizer in $F_1$ and $F_2$ 
defined in Table~\ref{Table:uniformizer} 
in the cases (1) and (2), respectively. 
Furthermore, define also the quantities 
\[
u=\pi^4, \qquad r=1+\pi^4, \qquad s=1 + \pi^2, \qquad t = \pi^4 \mu.
\]
Then, in case $(1)$, the change of coordinates \eqref{E:coordinates} 
transforms the model \eqref{E:usefulmodel} into a minimal model with good reduction over $F_1$ and
residual curve $\Ebar$ given by
\begin{enumerate}
 \item[(i)] $\Ebar = \Ebar_1$ if $n \geq 9$ and $\tilde{c}_4 \equiv 7,11 \pmod{16}$ or 
 $n = 8$ and $\tilde{c}_4 \equiv 3,15 \pmod{16}$;
 \item[(ii)] $\Ebar = \Ebar_2$ if $n \geq 9$ and $\tilde{c}_4 \equiv 3,15 \pmod{16}$ or
 $n = 8$ and $\tilde{c}_4 \equiv 7,11 \pmod{16}$;
\end{enumerate}
and, in case $(2)$, the change of coordinates \eqref{E:coordinates} gives rise to 
a minimal model $E/F_2$ with good 
reduction and $\Ebar$ given by
\begin{enumerate}
 \item[(i)] $\Ebar = \Ebar_2$  if $\tilde{c}_4 \equiv 3,7 \pmod{16}$ and 
 $\tilde{c}_6 \equiv 1,5,9,13 \pmod{16}$ or $\tilde{c}_4 \equiv 11,15 \pmod{16}$ and 
 $\tilde{c}_6 \equiv 3,7,11,15 \pmod{16}$;
 \item[(ii)] $\Ebar = \Ebar_1$ if $\tilde{c}_4 \equiv 3,7 \pmod{16}$ and 
 $\tilde{c}_6 \equiv 3,7,11,15 \pmod{16}$ or $\tilde{c}_4 \equiv 11,15 \pmod{16}$ and 
 $\tilde{c}_6 \equiv 1,5,9,13 \pmod{16}$.
\end{enumerate}
\end{lemma}
\begin{proof} 
Suppose we are in case (1). Let $\vv_2$ denote the usual valuation in $\Q_2$ and write 
$\vv$ for the valuation in $F_1$ normalized such that $\vv(2) = 8$. In $F_1$, we have
\[
 2 = \beta \pi^8 \quad \text{ where } \quad \beta = 1 + \pi^6 + \pi^9 + \pi^{11}+ O(\pi^{12}). 
\]
Consider the model $W$ for $E/F_1$ 
\[
 W \; : \; y^2 + a_1' yx + a_3'y = x^3 + a_2' x^2 + a_4'x + a_6'
\]
obtained from \eqref{E:usefulmodel} 
by applying the change of coordinates
given in the statement. It satisfies
\[ 
\vv(\Delta(W)) = \vv (u^{-12} \Delta_m) = -12 \vv(u) + 8 \cdot \vv_2(\Delta_m)
= -12 \cdot 4 + 8 \cdot 6 = 0,
\]
hence it will be a minimal model with good reduction if we can show that $\vv(a_i') \geq 0$ for all~$i$,
where the $a_i'$ are given by the formulas in \eqref{E:coordchange}. We first check 
\[
 \vv(a_1') = \vv(2) + \vv(s) - \vv(u) = 8+0-4=4, \quad  \vv(a_3') = \vv(2) + \vv(t) - 3\vv(u) = 8+4-12=0;
\]
observe also that
\[
 3r-s^2 = 3+3\pi^4 - (1+2\pi^2 + \pi^4) = 2-2\pi^2 + 2\pi^4 
\]
therefore $\vv(a_2') = \vv(2-2\pi^2 + 2\pi^4) - 2\vv(u) = 8-8 = 0$. 

Note that $c_4 \neq 0$ and recall the quantity $\tilde{c}_4 = c_4/2^{\vv_2(c_4)}$. We have
\[
 a_4 = -\frac{\tilde{c}_4}{3} \qquad \text{ and } \qquad a_4' = \frac{a_4 + 3r^2 - 2st}{u^4},
\]
and we need to prove that the numerator of $a_4'$ is divisible by $\pi^{16}$. For this, we compute
\[
 3r^2 = 3 + 6\pi^4 + 3\pi^8, \qquad 2st \equiv 2\pi^4 + 2\pi^6 \pmod{\pi^{16}}
\]
where the second congruence holds for the two values of $t$ in the statement.
So we need that 
\[ 
-\tilde{c}_4/3 + 3 + 6\pi^4 + 3\pi^8 - 2\pi^4 - 2\pi^6 \equiv 
-\tilde{c}_4/3 + 3 - 2\pi^6 + \pi^8 \equiv 0 \pmod{\pi^{16}} 
\]
which is the same as having $\tilde{c}_4 - 9 + 6\pi^6 - 3\pi^8 \equiv 0 \pmod{\pi^{16}}$.

Recall that $\tilde{c}_4 \equiv 3 \pmod{4}$, so we have a $2$-adic expansion
\[
 \tilde{c}_4 = 1 \cdot 2^0 + 1 \cdot 2 + \alpha_1 \cdot 2^2 + O(2^3)
 \quad \text{ with} \quad  \alpha_1 \in \{0,1\};
\]
hence $\tilde{c}_4 \equiv 1 + \beta \pi^8 \pmod{\pi^{16}}$. 
We can now compute
\[
 \tilde{c}_4 - 9 + 6\pi^6 - 3\pi^8  \equiv - 8 + (\beta-3)\pi^8 + 3s\pi^{14} \equiv (3\beta + 1)\pi^{14} \equiv 0 \pmod{\pi^{16}},
\]
as desired. To finish the proof we are left to show $\vv(a_6') \geq 0$; 
indeed, we have 
\[
 a_6' = \frac{a_6 + r a_4 + r^3 - t^2}{u^6} \qquad \text{ where } \qquad a_6 = - \frac{2^{n-5} \tilde{c}_6}{3^3}  
\]
and want to prove $a_6 + r a_4 + r^3 - t^2 \equiv 0 \pmod{\pi^{24}}$. 

Since $n \geq 8$ it follows that $2^3 \mid a_6$ thus $a_6 \equiv 0 \pmod{\pi^{24}}$.

Assume further $\tilde{c}_4 \equiv 7 \pmod{8}$ so that $t = \pi^4 + \pi^9 + \pi^{11}$;
in particular, $\tilde{c}_4 \equiv 7 \pmod{\pi^{24}}$.

We check that $r^3 = 1 + 3\pi^4 + 3\pi^8 + \pi^{12}$ and
\[
 r a_4 \equiv \frac{-7}{3} - 7\pi^{4} \pmod{\pi^{24}}, \qquad 
 t^2  \equiv \pi^{8} + \pi^{18} + \pi^{21} + \pi^{22} + \pi^{23} \pmod{\pi^{24}}
\]
and further calculations show 
\[
 a_6 + r a_4 + r^3 - t^2 \equiv 4 - 9\pi^8 - (3+s)\pi^{12} + 3t^2 
 \equiv 4 - 6\pi^8 + 3 \pi^{22} \equiv 0 \pmod{\pi^{24}},
\]
concluding the proof when $\tilde{c}_4 \equiv 7 \pmod{8}$; for the
case $\tilde{c}_4 \equiv 3 \pmod{8}$ we proceed with similar calculations, 
which completes the proof that the model is minimal and with good 
reduction in case (1). For the claim regarding the residual curve we proceed
with same kind of computations using one further step of precision; this will 
allow to decide if $\vv(a_6') = 0$, hence $\Ebar = \Ebar_2$, or $\vv(a_6') > 0$ 
giving $\Ebar = \Ebar_1$.

Case (2) follows from analogous computations over the field $F_2$. 
\end{proof}

\section{The morphism $\gamma_E$ in the wild case $e=12$}

Let $E/\Q_3$ be an elliptic curve with $e=12$, hence conductor $3^3$ or $3^5$.
From parts (6) and~(7) of Theorem~\ref{T:goodOverF} we know that $E$ obtains good reduction over
a field $F$ defined by one of the polynomials $h_i$, for $i=1,2,3,4,5$.

For the field defined by $h_i$ we fix an uniformizer $\pi$ satisfying $h_i(\pi) = 0$.

As in Section~\ref{S:coordchanges12} we will not describe the image of $\gamma_E$ 
explicitly. Again, it will 
be enough for us to describe change of coordinates leading 
the model \eqref{E:usefulmodel} into a minimal model of $E/F$ with good reduction.
More precisely, we consider change of coordinates of the form
\begin{equation}
\label{E:coordinates12} 
 x=u^2 x' + r, \qquad y=u^3 y'  \qquad \text{ where } \quad u,r \in \calO_{F}
\end{equation}
and we will prescribe the value of $(u,r)$, in terms of the standard invariants of $E$.

Fix the following elliptic curves over $\F_3$:
\[
 \Ebar_0 : y^2 = x^3 + x, \quad \Ebar_1 : y^2 = x^3 + x + 1 \quad 
 \Ebar_2 : y^2 = x^3 + x + 2 \quad \Ebar_3 : y^2 = x^3 + 2x. 
\]
\begin{theorem} Let $E/\Q_3$ be an elliptic curve with potentially 
good reduction with $e=12$, so that one of the $G$ cases in 
Table~\ref{Table:model} is satisfied. 

Then the change of coordinates \eqref{E:coordinates12} 
transforms the model \eqref{E:usefulmodel} into a minimal model with good reduction over $F$ and
residual curve $\Ebar$, where the relevant information is given by 
Tables~\ref{Table:(a12)}, ~\ref{Table:(b12)}, ~\ref{Table:(c12)}, ~\ref{Table:(d12)}, ~\ref{Table:(e12)},
~\ref{Table:(f12)},~\ref{Table:(g12)},~\ref{Table:(h12)},~\ref{Table:(i12)}~and~~\ref{Table:(j12)} 
in case $G_a$, $G_b$, $G_c$, $G_d$, $G_e$, $G_f$, $G_g$, $G_h$, $G_i$ and $G_j$, respectively.
\label{T:coordchanges12}
\end{theorem}
\begin{proof} This follows from the same kind of lengthy calculations as Theorem~\ref{T:coordchanges8}.
\end{proof}

\section{Tables with coordinate changes}
The following tables give the coordinate changes used in Theorems~\ref{T:coordchanges8}~and~\ref{T:coordchanges12}.

\label{S:coordchanges}
\begin{footnotesize}
\begin{table}[htb]
$$
\begin{array}{|c|c|c|c|c|c|c|c|} \hline
\mbox{Case } D_a   & \tilde{c}_4 \pmod{8} &  u    &r& s & t & \mbox{Field} \\ \hline
(4, n \geq 8,  6) &  7                   & \pi^4 & 1+\pi^4  & 1 + \pi^2  & \pi^4 (1 + \pi^5 + \pi^7) & g_1 \\ 
(4, n \geq 8,  6) &  3                   & \pi^4 & 1+\pi^4  & 1 + \pi^2  & \pi^4 (1 + \pi^4 + \pi^5 + \pi^7) & g_1 \\
\hline
(4, n = 7,  6)    &  3                   & \pi^4 & 1+\pi^4  & 1 + \pi^2  & \pi^4 (1 + \pi^5) & g_2 \\ 
(4, n = 7,  6)    &  7                   & \pi^4 & 1+\pi^4  & 1 + \pi^2  & \pi^4 (1 + \pi^4 + \pi^5) & g_2 \\
\hline
\end{array}
$$
$$
\begin{array}{|c|c|c|} \hline
\mbox{Field} & n \ge 7 \; \text{ and } \; \tilde{c}_4, \; \tilde{c}_6 \pmod{16} &  \Ebar \\ \hline
  g_1        & n \geq 9, \; \tilde{c}_4 \; \equiv 7, 11 \quad \text{or} \quad n = 8, \; \tilde{c}_4 \equiv 3, 15 &  \Ebar_1 \\ 
  g_1        & n \geq 9, \; \tilde{c}_4 \equiv 3, 15 \quad \text{or} \quad  n = 8, \; \tilde{c}_4 \equiv 7, 11   &  \Ebar_2 \\ 
  \hline
  g_2        & n=7, \; \tilde{c}_4 \equiv 3, 7, \;  \tilde{c}_6 \equiv   3, 7, 11, 15 \quad  \text{or} \quad
   n=7, \; \tilde{c}_4 \equiv 11, 15, \tilde{c}_6 \equiv   1, 5, 9,  13 & \Ebar_1 \\
  g_2        & n=7, \; \tilde{c}_4 \equiv 3, 7, \;  \tilde{c}_6 \equiv   1, 5, 9,  13 \quad \text{or} \quad
  n=7, \; \tilde{c}_4 \equiv 11, 15, \; \tilde{c}_6 \equiv   3, 7, 11, 15 & \Ebar_2 \\
\hline
\end{array}
$$
\caption{Description of $F$, $\Ebar$ and $(u,r,s,t)$ in case~$D_a$.}
\label{Table:(a)}
\end{table}
\end{footnotesize}

\begin{footnotesize}
\begin{table}[htb]
$$
\begin{array}{|c|c|c|c|c|c|c|} \hline
\mbox{Case } D_b   & \tilde{c}_4 \pmod{8} &  u    &r& s & t & \mbox{Field} \\ \hline
(6, n \geq 11, 12) &  5                   & \pi^8 & \pi^8  & \pi^4  & \pi^{16} (1 + \pi^2 + \pi^6 + \pi^7) & g_1 \\ 
(6, n \geq 11, 12) &  1                   & \pi^8 & \pi^8  & \pi^4  & \pi^{16} (1 + \pi^2 + \pi^4 + \pi^7) & g_1 \\
\hline
(6, n = 10, 12)  &  5                   & \pi^8 & \pi^8  & \pi^4  & \pi^{16} (1 + \pi^2 + \pi^7) & g_2 \\ 
(6, n = 10, 12)  &  1                   & \pi^8 & \pi^8  & \pi^4  & \pi^{16} (1 + \pi^2 + \pi^4 + \pi^6 + \pi^7) & g_2 \\
\hline
\end{array}
$$
$$
\begin{array}{|c|c|c|} \hline
\mbox{Field} & n \ge 10 \; \text{ and } \; \tilde{c}_4, \; \tilde{c}_6 \pmod{16} &  \Ebar \\ \hline
  g_1        & n \geq 12, \; \tilde{c}_4 \; \equiv 1, 5 \quad \text{or} \quad n = 11, \; \tilde{c}_4 \equiv 9, 3 &  \Ebar_1 \\ 
  g_1        & n \geq 12, \; \tilde{c}_4 \equiv 9,3  \quad \text{or} \quad  n = 11, \; \tilde{c}_4 \equiv 1, 5   &  \Ebar_2 \\ 
  \hline
  g_2        & n=10, \; \tilde{c}_4 \equiv 9, 13, \;  \tilde{c}_6 \equiv   3, 7, 11, 15 \quad  \text{or} \quad   n=10, \; \tilde{c}_4 \equiv 1, 5, \tilde{c}_6 \equiv   1, 5, 9, 13 & \Ebar_1 \\
  g_2        & n=10, \; \tilde{c}_4 \equiv 9, 13, \;  \tilde{c}_6 \equiv   1, 5, 9,  13 \quad \text{or} \quad
  n=10, \; \tilde{c}_4 \equiv 1, 5, \; \tilde{c}_6 \equiv   3, 7, 11, 15 & \Ebar_2 \\
\hline
\end{array}
$$
\caption{Description of $F$, $\Ebar$ and $(u,r,s,t)$ in case~$D_b$.}
\label{Table:(b)}
\end{table}
\end{footnotesize}

\begin{footnotesize}
\begin{table}[htb]
$$
\begin{array}{|c|c|c|c|c|c|c|} \hline
\mbox{Case } D_c   & \tilde{c}_4, \tilde{c}_6 \pmod{4} &  u    &r& s & t & \mbox{Field } \\ \hline
(7, 9, 12) &  1, 1                	 & \pi^8 & \pi^{12}  & \pi^6  & \pi^{16} (1 + \pi^5 + \pi^6 + \pi^7) & g_1 \\ 
(7, 9, 12) &  1, 3                  & \pi^8 & \pi^{12}  & \pi^6  & \pi^{16} (1 + \pi^4 + \pi^5 + \pi^7) & g_1 \\
\hline
(7, 9, 12)  &  3,1                    & \pi^8 & \pi^{12}  & \pi^6  & \pi^{16} (1 + \pi^5 + \pi^6) & g_2 \\ 
(7, 9, 12)  &  3,3                    & \pi^8 & \pi^{12}  & \pi^6  & \pi^{16} (1 + \pi^4 + \pi^5) & g_2 \\
\hline
\end{array}
$$
$$
\begin{array}{|c|c|c|} \hline
\mbox{Field} & \tilde{c}_6 \pmod{8} &  \Ebar \\ \hline
  g_1        & 5, 7 &  \Ebar_3 \\ 
  g_1        & 1, 3 &  \Ebar_4 \\ 
  \hline
  g_2        & 1, 7 &  \Ebar_3 \\ 
  g_2        & 3, 5 &  \Ebar_4 \\
\hline
\end{array}
$$
\caption{Description of $F$, $\Ebar$ and $(u,r,s,t)$ in case~$D_c$.}
\label{Table:(c)}
\end{table}
\end{footnotesize}

\begin{footnotesize}
\begin{table}[htb]
$$
\begin{array}{|c|c|c|c|c|c|c|} \hline
\mbox{Case } D_d   & \tilde{\Delta}  \pmod{4}, \; \tilde{c}_6 \pmod{8} &  u    &r& s & t & \mbox{Field} \\ \hline
(4, 6, 9) &  1, 1                  & \pi^6 & 1+\pi^{10}  & 1 + \pi^4 + \pi^5  & \pi^{10} (1 + \pi^2 + \pi^4 + \pi^5) & g_1 \\ 
(4, 6, 9) &  1, 3                  & \pi^6 & 1+\pi^{8} + \pi^{10}  & 1+\pi^5  & \pi^{10} (1 + \pi^2 + \pi^4 + \pi^5 + \pi^7) & g_1 \\
(4, 6, 9) &  1, 5                  & \pi^6 & 1+\pi^{10}  & 1 + \pi^4 + \pi^5  & \pi^{10} (1 + \pi^2 + \pi^4 + \pi^5 + \pi^6) & g_1 \\
(4, 6, 9) &  1, 7                  &  \pi^6 & 1+\pi^{8} + \pi^{10}  & 1+\pi^5  & \pi^{10} (1 + \pi^2 + \pi^4 + \pi^5 + \pi^6 + \pi^7) & g_1 \\
\hline
(4, 6, 9) &  3, 1                  & \pi^6 & 1+\pi^{10}  & 1 + \pi^4 + \pi^5  & \pi^{10} (1 + \pi^2 + \pi^4 + \pi^5 + \pi^7) & g_2 \\ 
(4, 6, 9) &  3, 3                  & \pi^6 & 1+\pi^{8} + \pi^{10}  & 1+\pi^5  & \pi^{10} (1 + \pi^2 + \pi^4 + \pi^5 + \pi^6) & g_2 \\
(4, 6, 9) &  3, 5                  & \pi^6 & 1+\pi^{10}  & 1 + \pi^4 + \pi^5  & \pi^{10} (1 + \pi^2 + \pi^4 + \pi^5 + \pi^6 + \pi^7) & g_2 \\
(4, 6, 9) &  3, 7                  &  \pi^6 & 1+\pi^{8} + \pi^{10}  & 1+\pi^5  & \pi^{10} (1 + \pi^2 + \pi^4 + \pi^5) & g_2 \\
\hline
\end{array}
$$
$$
\begin{array}{|c|c|c|} \hline
\mbox{Field} & \tilde{c}_6 \pmod{8} &  \Ebar \\ \hline
  g_1        & 1,3,5,7 &  \Ebar_2, \Ebar_3, \Ebar_1, \Ebar_4, \text{ respectively} \\ 
  \hline
  g_2        & 1,3,5,7 &  \Ebar_1, \Ebar_3, \Ebar_2, \Ebar_4, \text{ respectively} \\ 
\hline
\end{array}
$$
\caption{Description of $F$, $\Ebar$ and $(u,r,s,t)$ in case~$D_d$.}
\label{Table:(d)}
\end{table}
\end{footnotesize}

\begin{footnotesize}
\begin{table}[htb]
$$ 
\begin{array}{|c|c|c|c|c|c|c|} \hline
\mbox{Case } D_e   & \tilde{c}_4 \pmod{4} &  u  &r& s & t & \mbox{Field} \\ \hline
(5, n = 9,  9) &  1   & \pi^6 & \pi^8 + \pi^6 + \pi^4  & \pi^4 + \pi^3 + \pi^2 
& \pi^{10} (\pi^{7} + \pi^4 + \pi^3 + \pi + 1) & g_3 \\ 
(5, n = 9,  9) &  3   & \pi^6 & \pi^{10} + \pi^6 + \pi^4  & \pi^5 + \pi^3 + \pi^2
& \pi^{10} (\pi^7 + \pi^6 + \pi^5 + \pi^4 + \pi^2 + \pi + 1) & g_3 \\ 
(5, n \geq 10 ,  9) &  1   & \pi^6 & \pi^8 + \pi^6 + \pi^4  & \pi^4 + \pi^3 + \pi^2 
& \pi^{10} (\pi^6 + \pi^4 + \pi^3 + \pi + 1) & g_3 \\ 
(5, n \geq 10,   9) &  3   & \pi^6 & \pi^{10} + \pi^6 + \pi^4  & \pi^5 + \pi^3 + \pi^2
& \pi^{10} (\pi^5 + \pi^4 + \pi^2 + \pi + 1) & g_3 \\ 
\hline
\end{array}
$$
$$
\begin{array}{|c|c|c|} \hline
\mbox{Field} & n \ge 9, \; \tilde{c}_4 \pmod{16} &  \Ebar \\ \hline
  g_3        & n = 9, \; \tilde{c}_4 \; \equiv 1, 11 \quad \text{or} \quad n \geq 10, \; \tilde{c}_4 \equiv 9, 11 &  \Ebar_1 \\ 
  g_3        & n = 9, \; \tilde{c}_4 \; \equiv 3, 9 \quad \text{or} \quad n \geq 10, \; \tilde{c}_4 \equiv 1, 3 &  \Ebar_2 \\ 
  \hline
\end{array}
$$
\caption{Description of $F$, $\Ebar$ and $(u,r,s,t)$ in case~$D_e$.}
\label{Table:(e)}
\end{table}
\end{footnotesize}

\begin{footnotesize}
\begin{table}[htb]
$$
\begin{array}{|c|c|c|c|c|c|c|} \hline
\mbox{Case } D_f   & \tilde{c}_4 \pmod{4} &  u  &r& s & t & \mbox{Field} \\ \hline
(7, n = 12, 15) &  1   & \pi^{10} & \pi^{18} + \pi^{14} + \pi^{12}  & \pi^{9} + \pi^{7} + \pi^{6}
& \pi^{22} (\pi^{6} + \pi^2 + \pi + 1) & g_3 \\ 
(7, n = 12, 15) &  3   & \pi^{10} & \pi^{16} + \pi^{14} + \pi^{12}  & \pi^8 + \pi^7 + \pi^6
& \pi^{22} (\pi^{7} + \pi^6 + \pi^5 + \pi^3 + \pi + 1) & g_3 \\ 
(7, n \geq 13, 15) &  1   & \pi^{10} & \pi^{18} + \pi^{14} + \pi^{12}  & \pi^9 + \pi^7 + \pi^6
& \pi^{22} (\pi^7 + \pi^2 + \pi + 1) & g_3 \\ 
(7, n \geq 13, 15) &  3   & \pi^{10} & \pi^{16} + \pi^{14} + \pi^{12}  & \pi^{8} + \pi^{7} + \pi^{6}
& \pi^{22} (\pi^5 + \pi^3  + \pi + 1) & g_3 \\ 
\hline
\end{array}
$$
$$
\begin{array}{|c|c|c|} \hline
\mbox{Field} & n \ge 12, \; \tilde{c}_4 \pmod{16} &  \Ebar \\ \hline
  g_3        & n = 12, \; \tilde{c}_4 \; \equiv 5, 15 \quad \text{or} \quad n \geq 13, \; \tilde{c}_4 \equiv 13, 15 &  \Ebar_1 \\ 
  g_3        & n = 12, \; \tilde{c}_4 \; \equiv 7, 13 \quad \text{or} \quad n \geq 13, \; \tilde{c}_4 \equiv 5, 7 &  \Ebar_2 \\ 
  \hline
\end{array}
$$
\caption{Description of $F$, $\Ebar$ and $(u,r,s,t)$ in case~$D_f$.}
\label{Table:(f)}
\end{table}
\end{footnotesize}

\begin{footnotesize}
\begin{table}[htb]
$$
\begin{array}{|c|c|c|c|c|c|c|} \hline
\mbox{Case } G_a  & \tilde{c}_6 \pmod{9} &  u       &   r                & \mbox{Field} &  \Ebar  \\ \hline
(n \geq 3, 3, 3)  &  1                   & \pi^{3}  & 2\pi^4 + 2         &   h_2        & \Ebar_0 \\
(n \geq 3, 3, 3)  &  2                   & \pi^{3}  & 2\pi^4 + 1         &   h_2        & \; \Ebar_0 \\  
(n \geq 3, 3, 3)  &  7                   & \pi^{3}  & \pi^4 + 2          &   h_2        &  \; \Ebar_0 \\
(n \geq 3, 3, 3)  &  8                   & \pi^{3}  & \pi^4 + 1          &   h_2        & \Ebar_0 \\
\hline
(n = 2, 3, 3)    &  1                   & \pi^{3}  & 2\pi^5 + \pi^4 + 2 &   h_1        & \Ebar_2 \\
(n = 2, 3, 3)    &  4                   & \pi^{3}  & \pi^5 + 2\pi^4 + 2 &   h_1        & \Ebar_1 \\
(n = 2, 3, 3)    &  5                   & \pi^{3}  & 2\pi^5 + \pi^4 + 1 &   h_1        & \Ebar_2 \\
(n = 2, 3, 3)    &  8                   & \pi^{3}  & \pi^5 + 2\pi^4 + 1 &   h_1        & \Ebar_1 \\ \hline
\end{array}
$$
\caption{Description of $F$, $\Ebar$ and $(u,r)$ in case~$G_a$.}
\label{Table:(a12)}
\end{table}
\end{footnotesize}

\begin{footnotesize}
\begin{table}[htb]
$$
\begin{array}{|c|c|c|c|c|c|c|} \hline
\mbox{Case } G_b  & \tilde{c}_6 \pmod{9} &  u       &   r                  & \mbox{Field} &  \Ebar  \\ \hline
(n \geq 5, 6, 9)  &  1                   & \pi^{9}  & \pi^{12}             &   h_2        & \Ebar_3 \\
(n \geq 5, 6, 9)  &  2                   & \pi^{9}  & \pi^{12}(2\pi^4 + 2) &   h_2        & \Ebar_3 \\  
(n \geq 5, 6, 9)  &  7                   & \pi^{9}  & \pi^{12}(\pi^4 + 1)  &   h_2        & \Ebar_3 \\
(n \geq 5, 6, 9)  &  8                   & \pi^{9}  & 2\pi^{12}            &   h_2        & \Ebar_3 \\
\hline
(n = 4, 6, 9) &  1                   & \pi^{9}  & \pi^{12}(2\pi^3 + 1) &   h_1        & \Ebar_2 \\
(n = 4, 6, 9) &  4                   & \pi^{9}  & \pi^{12}(\pi^5 + 2\pi^4 + 2\pi^3 + 1) &  h_1        & \Ebar_0 \\
(n = 4, 6, 9) &  5                   & \pi^{9}  & \pi^{12}(2\pi^5 + \pi^4 + \pi^3 + 2) &   h_1        & \Ebar_0 \\
(n = 4, 6, 9) &  8                   & \pi^{9}  & \pi^{12}(\pi^3 + 2) &   h_1        & \Ebar_1 \\ \hline
\end{array}
$$
\caption{Description of $F$, $\Ebar$ and $(u,r)$ in case~$G_b$.}
\label{Table:(b12)}
\end{table}
\end{footnotesize}

\begin{footnotesize}
\begin{table}[htb]
$$
\begin{array}{|c|c|c|c|c|c|c|} \hline
\mbox{Case } D_c  & \tilde{\Delta}_m, \; \tilde{c}_6 \pmod{3} &  u   &   r             & \mbox{Field} &  \Ebar  \\ \hline
(2,4,3) &  1, 1                   & \pi^{3}  & \pi^4(2\pi + 1)         &   h_1        & \Ebar_2 \\
(2,4,3) &  1, 2                   & \pi^{3}  & \pi^4(\pi + 2)          &   h_1        & \Ebar_1 \\
(2,4,3) &  2, 1                   & \pi^{3}  & \pi^4                   &   h_2        & \Ebar_3 \\
(2,4,3) &  2, 2                   & \pi^{3}  & 2\pi^4                  &   h_2        & \Ebar_3 \\
\hline
\end{array}
$$
\caption{Description of $F$, $\Ebar$ and $(u,r)$ in case~$G_c$.}
\label{Table:(c12)}
\end{table}
\end{footnotesize}

\begin{footnotesize}
\begin{table}[htb]
$$
\begin{array}{|c|c|c|c|c|c|c|} \hline
\mbox{Case } G_d  & \tilde{\Delta}_m, \; \tilde{c}_6 \pmod{3} &  u   &   r             & \mbox{Field} &  \Ebar  \\ \hline
(2,3,5) &  1, 1                   & \pi^{5}  & 2\pi^9 + 2\pi^8 + 2     &   h_1        & \Ebar_1 \\
(2,3,5) &  1, 2                   & \pi^{5}  & \pi^9 +  \pi^8 + 1      &   h_1        & \Ebar_2 \\
(2,3,5) &  2, 1                   & \pi^{5}  & \pi^8 + 2               &   h_2        & \Ebar_3 \\
(2,3,5) &  2, 2                   & \pi^{5}  & 2\pi^8 + 1              &   h_2        & \Ebar_3 \\
\hline
\end{array}
$$
\caption{Description of $F$, $\Ebar$ and $(u,r)$ in case~$G_d$.}
\label{Table:(d12)}
\end{table}
\end{footnotesize}

\begin{footnotesize}
\begin{table}[htb]
$$
\begin{array}{|c|c|c|c|c|c|c|} \hline
\mbox{Case } G_e  & \tilde{\Delta}_m, \; \tilde{c}_6 \pmod{3} &  u   &   r             & \mbox{Field} &  \Ebar  \\ \hline
(4,7,9) &  1, 1                   & \pi^{9}  & \pi^{16}(\pi  + 2)     &   h_1        & \Ebar_1 \\
(4,7,9) &  1, 2                   & \pi^{9}  & \pi^{16}(2\pi + 1)     &   h_1        & \Ebar_2 \\
(4,7,9) &  2, 1                   & \pi^{9}  & 2\pi^{16}              &   h_2        & \Ebar_3 \\
(4,7,9) &  2, 2                   & \pi^{9}  & \pi^{16}               &   h_2        & \Ebar_3 \\
\hline
\end{array}
$$
\caption{Description of $F$, $\Ebar$ and $(u,r)$ in case~$G_e$}
\label{Table:(e12)}
\end{table}
\end{footnotesize}

\begin{footnotesize}
\begin{table}[htb]
$$
\begin{array}{|c|c|c|c|c|c|c|} \hline
\mbox{Case } G_f  & \tilde{\Delta}_m, \; \tilde{c}_6 \pmod{3} &  u   &   r             & \mbox{Field} &  \Ebar  \\ \hline
(4,6,11) &  1, 1                   & \pi^{11}  & \pi^{12}(\pi^8 + \pi^7 + \pi^6 + 2\pi^5 + \pi^4 + 2\pi^3 + 1) 
&   h_1        & \Ebar_1 \\
(4,6,11) &  1, 2                   & \pi^{11}  & \pi^{12}(2\pi^8 + 2\pi^7 + 2\pi^6 + \pi^5 + 2\pi^4 + \pi^3 + 2) 
&   h_1        & \Ebar_2 \\
(4,6,11) &  2, 1                   & \pi^{11}  & \pi^{12}(2\pi^4 + 1)              &   h_2        & \Ebar_3 \\
(4,6,11) &  2, 2                   & \pi^{11}  & \pi^{12}(\pi^4 + 2)               &   h_2        & \Ebar_3 \\
\hline
\end{array}
$$
\caption{Description of $F$, $\Ebar$ and $(u,r)$ in case~$G_f$.}
\label{Table:(f12)}
\end{table}
\end{footnotesize}

\begin{footnotesize}
\begin{table}[htb]
$$
\begin{array}{|c|c|c|c|c|c|c|} \hline
\mbox{Case } G_g  & \tilde{\Delta} \pmod{9}, \; \tilde{c}_6 \pmod{3} &  u       &   r                & \mbox{Field} &  \Ebar  \\ \hline
(n \geq 4, 4, 5) &  8, 1                & \pi^{5}  & \pi^4(2\pi^5 + 2\pi^3 + \pi + 2)  &   h_3        & \Ebar_0 \\
(n \geq 4, 4, 5) &  8, 2                & \pi^{5}  & \pi^4(\pi^5 + \pi^3 + 2\pi + 1)   &    h_3        & \Ebar_0 \\
(n \geq 4, 4, 5) &  5, 1                & \pi^{5}  & \pi^4(\pi^4 + \pi^2 + 2)    &   h_4                & \Ebar_0 \\
(n \geq 4, 4, 5) &  5, 2                & \pi^{5}  & \pi^4(2\pi^4 + 2\pi^2 + 1)  &   h_4              & \Ebar_0 \\
(n \geq 4, 4, 5) &  2, 1                & \pi^{5}  & \pi^4(2\pi^5 + 2\pi^4 + \pi^3 + 2 \pi + 2)    &   h_5                & \Ebar_0 \\
(n \geq 4, 4, 5) &  2, 2                & \pi^{5}  & \pi^4(\pi^5 + \pi^4 + 2\pi^3 + \pi + 1)  &   h_5              & \Ebar_0 \\
\hline
\end{array}
$$
$$
\begin{array}{|c|c|c|c|c|c|c|c|} \hline
\mbox{Case } G_g  & \tilde{\Delta} \pmod{9}, \; \tilde{c}_4 \pmod{3} & \tilde{c}_6 \pmod{9} &  u       &   r                & \mbox{Field} &  \Ebar  \\ \hline
(n = 3, 4, 5) &  2, 2 \text{ or } 5, 1  &   2  & \pi^{5}  & \pi^4(2\pi^4 + \pi^3 + 2\pi + 1)  &   h_3  & \Ebar_2 \\
(n = 3, 4, 5) &  2, 2 \text{ or } 5, 1  &   4  & \pi^{5}  & \pi^4(\pi^5 + 2\pi^4 + 2\pi^3 + \pi + 2)  &   h_3  & \Ebar_2 \\
(n = 3, 4, 5) &  2, 2 \text{ or } 5, 1  &   5  & \pi^{5}  & \pi^4(2\pi^5 + \pi^4 + \pi^3 + 2\pi + 1)  &   h_3  & \Ebar_1 \\
(n = 3, 4, 5) &  2, 2 \text{ or } 5, 1  &   7  & \pi^{5}  & \pi^4(\pi^4 + 2\pi^3 + \pi + 2)  &   h_3  & \Ebar_1 \\

(n = 3, 4, 5) &  2, 1 \text{ or } 8, 2  &   1  & \pi^{5}  & \pi^4(\pi^2 + 2)  &   h_4  & \Ebar_2 \\
(n = 3, 4, 5) &  2, 1 \text{ or } 8, 2  &   4  & \pi^{5}  & \pi^4(2\pi^4 + \pi^2 + 2)  &   h_4  & \Ebar_1 \\
(n = 3, 4, 5) &  2, 1 \text{ or } 8, 2  &   5  & \pi^{5}  & \pi^4(\pi^4 + 2\pi^2 + 1)  &   h_4  & \Ebar_2 \\
(n = 3, 4, 5) &  2, 1 \text{ or } 8, 2  &   8  & \pi^{5}  & \pi^4(2\pi^2 + 1)  &   h_4  & \Ebar_1 \\

(n = 3, 4, 5) &  5, 2 \text{ or } 8, 1  &   1  & \pi^{5}  & \pi^4(\pi^5 + \pi^3 + 2\pi + 2)  &   h_5  & \Ebar_1 \\
(n = 3, 4, 5) &  5, 2 \text{ or } 8, 1  &   2  & \pi^{5}  & \pi^4(2\pi^4 + 2\pi^3 + \pi + 1)  &   h_5  & \Ebar_1 \\
(n = 3, 4, 5) &  5, 2 \text{ or } 8, 1  &   7  & \pi^{5}  & \pi^4(\pi^4 + \pi^3 + 2\pi + 2)  &   h_5  & \Ebar_2 \\
(n = 3, 4, 5) &  5, 2 \text{ or } 8, 1  &   8  & \pi^{5}  & \pi^4(2\pi^5 + 2\pi^3 + \pi + 1)  &   h_5  & \Ebar_2 \\
\hline
\end{array}
$$
\caption{Description of $F$, $\Ebar$ and $(u,r)$ in case~$G_g$.}
\label{Table:(g12)}
\end{table}
\end{footnotesize}

\begin{footnotesize}
\begin{table}[htb]
$$
\begin{array}{|c|c|c|c|c|c|c|} \hline
\mbox{Case } G_h  & \tilde{\Delta} \pmod{9}, \; \tilde{c}_6 \pmod{3} &  u       &   r                & \mbox{Field} &  \Ebar  \\ \hline
(n \geq 4, 5, 7) &  5, 1                & \pi^{7}  & \pi^8(\pi^5 + 2\pi^4 + \pi^3 + 2\pi^2 + 2\pi + 2)  &   h_3        & \Ebar_1 \\
(n \geq 4, 5, 7) &  5, 2                & \pi^{7}  & \pi^8(2\pi^5 + \pi^4 + 2\pi^3 + \pi^2 + \pi + 1)  &   h_3        & \Ebar_2 \\

(n \geq 4, 5, 7) &  8, 1                & \pi^{7}  & \pi^8(\pi^4 + 2\pi^2 + 2)  &   h_4        & \Ebar_1 \\
(n \geq 4, 5, 7) &  8, 2                & \pi^{7}  & \pi^8(2\pi^4 + \pi^2 + 1)  &   h_4        & \Ebar_2 \\

(n \geq 4, 5, 7) &  2, 1                & \pi^{7}  & \pi^8(2\pi^5 + 2\pi^3 + 2\pi^2 + \pi + 2) &   h_5        & \Ebar_0 \\
(n \geq 4, 5, 7) &  2, 2                & \pi^{7}  & \pi^8(\pi^5 + \pi^3 + \pi^2 + 2\pi + 1)  &   h_5        & \Ebar_0 \\
\hline
\end{array}
$$
\caption{Description of $F$, $\Ebar$ and $(u,r)$ in case~$G_h$.}
\label{Table:(h12)}
\end{table}
\end{footnotesize}

\begin{footnotesize}
\begin{table}[htb]
$$
\begin{array}{|c|c|c|c|c|c|c|} \hline
\mbox{Case } G_i  & \tilde{\Delta} \pmod{9}, \; \tilde{c}_6 \pmod{3} &  u       &   r                & \mbox{Field} &  \Ebar  \\ \hline
(n \geq 6, 7, 11) &  8, 1                & \pi^{11}  & \pi^{16}(2\pi^5 + 2\pi^4 + \pi + 2)  &   h_3      & \Ebar_1 \\
(n \geq 6, 7, 11) &  8, 2                & \pi^{11}  & \pi^{16}(\pi^5 + \pi^4 + 2\pi + 1)  &   h_3      & \Ebar_2 \\

(n \geq 6, 7, 11) &  5, 1                & \pi^{11}  & \pi^{16}(\pi^4 + \pi^2 + 2)  &   h_4      & \Ebar_1 \\
(n \geq 6, 7, 11) &  5, 2                & \pi^{11}  & \pi^{16}(2\pi^4 + 2\pi^2 + 1)  &   h_4      & \Ebar_2 \\

(n \geq 6, 7, 11) &  2, 1                & \pi^{11}  & \pi^{16}(2\pi^5 + \pi^4 + 2\pi + 2)  &   h_5      & \Ebar_1 \\
(n \geq 6, 7, 11) &  2, 2                & \pi^{11}  & \pi^{16}(\pi^5 + 2\pi^4 + \pi + 1)  &   h_5      & \Ebar_2 \\

\hline
\end{array}
$$
$$
\begin{array}{|c|c|c|c|c|c|c|c|} \hline
\mbox{Case } G_i  & \tilde{\Delta} \pmod{9}, \; \tilde{c}_4 \pmod{3} & \tilde{c}_6 \pmod{9} &  u       &   r                & \mbox{Field} &  \Ebar  \\ \hline
(n = 5, 7, 11) &  2, 2 \text{ or } 5, 1  &   2  & \pi^{11}  & \pi^{16}(2\pi + 1)  &   h_3  & \Ebar_1 \\
(n = 5, 7, 11) &  2, 2 \text{ or } 5, 1  &   4  & \pi^{11}  & \pi^{16}(\pi^5 + \pi^4 + \pi + 2)  &   h_3  & \Ebar_0 \\
(n = 5, 7, 11) &  2, 2 \text{ or } 5, 1  &   5  & \pi^{11}  & \pi^{16}(2\pi^5 + 2\pi^4 + 2\pi + 1) & h_3  & \Ebar_0 \\
(n = 5, 7, 11) &  2, 2 \text{ or } 5, 1  &   7  & \pi^{11}  & \pi^{16}(\pi + 2)  &   h_3  & \Ebar_2 \\

(n = 5, 7, 11) &  2, 1 \text{ or } 8, 2  &   1  & \pi^{11}  & \pi^{16}(\pi^2 + 2)  &   h_4  & \Ebar_0 \\
(n = 5, 7, 11) &  2, 1 \text{ or } 8, 2  &   4  & \pi^{11}  & \pi^{16}(2\pi^4 + \pi^2 + 2)  &   h_4  & \Ebar_2 \\
(n = 5, 7, 11) &  2, 1 \text{ or } 8, 2  &   5  & \pi^{11}  & \pi^{16}(\pi^4 + 2\pi^2 + 1) & h_4  & \Ebar_1 \\
(n = 5, 7, 11) &  2, 1 \text{ or } 8, 2  &   8  & \pi^{11}  & \pi^{16}(2\pi^2 + 1)  &   h_4  & \Ebar_0 \\
(n = 5, 7, 11) &  5, 2 \text{ or } 8, 1  &   1  & \pi^{11}  & \pi^{16}(\pi^5 + 2\pi^4 + 2\pi + 2) &   h_5  & \Ebar_2\\
(n = 5, 7, 11) &  5, 2 \text{ or } 8, 1  &   2  & \pi^{11}  & \pi^{16}(\pi + 1) &   h_5  & \Ebar_0\\
(n = 5, 7, 11) &  5, 2 \text{ or } 8, 1  &   7  & \pi^{11}  & \pi^{16}(2\pi + 2) &   h_5  & \Ebar_0\\
(n = 5, 7, 11) &  5, 2 \text{ or } 8, 1  &   8  & \pi^{11}  & \pi^{16}(2\pi^5 + \pi^4 + \pi + 1) &   h_5  & \Ebar_1\\
\hline
\end{array}
$$
\caption{Description of $F$, $\Ebar$ and $(u,r)$ in case~$G_i$.}
\label{Table:(i12)}
\end{table}
\end{footnotesize}

\begin{footnotesize}
\begin{table}[htb]
$$
\begin{array}{|c|c|c|c|c|c|c|} \hline
\mbox{Case } G_j  & \tilde{\Delta} \pmod{9}, \; \tilde{c}_6 \pmod{3} &  u       &   r                & \mbox{Field} &  \Ebar  \\ \hline
(n \geq 6, 8, 13) &  5, 1   & \pi^{13}  & \pi^{20}(2\pi^5 + 2\pi^3 + 2\pi^2 + 2\pi + 2)  &   h_3        & \Ebar_0 \\
(n \geq 6, 8, 13) &  5, 2   & \pi^{13}  & \pi^{20}(\pi^5 + \pi^3 + \pi^2 + \pi + 1)      &   h_3        & \Ebar_0 \\

(n \geq 6, 8, 13) &  8, 1   & \pi^{13}  & \pi^{20}(\pi^4 + 2\pi^2 + 2)  &   h_4        & \Ebar_2 \\
(n \geq 6, 8, 13) &  8, 2   & \pi^{13}  & \pi^{20}(2\pi^4 + \pi^2 + 1)      &  h_4        & \Ebar_1 \\
(n \geq 6, 8, 13) &  2, 1   & \pi^{13}  & \pi^{20}(\pi^5 + \pi^4 + \pi^3 + 2\pi^2 + \pi + 2)  &   h_5        & \Ebar_2 \\
(n \geq 6, 8, 13) &  2, 2   & \pi^{13}  & \pi^{20}(2\pi^5 + 2\pi^4 + 2\pi^3 + \pi^2 + 2\pi + 1)  &   h_5   & \Ebar_1 \\
\hline
\end{array}
$$
\caption{Description of $F$, $\Ebar$ and $(u,r)$ in case~$G_j$.}
\label{Table:(j12)}
\end{table}
\end{footnotesize}

\FloatBarrier

{\large \part{Proof of the criteria}}
\label{Part:proofsCyclic}

We will now use the previous results 
to prove some of our main criteria, more precisely,  Theorems~\ref{T:mainTame3},~\ref{T:mainWild3},~\ref{T:mainTame4}
and~\ref{T:mainWild4}. Let us first recall some useful facts and notation.

Let $E/\Q_\ell$ an elliptic curve with potentially good reduction. 
Let $F \subset \Qbar_\ell$ be a field where $E/F$ has good reduction. Write $\pi_F$
for an uniformizer in $F$ and $\Frob_F \in G_F = \Gal(\Qbar_\ell / F)$ for 
a Frobenius element. Write $\Ebar$ for the residue elliptic curve 
obtained by reduction mod~$\pi_F$ of a model of $E / F$ with good reduction. 
Let $E_F[p]$ denote $E[p]$ as a $G_F$-module.
We write $\varphi : E_F[p] \rightarrow \overline{E}[p]$ for 
the reduction morphism, which is a symplectic $G_F$-isomorphism. 

Let $\Aut(\Ebar)$ denote the group of $\Fbar_\ell$-automorphisms of $\Ebar$.
Let $L = \Q_\ell^{un}(E[p])$ and $\Phi = \Gal(L/\Q_\ell^{un})$.
The action of $\Phi$ on $L$ induces an injective morphism 
$\gamma_E : \Phi \rightarrow \Aut(\overline{E})$ satisfying, 
for all $\sigma \in \Phi$, the relation
\begin{equation}
\label{E:phi2}
 \varphi \circ \rhobar_{E,p}(\sigma) = \psi(\gamma_E(\sigma)) \circ \varphi,
\end{equation}
where $\psi : \Aut(\Ebar) \rightarrow \GL(\Ebar[p])$ is the natural
injective morphism. 

Let $E/\Q_\ell$ and $E'/\Q_\ell$ denote elliptic curves
with isomorphic $p$-torsion modules. Assume further 
they have potentially good reduction with semistability
defect $e=e' \in \{ 3, 4 \}$. In particular, 
they have the same $p$-torsion field $K$ and we
let the generators $\sigma, \tau \in \Gal(K/\Q_\ell)$ 
and the fields $F$, $K_2 = F \cap K$ 
be given by Proposition~\ref{P:groupStructure}
applied to either $E$ or $E'$.

Write $\bar{\tau}$ the Frobenius element in $\Gal(\Fbar_\ell / \F_\ell)$. 
The element $\tau \in \Gal(K/K_2)$ acts on the residue
field of $K$ as $\bar{\tau}$;
there is a natural identification of $\sigma$ with an element 
$\sigma \in \Phi \simeq \Gal(K_{un}/K)$.

\section{Proof of Theorem~\ref{T:mainTame3}}
\label{S:thm1}

Here we are in the case of tame reduction with $e=e'=3$, hence
$F = \Q_\ell(\ell^{1/3})$ is totally ramified and non-Galois 
since $\ell \equiv 2 \pmod{3}$. We will need the following key lemma.

\begin{lemma} Let $\ell \equiv 2 \pmod{3}$ be a prime.
Suppose $E/\Q_\ell$, $E'/\Q_\ell$ satisfy $e=e'=3$ 
and both have a $3$-torsion point over $\Q_\ell$.
Suppose that $E[p]$ and $E'[p]$ are isomorphic $G_{\Q_{\ell}}$-modules.

Set $s=1$ if $\upsilon_\ell(\Delta_m) \equiv \upsilon_\ell(\Delta_m') \pmod{3}$ 
and $s=-1$ otherwise. 

Then we can choose symplectic bases of $E[p]$ and $E'[p]$ such that the following holds
\[
N = N' \qquad \text{ and } \qquad A^s = A', 
\]
where $A$, $A'$, $N$ and $N'$ are the matrices representing
$\rhobar_{E,p}(\sigma)$, $\rhobar_{E',p}(\sigma)$, $\rhobar_{E,p}(\tau)$ and $\rhobar_{E',p}(\tau)$
in the fixed bases. 
\label{L:goodbasis}
\end{lemma}

\begin{proof} It follows from Lemma~\ref{L:gammaEe3} we 
can choose minimal models for $E$, $E'$ over $F$ reducing to $ \Ebar : Y^2 + Y = X^3$.
Moreover, $\gamma_E(\sigma)^s = \gamma_{E'}(\sigma)$ in $\Aut(\Ebar)$ with 
$s=1$ if $\upsilon_\ell(\Delta_m) \equiv \upsilon_\ell(\Delta_m') \pmod{3}$ and $s=-1$ otherwise.
Thus $\psi(\gamma_E(\sigma))^s = \psi(\gamma_{E'}(\sigma))$. 

Write $\rhobar$ for the representation giving the action of $\Gal(\Fbar_\ell / \F_\ell)$ on $\Ebar[p]$. The reduction morphisms $\varphi : E_F[p] \to \Ebar[p]$ and $\varphi' : E'_F[p] \to \Ebar[p]$  satisfy $\varphi \circ \rhobar_{E,p}(\Frob_F) = \rhobar(\bar{\tau}) \circ \varphi$ and $\varphi' \circ \rhobar_{E',p}(\Frob_F) = \rhobar(\bar{\tau}) \circ \varphi'$. 

From Lemma~\ref{L:FinK} part~(i) we have $F \subset K$, hence we can further assume $\tau$
acts on $K$ as $\Frob_F$ to conclude that $\varphi \circ \rhobar_{E,p}(\tau) = \rhobar(\bar{\tau}) \circ \varphi$ and $\varphi' \circ \rhobar_{E',p}(\tau) = \rhobar(\bar{\tau}) \circ \varphi'$.

Fix a symplectic basis for $\Ebar[p]$. Let $\bar{N}$ be the matrix representing $\rhobar(\bar{\tau})$ in that basis.
Lift the fixed basis to basis of $E[p]$ and $E'[p]$ via the reduction morphisms. The lifted basis are symplectic
and in these basis the matrix representing $\varphi$ and $\varphi'$ are the identity.
Thus $N = N' = \bar{N}$. Finally, it follows from $\psi(\gamma_E(\sigma))^s = \psi(\gamma_{E'}(\sigma))$ and \eqref{E:phi2} that $A' = A^s$ in the same fixed bases, as desired. 
\end{proof}

\noindent{\bf Proof of Theorem~\ref{T:mainTame3}}
Since $\ell \equiv 2 \pmod{3}$, it follows from Corollary~\ref{C:nonabelianTame}
that $\rhobar_{E,p}$ has non-abelian image, hence Corollary~\ref{C:notBoth} implies the last statement. 

Let $\phi : E[p] \to  E'[p]$ be an isomorphism of $G_{\Q_{\ell}}$-modules.

We first make the following simplifications related to the definition of $t$.
The representation $\rhobar_{E,3}$ has image of order 6 and is 
conjugate to exactly one of
\begin{equation} \label{E:matrixgroups}
  \begin{pmatrix} 1 & * \\ 0 & \chi_3 \end{pmatrix} \qquad \text{ or } \qquad \begin{pmatrix} \chi_3 & * \\ 0 & 1 \end{pmatrix}
  \end{equation}
respectively if $E$ has a 3-torsion point over $\Q_\ell$ or not (see also Proposition~\ref{P:mod3}). 
The same holds for $E'$.
If both $E$ and $E'$ are in the second case the quadratic twist by $\chi_3$ will change both to the first case. Taking simultaneous quadratic twists by the same character preserves the hypothesis $E[p] \simeq E'[p]$
and does not affect the symplectic type of~$\phi$. Thus we can assume that either both curves have a 3-torsion point over $\Q_\ell$ 
or only one of them does.

We define $s=1$ if $r = 0$ and $s=-1$ if $r = 1$.

We now divide the proof in two natural cases.

{\sc Case I:} Suppose $t=0$; hence both curves have a $3$-torsion point over $\Q_\ell$. 

Fix symplectic basis for $E[p]$ and $E'[p]$ given by Lemma~\ref{L:goodbasis}.
Let also $A$, $A'$, $N$, $N'$ be as in Lemma~\ref{L:goodbasis}. Write $\Theta$ for the matrix
representing $\phi$ in the fixed bases. From the relations in Proposition~\ref{P:groupStructure} part (i)
and the fact that $\phi$ commutes with the action of $\sigma$ and $\tau$ we can write
\[
 NAN^{-1} = A^{-1}, \qquad A^s = \Theta A \Theta^{-1}, \qquad N = \Theta N \Theta^{-1}, \qquad A^3 = 1. 
\]
We have $\det A = 1$ (inertia lands in $\SL_2(\F_p)$) and $\det N = \chi_p (\Frob_\ell) \equiv \ell \pmod{p}$. 

Suppose $p \geq 5$ and write $G = \GL_2(\F_p)$. It follows from Proposition~\ref{P:centralizer} that
the centralizer $C_G(A)$ is a Cartan subgroup $C$. 
We have that $N \not\in C$ but $N$ is in the normalizer $N_G(C)$. From Proposition~\ref{P:groupStructure}
we know that $\Gal(K/\Q_\ell)$ has order coprime to $p$, hence $N$ has order coprime
to $p$ and the centralizer $C_G(N)$, again by Proposition~\ref{P:groupStructure}, 
is a Cartan subgroup $C'$ different from $C$.

Set $V = \Theta$ if $s=1$ or $V = N^{-1} \Theta$ if $s=-1$. It is easy to check that $V \in C$ and also $V \in C'$.
Since $C \ne C'$ the intersection $C \cap C'$ are scalar matrices. Thus $\det V$ is a square mod~$p$.

Suppose $s=1$. Then $\det V = \det \Theta$ is a square mod~$p$, hence $E[p]$ and $E'[p]$ are symplectically isomorphic. 

Suppose $s=-1$. Then $\det V = \det N^{-1} \det \Theta$ is a square mod~$p$. Therefore $\det \Theta$ is a square if and only if $\det N$ is a square, that is if and only if $(\ell/p) =1$. 

So far we have proved the following statements for  all $p \geq 5$:
\begin{enumerate}
 \item if $t=0$ and $s=1$ then $E[p]$ and $E'[p]$ are symplectically isomorphic;
 \item if $t=0$ and $s=-1$ then $E[p]$ and $E'[p]$ are symplectically isomorphic if and only if $(\ell/p) = 1$.
\end{enumerate}
Suppose now $p=3$ and write $G = \GL_2(\F_3)$. Then $C = C_G(A) = \langle -A \rangle \subset \SL_2(\F_3)$.
Define $V$ as before, hence $V \in C$. If $s=1$ then $\det V = \det \Theta = 1$ is a square. If $s=-1$ then $1 = \det V = \det N^{-1} \det \Theta$, hence $\det \Theta = \det N = (\ell/3)$. 
This proves (1) and (2) also for $p=3$. 

Now note that $(\ell / 3) = -1$ since $\ell \equiv 2 \pmod{3}$, therefore $(\ell / 3)^r = (-1)^r = 1$
if and only if $r=0$, proving (ii) when $t=0$.

{\sc Case II:} Suppose $t=1$. Since the matrix groups in~\eqref{E:matrixgroups} are not conjugate the $3$-torsion modules are not isomorphic in this case. This shows that $p=3$ and $t=1$ cannot occur, completing the proof of (ii). We thus have $p \geq 5$ from now on. 

We can assume $E'$ does not have a 3-torsion point over $\Q_\ell$.
Since $\rhobar_{E',3}$ is reducible there is a 3-isogeny $h : E' \to E''$, inducing an isomorphism
$\phi_h : E'[p] \to E''[p]$. Moreover, the composition $\phi_h \circ \phi : E[p] \to E''[p]$ is a $G_{\Q_\ell}$-modules isomorphism which is symplectic if and only if $(3/p)=1$.

The curve $E''$ has a 3-torsion point over $\Q_\ell$ and $\upsilon(\Delta_m') = \upsilon(\Delta_m'')$ (see \cite[Table~1]{DD2015}).  We can now apply case (1) or (2) with $E$ and $E''$ to decide the symplectic type of $\phi_h \circ \phi$. 
Finally, note that if $\phi$ is symplectic the composition $\phi_h \circ \phi$ is symplectic if and only if
$\phi_h$ is symplectic; if $\phi$ is anti-symplectic the composition $\phi_h \circ \phi$ is symplectic if and only if $\phi_h$ is anti-symplectic. 

By now we have also established the following statements for all $p \geq 5$:
\begin{enumerate}
 \item[(3)] if $t=1$ and $s=1$ then $E[p]$ and $E'[p]$ are symplectically isomorphic if and only if $(3/p) = 1$; 
 \item[(4)] if $t=1$ and $s=-1$ then $E[p]$ and $E'[p]$ are symplectically isomorphic if and only if $(3/p)(\ell/p) = 1$; 
\end{enumerate}

Recall that we have defined $s=1$ if $r = 0$ and $s=-1$ if $r = 1$. To finish the proof we note
that the statements $(1)$, $(2)$, $(3)$ and $(4)$ can be summarized has
\[
 E[p] \text{ and } E'[p] \quad \text{are symplectically isomorphic} \quad \Leftrightarrow \quad \left(\frac{\ell}{p}\right)^r \left (\frac{3}{p} \right)^t = 1,
\]
as desired.

\section{Proof of Theorem~\ref{T:e=p=3}}
\label{S:thm3}

Let $F=\Q_\ell(\ell^{1/3})$. Since $\ell \equiv 1 \pmod{3}$, 
the field $F$ is a tame totally ramified
cyclic extension of $\Q_\ell$ of degree 3.
From \cite{Kraus1990} it follows that the Kodaira type of $E$ is IV or IV*
and a direct application of part (3) of \cite[Theorem~3]{DDKod} 
implies that $E/F$ has good reduction.
Now, Corollary~\ref{C:nonabelianTame} implies that the
$3$-torsion field extension of $E$ is abelian, so the image 
$\rhobar_{E,3}(G_{\Q_\ell}) \subset \GL_2(\F_3)$ 
is abelian of order multiple of $e=3$. 
There are only two abelian subgroups of $\GL_2(\F_3)$ with order 
multiple of $3$. These are, up to conjugation, the order 3 subgroup 
generated by $\left( \begin{smallmatrix} 1 & 1 \\ 0 & 1 \end{smallmatrix} \right)$
and the order 6 subgroup generated by
$\left( \begin{smallmatrix} 2 & 1 \\ 0 & 2 \end{smallmatrix} \right)$.

From Lemma~\ref{L:CentralizerAbelian}, we know that all 
the matrices in the centralizer of $\rhobar_{E,3}(G_{\Q_\ell})$ 
have square determinant. The last statement now follows from Lemma~\ref{L:sympcriteria}
since $\rhobar_{E,3} \simeq \rhobar_{E',3}$. Moreover, all the above also holds for $E'$.
In particular, both $E$ and $E'$ have good reduction over~$L = \Q_\ell^{un}F$, the unique 
cubic tame extension of $\Q_\ell^{un}$. Let $\sigma \in \Gal(L/\Q_\ell^{un})$ be as
defined in Section~\ref{S:gammaEtame3}.

Let $s=1$ if $\vv_\ell(\Delta_m) \equiv \vv_\ell(\Delta_m') \pmod{3}$ and $s=-1$ otherwise.
Write $A = \left( \begin{smallmatrix} 1 & 1 \\ 0 & 1 \end{smallmatrix} \right)$.
We claim that, after replacing $\sigma$ by $\sigma^2$ if necessary, 
we can choose symplectic bases of $E[3]$ and $E'[3]$ such that 
$\rhobar_{E,3}(\sigma) = A$ and $\rhobar_{E',3}(\sigma) = A^s$.

In such bases, we have $\rhobar_{E,3}(I_\ell) = \rhobar_{E',3}(I_\ell) = \langle A \rangle$, hence $\rhobar_{E,3} = M \rhobar_{E',3} M^{-1}$ for some $M$ in the normalizer $N_{\GL_2(\F_3)}(\langle A \rangle)$. From Lemma~\ref{L:sympcriteria}, $E[3]$ and $E'[3]$ are symplectically isomorphic if and only if $\det M = 1$. The elements in $N_{\GL_2(\F_3)}(\langle A \rangle)$ of determinant 1 are precisely those commuting with $A$, then $\det M = 1$ if and only if $s=1$. The result follows.

We will now prove the claim. From the discussion in the first paragraph, we know that
$\rhobar_{E,3}$ is reducible and, taking an unramified (since $e=3$) quadratic twist if necessary, 
we can assume $\rhobar_{E,3} \sim \left( \begin{smallmatrix} 1 & * \\ 0 & 1 \end{smallmatrix} \right)$
and $\rhobar_{E,3}(G_{\Q_\ell}) = \rhobar_{E,3}(G_{I_\ell})$. Clearly, $\rhobar_{E,3} \simeq \rhobar_{E',3}$ implies the same is true for $E'$. 

We conclude that both $E$ and $E'$ have a 3-torsion point defined over~$\Q_\ell$. 

Now, arguing as in the proof of Lemma~\ref{L:goodbasis} (ignoring the arguments regarding $\tau$), 
where we replace $F$ by $L$, Lemma~\ref{L:gammaEe3} by Lemma~\ref{L:gammaEe3II} 
and $\sigma$ by $\sigma^2$ if necessary, the claim follows.

\section{Proof of Theorem~\ref{T:mainWild3}}
\label{S:thm4}

The hypothesis $\tilde{\Delta} \equiv 2 \pmod{3}$ and Proposition~\ref{P:nonabelianWild3}
means that $K/\Q_3$ is non-abelian. Since $E$ and $E'$ have the same $p$-torsion field 
from the same proposition we conclude $\tilde{\Delta}^\prime \equiv 2 \pmod{3}$; moreover,
from Corollary~\ref{C:notBoth} the last statement follows.

Set $s=1$ if $r = 0$ and $s=-1$ if $r = 1$.
From Lemma~\ref{L:gammEwild3}, possibly after replacing $\sigma$ by $\sigma^2$, it follows 
that $\gamma_E(\sigma) = \gamma_{E'}(\sigma)^s$. 
Now, arguing as in the proof of Lemma~\ref{L:goodbasis} 
but replacing Lemma~\ref{L:gammaEe3} by Lemma~\ref{L:gammEwild3}
(and noting that the hypothesis on the existence of a $3$-torsion 
point is not needed for the latter) we can choose symplectic 
bases of $E[p]$ and $E'[p]$ such that 
\[
N = N' \qquad \text{ and } \qquad A^s = A', 
\]
where $A$, $A'$, $N$ and $N'$ are the matrices representing
$\rhobar_{E,p}(\sigma)$, $\rhobar_{E',p}(\sigma)$, 
$\rhobar_{E,p}(\tau)$ and $\rhobar_{E',p}(\tau)$.

Finally, arguing as in the case $t=0$ and $p\geq 5$
of the proof of Theorem~\ref{T:mainTame3} 
we conclude that 
\[
 E[p] \text{ and } E'[p] \quad \text{are symplectically isomorphic} 
 \quad \Leftrightarrow \quad \left(\frac{\ell}{p}\right)^r = 1
\]
and the result follows since $\ell = 3$.

\section{Proof of Theorem~\ref{T:mainTame4}}
\label{S:thm5}

Here we are in the case of tame reduction with $e=e'=4$, hence
$F = \Q_\ell(\ell^{1/4})= \Q_\ell(\pi)$ is totally ramified and non-Galois 
since $\ell \equiv 3 \pmod{4}$. Again we will need the following key lemma.

\begin{lemma} Let $\ell \equiv 3 \pmod{4}$ be a prime. 
Suppose $E/\Q_\ell$, $E'/\Q_\ell$ satisfy $e=e'=4$ 
and both have full $2$-torsion over $F$. 
Suppose that $E[p]$ and $E'[p]$ are isomorphic $G_{\Q_{\ell}}$-modules.

Set $s=1$ if $\upsilon_\ell(\Delta_m) \equiv \upsilon_\ell(\Delta_m') \pmod{4}$ and $s=-1$ otherwise. 

Then we can choose symplectic bases of $E[p]$ and $E'[p]$ such that the following holds
\[
N = N' \qquad \text{ and } \qquad A^s = A', 
\]
where $A$, $A'$, $N$ and $N'$ are the matrices representing
$\rhobar_{E,p}(\sigma)$, $\rhobar_{E',p}(\sigma)$, $\rhobar_{E,p}(\tau)$ and $\rhobar_{E',p}(\tau)$
in the fixed bases.
\label{L:goodbasisTame4}
\end{lemma}

\begin{proof} 
The element $\sigma$ generates $\Phi = \Gal(L/\Q_\ell^{un})$. From  
Lemma~\ref{L:modelEtame4} part (A) applied to $E$ and $E'$ we conclude that
there are models for $E$ and $E'$ over $F$ (hence also over $L$) 
with good reduction whose reduction
mod $(\pi)$ gives the curve $\Ebar : Y^2 = X^3 - X$. Furthermore,
\[
\gamma_E(\sigma)^s = \gamma_{E'}(\sigma) \quad \text{ in } \quad \Aut(\Ebar), 
\]
with $s=1$ if $\upsilon_\ell(\Delta_m) \equiv \upsilon_\ell(\Delta_m') \pmod{4}$ and $s=-1$ otherwise.
Thus $\psi(\gamma_E(\sigma))^s = \psi(\gamma_{E'}(\sigma))$.

Note in this setting we can have (i) $F \cap K = F$ or (ii) $F \cap K = K_2$. 

In case (i) the result follows exactly as in the proof of Lemma~\ref{L:goodbasis}.

Suppose we are in case (ii). Then the action of $\Frob_F$ 
on $K$ differs from that of $\tau$. Indeed, the element 
$\tilde{\tau} = \sigma^2 \cdot \Frob_F \in \Gal(\Qbar_\ell/ K \cap F)$ 
acts (modulo $\ker \rhobar_{E,p} =\ker \rhobar_{E',p}$) as $\tau$ on $K$. 

Write $\rhobar$ for the representation giving the action of $\Gal(\Fbar_\ell / \F_\ell)$ on $\Ebar[p]$. The reduction morphisms $\varphi : E_F[p] \to \Ebar[p]$ and $\varphi' : E'_F[p] \to \Ebar[p]$ satisfy $\varphi \circ \rhobar_{E/F,p}(\Frob_F) = \rhobar(\bar{\tau}) \circ \varphi$ and 
$\varphi' \circ \rhobar_{E'/F,p}(\Frob_F) = \rhobar(\bar{\tau}) \circ \varphi'$. 

Fix a symplectic basis for $\Ebar[p]$. Let $\bar{N}$ be the matrix representing $\rhobar(\bar{\tau})$ in that basis.
Lift the fixed basis to bases of $E_F[p]$ and $E'_F[p]$  via the reduction morphisms. 
The lifted basis are symplectic
and in these basis the matrix representing $\varphi$ and $\varphi'$ are the identity.
Thus $\rhobar_{E/F,p}(\Frob_F) = \rhobar_{E'/F,p}(\Frob_F) = \bar{N}$.
Since $E$ and $E'$ are originally defined over $\Q_\ell$ with a linear coordinate 
change (defined over $F$) we transform the fixed bases of $E_F[p]$ and $E'_F[p]$ 
into symplectic bases of $E[p]$ and $E'[p]$ (satisfying the models over $\Q_\ell$).
With respect to these bases, the matrix representations satisfy
\[ \rhobar_{E,p}(\Frob_F) = \rhobar_{E/F,p}(\Frob_F) = \bar{N}
 = \rhobar_{E'/F,p}(\Frob_F)=\rhobar_{E',p}(\Frob_F).
\]
Also, it follows from $\psi(\gamma_E(\sigma))^s = \psi(\gamma_{E'}(\sigma))$ and \eqref{E:phi2}
that $A' = A^s$ in the same bases. 

Finally, observe that $A^2 = -I$ and compute
\[ N = \rhobar_{E,p}(\tau) = \rhobar_{E,p}(\tilde{\tau}) = A^2 \bar{N} = - \bar{N} \]
and
\[ N' = \rhobar_{E',p}(\tau) = \rhobar_{E',p}(\tilde{\tau}) 
= A^{2s} \bar{N} = (-I)^s \bar{N} = -\bar{N} \]
to conclude (since $s= \pm 1$) that $N = N'$, as desired.
\end{proof}

\noindent{\bf Proof of Theorem~\ref{T:mainTame4}}
From Lemma~\ref{L:full2torsion} we see that the condition of exactly one of $\tilde{\Delta}$, $\tilde{\Delta}'$ being a square mod~$\ell$ is equivalent to exactly one of $E$, $E'$ having full 2-torsion over $F$. 

Let $\phi : E[p] \to E'[p]$ be a $G_{\Q_\ell}$-modules isomorphism.

Suppose that $t=0$ because both $E$ and $E'$ do not have full 2-torsion over $F$ then by
part~(B) of Lemma~\ref{L:modelEtame4} we can interchange $E$ and $E'$ respectively
by $2$-isogenous curves $W$ and $W'$. The curves $W$ and $W'$ satisfy the following. 
\begin{itemize} 
 \item both $W$ and $W'$ have full 2-torsion over $F$;
 \item $\upsilon_\ell(\Delta_m(E)) =  \upsilon_\ell(\Delta_m(W))$ 
  and $\upsilon_\ell(\Delta_m(E')) =  \upsilon_\ell(\Delta_m(W'))$;
 \item there is a $G_{\Q_\ell}$-modules isomorphism $\phi_W : W[p] \to W'[p]$ which
 is symplectic if and only if $\phi$ is symplectic.
\end{itemize}
Therefore, when $t=0$ we can assume that both $E$ and $E'$ have full 2-torsion over $F$.

We define $s=1$ if $r = 0$ and $s=-1$ if $r = 1$.

We now divide the proof in two natural cases.

{\sc Case I:} Suppose $t=0$. Proceeding exactly as in the case $t=0$, $p \geq 5$ of the proof of Theorem~\ref{T:mainTame3}, where we replace Lemma~\ref{L:goodbasis} by Lemma~\ref{L:goodbasisTame4}, 
it follows that:
\begin{enumerate}
 \item if $t=0$ and $s=1$ then $E[p]$ and $E'[p]$ are symplectically isomorphic;
 \item if $t=0$ and $s=-1$ then $E[p]$ and $E'[p]$ are symplectically isomorphic if and only if $(\ell/p) = 1$.
\end{enumerate}

{\sc Case II:} Suppose $t=1$. 

We can assume that $E'$ does not have full $2$-torsion over $F$.
From part~(B) of Lemma~\ref{L:modelEtame4}
there is a 2-isogeny $h : E' \to E''$ where $E''/\Q_\ell$ 
satisfy 
\begin{itemize}
 \item $E''$ has full 2-torsion over $F$;
 \item $\upsilon_\ell(\Delta_m(E')) =  \upsilon_\ell(\Delta_m(E''))$;
 \item there is a $G_{\Q_\ell}$-modules isomorphism $\phi_h : E'[p] \to E''[p]$ which 
 is symplectic if and only if $(2/p)=1$.
\end{itemize}
The composition $\phi_h \circ \phi : E[p] \to E''[p]$ is a $G_{\Q_\ell}$-modules isomorphism
and we can now apply case (1) or (2) with $E$ and $E''$ to decide the 
symplectic type of $\phi_h \circ \phi$. Finally, note that when $\phi$ is symplectic the composition $\phi_h \circ \phi$ is symplectic if and only if $\phi_h$ is symplectic, that is $(2/p)=1$; when $\phi$ is anti-symplectic the composition $\phi_h \circ \phi$ is symplectic if and only if $\phi_h$ is anti-symplectic, i.e. $(2/p)=-1$.

We have now also established the following statements for all $p \geq 5$:
\begin{enumerate}
 \item[(3)] if $t=1$ and $s=1$ then $E[p]$ and $E'[p]$ are symplectically isomorphic if and only if $(2/p) = 1$; 
 \item[(4)] if $t=1$ and $s=-1$ then $E[p]$ and $E'[p]$ are symplectically isomorphic if and only if $(2/p)(\ell/p) = 1$; 
\end{enumerate}
Recall that we have defined $s=1$ if $r = 0$ and $s=-1$ if $r = 1$. To finish the proof we note
that the statements $(1)$, $(2)$, $(3)$ and $(4)$ can be summarized has
\[
 E[p] \text{ and } E'[p] \quad \text{are symplectically isomorphic} \quad \Leftrightarrow \quad \left(\frac{\ell}{p}\right)^r \left (\frac{2}{p} \right)^t = 1
\]
as desired.

\section{Proof of Theorem~\ref{T:mainWild4}}
\label{S:thm6}

The hypothesis $\tilde{c}_4 \equiv 5\Delta_m \pmod{8}$ and Proposition~\ref{P:nonabelianWild4}
mean that $K/\Q_2$ is non-abelian and since $E$ and $E'$ have the same $p$-torsion field 
the same proposition implies $\tilde{c}_4' \equiv 5\Delta_m' \pmod{8}$; 
moreover, from Corollary~\ref{C:notBoth} the last statement follows.

Both $E$ and $E'$ have good reduction 
over $L = \Q_2^{un} K = \Q_2^{un} F$, where $F$ is defined by one of 
the polynomials in Theorem~\ref{T:goodOverF} part (2). Taking 
quadratic twists does not affect the existence and the symplectic type 
of a $p$-torsion isomorphism so, 
if necessary, we can (still Theorem~\ref{T:goodOverF})  twist both curves by -1
and assume that $F = F_1$ is defined by the polynomial~$f_1$.
Now, replacing $\sigma$ by $\sigma^3$ if necessary, we can also assume that
$\sigma \in \Gal(K/K_{un})$ lifts to the generator $\sigma \in \Gal(L/\Q_2^{un})$ 
given by Lemma~\ref{L:gammEwild4}.

Set $s=1$ if $r = 0$ and $s=-1$ if $r = 1$.
Observe that $\tilde{c}_6 \equiv \pm 1 \pmod{4}$ and the value of
$\alpha_1 = \alpha_1(E)$ in \eqref{E:c62adic} is $\alpha_1 = 0$ if $\tilde{c}_6 \equiv 1 \pmod{4}$ and 
$\alpha_1 = 1$ if $\tilde{c}_6 \equiv -1 \pmod{4}$, respectively. The same relation is true 
between $\alpha_1(E')$ and $\tilde{c}_6'$. 
Thus, from Lemma~\ref{L:gammEwild4} applied to both $E$ and $E'$
we conclude that in $\Aut(\Ebar)$ we have
$\gamma_E(\sigma)^s = \gamma_{E'}(\sigma)$.

Now it follows as in the proof of Lemma~\ref{L:goodbasisTame4} 
(noting that the hypothesis of $E$ and $E'$ have full 2-torsion over $F$ is not
used here) that we can choose symplectic basis of $E[p]$ and $E'[p]$ such that
\[
N = N' \qquad \text{ and } \qquad A^s = A', 
\]
where $A$, $A'$, $N$ and $N'$ are the matrices representing
$\rhobar_{E,p}(\sigma)$, $\rhobar_{E',p}(\sigma)$, $\rhobar_{E,p}(\tau)$ and $\rhobar_{E',p}(\tau)$.

Moreover, the same arguments as in the case $t=0$, $p\geq 5$ of the proof of Theorem~\ref{T:mainTame4},
with the extra condition $\ell = 2$, lead to the conclusion
\[
 E[p] \text{ and } E'[p] \quad \text{are symplectically isomorphic} 
 \quad \Leftrightarrow \quad \left(\frac{2}{p}\right)^r = 1.
\]
Since the order of $A$ is 4, hence not divisible by 3, 
the same arguments and the above formula also 
hold for $p=3$, concluding the proof.

\section{Proof of Theorem~\ref{T:mixedReduction}}
\label{S:mixedReduction}

Let $L = \Q_\ell^{un}(E[3]) = \Q_\ell^{un}(E'[3])$
be the fixed field of
$\rhobar_{E,3}|_{I_\ell} \simeq \rhobar_{E',3}|_{I_\ell}$. 

Since $\ell \neq 3$ and $e'=3$ it follows that $L$ is the unique cubic tame extension of~$\Q_\ell^{un}$. 
Moreover, there is only one conjugacy class of subgroups of order 3 in $\GL_2(\F_3)$ therefore,
up to conjugation, we have $\rhobar_{E,3}(I_\ell) = \rhobar_{E',3}(I_\ell) = \langle A \rangle$, 
where $A = \left( \begin{smallmatrix} 1 & 1 \\ 0 & 1 \end{smallmatrix} \right)$.
From Lemma~\ref{L:CentralizerAbelian}, we know that all the matrices in the centralizer 
of $\rhobar_{E,3}(I_\ell)$ 
have square determinant and the last statement follows from Lemma~\ref{L:sympcriteria}.

We can write $L = \Q_\ell^{un}(\pi)$ where $\pi^3 = \ell$. Let $\sigma \in \Phi = \Gal(L/\Q_\ell^{un})$ be given by $\sigma(\pi) = \zeta_3 \pi$, which is the same element as defined in the beginning of Section~\ref{S:gammaEtame3}.

Let $j_E$ denote the $j$-invariant of~$E$.
We can write $j_E = \mu \ell^k c^3$, where 
$c \in \Q_\ell$, $k \in \{0, 1, 2 \}$
is given by $k \equiv \vv_\ell(\Delta_m(E)) \pmod{3}$
and $\mu$ is a unit in~$\Z_\ell$.
Further, over $\Q_\ell^{un}$ every unit is a cube (by Hensel's lemma since $\ell \neq 3$) and we have $j_E = \pi^{3k} c_0^3$, where $c_0 \in \Q_\ell^{un}$. 

We fix the cube root of~$j_E$ given by $j_E^{1/3} := \pi^{k} c_0$ and compute 
\[
\sigma(j_E^{1/3}) 
= \sigma(\pi)^k c_0 = \zeta_3^k \pi^k c_0 = \zeta_3^k \cdot j_E^{1/3}, 
\]
therefore, from Proposition~\ref{P:theoryTate} it follows that, in a symplectic basis, we have
\begin{equation}\label{E:k}
\rhobar_{E,3}(\sigma) = \begin{pmatrix} 1 & k \\ 0 & 1
\end{pmatrix} \quad \text{ with } \quad k \neq 0,
\end{equation}
where $k \neq 0$ is due to
$\rhobar_{E,3}(I_\ell) = \langle A \rangle$; 
the same proposition
implies $\vv_\ell(\Delta_m) \not\equiv 0 \pmod{3}$, proving the first statement.
Still from Proposition~\ref{P:theoryTate} we know that
$\rhobar_{E',3} \simeq \rhobar_{E,3}$ is isomorphic to either 
\[
  \begin{pmatrix} \chi_3 & * \\ 0 & 1
\end{pmatrix} \quad \text{ or } \quad 
\begin{pmatrix} 1 & * \\ 0 & \chi_3
\end{pmatrix},
\]
depending on wether $E$ has split or non-split multiplicative reduction. Thus, after at most an unramified twist by $\chi_3$, we can suppose that 
we are in the latter case, hence both curves have 
a $3$-torsion point over $\Q_\ell$. Note that the reasoning leading to~\eqref{E:k} is not affected by the (unramified) twist because it only concerns the action of inertia.

Since $E'$ has a 3-torsion and~$\sigma$ is defined as in Section~\ref{S:gammaEtame3},
from Lemma~\ref{L:gammaEe3II} we know there is a model 
for $E'/L$ with good reduction and residual curve~$\Ebar : Y^2 + Y = X^3$ over $\F_\ell$.
Moreover, $\gamma_{E'}(\sigma) : (X,Y) \mapsto (\omega_3^{2\alpha}X,Y)$ where $\alpha = \vv_\ell(\Delta_m')/4 \in \{1,2\}$ and $\omega_3 \in \Fbar_\ell$ is the cube root of unity obtained by reducing~$\zeta_3$. Note that $\alpha \equiv \vv_\ell(\Delta_m') \pmod{3}$. 

We claim that~\eqref{E:phi} implies that there exists a symplectic basis for $E'[3]$ such that
\begin{equation}\label{E:alpha}
 \rhobar_{E',3}(\sigma) = \begin{pmatrix} 1 & \alpha \\ 0 & 1 \end{pmatrix}.
\end{equation}
By assumption, we have
$\rhobar_{E,3} = M \rhobar_{E',3} M^{-1}$ for some $M$ in the normalizer $N_{\GL_2(\F_3)}(\langle A \rangle)$. From Lemma~\ref{L:sympcriteria}, $E[3]$ and $E'[3]$ are symplectically isomorphic if and only if $\det M = 1$. The elements in $N_{\GL_2(\F_3)}(\langle A \rangle)$ of determinant 1 are precisely those commuting with $A$. Thus, from equations \eqref{E:k} and~\eqref{E:alpha}, we have
$\det M = 1$ if and only if 
\[ \vv_\ell(\Delta_m) \equiv k = \alpha \equiv \vv_\ell(\Delta_m') \pmod{3},\] 
as desired.  We shall now prove the claim to complete the proof.

Recall that the reduction map~$\varphi' : E_L'[3] \to \Ebar[3]$ is a symplectic isomorphism. 
Fix a symplectic basis for $\Ebar[3]$ and lift it to a symplectic basis of $E'[3]$ using $\varphi'$, thus 
the matrix representing $\varphi'$ in these basis
is the identity. Furthermore, in the same basis, from~\eqref{E:phi} we have
\begin{equation}
\rhobar_{E',3}(\sigma) = \psi(\gamma_{E'}(\sigma)) 
\in \GL_2(\F_3) \simeq \GL(\Ebar[3]) \simeq \GL(E'[3]),
\end{equation}
hence we need to show that the action 
of $\gamma_{E'}(\sigma) \in \Aut \Ebar$ on 
$\Ebar[3]$ is given by the matrix in~\eqref{E:alpha}.
This follows from explicit calculations with the $3$-torsion points of $\Ebar : Y^2 + Y = X^3$.

In fact we consider the curve $W : y^2 + y = x^3$ 
of conductor 27
over the quadratic field~$\Q(\zeta_3)$. 
The curve $W$ has all the 3-torsion points defined over~$\Q(\zeta_3)$. We let
\[
 P_1 = (0,0), \quad P_2 = (-1,\zeta_3^2), \quad 
 Q = (-\zeta_3^{2\alpha},\zeta_3^2).
\]
Direct calculation shows that $P_1$ and $P_2$ form a basis of the $3$-torsion and
\begin{equation}
 \label{E:Q}
Q = \alpha \cdot P_1 + P_2.
\end{equation}
Since $\Delta(W) = -27$, the model is minimal at 
all $\ell \neq 3$ and we write 
$\varphi_\ell : W[3] \to \Ebar[3]$ 
for the reduction map. Reducing the points above 
we obtain 
\[
 \varphi_\ell(P_1) = (0,0), \quad \varphi_\ell(P_2) = (-1,\omega_3^2), \quad 
 \varphi_\ell(Q) = (-\omega_3^{2\alpha},\omega_3^2)
\]
and we easily check that 
\[ 
\gamma_{E'}(\sigma)(\varphi_\ell(P_1)) = \varphi_\ell(P_1)
\quad \text{ and } \quad 
\gamma_{E'}(\sigma)(\varphi_\ell(P_2)) = \varphi_\ell(Q).
\]
Thus, from~\eqref{E:Q} we conclude
\[ \gamma_{E'}(\sigma)(\varphi_\ell(P_2)) = \alpha \cdot \varphi_\ell(P_1) + \varphi_\ell(P_2),\]
showing that $\gamma_{E'}(\sigma)$ acts on $\Ebar[3]$ via the matrix~\eqref{E:alpha} on the 
$\F_\ell$-basis $\varphi_\ell(P_1)$, $\varphi_\ell(P_2)$.

Finally, we observe that $\varphi_\ell(P_1)$, $\varphi_\ell(P_2)$ is a symplectic basis if and only if $P_1$ and $P_2$ is a symplectic basis, therefore it is enough to check that $\varphi_\ell(P_1)$, $\varphi_\ell(P_2)$ is a symplectic basis for one value of $\ell$. We do this by taking $\ell =2$ and using the {\tt Magma} command {\tt WeilPairing}, completing the proof of the claim.

\section{Proof of Theorems~\ref{T:mainWilde8II} and~\ref{T:mainWilde12II}}
\label{S:pfThm9}

We will only prove here Theorem~\ref{T:mainWilde8II} because the proof of 
Theorem~\ref{T:mainWilde12II} follows by the analogous arguments 
over $\Q_3$ with Theorem~\ref{T:coordchanges12} replacing Theorem~\ref{T:coordchanges8}.

Let $E$, $E'$ and $p$ be as in the statement of Theorem~\ref{T:mainWilde8II}. 
Since $E[p]$ and $E'[p]$ are isomorphic $I_2$-modules
and $e=8$ we have that $E$ and $E'$ have the same conductor
equal to $2^k$ with $k=5,6$ or $8$. Moreover, we  see from~\cite[Table~1]{DFV} that the corresponding inertial types for $k=6$ are twists by $d=2$ 
(i.e. twist by $\varepsilon_8$ in the notation of {\it loc. cit}) of those for $k=5$. Therefore, twisting both curves by $d=2$ reduces the case of conductor~$2^6$ to the case of conductor~$2^5$,
proving the first statement (without using $(2/p)=-1$).

To prove (A) and (B) we will
first show that we only have to 
compare finitely many pairs of curves $(E,E')$ and
then, with the help of a computer, effectively comparing a set of representatives
for those pairs. 

For $i=1,2,3,4$, we write $F_i$ for
the field defined by the polynomial $g_i$ in cases (3) and (4) 
of Theorem~\ref{T:goodOverF}. 
From tables~\ref{Table:(a)},~\ref{Table:(b)} and ~\ref{Table:(c)} we see that,
in the cases $D_a$, $D_b$, $D_c$ of Table~\ref{Table:model},
to describe $F_i$, $\Ebar$ and the change of coordinates
it is enough to know the triple of 
valuations $(\vv(c_4), \vv(c_6), \vv(\Delta_m))$ and 
$\tilde{c}_4, \tilde{c}_6 \pmod{16}$. 
For case $D_d$ we are
required to know $\tilde{\Delta} \pmod{4}$, so it is convenient 
to translate the 
conditions $\tilde{\Delta} \equiv 1,3 \pmod{4}$ 
into congruence conditions on 
$\tilde{c}_4$ and $\tilde{c}_6$. 
Indeed, let $E$ be an elliptic curve 
satisfying case~$D_d$ in Table~\ref{Table:model},
we have
\[
 c_4^2 - c_6^3 = 12^3 \Delta_m \quad \Leftrightarrow \quad \tilde{c}_4^2 - \tilde{c}_6^3 = 2^3 3^3 \tilde{\Delta},
\]
hence $\tilde{c}_4^2 - \tilde{c}_6^3 \equiv 8 \pmod{16}$. Running through all the values for 
$\tilde{c}_4, \tilde{c}_6 \pmod{2^5}$ such that
\[
 \tilde{c}_4 \equiv \tilde{c}_6 \equiv 1\pmod{2}, \quad \tilde{c}_4^2 - \tilde{c}_6^3 \equiv 8 \pmod{16}, \quad
 \frac{\tilde{c}_4^2 - \tilde{c}_6^3}{8} \equiv 3^3 \tilde{\Delta} \pmod{4}
\]
we find that 
\[
\tilde{\Delta} \equiv 1 \pmod{4} \quad \Leftrightarrow \quad \begin{cases}
  \tilde{c}_4 \equiv 1 \pmod{32}, \; \tilde{c}_6 \equiv 3,13,19,29 \pmod{32}, \\
  \tilde{c}_4 \equiv 9 \pmod{32}, \; \tilde{c}_6 \equiv 1,15,17,31 \pmod{32},  \\
  \tilde{c}_4 \equiv 17 \pmod{32}, \; \tilde{c}_6 \equiv 5,11,21,27 \pmod{32}, \\
  \tilde{c}_4 \equiv 25 \pmod{32}, \; \tilde{c}_6 \equiv 7,9,23,25 \pmod{32}, \
  \end{cases}
\]
and
\[
 \tilde{\Delta} \equiv 3 \pmod{4} \quad  \Leftrightarrow \quad \begin{cases}
  \tilde{c}_4 \equiv 1 \pmod{32}, \; \tilde{c}_6 \equiv 5,11,21,27 \pmod{32}, \\
  \tilde{c}_4 \equiv 9 \pmod{32}, \; \tilde{c}_6 \equiv 7,9,23,25 \pmod{32}, \\ 
  \tilde{c}_4 \equiv 17 \pmod{32}, \; \tilde{c}_6 \equiv 3,13,19,29 \pmod{32}, \\
  \tilde{c}_4 \equiv 25 \pmod{32}, \; \tilde{c}_6 \equiv 1,15,17,31 \pmod{32}.  \
  \end{cases}
\]

Let now $E$ and $E'$ be as in part (A) of the theorem. 
In particular, the inertial field of $E$ and~$E'$
is $L = \Q_2^{un} F$, where $F = F_1$ or $F = F_2$
is determined by Table~\ref{Table:FoverQ2II}.

From Theorem~\ref{T:mainWilde8} and the paragraph after it, we know that, for all $p$ such that $(2/p)=-1$, we have $E[p]$ and $E'[p]$ symplectically isomorphic $I_2$-modules
if and only if the same holds for $E[3]$ and $E'[3]$.
To compute the image of $\gamma_E : \Gal(L/\Q_2^{un}) \to \Aut(\Ebar) \hookrightarrow \GL(\Ebar[3])$ 
it is enough to know (see Section~\ref{S:mapgammaE}) the Galois action on $L/\Q_2^{un}$ and a change of coordinates 
leading to a minimal model of $E/L$ with good reduction and residual curve $\Ebar$; 
this is provided by Theorem~\ref{T:coordchanges8}.
Moreover, this already determines $\rhobar_{E,3}|_{I_2}$ since the action of inertia
is given via $\gamma_E$; the same holds for $E'$. 
Furthermore, we also know that $E[3]$ and $E'[3]$ cannot be both 
symplectically and anti-symplectically isomorphic, so it is 
enough to compare them for one choice of $I_2$-modules isomorphism $\phi$.

From the discussion above it follows that, to determine the symplectic type of $\phi$, for a
fixed choice of the valuations of $(c_4,c_6,\Delta_m)$ and $(c_4',c_6',\Delta_m')$
(note that $E'$ does not have to be in the same case 
of Table~\ref{Table:model} as $E$), we can replace
$E$, $E'$ by curves $E_0$, $E_0'$ with the same triple of valuations and satisfying
\[
 \tilde{c}_4 \; \equiv \; \tilde{c}_4(E_0), \quad \tilde{c}_4' \equiv \tilde{c}_4(E_0') \pmod{32}, \quad
 \tilde{c}_6 \equiv \tilde{c}_6(E_0), \quad \tilde{c}_6' \equiv \tilde{c}_6(E_0') \pmod{32},
\]
hence, in particular, we only need to compute with a finite amount of curves. Note that if neither $E$ or
$E'$ is in case $D_d$ we can instead use a congruence mod~$16$ giving less cases.

Note that, in cases $D_a$ and $D_b$,
there are infinitely many possible triples of valuations which depend on $n$.
However, from Table~\ref{Table:(a)}, we see that the 
value of $n$ influences the field $F_i$ and the curve $\Ebar$ in such a way that it suffices 
to know if $n=7,8$ or $\geq 9$; similarly, from Table~\ref{Table:(b)} we only 
need to know if $n=10,11$ or $\geq 12$.
So, for each choice of the field $F_i$ of good reduction, 
if at least one of $E$ or $E'$ is in case (a) or (b) 
we can further choose $E_0$ or $E_0'$ (or both) with the minimal relevant values of $n$. 
Since there are only the two possibilities for the inertial field 
we are reduced to a finite amount of comparisons in total.

We finish the proof of part (A) by using {\tt Magma} to do the following.
First we compute $E[3]$ and $E'[3]$
for a list of representative pairs $(E,E')$. We note that
as $G_{\Q_2}$-modules they are not necessarily isomorphic,
however, we know they are isomorphic $I_2$-modules. Thus 
fixing symplectic basis for both determines an isomorphism $\phi$
as explained above and we can determine its symplectic type by explicit 
calculations, giving the desired results.

We now prove (B). Since taking a quadratic twist preserves the symplectic type, from
the last paragraph of Theorem~\ref{T:goodOverF}, after twisting both curves by 2 if necessary, 
we can assume that the inertial field is $L = \Q_2^{un} F_3$. 
Thus we can use Tables~\ref{Table:(e)}~and~\ref{Table:(f)} to determine the relevant information and 
the result follows as in case~(A).

{\large \part{Applications} \label{Part:applications}}

\section{Revisiting a question of Mazur}
\label{S:appMazur}

Using the LMFDB database~ \cite{lmfdb}, we can find 
triples $(E,E',p)$ as in Problem~A. In this section, we will apply our main theorems to 
describe the symplectic 
type of triples we found this way. Indeed, using {\tt Magma} and \cite[Proposition~4]{KO} we run through 
the elliptic curves in the database having conductor $\leq  360000$ 
looking for non-isogenous curves $E$ and~$E'$ such that $\rhobar_{E,p}^{ss} \simeq \rhobar_{E',p}^{ss}$
for some prime $p$. When $\rhobar_{E,p}$ is irreducible we have
$\rhobar_{E,p} \simeq \rhobar_{E',p}$ and all the 
$G_\Q$-modules isomorphisms $\phi : E[p] \to E'[p]$
have the same symplectic type (by Corollary~\ref{C:problemA}).
The symplectic type of one such $\phi$ is the type of~$(E,E',p)$ 
which we determine using one of our criteria; 
in Theorem~\ref{T:triples} we provide a few selected examples. 

\begin{theorem} Let $E$, $E'$ and $p$ be given by one of the lines in Table~\ref{Table:triples}.
Then $E[p]$ and $E'[p]$ are isomorphic as $G_\Q$-modules and the symplectic type of $(E,E',p)$ 
is given by the sign in the last column of the table, where `+' denotes symplectic and `-' anti-symplectic.
\label{T:triples} 
\end{theorem}

Before proving this theorem we make the following remarks:
\begin{itemize}
 \item The examples in Theorem~\ref{T:triples} are not concerned with the criteria in the case of good reduction; instead Theorem~\ref{T:simpleAbelian} is 
 illustrated within the Diophantine application in Section~\ref{S:2319} and Theorem~\ref{T:mainAbelian} in Example~\ref{Ex:mod7}.
 \item For many of the examples in Table~\ref{Table:triples} we include more than one proof, 
 to illustrate the flexibility of the criteria 
    and that our new theorems are compatible with previously known criteria (as expected).
 \item We found many examples for both types of isomorphism for $p=7,11$. This 
 was expected, since such examples can already be found in the work of the first author for $p=7$ 
 and in the work of Fisher \cite{Fisher}, \cite{FisherList} for $p=11$; nevertheless, 
 our methods provide alternative proofs, since it avoids computing rational 
 points on the relevant modular curves $X_E^{\pm}(p)$. Indeed, example (d) is from \cite{FisherList}
 and (c) is obtained by taking the quadratic twist by $-23$ of $(12696b1,12696c1,11)$ 
 which is also due to Fisher \cite{FisherList}; 
 \item The existence of pairs of elliptic curves admitting a 
 $13$-torsion isomorphism is well known but, to our knowledge, 
 the symplectic type was never determined for the older examples;
 in (e) and (f) we determine it for
 the isomorphism in \cite[\S 3.1, Example~5]{darmon-SerreConjs};
 (note that most of the recent infinitely many examples found by Fisher~\cite{Fisher13cong} fall outside the range of conductors in the database.)
 \item For $p=17$, up to quadratic twist, only 
 one pair of non-isogenous curves admitting a $17$-torsion isomorphism 
 was known for a long time; we will show it is an anti-symplectic isomorphism (see case~(g) in Table~\ref{T:triples}). 
 Very recently, Fisher~\cite{Fisher17cong}  found another $17$-torsion isomorphism and proved it is symplectic.
 \item We are not aware of any $p$-torsion isomorphism for $p \geq 19$ between non-isogenous curves.
\item In examples (a), (e), (f) and (h)
no previously known criterion applies, 
so our new theorems play an essential r\^ole.   
\item The mixed reduction case (Theorem~\ref{T:mixedReduction}) is illustrated in example (i); all other examples assume both curves have the same kind of reduction.
\end{itemize}

\begin{small}
\begin{table}[htb]
$$
\begin{array}{|c|c|c|c|c|c|c|} \hline
& E      & \Delta_m           & E'       &      \Delta_m'                 &  p            & \text{Type} \\ \hline 
(a) &2116a1 &  -2^8 \cdot 23^8          & 10580a1  &  2^8 \cdot 5^7 \cdot 23^4                  &  7            & - \\ 
(b) &648a1  &  -2^{10} \cdot 3^4  & 12312a1  &  -2^{8} \cdot 3^{12} \cdot 19^7 &  7            & + \\
(c) &12696e1&  -2^{11} \cdot 3^3 \cdot 23^8  & 12696f1     &   -2^{8} \cdot 3^5 \cdot 23^4 &  11   & + \\
(d) &4536c1&  -2^{11} \cdot 3^{12} \cdot 7^{11}  & 648b1     &   2^{4} \cdot 3^4  &  11   & - \\
(e) &52a2   &  2^4 \cdot 13      & 988b1    &  -2^4 \cdot 13 \cdot 19^{13} &  13    & + \\
(f) &52a1   &  -2^8 \cdot 13^2      & 988b1  &  -2^4 \cdot 13 \cdot 19^{13} &  13    & - \\
(g) &3675k1 &  -3^5 \cdot 5^2 \cdot 7^2      & 47775cq1 &  -3^2 \cdot 5^2 \cdot 7^2 \cdot 13^{17} & 17  & - \\ 
(h) &882a1 &  -2^3 \cdot 3^3 \cdot 7^8      & 441b1 &  -3^3 \cdot 7^4 & 3  & -\\ 
(i) &26b1 &  -2^7 \cdot 13  & 52a1 &  -2^8 \cdot 13^2 & 3  & - \\ \hline
\end{array}
$$
\caption{Examples of pairs $(E,E',p)$ and their symplectic types.}
\label{Table:triples}
\end{table}
\end{small}

\begin{proof}[Proof of Theorem~\ref{T:triples}] 
Our {\tt Magma} computations show that if $(E,E',p)$ is given by one 
of the lines in Table~\ref{Table:triples} then
$E[p]$ and $E'[p]$ are  
isomorphic $G_\Q$-module; in particular, for all primes $\ell$, 
they are also isomorphic $G_{\Q_\ell}$-modules. 
In each of the following cases, we will apply symplectic criteria at convenient primes 
$\ell$ to determine the symplectic type.

{\sc Case (a).}
The conductors are $N_E = 2^2 \cdot 23^2$ and $N_{E'} = 2^2 \cdot 5 \cdot 23^2$
and both $\ell=2,23$ satisfy $\ell \equiv 2 \pmod{3}$.
Let $\ell = 2$ and write $\vv = \vv_\ell$. 
We have that 
\[
(\vv(c_4), \vv(c_6), \vv(\Delta_m)) = (4,6,8), 
\qquad (\vv(c_4'), \vv(c_6'), \vv(\Delta_m')) = (9,7,8),
\]
$\tilde{c}_6 \equiv -1 \pmod{4}$ and $\tilde{c}'_6 \equiv 1 \pmod{4}$,
hence $e=e'=3$ by \cite[p. 358]{Kraus1990}.

We have $\vv_\ell(\Delta_m) \equiv \vv_\ell(\Delta_m') \pmod{3}$, hence $r=0$. 
We also have $\tilde{c}_4 \equiv 21 \pmod{32}$, $\tilde{c}_6 \equiv 11 \pmod{16}$ 
and $\tilde{c}_6' \equiv 1 \pmod{8}$.
From Theorem~\ref{T:main3torsion} we conclude that $E$ 
has a 3-torsion point defined over $\Q_2$ and $E'$ has not, hence $t=1$.
Since $(3/7)=-1$ it follows from Theorem~\ref{T:mainTame3} that $E[7]$ and $E'[7]$
are anti-symplectically isomorphic.

Let $\ell = 23$ and write $\vv = \vv_\ell$. 
We compute
\[\vv(\Delta_m) = 8, \quad \vv(\Delta_m') = 4, \quad \tilde{c}_6 = 1728, \quad \tilde{c}_6' = -1372032\] 
and from \cite[Proposition~1]{Kraus1990} we see that $e=e'=3$.
We also have $\vv(\Delta_m) \not\equiv \vv(\Delta_m') \pmod{3}$ thus $r=1$. Moreover, $t=1$ by Theorem~\ref{T:main3torsion} because $-6\tilde{c}_6$ is not a square in $\Q_{23}$ but $-6\tilde{c}_6'$ is.
Since $(3/7)(23/7) = -1$ it follows  again  from Theorem~\ref{T:mainTame3} that $E[7]$ and $E'[7]$
are anti-symplectically isomorphic, as expected.

{\sc Case (b).} We will apply Theorem~\ref{T:mainWild3}. We can easily compute that 
\[
(\vv_3(c_4), \vv_3(c_6), \vv_3(\Delta_m)) = (2,3,4), \qquad (\vv_3(c_4'), \vv_3(c_6'), \vv_3(\Delta_m')) = (5,8,12)
\]
and $\tilde{\Delta} = \Delta_m/3^4 = - 1024 \equiv 2 \pmod{3}$. We have $r=0$ because $\tilde{c}_6 = -448$ and $\tilde{c}_6' = -1703296$ satisfy $\tilde{c}_6 \equiv \tilde{c}_6' \pmod{3}$. We conclude the $7$-torsion isomorphism is symplectic.
 
Note that $E$ has potentially good reduction at $\ell = 2$ with $e(E/\Q_2)=24$; since $(2/7) = 1$
it follows from part (1) of Theorem~\ref{T:Wilde24} that $E[7]$ and $E'[7]$ are symplectically 
isomorphic, as expected.

{\sc Case (c).}
We have $N_E = N_{E'} = 2^3 \cdot 3 \cdot 23^2$.
We have $\vv_{23}(\Delta_m) = 4$ and $\vv_{23}(\Delta_m') = 8$, hence 
from \cite[Proposition~1]{Kraus1990} we see that $e=e'=3$ and we can
apply apply Theorem~\ref{T:mainTame3} at $\ell = 23$.

Since $\vv_{23}(\Delta_m) \not\equiv \vv_{23}(\Delta_m') \pmod{3}$ we have $r=1$;
from Theorem~\ref{T:main3torsion} we compute $t=1$ because $\tilde{c}_6 = 552064$ and $\tilde{c}'_6 = -1263022400$,
hence $-6\tilde{c}_6$ is not a square in $\Q_{23}$ and $-6\tilde{c}'_6$ is.
Finally, $(3/11)=(23/11)=1$ and we conclude that $E[11]$ and $E'[11]$ are symplectically ismorphic.

Since $5 \cdot 3^{-1} \equiv 3^2  \pmod{11}$ we have
$\vv_{3}(\Delta_m')/ \vv_{3}(\Delta_m)$ is a square mod~$11$,  
hence Theorem~\ref{T:potMult} implies that
$E[11]$ and $E'[11]$ are 
symplectically isomorphic, as expected.

Note that $E$ has potentially good reduction at $\ell = 2$ with $e(E/\Q_2)=24$; since $(2/11) = -1$
and $\vv_{2}(\Delta_m) \equiv \vv_{2}(\Delta_m') \pmod{3}$
it follows part (2) of Theorem~\ref{T:Wilde24} that $E[11]$ and $E'[11]$ are symplectically 
isomorphic, as expected.

{\sc Case (d).} We have $N_E = 2^3 \cdot 3^4 \cdot 7$ and 
$N_{E'} = 2^3 \cdot 3^4$. 
Note that $E$ has potentially good reduction at $\ell = 2$ with $e(E/\Q_2)=24$; since $(2/11) = -1$
and $\vv_{2}(\Delta_m) \not\equiv \vv_{2}(\Delta_m') \pmod{3}$
it follows part (2) of Theorem~\ref{T:Wilde24} that $E[11]$ and $E'[11]$ are anti-symplectically 
isomorphic.

{\sc Case (e).} We have $N_E = 2^2 \cdot 13$, $N_{E'} = 2^2 \cdot 13 \cdot 19$ and
\[
(\vv_2(c_4), \vv_2(c_6), \vv_2(\Delta_m)) = (6,5,4), \qquad (\vv_2(c_4'), \vv_2(c_6'), \vv_2(\Delta_m')) = (4,5,4);
\]
also, $\tilde{c}_6 \equiv \tilde{c}'_6 \equiv 1\pmod{4}$ and $\tilde{c}'_4 \equiv -1 \pmod{4}$, 
so $e=e'=3$ by \cite[p. 358]{Kraus1990}, so we can apply Theorem~\ref{T:mainTame3} at $\ell=2$.
Indeed, we have $\vv_2(\Delta_m) \equiv \vv_2(\Delta_m')\pmod{3}$, hence $r=0$; 
we also have $\tilde{c}_6 \equiv 1 \pmod{8}$ 
and $(\tilde{c}_4', \tilde{c}_6') \equiv (3,5) \pmod{8}$.
From Theorem~\ref{T:main3torsion} we conclude that $E'$ 
has a 3-torsion point defined over $\Q_2$ and $E$ has not, hence $t=1$.
Since $(3/13)=1$ we conclude that $E[13]$ and $E'[13]$ are symplectically isomorphic.

{\sc Case (f).} Since $52a1$ and $52a2$ are related by a $2$-isogeny and $(2/13)=-1$ 
the symplectic type is the opposite of the type in case (d).

{\sc Case (g).} 
The conductors are $N_E = 3 \cdot 5^2 \cdot 7^2$ and $N_{E'} = 3 \cdot 5^2 \cdot 7^2 \cdot 13$, hence $\ell = 5$ is the unique prime of bad reduction such that $\ell \equiv 2\pmod{3}$. We check that $e= e'= 6$, hence to apply Theorem~\ref{T:mainTame3}
we first twist the curves by $d=5$. 
Write $E$ and $E'$ for the quadratic twists $dE$ and $dE'$ which
have Cremona label $3675g1$ and $47775bf1$ respectively; we have
\[
N_E = 3 \cdot 5^2 \cdot 7^2, \quad N_{E'} = 3 \cdot 5^2 \cdot 7^2 \cdot 13, \quad
\Delta_m = -3^5 \cdot 5^8 \cdot 7^2, \quad \Delta_m' = -3^2 \cdot 5^8 \cdot 7^2 \cdot 13^{17}.
\]
Let $\ell = 5$ and write $\vv = \vv_\ell$. We have
\[
\vv(\Delta_m) = 8, \quad \vv(\Delta_m') = 8, \quad \tilde{c}_6 =-4781, \quad \tilde{c}_6' =2054214213667,
\] 
so that $e=e'=3$ by \cite[Proposition~1]{Kraus1990}
and $r=0$ because $\vv(\Delta_m) \equiv \vv(\Delta_m')\pmod{3}$. 
Moreover, $t=1$ by Theorem~\ref{T:main3torsion} because $-6\tilde{c}_6$ is a square in $\Q_{5}$ but $-6\tilde{c}_6'$ is not.
Since $(3/17) = -1$ it follows from Theorem~\ref{T:mainTame3} that $E[17]$ and $E'[17]$
are anti-symplectically isomorphic.

Note that $\vv_3(\Delta_m) = 5$, $\vv_3(\Delta_m') = 2$ and $5\cdot 2^{-1} \equiv 11 \pmod{17}$.
Since $(11/17) = -1$ we conclude from Theorem~\ref{T:potMult} that $E[17]$ and $E'[17]$
are anti-symplectically isomorphic, as expected.

{\sc Case (h).} 
In this case, the elliptic curves $E$ and $E'$ have a rational $3$-torsion point. The representations
$\rhobar_{E,3}$ and $\rhobar_{E’,3}$ have reducible image isomorphic to $S_3$. One easily
checks that the $3$-torsion fields are the same. Then $\rhobar_{E,3} \sim \rhobar_{E',3}$
and their image is non-abelian, so the symplectic type of 
$(E,E',3)$ is well defined. The conductors are
$N_E = 2 \cdot 3^2 \cdot 7^2$, $N_{E'} = 3^2 \cdot 7^2$
and it is easy to check that $e(E/\Q_7) = e(E'/\Q_7) = 3$. 
Moreover, $7 \equiv 1 \pmod{3}$ and
$8 = \vv_7(\Delta_m) \not\equiv \vv_7(\Delta'_m) = 4 \pmod{3}$.
Then $E[3]$ and $E'[3]$ 
are anti-symplectically isomorphic by Theorem~\ref{T:e=p=3}. 

{\sc Case (i).} The conductors are $N_E = 2 \cdot 13$ and $N_{E'} = 2^2 \cdot 13$. 

Both curves have multiplicative reduction at $\ell=13$. Also, $\vv_{13}(\Delta_m')/\vv_{13}(\Delta_m)=2/1=2$ is not a square mod~$3$. 
Then $E[3]$ and $E'[3]$ 
are anti-symplectically isomorphic by Theorem~\ref{T:potMult}. 

We now work with $\ell = 2$. 
Since $\vv_2(N_{E'}) = 2$ it follows that  
$E'$ has potentially good reduction with 
$e'=e(E'/\Q_2) = 3$. Clearly, $E$ has multiplicative reduction at 2. Thus we are in the mixed reduction case
and will apply Theorem~\ref{T:mixedReduction}. From the discriminants we promptly check
\[ \vv(\Delta_m) = 7 \not\equiv 8 = \vv(\Delta_m')\pmod{3} \]
and conclude that 
$E[3]$ and $E'[3]$ 
are anti-symplectically isomorphic by Theorem~\ref{T:mixedReduction}. Note also that $3 \nmid \vv_2(\Delta_m)$ as also predicted by the theorem.
\end{proof}

Recall from the introduction that,
for a triple $(E,E',p)$ as in Problem A, we defined the set
$\calL_{(E,E',p)}$ 
of primes $\ell \ne p$ such 
a symplectic criterion exists for $E/\Q_\ell$ and $E'/\Q_\ell$.
In other words, $\calL_{(E,E',p)}$ is the set of primes for which
the curves $E/\Q_\ell$ and $E/\Q_\ell$ 
contain sufficient information to determine the symplectic type of $(E,E',p)$. The set $\calL_{(E,E',p)}$ can be empty;
this occurs exactly when,
for all primes~$\ell \neq p$,
the $G_{\Q_\ell}$-isomorphic modules
$(E/\Q_\ell)[p]$ 
and $(E/\Q_\ell)[p]$ are
simultaneously symplectically and anti-symplectically isomorphic.
The following proposition gives such an example.

\begin{proposition}\label{P:emptyL} Let $E$ and $E'$ be the elliptic curves
given by the models
\[
 E : y^2 + xy = x^3 - x^2 - 1915152x + 713106917
\]
and
\[
 E' : y^2 + y = x^3 - 864116667x + 78257970610832.
\]
Then the triple $(E,E',3)$ has a well defined symplectic type 
and $\calL_{(E,E',3)}$ is empty.
\end{proposition}
\begin{proof} The elliptic curves in the statement have conductors 
\[
 N_E = 3^3 \cdot 61^2 \quad \text{ and } \quad N_{E'} = 3^3 \cdot 5 \cdot 61^2 \cdot 449.
\]
The curve $E$ has Cremona label 
$100467f1$. Using the LMFDB \cite{lmfdb} we can easily 
check that $\rhobar_{E,3}$ has image the normalizer of the non-split 
Cartan subgroup of $\GL_2(\F_3)$. Thus $\rhobar_{E,3}^{ss} \simeq \rhobar_{E,3}$ 
has non-abelian image. Now an application of \cite[Proposition~4]{KO} 
shows that $E[3]$ and $E'[3]$ are $G_\Q$-isomorphic, therefore 
$(E,E',3)$ has a well defined symplectic type. 

The primes $\ell \neq 3$ where the curves have different reduction types are 
$\ell = 5, 449$. The curve~$E'$ has multiplicative reduction at these primes, however, Theorem~\ref{T:mixedReduction} cannot be applied, as 
it would require $E$ to have
potentially 
good reduction with $e=3$ 
at~$\ell$.

Both curves have potentially good reduction at $\ell=61$ 
with $e=e'=4$. Therefore, the primes $\ell \neq 3$ for which $E$ and $E'$ have the 
same reduction type are $\ell = 61$ and primes $\ell \nmid N_{E'}$, i.e. primes of 
simultaneously good reduction of $E'$ and $E$. 

Since $\#\rhobar_{E,3}(G_\Q) = \#\rhobar_{E',3}(G_\Q) = 16$, we 
have that $3$ does not divide the order of $\rhobar_{E,3}(\Frob_\ell)$ 
for all $\ell \nmid N_{E'}$, so Theorems~\ref{T:simpleAbelian}~and~\ref{T:mainAbelian} cannot be applied. 
For $\ell=61$, since the semistability is $e=4$, we have $\ell \equiv 1 \pmod{e}$
and Theorem~\ref{T:mainTame4} cannot be applied either. 

Thus none of the cases in Table~\ref{Table:CriteriaList} applies,
hence 
$\calL_{(E,E',3)}$ is empty.
\end{proof}

From the previous proof, we see that
for the set $\calL_{(E,E',p)}$ to be empty, we need that the image of $\rhobar_{E,p}$ does not contain elements of order~$p$; otherwise, from Chebotarev density theorem, there will be a prime~$\ell$ of good reduction for $E$ and~$E'$, for which $\Frob_\ell$ has order multiple of~$p$, and so we can apply Theorem~\ref{T:mainAbelian} (possibly also Theorem~\ref{T:simpleAbelian}). The following example illustrates a case where all the criteria at primes of bad reduction and Theorem~\ref{T:simpleAbelian} fail, but Theorem~\ref{T:mainAbelian} applies since the image contain elements of order multiple of~$p$.

\begin{example} \label{Ex:mod7}
Let $E$ and $E'$ be the elliptic curves with Cremona labels $338b2$ and $12506d2$, respectively. Using {\tt Magma} and \cite[Proposition~4]{KO} we check that $\rhobar_{E,7}^{ss} \simeq \rhobar_{E',7}^{ss}$. Consulting LMFDB~\cite{lmfdb} we verify that $\rhobar_{E,7}$ is reducible; thus, to establish that 
$E[7] \simeq E'[7]$ we need to verify that the curves have isomorphic $7$-torsion fields, which we do using {\tt Magma}.

The conductors and discriminats of $E$ and $E'$ are 
\[
 N_E = 2 \cdot 13^2, \quad \Delta_m = -2^{14} \cdot 13^8 \quad \text{ and } \quad N_{E'} = 2 \cdot 13^2 \cdot 37, 
 \quad \Delta_m' = -2^{7} \cdot 13^4 \cdot 37^7.
\]
Since $p=7 \neq 3$ we cannot apply Theorems~\ref{T:e=p=3}~and~\ref{T:mixedReduction}. So we focus on primes where the type of reduction is the same. At $\ell = 2$ both curves 
have multiplicative reduction, but 
$\vv_2(\Delta_m)$ and~$\vv_2(\Delta_m')$ are multiples of 7, so Theorem~\ref{T:potMult} cannot be applied. At $\ell=13$ we have semistability defect 
$e=e(E/\Q_{13}) = e(E'/\Q_{13}) = 3$ and since 
$13 \equiv 1 \pmod{3}$ we cannot apply Theorem~\ref{T:mainTame3}. We conclude that no criteria  at primes of bad reduction can be applied.

However, 
$\rhobar_{E,7}(G_\Q)$ has order multiple of 7\footnote{This follows from the $7$-torsion field computation, but also from the fact that~$\rhobar_{E,7}$ is reducible and Proposition~\ref{P:typeTriple}.}, so by Chebotarev density theorem there are primes~$\ell$ 
of good reduction for both curves
for which $\rhobar_{E,7}(\Frob_\ell)$ has order multiple of~7, therefore $\calL_{(E,E',7)}$ is non-empty.
For all such primes~$\ell$ up to $1000$  
we verified that the residual curves $\Ebar$ and $\Ebar'$ are not isomorphic over~$\F_\ell$, and so Theorem~\ref{T:simpleAbelian} cannot be applied either at any of these primes. Note that, when applicable, Theorem~\ref{T:simpleAbelian} concludes the isomorphism is symplectic so, if we are in an anti-symplectic situation, we will never find a prime~$\ell$ to apply that theorem.
Thus, we move to the only other criterion left to apply, that is Theorem~\ref{T:mainAbelian}. The smallest primes we can try to use for this theorem are $\ell = 5, 11$.

Using the {\tt Magma} code
provided by Centeleghe and Tsaknias~\cite{Ccode}, 
we compute the action of $\Frob_{5}$ as in \eqref{E:FrobCharp} and after reducing mod~$7$ we get
\[
 \rhobar_{E,7}(\Frob_{5}) = \rhobar_{E',7}(\Frob_{5}) 
  = \begin{pmatrix} 2 & 3 \\ 1 & 4 \end{pmatrix} \in \GL_2(\F_7),
\]
which is a matrix of order~$6$, 
hence $5 \not\in \calL_{(E,E',7)}$.\footnote{Since $(5/7) = -1$, it follows also from Proposition~\ref{P:tableConditionI} that~$5 \not\in \calL_{(E,E',7)}$.}
The same calculation 
with $\Frob_{11}$ gives
\begin{equation} \label{Ex:Frob}
 \rhobar_{E,7}(\Frob_{11}) = \rhobar_{E',7}(\Frob_{11}) 
  = \begin{pmatrix} 2 & 0 \\ 1 & 2 \end{pmatrix},
\end{equation}
which has order 21, hence $11 \in \calL_{(E,E',7)}$.
We also see that $\beta_{11} \equiv \beta_{11}' \equiv 1 \pmod{7}$. Unfortunately, from this computation we cannot get
the quantity $s$ in the statement of Theorem~\ref{T:mainAbelian}.

To obtain this value, we need to do much more calculations, which become impractical for large values of~$p$. We proceed as follows. Since $E$ and $E'$ have good reduction at~$11$, the reduction maps 
$(E/\Q_{11})[7] \to (\Ebar/\F_{11})[7]$  and 
$(E'/\Q_{11})[7] \to (\Ebar'/\F_{11})[7]$
are symplectic Galois isomorphisms, hence 
$\Ebar[7]\simeq \Ebar'[7]$ and
we can instead work with the $p$-torsion modules of the residual curves (which is a good computational improvement compared to work in characteristic 0).
Using {\tt Magma} we start by finding the common $p$-torsion field~$K$ of the residual curves. Then, we
find a choice of symplectic basis for $\Ebar[p]$ and $\Ebar'[p]$. 
After identifying the Frobenius element in $\Gal(K/\F_{11})$
we act with it on the choosen basis and write down the corresponding 
matrices. This gives
\begin{equation} \label{Ex:Frob2}
  M_E = \rhobar_{E,7}(\Frob_{11}) 
  = \begin{pmatrix} 2 & 5 \\ 0 & 2 \end{pmatrix} \quad 
  \text{ and } \quad M_{E'} = \rhobar_{E',7}(\Frob_{11}). 
  = \begin{pmatrix} 2 & 2 \\ 0 & 2 \end{pmatrix}. 
\end{equation}
Write $M$ for the matrix in~\eqref{Ex:Frob}.
A direct calculation shows that the matrices
\[
 C_E = \begin{pmatrix} 0 & 5 \\ 1 & 0 \end{pmatrix}
 \quad \text{ and } \quad 
 C_{E'} = \begin{pmatrix} 0 & 2 \\ 1 & 0 \end{pmatrix}
\]
satisfy $C_EM_EC_E^{-1} = M$ and $C_{E'}M_{E'}C_{E'}^{-1} = M$. Note that $\det C_E = 2$ which is a square mod~$7$ while $\det C_{E'} = 5$ which is not. We conclude that the action of $\Frob_{11}$ represented by~$M$ in~\eqref{Ex:Frob} corresponds to a symplectic basis for~$E$ and a non-symplectic basis for~$E'$. 
Thus $s=1$ and since $\beta_{11}/\beta_{11}' \equiv 1 \pmod{7}$, we conclude from Theorem~\ref{T:mainAbelian} that $E[7]$ and $E'[7]$ are anti-symplectically isomorphic, so
the symplectic type of $(E,E',7)$ is anti-symplectic.
\end{example}

\section{The Generalized Fermat equation $x^2 + y^3 = z^{p}$}
\label{S:2319}

Among all the generalized Fermat equations of the form 
\[
 x^r + y^q = z^p, \qquad  1/r + 1/q + 1/p < 1, \qquad \gcd(x,y,z) = 1,
\] 
where $p$, $q$, $r$ are primes,
there are only a few particular cases to which we know how to 
attach a Frey elliptic curve. Among those cases 
the equation $x^2 + y^3 = z^{p}$ 
is arguably the most challenging one.
Contributing to this is 
the fact that the exponents are all different, which does not allowing 
for any helpful factorization, but more relevant is the fact that it admits 
the solutions $(\pm 1,0,1)$, $\pm (0,1,1)$ and $(\pm 3,2,1)$

In~\cite{FNS23n} it is shown that to solve $x^2 + y^3 = z^{p}$ one needs 
to determine rational points on a finite list of twists of the modular 
curve $X(p)$; further, this very challenging task is carried 
out under GRH for $p=11$. 
In this section we will show how Theorem~\ref{T:simpleAbelian} sometimes 
applies to reduce the amount of twists of $X(p)$ that need 
to be considered (see Theorem~\ref{T:Diophantine}). The 
results here do not give a complete resolution of the equation for any 
new exponent. Nevertheless, they confirm the expectation 
that twists of $X(p)$ not attached 
to any of the obvious solutions should not have relevant rational points.

We start by recalling and summarizing the relevant notation and facts from~\cite{FNS23n}.

Given a $G_\Q$-isomorphism $\phi : W[p] \to E[p]$ recall
the quantity $d(\phi) \in \F_p^*$ given in~\eqref{E:dphi}.
Fix a non-square~$d_p \in \F_p^\times$.
We say that $\phi$ is \emph{strictly symplectic}, if $d(\phi) = 1$, and
\emph{strictly anti-symplectic}, if $d(\phi) = d_p$.

For a fixed elliptic curve $W$ we denote by 
$X_W^+(p)$ the curve parameterizing isomorphism classes of pairs $(E,\phi)$
where $E$ is an elliptic curve and $\phi : W[p] \to E[p]$ 
a strictly symplectic isomorphism. We write $X_W^-(p)$ for the 
analogous curve in the case of strictly anti-symplectic isomorphisms.
These are twist of the modular curve $X(p)$. 

Let $p \geq 7$ be a prime and suppose that $(a,b,c) \in \Z^3$ satisfies the equation
\begin{equation} \label{E:GFE}
  a^2 + b^3 = c^p \qquad \text{with} \qquad \gcd(a,b,c) = 1
\end{equation}
and consider the associated Frey elliptic curve
\[ E_{(a,b,c)} \; \colon \; y^2 = x^3 + 3bx - 2a \]
whose invariants are
\begin{equation} \label{E:jInvariant}
  c_4 = -12^2 b, \qquad c_6 = -12^3 a, \qquad  \Delta = -12^3 c^p.
\end{equation}

Consider also the seven elliptic curves (specified by their Cremona label):
  \[ 27a1,\; 54a1,\; 96a1,\; 288a1,\; 864a1,\; 864b1,\; 864c1 \,. \]
It is shown in \cite{FNS23n} that, for some $d \in \{ \pm 1, \pm 2, \pm 3, \pm 6 \}$, the 
quadratic twist $dE_{(a,b,c)}$ of the Frey curve gives rise to a rational point
(satisfying certain $2$-adic and $3$-adic conditions)
on the modular curve $X_W^+(p)$ or $X_W^-(p)$, 
where $W$ is one of the curves in the previous list. 
Moreover, using the symplectic criteria provided by
Theorems~\ref{T:mainWilde8},~\ref{T:mainWilde12}~and~\ref{T:potMult}
and information on the rational 
points of $X_{\text{split}}^+(p)$\footnote{More precisely, it follows from~\cite{BPR2013,BDMTV}
that, for primes $p \ge 11$, the image of the mod~$p$ Galois representation of any elliptic curve 
$E/\Q$ without CM is never contained in the normalizer of a split Cartan subgroup. Thus, if $dE_{(a,b,c)}[p] \simeq W[p]$ with $W = 27a1$ or $288a1$ and $p$ splits in the CM field of~$W$, then 
$dE_{(a,b,c)}$ has CM; hence $c=1$ as CM curves have no primes of multiplicative reduction.},
the authors also establish that, according to the values of $p$ mod~$24$, 
the quantity of twists that need to be considered varies between 4 and 10, as summarized 
in Table~\ref{Table-twist}.

\begin{table}[htb]
\[ \begin{array}{|r||c|c|c|c|c|c|c|} \hline
     p \bmod 24 & 27a1 & 54a1 & 96a1 & 288a1 & 864a1 & 864b1 & 864c1 \text{\large\strut}\\\hline
              1 &      &   +  &   +  &       &   +   &   +   &   +   \\
              5 &   +  &   -  &   +  &       &  + -  &  + -  &  + -  \\
              7 &      &   -  &   +  &   +   &   +   &   +   &   +   \\
             11 &   +  &   +  &   +  &  + -  &   +   &   +   &   +   \\
             13 &      &      &   -  &       &   +   &   +   &   +   \\
             17 &   +  &   +  &      &       &   +   &   +   &   +   \\
             19 &      &   +  &   -  &  + -  &  + -  &  + -  &  + -  \\
             23 &   +  &      &      &   +   &   +   &   +   &   +   \\\hline
   \end{array}
\]
\caption{Twists of~$X(p)$, for $p \ge 11$, 
	remaining after local considerations at $\ell=2,3$ and using
         information on $X_{\text{split}}^+(p)$, according to $p \bmod 24$.}
\label{Table-twist}
\end{table}
Furthermore, the table is optimal locally at $2$ and $3$, 
in the sense that all the twists included in it have
$2$-adic and $3$-adic points arising from a twisted Frey curve. 

From this point, to actually solve equation \eqref{E:GFE} completely for a specific value of $p$,  
it is necessary to determine all the rational points that could correspond to $dE_{(a,b,c)}$ 
on the relevant twists $X_W^{\pm}(p)$. This is a very hard task, requiring complicated global methods, 
and so far has only been accomplished  for $p=7$ in \cite{PSS} (where the twists have genus 3) 
and, assuming GRH, in \cite{FNS23n} for $p=11$ (where the genus of the twists is 26). 
Therefore, it is natural to hope that new easily applicable local techniques 
could be used to shrink further the list of twists that needs to be treated
with global methods. It is worth noting that, as explained in \cite{FNS23n},
the existence of the solutions $(\pm 1,0,1)$, $\pm (0,1,1)$ and 
$(\pm 3,2,1)$ to \eqref{E:GFE} implies that, for all $p$, the twists $X_{27a1}^+(p)$,
$X_{288a1}^+(p)$, $X_{864b1}^+(p)$ have a rational point and so 
$\ell$-adic points for every prime $\ell$; 
further the same is true for $X_{288a1}^-(p)$ when $(2/p) = -1$. 

We already mentioned that local methods at $\ell = 2,3$ 
cannot reduce the list of twists further. Since for all primes $\ell \ne 2,3$ 
the elliptic curves $W$ in the table above have good reduction at~$\ell$ 
the only criteria we can still apply are Theorems~\ref{T:simpleAbelian}~and~\ref{T:mainAbelian}. 
To apply them, we need that
$\rhobar_{W,p}(\Frob_\ell)$ has order 
divisible by $p$, so we cannot use it for the CM elliptic curves $27a1$ and $288a1$;
also, when applicable, Theorem~\ref{T:simpleAbelian} concludes that the $p$-torsion 
modules of the subject elliptic curves are symplectically isomorphic, hence we 
may use it only to discard the `negative' twists.
Therefore, we can try to apply Theorem~\ref{T:simpleAbelian}
to discard $X_{W}^-(p)$ for $W=54a1$, $96a1$, $864a1$, $864b1$ or $864c1$;
more precisely, we will use it to prove the following theorem.

\begin{theorem} \label{T:Diophantine}
  Let $p  = 19, 43$ or $67$. Let 
  $(a,b,c) \in \Z^3$ be a solution to the equation
  \[ x^2 + y^3 = z^p,  \qquad \gcd(x,y,z) = 1.  \]
  Then, for all $d \in \{ \pm 1, \pm 2, \pm 3, \pm 6 \}$, 
  the twisted Frey curve~$dE_{(a,b,c)}$
  does not give rise to a rational point on the following twists of~$X(p)$:
  \[ \renewcommand{\arraystretch}{1.2}
     \begin{array}{|c|l|} \hline
       p  &             \text{twists of $X(p)$} \\\hline\hline
 
       19, 43         & X^-_{864a1}(p), \; X^-_{864b1}(p) \\ \hline
       67             & X^-_{864b1}(p) \\\hline
     \end{array}
  \]
\end{theorem}
\begin{proof} We will prove the table in the statement, by
using an auxiliary prime $\ell$ of good reduction according to the cases:
\[
 (p,W,\ell) = (19,864a1,5), \; (19,864b1,7), \; (43,864a1,31), \; (43,864b1,13), \; (67,864b1,19).
\]
We will explain the arguments only in the case $X^-_{864a1}(19)$ and $\ell=5$,
since the remaining cases follow similarly; we note that $\ell = 7$ could also be used
to discard $X^-_{864a1}(19)$.

Write $E = dE_{(a,b,c)}$ and suppose it gives rise to a rational point in $X^-_{864a1}(19)$, in particular,
\[ \rhobar_{E,19} \simeq \rhobar_{W,19}.\] 
If $E$ has bad reduction at $\ell = 5$, then $\ell \mid c$ and from the invariants of $E$ and Tate's algorithm 
we see that 
the reduction is multiplicative. Thus, level lowering at $\ell = 5$ is happening 
in the previous isomorphism of representations. Hence we have
\[ a_\ell(W) \equiv \pm (\Norm(\ell) +1) \pmod{19} \] 
and since $a_5(W) = -1$, we see this is not possible. We
conclude that both $E$ and $W$ have good reduction at 5. 
It also follows from the previous isomorphism that 
$\rhobar_{E,19}(\Frob_5)$ and $\rhobar_{W,19}(\Frob_5)$ 
are conjugated matrices. Using the reduction map we can
compute the action of $\rhobar_{W,19}(\Frob_5)$ from 
the action of $\Frob_5$ on $\overline{W}[19]$.
Using {\tt Magma} we compute that, 
\[
 \overline{W} \; : \; y^2 = x^3 + 2x + 1 \quad \text{ and } 
 \quad \rhobar_{W,19}(\Frob_5) = \begin{pmatrix} 9 & 0 \\ 1 & 9 \end{pmatrix}
\]
in some basis of $W[19]$. We now need to compute the matrix of $\rhobar_{E,19}(\Frob_5)$ 
and check if it is conjugated to the matrix above.
Indeed, since $E$ has good reduction it follows from the expression of 
$\Delta_E$ that $5 \nmid a^2 + b^3$. 
We have $\gcd(a,b,c) = 1$, so 
running through all the possibilities for $(a,b) \neq (0,0)$ with $0 \leq a, b \leq 4$ 
we obtain all the possible residual curves $\Ebar/\F_5$. In each case, we compute
$\rhobar_{E,19}(\Frob_5)$ as we did for $W$. The pairs $(a,b)$ that
give rise conjugated matrices are 
\[
 (a,b) = (2,4), \quad \Ebar \; : \; y^2 = x^3 + 2x + 1, \quad
 \rhobar_{E,19}(\Frob_5) = \begin{pmatrix} 9 & 0 \\ 1 & 9 \end{pmatrix}
\]
and 
\[
 (a,b) = (3,4), \quad \Ebar \; : \; y^2 = x^3 + 2x + 4, \quad
 \rhobar_{E,19}(\Frob_5) = \begin{pmatrix} 9 & 0 \\ 1 & 9 \end{pmatrix}.
\]
We observe that $\overline{W}$ is $\F_5$-isomorphic to the two possibilities for $\Ebar$, 
so by Theorem~\ref{T:simpleAbelian} we conclude that $E[19]$ and $W[19]$ are symplectically 
isomorphic in both cases. 

Note, however, that we have $dE[19] \simeq W[19]$ for some
$d \in \{ \pm 1, \pm 2, \pm 3, \pm 6 \}$ and the previous calculations assumed $d=1$, 
so we need to repeat them for the remaining values of $d$. Note
it is enough to compute with a representative of non-squares in $\F_5$ (otherwise we 
obtain isomorphic curves over $\F_5$) so we take $d=2$. Computing as before we get
\[
 (a,b) = (1,1),  \quad \Ebar \; : \; y^2 = x^3 + 2x + 4, \quad
 \rhobar_{E,19}(\Frob_5) = \begin{pmatrix} 9 & 0 \\ 1 & 9 \end{pmatrix}
\]
and 
\[
 (a,b) = (4,1), \quad \Ebar \; : \; y^2 = x^3 + 2x + 1, \quad
 \rhobar_{E,19}(\Frob_5) = \begin{pmatrix} 9 & 0 \\ 1 & 9 \end{pmatrix}
\]
and the same conclusion holds. We conclude that $dE[19]$ and $864a1[19]$ are 
always symplectic isomorphic,
so $dE$ cannot define a $\Q$-point in $X_W^-(19)$, thereby
concluding the proof of the theorem in this case. 
\end{proof}

We remark that the cases in Theorem~\ref{T:Diophantine} 
were the only ones where we succeeded in applying Theorem~\ref{T:simpleAbelian}
among the following search. For each of the five curves $W=54a1$, $96a1$, $864a1$, $864b1$ or $864c1$, 
using {\tt Magma}, we proceeded as follows.  
For all primes $p \leq 400$ and
all primes $ 5 \leq \ell \leq 200$ 
such that $(\ell/p) = 1$ (see Proposition~\ref{P:tableConditionI}) 
we compute the order $N_{\ell,p}$ of $\rhobar_{W,p}(\Frob_\ell)$ 
and check if $p \mid N_{\ell,p}$, as this is 
necessary for applying Theorem~\ref{T:simpleAbelian}.  
In particular, for $p=19$ we find the possible pairs $(W,\ell)$: 
\[ (54a1,47), \; (54a1,83), \; (96a1,197), \; (864a1,5), \; (864a1,7), \]
\[ (864a1,47), \; (864a1,83), \; (864a1,163), \; (864b1,7) \; \text{ and } \; (864b1,47), \]
but, except for those in Theorem~\ref{T:Diophantine}, there is always a choice of $(a,b) \pmod{\ell}$ 
that $j(\Ebar) \ne j(\overline{W})$, so the condition on the hypothesis residual curves fails.
The same occurs also for the other values of $p$.

Finally, one could hope that the more general Theorem~\ref{T:mainAbelian} would discard further twists but,  unfortunately, this is not the case. 
Indeed, withing the described search range, we tried to use Theorem~\ref{T:mainAbelian} but, for each fixed curve~$W$, the symplectic type of the isomorphism varies with $(a,b) \pmod{\ell}$ 
and so no other~$X^{\pm}_W(p)$ could be discarded completely.

\section{On the hyperelliptic curves $y^2 = x^p - \ell$ and $y^2 = x^p - 2\ell$}
\label{S:hyperelliptic}

In this section, we consider the family of curves 
\[
 C_{\ell,p} \; : \; y^2 = x^p - \ell \quad \text{ and } \quad C'_{\ell,p} \; : \; y^2 = x^p - 2\ell
\]
where $\ell$ and $p$ are primes. Let us define the quantity 
\[
f(\ell) = (\sqrt{32(\ell+1)}+1)^{8(\ell-1)}.                                
\]
From \cite[Theorem 1.1]{IK2006}, we know that,  if $\ell\equiv 3 \pmod 8$ and $\ell\neq 3$, or if 
$\ell \equiv 5 \pmod{8}$ and $\ell - 1$ is not a square, 
then for all primes $p > f(\ell)$ there are no rational points 
on the curve~$C_{\ell,p}$.
Furthermore, if 
$\ell \equiv 3 \pmod{8}$ and $\ell - 2$ is not a square, 
we have the same conclusion for the curve~$C'_{\ell,p}$.
Our   objective in this section is  to use the symplectic argument 
to extend these results.

\subsection{The curve $y^2 = x^p - \ell$}
Let us prove the following theorem.
\begin{theorem} \label{T:2pl}
Let $\ell \equiv 5 \pmod{8}$ be a prime such that $\ell-1$ is a square. Then, 
for all primes $p > f(\ell)$ satisfying $(2/p)=-1$, the set   $C_{\ell,p}(\Q)$ is empty.
\end{theorem}

Note that, from the assumptions   made on $\ell$, we have $\ell=5$ or $\ell\geq 29$.

We need to recall several definitions and results from \cite{IK2006}.
Suppose there are $x, y \in \Q$ such that $y^2 = x^p - \ell$ and
write $x=a/b$ and $y=c/d$ where $(a,b)=(c,d)=1$. Note that, for all primes
$q$, we have $\vv_q(y)<0$ if and only if $v_q(x)<0$; further $pv_q(b)=2v_q(d)$.
Therefore, there are 
integers $u,v,w$ satisfying $(u,v)=(w,v)=1$ such that the substitution
\[
 x = \frac{u}{v^2} \quad \quad y = \frac{w}{v^p}
\]
leads to
\begin{equation}
 u^p + \ell(-v^2)^p = w^2.
 \label{E:pp2}
\end{equation}
This shows that, with the terminology used in  \cite{IK2006}, $(u,-v^2,w)$ is a non-trivial primitive solution to the equation 
$x^p + \ell y^p = z^2$. In the notation of {\it{loc. cit.}}, we have 
\[
 a = 1, \quad b = \ell, \quad c = 1, \quad x = u, \quad y = - v^2, \quad z = w.
\]
It  is easy to check the conditions $(C_1)$, $(C_2)$ and $(C_4)$ of \cite[p. 122]{IK2006} 
and $(C_5)$ in \cite[p. 128]{IK2006} are satisfied; furthermore, if $w$ is odd, 
replacing it by $-w$ if necessary, we can also assume $w \equiv -1 \pmod{4}$, so that condition $(C_3)$ is also satisfied.

We now attach to the solution $(u,-v^2,w)$ the two Frey curves 
\[
 E_1 : Y^2 =X^3 + 2wX^2 + u^pX \quad \text{and} \quad E_2 : Y^2 =X^3 + 2wX^2 - \ell v^{2p}X.
\]
Their standard invariants are given by 
\[
 c_4(E_1) = 2^4(4w^2 - 3u^p), \quad c_6(E_1) = 2^6w(9u^p-8w^2), 
 \quad \Delta(E_1) = -2^6 \ell (uv)^{2p}
\]
and
\[
 c_4(E_2) = 2^4(4w^2 + 3\ell v^{2p}), \quad c_6(E_2) = -2^6w(9\ell v^{2p} + 8w^2), 
 \quad \Delta(E_2) = 2^6 \ell^2 (uv^4)^p.
\]
From Lemmas~2.1~and~2.4 in \cite{IK2006} we know that the models for $E_1$ and $E_2$ above are minimal away from 2 and have multiplicative reduction at all primes dividing $\ell uv$.

\begin{lemma} Let $u, v, w$ be as above. Then $u$ is odd.

Moreover, if $v$ is odd, then $u \equiv 1 \pmod{4}$ and $w$ is even.
\label{L:aux1}
\end{lemma}
\begin{proof} Suppose $u$ is even, so that $v$ is odd. From~\eqref{E:pp2} it follows 
that $w^2 \equiv -5 \pmod{8}$ which is impossible. Thus $u$ is odd.

Suppose $v$ is odd. If $u \equiv -1 \pmod{4}$, then $w^2 \equiv -1 - 5 \equiv 2 \pmod{4}$, a contradiction. We conclude that $u \equiv 1 \pmod{4}$ and since $uv$ is odd it follows
from~\eqref{E:pp2} that $w$ is even.
\end{proof}

Suppose now $p \geq 11$ and $\ell \neq p$. From \cite[Propositions~3.1 and~3.2]{IK2006} 
we know that the mod~$p$ Galois representations $\rhobar_{E_1,p}$ and $\rhobar_{E_2,p}$ 
attached to $E_1$ and $E_2$ are irreducible and of weight 2. 
\begin{lemma} Let $N(\rhobar_{E_1,p})$ and $N(\rhobar_{E_2,p})$ 
denote the Serre conductor of
$\rhobar_{E_1,p}$ and 
$\rhobar_{E_2,p}$, respectively. 
Then
 \begin{enumerate}
  \item if $v$ is even, we have  $N(\rhobar_{E_1,p}) = 2 \ell$;
  \item if $v$ is odd, we have  $N(\rhobar_{E_2,p}) = 32 \ell$.
 \end{enumerate}
 \label{L:Serrelevel}
\end{lemma}
\begin{proof} Part (1) follows directly from part (3.1) in \cite[Propositions~3.3]{IK2006}.
For part (2), we have $u \equiv 1 \pmod{4}$ by Lemma~\ref{L:aux1}, hence the result 
follows from part (3.3) in \cite[Propositions~3.4]{IK2006}.
\end{proof}

Suppose $p > f(\ell)$. As explained in \cite{IK2006}, it follows from 
modularity, level lowering and a standard argument that there are
elliptic curves $F_i/\Q$ of conductor $N(\rhobar_{E_i,p})$, having 
at least one rational 2-torsion point, such that 
$\rhobar_{E_i,p} \simeq \rhobar_{F_i,p}$. 

\begin{lemma} Suppose $p > f(\ell)$. Then $v$ is odd.
\label{L:vodd}
\end{lemma}
\begin{proof} Suppose $v$ is even. Then $N(\rhobar_{E_1,p}) = 2\ell$ by Lemma~\ref{L:Serrelevel}. 
In the case $\ell\geq 29$, since $\ell \equiv 5 \pmod{8}$ it follows from \cite{Ivorra} there 
are no elliptic curves over $\Q$ with at least one rational 2-torsion point. 
This assertion is also true for $\ell=5$, so we obtain a contradiction.
\end{proof}

\begin{lemma} Suppose $p > f(\ell)$ and let $N_2$ denote the conductor of the curve $E_2$. 
Then the model of $E_2$ is minimal at 2 and its invariants satisfy
\[
 (\vv_2(c_4),\vv_2(c_6),\vv_2(\Delta)) = (4, \geq 7, 6).
\]
Moreover, $\vv_2(N_2) = 5$ and the defect of semistability of $E_2$ is $e=8$.
\end{lemma}
\begin{proof} From Lemma~\ref{L:vodd} we know that $v$ is odd. Then, from Lemma~\ref{L:aux1}, 
we have $u \equiv 1 \pmod{4}$ and $w$ even; 
from the formulas for the invariants of $E_2$ and part (3.3) of 
\cite[Lemma~2.5]{IK2006} it follows that $\vv_2(N_2) = 5$ and
$(\vv_2(c_4),\vv_2(c_6),\vv_2(\Delta)) = (4, \geq 7, 6)$.
Now \cite{Kraus1990} gives $e=8$.
\end{proof}

We can finally prove the announced theorem.

\begin{proof}[Proof of Theorem~\ref{T:2pl}] 
The case $\ell=5$ is proved below in Theorem~\ref{T:hyperl=5} with a sharper bound on the exponent $p$.
Consequently, we will  assume  $\ell\geq 29.$
Suppose $p > f(\ell)$ and $(2/p) = -1$.
It follows from Lemmas~\ref{L:Serrelevel}~and~\ref{L:vodd} that there exits an 
elliptic curve $F/\Q$ of conductor $32\ell$ with at least one $2$-torsion point over~$\Q$ 
satisfying $\rhobar_{E_2,p} \simeq \rhobar_{F,p}$. Since $\ell - 1$ is a square and $\ell\geq 29$, it follows from \cite[Theorem~5]{Ivorra} that, after taking a quadratic twist of $E_2$ and $F$ if necessary, we can assume that $F$ is defined by
\[
 y^2 = x^3 + 2\sqrt{\ell-1}x^2 - x,
\]
whose invariants are
\[
 c_4(F) = 2^4(4(\ell-1) + 3), \quad c_6(F) = -2^6\sqrt{\ell-1}(\ell + 8), \quad \Delta(F) = 2^6 \ell.
\]
In particular, $(\vv_2(c_4(F)),\vv_2(c_6(F)),\vv_2(\Delta(F))) = (4, \geq 7, 6)$.
From the formulas for $c_4(E_2)$ and $c_4(F)$, using the fact that $\ell \equiv 5 \pmod{8}$, 
we conclude that 
\[
 \frac{c_4(E_2)}{2^4} \equiv -1 \pmod{4} \quad \text{ and }\quad  \frac{c_4(F)}{2^4} \equiv -1 \pmod{4},
\]
showing that both curves $E_2$ and $F$ satisfy case (a) of Theorem~\ref{T:mainWilde8II}, therefore 
$E_2[p]$ and $F[p]$ are symplectically isomorphic. 
Finally, we note that both curves have multiplicative reduction at $\ell$
with minimal discriminants satisfying
\[
 \vv_\ell(\Delta(E_2)) = 2 \quad \text{ and } \quad \vv_\ell(\Delta(F)) = 1. 
\]
Then Theorem~\ref{T:potMult} implies $(2/p) = 1$, giving a contradiction.
\end{proof} 

\addtocontents{toc}{\SkipTocEntry}
\subsection*{The cases $\ell = 3, 5$} We will now show that, by making the value of $\ell$ 
concrete, we can give much sharper bounds for the exponent~$p$.

\begin{theorem} \label{T:hyperl=5}
For all primes $p \geq 7$ satisfying $(2/p) = -1$, the set 
$C_{5,p}(\Q)$ is empty.
\end{theorem}
\begin{proof} Suppose $p \geq 7$ and $(2/p) = -1$; thus $p \geq 11$. 
Assume there is a point $(x,y)$ in $C_{5,p}(\Q)$ hence, following the argument above, there is a solution $(u,-v^2,w)$ to \eqref{E:pp2}. Since $p \geq 11$, the representations $\rhobar_{E_1,p}$ and $\rhobar_{E_2,p}$ are irreducible of weight 2 (by \cite[Propositions~3.1 and~3.2]{IK2006}). 

From Lemma~\ref{L:Serrelevel}, modularity and level lowering, it follows that when $v$ is even
$\rhobar_{E_1,p}$ arises from a newform of level 10, weight 2 and trivial character. There are no such forms, so we conclude that $v$ is odd. 

Now, again by Lemma~\ref{L:Serrelevel}, modularity and level lowering, we conclude that 
$\rhobar_{E_2,p}$ arises from a newform of level 160, weight 2 and trivial character; in this newspace of dimension 4 there are two newforms with rational coefficients and two conjugated forms with coefficients 
in~$\Q(\sqrt{2})$. The latter newforms $\ff$ satisfy $a_3(\ff) = \pm 2 \sqrt{2}$. 

Suppose $\rhobar_{E_2,p} \simeq \rhobar_{\ff,\fp}$, where $\ff$ has coefficients in $\Q(\sqrt{2})$
and $\fp \mid p$ in that field. Taking traces of Frobenius at~$3$ gives
\[ 
\tr(\rhobar_{E_2,p}(\Frob_3)) \equiv \tr(\rhobar_{\ff,\fp}(\Frob_3)) \pmod{\fp}. 
\]
Assume $3 \mid uv$. From \cite[Lemma~2.4]{IK2006} we know that $E_2$ 
has multiplicative reduction at~$3$. From the theory of the Tate curve 
it follows that $\tr(\rhobar_{\ff,\fp}(\Frob_3)) \equiv \pm (\chi_p(\Frob_3) +1) \pmod{p}$, 
therefore
\[
  \pm (3 + 1) \equiv a_3(f) \pmod{\fp} \iff 8 = a_3(f)^2 \equiv 16 \pmod{\fp},
\]
which is a contradiction with $p \geq 11$. Thus $3 \nmid uv$ and $E_2$ has good reduction at $3$. 
Now, taking again traces of Frobenius at~$3$ gives the congruence
\[ a_3(E) \equiv a_3(\ff) \equiv \pm 2 \sqrt{2} \pmod{\fp}.\] 
Since $a_3(E_2)$ is an integer 
satisfying the Weil bound $|a_3(E_2)| \le 2 \sqrt{3}$
this congruence also does not hold for $p \geq 11$. 
We conclude $\rhobar_{E_2,p} \not\simeq \rhobar_{\ff,\fp}$.

Suppose $\rhobar_{E_2,p} \simeq \rhobar_{\ff,\fp}$, where $\ff$ corresponds to one of the two
isogeny classes of elliptic curves with conductor $160$. Up to isogeny and quadratic twist by $-1$, we can assume that $\ff$ corresponds to the elliptic curve $E=160a1$. We have
\[
 E : y^2 = x^3 + x^2 - 6x + 4 
\]
and 
\[
 c_4(E) = 2^4 \cdot 19, \quad c_6(E) = -2^7 \cdot 41, \quad \Delta(E) = 2^6 \cdot 5
\]
From Theorem~\ref{T:potMult} applied at $\ell=5$ we conclude that $E_2[p]$ and $E[p]$ 
are symplectically isomorphic if and only if $(2/p)=1$. On the other hand, if $(2/p)=-1$ 
it follows, using Theorem~\ref{T:mainWilde8II} as in the proof of Theorem~\ref{T:2pl}, 
that $E_2[p]$ and $E[p]$ are symplectically isomorphic, a contradiction.
\end{proof}

We will now extend the study initiated in \cite[Section~8.3]{IK2006} 
by proving the following result.

\begin{theorem} For all primes $p \geq 7$ satisfying $(2/p) = -1$, the set 
$C_{3,p}(\Q)$ is empty.
\end{theorem}
\begin{proof}
From \cite[Section~8.3]{IK2006}, we have $u \equiv -1 \pmod{4}$, $v$ odd, $w$ even and we will use the Frey curve $E_1$. 
We check that
\[
 (\vv_2(c_4(E_1)),\vv_2(c_6(E_1)),\vv_2(\Delta(E_1))) = (4, \geq 7, 6), \quad
 \tilde{c}_4(E_1) \equiv -1 \pmod{4}
\]
and also $\Delta(E_1) = -2^6 \cdot 3 \cdot (uv)^{2p}$. We obtain that
$\rhobar_{E_1,p} \simeq \rhobar_{\ff,\fp}$, where $\ff$ is a newform of level~96 
with rational coefficients. Further, up to isogeny and quadratic 
twist by~$-1$, we can assume that $\rhobar_{E_1,p} \simeq \rhobar_{E,p}$ where $E=96a1$.
We have 
\[
 E : y^2 = x^3 + x^2 - 2x
\]
and 
\[
 c_4(E) = 2^4 \cdot 7, \quad c_6(E) = -2^7 \cdot 5, \quad \Delta(E) = 2^6 \cdot 3^2.
\]
From Theorem~\ref{T:potMult} applied at $\ell = 3$ we conclude that 
$E_1[p]$ and $E[p]$ are symplectically isomorphic 
if and only if $(2/p)=1$. Finally, when $(2/p)=-1$,
it follows from Theorem~\ref{T:mainWilde8II} that
$E_1[p]$ and $E[p]$ are symplectically isomorphic, a contradiction.
\end{proof}

\subsection{The curve $y^2 = x^p - 2\ell$}
We  assume in this section 
$$\ell\geq 29$$
and we show the following result.

\begin{theorem} \label{T:2plII}
Let $\ell \equiv 3 \pmod{8}$ be a prime such that $\ell-2$ is a square. Then, 
for all primes $p > f(\ell)$ satisfying $(2/p)=-1$, the set $C'_{\ell,p}(\Q)$ is empty.
\end{theorem}
Suppose there are $x, y \in \Q$ such that $y^2 = x^p - 2\ell$. 
Therefore, there are 
integers $u,v,w$ satisfying $(u,v)=(w,v)=1$ such that the substitution
\[
 x = \frac{u}{v^2} \quad \quad y = \frac{w}{v^p}
\]
leads to
\begin{equation}
 2\ell(-v^2)^p + u^p = w^2.
 \label{E:pp2II}
\end{equation}
This shows that $(-v^2,u,w)$ is a non-trivial primitive solution to the equation 
$2\ell x^p + y^p = z^2$. In the notation of \cite{IK2006}, we have 
\[
 a = 2\ell, \quad b = 1, \quad c = 1, \quad x = -v^2, \quad y = u, \quad z = w.
\]
It  is easy to check the conditions $(C_1)$, $(C_2)$ and $(C_4)$ of \cite[p. 122]{IK2006} 
and $(C_5)$ in \cite[p. 128]{IK2006} are satisfied; furthermore, if $w$ is odd, 
replacing it by $-w$ if necessary, we can also assume $w \equiv -1 \pmod{4}$, so that condition $(C_3)$ is also satisfied.

We consider the Frey curve attached $(-v^2,u,w)$ defined by
\[
 E_3 : Y^2 = X^3 + 2wX^2 + u^{p}X 
\]
whose standard invariants are given by 
\[
 c_4(E_3) = 2^4(4w^2 - 3 u^{p}), \quad c_6(E_3) = 2^6 w (9 u^{p} - 8w^2), 
 \quad \Delta(E_3) = -2^7 \ell (uv)^{2p}.
\]
From \cite[Lemma~2.4]{IK2006} the model for $E_3$ is minimal away from 2 
and has multiplicative reduction at all primes dividing $\ell uv$.
The following lemma follows easily from~\eqref{E:pp2II}.

\begin{lemma} Let $u, v, w$ be as above. Then $uw$ is odd.

Moreover, if $v$ is even, then $u \equiv 1 \pmod{4}$ and if $v$ is odd, 
we have $u \equiv -1 \pmod{4}$.
\label{L:aux2}
\end{lemma}

Suppose now $p \geq 11$ and $\ell \neq p$. From \cite[Propositions~3.1 and~3.2]{IK2006} 
we know that the mod~$p$ Galois representation $\rhobar_{E_3,p}$ 
attached to $E_3$ is irreducible and of weight 2. 
\begin{lemma} Let $N(\rhobar_{E_3,p})$ 
denote the Serre conductor of $\rhobar_{E_3,p}$.
Then
 \begin{enumerate}
  \item if $v$ is even, we have $N(\rhobar_{E_3,p}) = 2 \ell$;
  \item if $v$ is odd, we have $N(\rhobar_{E_3,p}) = 2^7 \ell$.
 \end{enumerate}
 \label{L:SerrelevelII}
\end{lemma}
\begin{proof} This follows from parts (2.1) and (2.2) of 
\cite[Propositions~3.4]{IK2006}.
\end{proof}

Suppose $p > f(\ell)$. As explained in \cite{IK2006}, it follows from 
modularity, level lowering and a standard argument that there is
an elliptic curves $F/\Q$ of conductor $N(\rhobar_{E_3,p})$, having 
at least one rational 2-torsion point, such that 
$\rhobar_{E_3,p} \simeq \rhobar_{F,p}$. 

\begin{lemma} Suppose $p > f(\ell)$. Then $v$ is odd.
\label{L:voddII}
\end{lemma}
\begin{proof} Suppose $v$ is even. Then $N(\rhobar_{E_3,p}) = 2\ell$ by Lemma~\ref{L:SerrelevelII}. 
Since $\ell \equiv 3 \pmod{8}$, it follows from \cite{Ivorra} there 
are no elliptic curves over $\Q$ with at least one rational 2-torsion point, 
a contradiction.
\end{proof}

\begin{lemma} \label{L:recall2}
Suppose $p > f(\ell)$ and let $N_2$ denote the conductor of the curve $E_3$. 
Then the model of $E_3$ is minimal at 2 and its invariants satisfy
\[
 (\vv_2(c_4),\vv_2(c_6),\vv_2(\Delta)) = (4, 6, 7).
\]
Moreover, $\vv_2(N_2) = 7$ and the defect of semistability of $E_3$ is $e=24$.
\end{lemma}
\begin{proof} From Lemma~\ref{L:vodd} we know that $v$ is odd and from Lemma~\ref{L:aux2} 
we have also $u,w$ odd. From the formulas for the invariants of $E_3$ and \cite{Kraus1990}
it follows that $(\vv_2(c_4),\vv_2(c_6),\vv_2(\Delta)) = (4, 6, 7)$,
$\vv_2(N_2) = 7$ and $e=24$.
\end{proof}

We can now prove our  result.

\begin{proof}[Proof of Theorem~\ref{T:2plII}]
Suppose $p > f(\ell)$ and $(2/p) = -1$.
It follows from Lemmas~\ref{L:SerrelevelII}~and~\ref{L:voddII} that there exits an 
elliptic curve $F/\Q$ of conductor $128\ell$ with at least one $2$-torsion point over~$\Q$ 
satisfying $\rhobar_{E_3,p} \simeq \rhobar_{F,p}$. Since $\ell - 2$ is a square and $\ell\geq 29$, it follows from \cite[Theorem~7]{Ivorra} that, after twisting  both $E_3$ and $F$ by $-1$ if necessary, we can assume that $F$ is one of the two curves
\[
 F_1 : y^2 = x^3 + 2\sqrt{\ell-2}x^2 + \ell x \quad \text{ or } \quad
 F_2 : y^2 = x^3 + 2\sqrt{\ell-2}x^2 - 2 x.
\]
whose semistability defect is $e(F_1) = e(F_2) = 24$ and minimal discriminants
\[
 \Delta(F_1) = -2^7 \ell^2 \quad \text{ and } \quad \Delta(F_2) = 2^8 \ell.
\]
From Lemma~\ref{L:recall2} we have $\Delta(E_3) = -2^7 \ell (uv)^{2p}$ is minimal.

Suppose $F=F_1$; from Theorem~\ref{T:Wilde24} part (2) we conclude that $E_3[p]$ and 
$F[p]$ are symplectically isomorphic. Now, applying 
Theorem~\ref{T:potMult} at $\ell$ forces $(2/p) = 1$, a contradiction.

Suppose $F=F_2$; from Theorem~\ref{T:potMult} at $\ell$ it follows 
that $E_3[p]$ and $F[p]$ are symplectically isomorphic. 
On the other hand, Theorem~\ref{T:Wilde24} part (2) implies $E_3[p]$ and 
$F[p]$ are anti-symplectically isomorphic, a contradiction. 
\end{proof} 

\addtocontents{toc}{\SkipTocEntry}


\begin{thebibliography}{10}

\bibitem{BDMTV}
Jennifer~S. Balakrishnan, Netan Dogra, J.~Steffen M\"uller, Jan Tuitman, and
  Jan Vonk.
\newblock Explicit {C}habauty-{K}im for the split Cartan modular curve of level 13. 
\newblock {\em Annals of Math.} 189(3):885--944, 2019.

\bibitem{BenBruFre}
Michael~A. Bennett, Carmen Bruni, and Nuno Freitas.
\newblock Sums of two cubes as twisted perfect powers, revisited.
\newblock {\em Algebra Number Theory}, 12(4):959--999, 2018.



\bibitem{BPR2013}
Yuri Bilu, Pierre Parent, and Marusia Rebolledo.
\newblock Rational points on {$X^+_0(p^r)$}.
\newblock {\em Ann. Inst. Fourier (Grenoble)}, 63(3):957--984, 2013.

\bibitem{MAGMA}
Wieb Bosma, John Cannon, and Catherine Playoust.
\newblock The {M}agma algebra system. {I}. {T}he user language.
\newblock {\em J. Symbolic Comput.}, 24(3-4):235--265, 1997.
\newblock Computational algebra and number theory (London, 1993).

\bibitem{BourbAlgII}
Nicolas Bourbaki.
\newblock {\em Algebra {II}. {C}hapters 4--7}.
\newblock Elements of Mathematics (Berlin). Springer-Verlag, Berlin, 2003.
\newblock Translated from the 1981 French edition by P. M. Cohn and J. Howie,
  Reprint of the 1990 English edition [Springer, Berlin; MR1080964
  (91h:00003)].

\bibitem{CaliThesis}
{\'E}lie Cali.
\newblock {\em Points de torsion des courbes elliptiques et quartiques de
  Fermat.} Th\`ese de Doctorat de l'Universit\'e de Paris 6, 2005.
  \newblock \url{https://www.imj-prg.fr/theses//pdf/elie_cali.pdf}
  
\bibitem{CaliKraus}
\'{E}lie Cali and Alain Kraus.
\newblock Sur la {$p$}-diff\'{e}rente du corps des points de {$l$}-torsion des
  courbes elliptiques, {$l\ne p$}.
\newblock {\em Acta Arith.}, 104(1):1--21, 2002.

\bibitem{Centeleghe}
Tommaso~Giorgio Centeleghe.
\newblock Integral {T}ate modules and splitting of primes in torsion fields of
  elliptic curves.
\newblock {\em Int. J. Number Theory}, 12(1):237--248, 2016.

\bibitem{Ccode}
Tommaso~Giorgio Centeleghe and Panagiotis Tsaknias.
\newblock {I}ntegral {F}robenius, {\tt {M}agma} package, available as ancillary files at
\newblock \url{https://arxiv.org/abs/1201.2124}.

\bibitem{CremonaF}
 John Cremona and Nuno Freitas.
\newblock Global methods for the symplectic type of congruences between elliptic curves,
\newblock {\em Rev. Mat. Iberoamericana}, 38(1):1--32, 2022.

\bibitem{DahmenPhD}
Sander\thinspace{R}. Dahmen.
\newblock {\em Classical and modular methods applied to {D}iophantine
  equations.}
\newblock \url{https://dspace.library.uu.nl/handle/1874/29640}.

\bibitem{darmon-SerreConjs}
Henri Darmon.
\newblock Serre's conjectures.
\newblock In {\em Seminar on {F}ermat's {L}ast {T}heorem ({T}oronto, {ON},
  1993--1994)}, volume~17 of {\em CMS Conf. Proc.}, pages 135--153. Amer. Math.
  Soc., Providence, RI, 1995.

\bibitem{DFV}
Lassina Demb\'el\'e, Nuno Freitas, and John Voight.
\newblock On {G}alois inertial types of elliptic curves over $\mathbb{Q}_\ell$
  (preprint).
\newblock \url{https://www.icmat.es/miembros/nfreitas/preprint6.pdf}.

\bibitem{Diamond}
Fred Diamond.
\newblock An extension of {W}iles' results.
\newblock In {\em Modular forms and {F}ermat's last theorem ({B}oston, {MA},
  1995)}, pages 475--489. Springer, New York, 1997.

\bibitem{DiamondShurman}
Fred Diamond and Jerry Shurman.
\newblock {\em A first course in modular forms}, volume 228 of {\em Graduate
  Texts in Mathematics}.
\newblock Springer-Verlag, New York, 2005.

\bibitem{DDRoot}
Tim Dokchitser and Vladimir Dokchitser.
\newblock Root numbers of elliptic curves in residue characteristic 2.
\newblock {\em Bull. Lond. Math. Soc.}, 40(3):516--524, 2008.

\bibitem{DDKod}
Tim Dokchitser and Vladimir Dokchitser.
\newblock A remark on {T}ate's algorithm and {K}odaira types.
\newblock {\em Acta Arith.}, 160(1):95--100, 2013.

\bibitem{DD2015}
Tim Dokchitser and Vladimir Dokchitser.
\newblock Local invariants of isogenous elliptic curves.
\newblock {\em Trans. Amer. Math. Soc.}, 367(6):4339--4358, 2015.

\bibitem{FisherList}
Tom Fisher.
\newblock A table of $11$-congruent elliptic curves over the rationals.
\newblock \url{https://www.dpmms.cam.ac.uk/~taf1000/papers/congr-11}.

\bibitem{Fisher}
Tom Fisher.
\newblock On families of 7- and 11-congruent elliptic curves.
\newblock {\em LMS J. Comput. Math.}, 17(1):536--564, 2014.

\bibitem{Fisher13cong}
Tom Fisher.
\newblock On families of $13$-congruent elliptic curves (preprint).
\newblock \url{https://arxiv.org/abs/1912.10777}

\bibitem{Fisher17cong}
Tom Fisher.
\newblock On pairs of 17-congruent elliptic curves.
\newblock \url{https://arxiv.org/abs/2106.02033}

\bibitem{F33p}
Nuno Freitas.
\newblock On the {F}ermat-type equation {$x^3+y^3=z^p$}.
\newblock {\em Comment. Math. Helv.}, 91(2):295--304, 2016.

\bibitem{FK2}
Nuno Freitas and Alain Kraus.
\newblock On the degree of the $p$-torsion field of elliptic curves over
  $\mathbb{Q}_\ell$ for $\ell \neq p$.
\newblock {\em Acta Arith.}, 195(1):13--55, 2020.

\bibitem{FK1}
Nuno Freitas and Alain Kraus.
\newblock An application of the symplectic argument to some {F}ermat-type
  equations.
\newblock {\em C. R. Math. Acad. Sci. Paris}, 354(8):751--755, 2016.

\bibitem{FNS23n}
Nuno Freitas, Bartosz Naskr{\k e}cki, and Michael Stoll.
\newblock The generalized Fermat equation with exponents $2,3,n$. 
\newblock {\em Compositio Math.}, 156(1):77--113, 2020. 

\bibitem{FS3}
Nuno Freitas and Samir Siksek.
\newblock Fermat's last theorem over some small real quadratic fields.
\newblock {\em Algebra Number Theory}, 9(4):875--895, 2015.

\bibitem{FukudaWeilPairing}
Kohei Fukuda and Sho Yoshikawa.
\newblock The 12th roots of the discriminant of an elliptic curve and the
  torsion points.
\newblock {\em Int. J. Number Theory}, 13(4):1037--1060, 2017.

\bibitem{GL}
Enrique Gonz\'{a}lez-Jim\'{e}nez and \'{A}lvaro Lozano-Robledo.
\newblock Elliptic curves with abelian division fields.
\newblock {\em Math. Z.}, 283(3-4):835--859, 2016.

\bibitem{GRSS}
R.~Greenberg, K.~Rubin, A.~Silverberg, and M.~Stoll.
\newblock On elliptic curves with an isogeny of degree 7.
\newblock {\em Amer. J. Math.}, 136(1):77--109, 2014.

\bibitem{HK2002}
Emmanuel Halberstadt and Alain Kraus.
\newblock Courbes de {F}ermat: r\'{e}sultats et probl\`emes.
\newblock {\em J. Reine Angew. Math.}, 548:167--234, 2002.

\bibitem{HK2003}
Emmanuel Halberstadt and Alain Kraus.
\newblock Sur la courbe modulaire {$X_E(7)$}.
\newblock {\em Experiment. Math.}, 12(1):27--40, 2003.

\bibitem{Husemoller}
Dale Husem\"{o}ller.
\newblock {\em Elliptic curves}, volume 111 of {\em Graduate Texts in
  Mathematics}.
\newblock Springer-Verlag, New York, second edition, 2004.
\newblock With appendices by Otto Forster, Ruth Lawrence and Stefan Theisen.

\bibitem{IK2006}
W.~Ivorra and A.~Kraus.
\newblock Quelques r\'{e}sultats sur les \'{e}quations {$ax^p+by^p=cz^2$}.
\newblock {\em Canad. J. Math.}, 58(1):115--153, 2006.

\bibitem{Ivorra}
Wilfrid Ivorra.
\newblock Courbes elliptiques sur {$\Bbb Q$}, ayant un point d'ordre 2
  rationnel sur {$\Bbb Q$}, de conducteur {$2^Np$}.
\newblock {\em Dissertationes Math. (Rozprawy Mat.)}, 429:55, 2004.

\bibitem{KaniRizzo}
E.\ Kani and O.\ G.\ Rizzo.
\newblock Mazur's question on mod 11 representations of elliptic curves
  (preprint).
\newblock \url{https://mast.queensu.ca/~kani/papers/mazur8pl.pdf}.

\bibitem{KaniSchanz}
E.~Kani and W.~Schanz.
\newblock Modular diagonal quotient surfaces.
\newblock {\em Math. Z.}, 227(2):337--366, 1998.

\bibitem{KO}
A.~Kraus and J.~Oesterl\'{e}.
\newblock Sur une question de {B}. {M}azur.
\newblock {\em Math. Ann.}, 293(2):259--275, 1992.

\bibitem{Kraus1990}
Alain Kraus.
\newblock Sur le d\'{e}faut de semi-stabilit\'{e} des courbes elliptiques \`a
  r\'{e}duction additive.
\newblock {\em Manuscripta Math.}, 69(4):353--385, 1990.

\bibitem{Kraus1996}
Alain Kraus.
\newblock Sur les modules des points de {$7$}-torsion d'une famille de courbes
  elliptiques.
\newblock {\em Ann. Inst. Fourier (Grenoble)}, 46(4):899--907, 1996.

\bibitem{LangEF}
Serge Lang.
\newblock {\em Elliptic functions}, volume 112 of {\em Graduate Texts in
  Mathematics}.
\newblock Springer-Verlag, New York, second edition, 1987.
\newblock With an appendix by J. Tate.

\bibitem{lmfdb}
The {LMFDB Collaboration}.
\newblock The {L}-functions and modular forms database.
\newblock \url{http://www.lmfdb.org}, 2013.

\bibitem{Mazur}
B.~Mazur.
\newblock Rational isogenies of prime degree (with an appendix by {D}.
  {G}oldfeld).
\newblock {\em Invent. Math.}, 44(2):129--162, 1978.

\bibitem{pap}
Ioannis Papadopoulos.
\newblock Sur la classification de {N}\'eron des courbes elliptiques en
  caract\'eristique r\'esiduelle {$2$} et {$3$}.
\newblock {\em J. Number Theory}, 44(2):119--152, 1993.

\bibitem{PSS}
Bjorn Poonen, Edward~F. Schaefer, and Michael Stoll.
\newblock Twists of {$X(7)$} and primitive solutions to {$x^2+y^3=z^7$}.
\newblock {\em Duke Math. J.}, 137(1):103--158, 2007.

\bibitem{roh94}
David~E. Rohrlich.
\newblock Elliptic curves and the {W}eil-{D}eligne group.
\newblock In {\em Elliptic curves and related topics}, volume~4 of {\em CRM
  Proc. Lecture Notes}, pages 125--157. Amer. Math. Soc., Providence, RI, 1994.

\bibitem{Serre72}
Jean-Pierre Serre.
\newblock Propri\'et\'es galoisiennes des points d'ordre fini des courbes
  elliptiques.
\newblock {\em Invent. Math.}, 15(4):259--331, 1972.

\bibitem{SerreBook}
Jean-Pierre Serre.
\newblock {\em Abelian {$l$}-adic representations and elliptic curves},
  volume~7 of {\em Research Notes in Mathematics}.
\newblock A K Peters, Ltd., Wellesley, MA, 1989.
\newblock With the collaboration of Willem Kuyk and John Labute, Revised
  reprint of the 1968 original.

\bibitem{ST1968}
Jean-Pierre Serre and John Tate.
\newblock Good reduction of abelian varieties.
\newblock {\em Ann. of Math. (2)}, 88:492--517, 1968.

\bibitem{SilvermanII}
Joseph~H. Silverman.
\newblock {\em Advanced topics in the arithmetic of elliptic curves}, volume
  151 of {\em Graduate Texts in Mathematics}.
\newblock Springer-Verlag, New York, 1994.

\bibitem{SilvermanI}
Joseph~H. Silverman.
\newblock {\em The arithmetic of elliptic curves}, volume 106 of {\em Graduate
  Texts in Mathematics}.
\newblock Springer, Dordrecht, second edition, 2009.

\bibitem{Wiles}
Andrew Wiles.
\newblock Modular elliptic curves and {F}ermat's {L}ast {T}heorem.
\newblock {\em Annals of Mathematics}, 144:443--551, 1995.

\end{thebibliography}
\end{document}